\documentclass[a4paper, 10pt, twoside]{article}

\usepackage{amsmath, amscd, amsfonts, amssymb, amsthm, latexsym, url, color, todonotes, bm, framed, rotating, enumerate, cancel, ulem} 
\usepackage{graphicx}
\usepackage[left=1in, right=1in, top=1.2in, bottom=1in, includefoot, headheight=13.6pt]{geometry}
\usepackage{mathtools}
\input{xy}
\xyoption{all}

\usepackage[colorlinks,breaklinks=true]{hyperref}
\usepackage[figure,table]{hypcap}
\hypersetup{
	bookmarksnumbered,
	pdfstartview={FitH},
	citecolor={black},
	linkcolor={black},
	urlcolor={black},
	pdfpagemode={UseOutlines}
}
\makeatletter
\newcommand\org@hypertarget{}
\let\org@hypertarget\hypertarget
\renewcommand\hypertarget[2]{%
  \Hy@raisedlink{\org@hypertarget{#1}{}}#2%
} 
\makeatother 

\usepackage{relsize} 
\usepackage[bbgreekl]{mathbbol}
\DeclareSymbolFontAlphabet{\mathbb}{AMSb} 
\DeclareSymbolFontAlphabet{\mathbbl}{bbold}

\newcommand{\Prism}{{ \mathbbl{\Delta}}}

\newtheorem{theorem}{Theorem}[section]
\newtheorem{lemma}[theorem]{Lemma}
\newtheorem{corollary}[theorem]{Corollary}
\newtheorem{proposition}[theorem]{Proposition}

\theoremstyle{definition}
\newtheorem{definition}[theorem]{Definition}
\newtheorem{remark}[theorem]{Remark}
\newtheorem{example}[theorem]{Example}

\newcommand{\xysquare}[8]{
\[\xymatrix{
#1 \ar@{#5}[r] \ar@{#6}[d] & #2 \ar@{#7}[d]\\
#3 \ar@{#8}[r] & #4
}\]
}

\DeclareMathOperator*{\projlimf}{``\varprojlim''}

\newcommand{\al}{\alpha}
\newcommand{\bb}{\mathbb}

\newcommand{\blob}{\bullet}

\newcommand{\bx}{{\scalebox{0.5}{$\square$}}}

\newcommand{\comment}[1]{}

\newcommand{\ep}{\varepsilon}

\newcommand{\into}{\hookrightarrow}

\newcommand{\isoto}{\stackrel{\simeq}{\to}}
\newcommand{\Isoto}{\stackrel{\simeq}{\longrightarrow}}

\newcommand{\onto}{\twoheadrightarrow}
\newcommand{\op}{\operatorname}
\newcommand{\ot}{\leftarrow}
\newcommand{\pid}[1]{\langle #1 \rangle}
\renewcommand{\phi}{\varphi}
\newcommand{\q}{\textrm{q}}
\newcommand{\quis}{\stackrel{\sim}{\to}}
\newcommand{\res}{\overline}
\newcommand{\roi}{\mathcal{O}}
\newcommand{\roiA}{\mathcal{O}\!A}

\newcommand{\sub}[1]{{\mbox{\rm \scriptsize #1}}}
\renewcommand{\inf}{\sub{inf}}
\newcommand{\crys}{\sub{crys}}

\newcommand{\To}{\longrightarrow}
\newcommand{\ul}[1]{\underline{#1}}
\newcommand{\ol}[1]{\overline{#1}}

\newcommand{\xto}{\xrightarrow}
\newcommand{\xot}{\xleftarrow}

\DeclareMathOperator*{\quiss}{\stackrel{\sim}{\longrightarrow}}
\DeclareMathOperator*{\quissleft}{\stackrel{\sim}{\longleftarrow}}

\newcommand \id {{\rm id}}
\newcommand \ev {{\rm ev}}

\newcommand \gr {\text{\rm gr}}

\def \txb {\color{black}}

\DeclareMathOperator{\CRYS}{CRYS}
\DeclareMathOperator{\CR}{CR}

\renewcommand{\cal}{\mathcal}
\renewcommand{\hat}{\widehat}
\renewcommand{\frak}{\mathfrak}
\newcommand{\indlim}{\varinjlim}
\renewcommand{\tilde}{\widetilde}
\renewcommand{\Im}{\operatorname{Im}}
\renewcommand{\ker}{\operatorname{Ker}}
\renewcommand{\projlim}{\varprojlim}

\DeclareMathOperator{\Aut}{Aut}

\DeclareMathOperator{\dlog}{dlog}
\DeclareMathOperator{\End}{End}

\DeclareMathOperator{\Fil}{Fil}
\DeclareMathOperator{\Frac}{Frac}
\DeclareMathOperator{\Gal}{Gal}
\DeclareMathOperator{\GL}{GL}
\DeclareMathOperator{\Hom}{Hom}
\DeclareMathOperator{\HIG}{HIG}

\DeclareMathOperator{\MIC}{MIC}

\DeclareMathOperator{\Spec}{Spec}
\DeclareMathOperator{\Spa}{Spa}
\DeclareMathOperator{\Spf}{Spf}
\DeclareMathOperator{\Sp}{Sp}

\DeclareMathOperator{\BKF}{BKF}
\DeclareMathOperator{\Rep}{Rep}

\newcommand{\Strat} {\text{\rm Strat}}

\newcommand{\categ}[2]{\left\{\parbox{#1}{\rm #2}\right\}}

\newcommand{\TA}{T\!\!A}

\newcommand{\lbar}[1]{\begin{leftbar}#1\end{leftbar}}

\newcommand{\dotimes}{\otimes^{\bb L}}
\newcommand{\xTo}[1]{\stackrel{#1}{\To}}

\usepackage{fancyhdr}

\usepackage{sectsty}
\sectionfont{\Large\sc\centering}
\chapterfont{\large\sc\centering}
\chaptertitlefont{\LARGE\centering}
\partfont{\centering}

\begin{document}
\itemsep0pt

\title{\vspace{-1cm}Generalised representations as  q-connections in integral $p$-adic Hodge theory}

\author{Matthew Morrow and Takeshi Tsuji}

\date{}

\maketitle

\begin{abstract}
We relate various approaches to coefficient systems in relative integral $p$-adic Hodge theory, working in the geometric context over the ring of integers of a perfectoid field. These include small generalised representations over $A_{\text{inf}}$ inspired by Faltings, modules with q-connection in the sense of q-de Rham cohomology, crystals on the prismatic site of Bhatt--Scholze, and q-deformations of Higgs bundles.
\end{abstract}

\tableofcontents

\section*{Introduction}
The goal of this paper is to explore and relate different approaches to coefficient systems in relative integral $p$-adic Hodge theory: the generalised representations of Faltings \cite{Faltings2005, Faltings2011}, which have been systematically studied by Abbes--Gros--Tsuji \cite{AbbesGrosTsuji2016} and Tsuji \cite{Tsuji_simons}; modules with $q$-connection in the sense of $q$-de Rham cohomology \cite{Scholze2017}; and crystals on the prismatic site \cite{BhattScholze2019} of Bhatt--Scholze, or equivalently modules with ``q-Higgs field'' in our context. The equivalence between modules with q-connections and modules with q-Higgs field applies to a range of situations arising in the theory of q-de Rham cohomology and may be viewed as a local, q-deformed Simpson correspondence.

The set-up throughout the introduction is as follows, though we often work in greater generality in the body of the paper when we are in the context of q-de Rham cohomology. Let $\roi$ be the ring of integers of a characteristic zero perfectoid field $C$ containing all $p$-power roots of unity; we fix a choice of $p$-power roots of unity and use it to define the elements $\ep:=(1,\zeta_p,\zeta_{p^2},\dots)\in\roi^\flat$ and $\mu:=[\ep]-1\in A_\inf=W(\roi^\flat)$; then $\xi:=\mu/\phi^{-1}(\mu)$ is a generator of the kernel of Fontaine's map $\theta:A_\inf\to\roi$, and $\tilde\xi:=\phi(\xi)$ is a generator of the kernel of $\theta\circ\phi^{-1}$.

In the first part of the introduction, corresponding to Sections \ref{section_small_reps}--\ref{section_crystalline} of the paper, we focus on the local theory; so let $R$ be a $p$-adically complete, formally smooth $\roi$-algebra which is moreover small in the sense that there exists a formally \'etale map $\roi\pid{\ul T^{\pm1}}=\roi\pid{T_1^{\pm1},\dots,T_d^{\pm1}}\to R$ (called a ``framing''), a choice of which we fix. The following diagram, in which all arrows will be shown to be equivalences of categories, serves as a road map to the various categories of coefficients which we will study:
\begin{equation}\xymatrix@C=1.2cm{
\BKF(\Spf R,\phi)\ar[r]^-{\Gamma(U_\infty,-)}&\Rep_\Gamma^{<\mu}(A_\inf(R_\infty),\phi)&\\
&\Rep_\Gamma^\mu(A_\inf(R_\infty),\phi)\ar@{^(->}[u]&&\\
&\Rep_\Gamma^\mu(A_\inf^\bx(R),\phi)\ar[u]^{-\otimes_{A_\sub{inf}^\bx(R)}A_\sub{inf}(R_\infty)}\ar[r]&\mathrm{qMIC}(A_\inf^\bx(R),\phi)&\\
&&\mathrm{qHIG}(A_\inf^{\bx}(R)^{(1)},\phi)\ar[u]_{(F,F_\Omega)^*}& \op{F-CR}_{\Prism}(R^{(1)}/(A_{\inf},\tilde\xi))\ar[l]_-{\op{ev}_{A_\inf^{\bx}(R)^{(1)}}}
}\label{eqn_intro}\end{equation}

Firstly, let $R_\infty$ be the $p$-adic completion of $\indlim_jR\otimes_{\roi\pid{\ul T^{\pm1}}}\roi\pid{\ul T^{\pm1/p^j}}$, where the base change is taken with respect to the choice of framing. The $R$-algebra $R_\infty$ is perfectoid in the sense of \cite[\S3]{BhattMorrowScholze1} and is equipped with its usual $R$-linear continuous action by $\Gamma:=\bb Z_p(1)^d$; by naturality this action extends to Fontaine's ring $A_\inf(R_\infty)=W(R_\infty^\flat)$. Inspired by the smallness condition on generalised representations which appears in Faltings' work on the $p$-adic Simpson correspondence, we study the following categories of generalised representations:

\begin{definition}\label{definition_intro_gen_rep}
A {\em generalised representation} over $A_\inf(R_\infty)$ is a finite projective $A_\inf(R_\infty)$-module $M$ equipped with a continuous, semi-linear action by $\Gamma$. The category of such generalised representations is denoted by $\Rep_\Gamma(A_\inf(R_\infty))$.

Given $\al\in A_\inf$, we say that $M$ is {\em trivial modulo $\al$} (or {\em $\al$-small}) if the invariants $(M/\al)^\Gamma$ are a finite projective $(A_\inf(R_\infty)/\al)^\Gamma$-module and the canonical map $(M/\al)^\Gamma\otimes_{(A_\inf(R_\infty)/\al)^\Gamma}A_\inf(R_\infty)/\al\to M/\al$ is an isomorphism. 
We are particularly interested in this smallness condition with respect to the element $\al=\mu$, and so we denote by $\Rep_\Gamma^\mu(A_\inf(R_\infty))$ the category of generalised representations which are trivial modulo $\mu$. Also let \begin{equation}\Rep_\Gamma^{<\mu}(A_\inf(R_\infty))\supseteq\Rep_\Gamma^\mu(A_\inf(R_\infty))\label{intro_<mu}\end{equation}
be those generalised representations which satisfy the (a priori weaker) condition of being trivial modulo $\mu/\phi^{-r}(\mu)$ for all $r\ge1$.
\end{definition}

\begin{remark}[Frobenius structures]
When we speak of a {\em Frobenius structure} on such a generalised representation $M$, we mean an isomorphism $\phi_M:(\phi^*M)[\tfrac1{\tilde\xi}]\isoto M[\tfrac1{\tilde\xi}]$ of $A_\inf(R_\infty)[\tfrac1{\tilde\xi}]$-modules compatible with the $\Gamma$-actions, where $\phi^*$ denotes base change along the Frobenius automorphism of $A_\inf(R_\infty)$. We denote by $\Rep_\Gamma(A_\inf(R_\infty),\phi)$ the category of generalised representations equipped with a Frobenius structure, and similarly in the presence of a triviality condition. We will implicitly use analogous notation to add Frobenius structures to other categories of representations.
\end{remark}

Our main class of generalised representations of interest is $\Rep_\Gamma^\mu(A_\inf(R_\infty),\phi)$. The equivalences of diagram (\ref{eqn_intro}) will ultimately offer two framing-independent approaches to this category, both of which may be globalised: either as our relative Breuil--Kisin--Fargues modules, or as prismatic F-crystals. However, in the local context a framing-independent approach is already provided by the almost purity theorem, assuming $\Spec(R/pR)$ is connected. Namely, letting $\res R$ be the $p$-adic completion of the integral closure of $R$ inside the maximal ind-\'etale extension of $R[\tfrac1p]$ inside some fixed algebraic closure of $\Frac R$, so that $\Delta:=\pi_1^\sub{\'et}(R[\tfrac1p])$ acts on $\res R$, then the base change functor \[-\otimes_{A_\inf(R_\infty)}A_\inf(\res R):\Rep_\Gamma^{\mu}(A_\inf(R_\infty))\To \Rep_\Delta^{\mu}(A_\inf(\res R))\] is an equivalence of categories (and similarly with Frobenius structures). Here the category of generalised representations $\Rep_\Delta^{\mu}(A_\inf(\res R))$ is defined analogously to Definition \ref{definition_intro_gen_rep}, and we refer to Remark \ref{remark_Delta} for further details.

\begin{example}\label{intro_examples}
In \S\ref{ss_examples} we give the following examples of generalised representations.
\begin{enumerate}
\item Dieudonn\'e theory: there is a fully faithful embedding of the category of $p$-divisible groups over $R$ into $\Rep_\Gamma^\mu(A_\inf(R_\infty),\phi)$, whose essential image consists of those $(M,\phi_M)$ such that $M\subseteq\phi_M(M)\subseteq\tfrac{1}{\tilde\xi}M$. This may either be deduced from the general prismatic Dieudonn\'e theory of J.~Anschutz and A.-C.~Le Bras \cite{AnschutzLeBras2019}, using our comparison to prismatic crystals, or else may be proved (when $p>2$) by combining Lau's Dieudonn\'e theory \cite{Lau2018} over complete intersection semiperfect rings, Grothendieck--Messing deformation theory, and our filtered version of Theorem \ref{thm_crystalline_intro}.
\item Relative prismatic cohomology: even before establishing any relation to prismatic crystals, we show that the prismatic cohomologies \[H^i_{A_\inf}(\cal Y/R):=H^i_\Prism(\cal Y\times_RR_\infty/(A_\inf(R_\infty),\tilde\xi)),\] equipped with their natural $\Gamma$-action, define objects of $\Rep_\Gamma^\mu(A_\inf(R_\infty),\phi)$ for all $i\ge0$; here $\cal Y$ is a proper smooth $p$-adic formal $R$-scheme whose relative de Rham cohomology groups $H^*_\sub{dR}(\cal Y/R)$ are assumed to be finite projective $R$-modules.
\item Relative Fontaine--Laffaille modules: the second author has shown that Faltings' relative Fontaine--Laffaille modules give rise to objects of $\Rep_\Gamma^\mu(A_\inf(\res R),\phi)$ \cite{Tsuji_simons}.
\end{enumerate}
\end{example}

Unfortunately the ring $A_\inf(R_\infty)$ is rather unwieldy in practice and we are only able to study generalised representations over it when they descend to its subring $A_\inf^\bx(R)\subseteq A_\inf(R_\infty)$; here {\em the framed period ring} $A_\inf^\bx(R)$ is the unique formally smooth $A_\inf$-algebra lifting (along Fontaine's map $\theta:A_\inf\to\roi$) the formally smooth $\roi$-algebra $R$ in a manner compatible with the chosen framing (see \S\ref{ss_framed} for more details). The $\Gamma$-action on $A_\inf(R_\infty)$ restricts to $A_\inf^\bx(R)$ and one may mimic Definition \ref{definition_intro_gen_rep} to introduce the category $\Rep_\Gamma^\mu(A_\inf^\bx(R))$ of generalised representations over $A_\inf^\bx(R)$ which are trivial modulo $\mu$; the fact that $\Gamma$ acts as the identity on $A_\inf^\bx(R)/\mu$ makes these representations much more manageable than those over $A_\inf(R_\infty)$. Fortunately it turns out that any of our generalised representations over $A_\inf(R_\infty)$ may be uniquely descended to the framed period ring:

\begin{theorem}[Thms.~\ref{theorem_descent_to_framed} \& \ref{theorem_<mu_implies_mu}]\label{theorem_intro2}
The base change functor \[-\otimes_{A_\sub{inf}^\bx(R)}A_\sub{inf}(R_\infty):\Rep_\Gamma^\mu(A_\inf^\bx(R))\To \Rep_\Gamma^\mu(A_\inf(R_\infty))\] is an equivalence of categories; similarly with Frobenius structures. Moreover, a generalised representation over $A_\inf(R_\infty)$ which is trivial modulo $<\mu$ is automatically trivial modulo $\mu$, i.e., the inclusion (\ref{intro_<mu}) is in fact an equality.
\end{theorem}


Having descended our representations to the framed period ring, Section \ref{section_q_connections} reinterprets them as modules with q-connection in the sense of q-de Rham cohomology \cite{Scholze2017}. Using our sequence of compatible $p$-power roots of unity to trivialise $\bb Z_p(1)\cong\bb Z_p$ and $\Gamma\cong \bb Z_p^d$, let $\gamma_1,\dots,\gamma_d\in\Gamma$ be the corresponding generators.

\begin{definition}[{c.f., \cite[Def.~7.3]{Scholze2017}}]\label{def_intro_q_conn}
A {\em flat q-connection} on an $A_\inf^\bx(R)$-module $N$ is the data of $d$ commuting $A_\inf$-linear endomorphisms $\nabla_1^\sub{log}, \dots,\nabla_d^\sub{log}$ of $N$ which satisfy the ``q-Leibniz rule'' \[\nabla_i^\sub{log}(fn)=\gamma_i(f)\nabla_i^\sub{log}(n)+\frac{\gamma_i(f)-f}{\mu}n,\] for all $f\in A_\inf^\bx(R)$, $n\in N$, $i=1,\dots,d$. A {\em Frobenius structure} on such a module is an isomorphism $\phi_N:(\phi^*N)[\tfrac1{\tilde\xi}]\isoto N[\tfrac1{\tilde\xi}]$ of $A_\inf^\bx(R)[\tfrac1{\tilde\xi}]$-modules compatible with connections (see \S\ref{ss_Frobenius} for further details).

The category of finite projective $A_\inf^\bx(R)$-modules with q-connection is denoted by $\textrm{qMIC}(A_\inf^\bx(R))$; similarly with Frobenius structures.
\end{definition}

Of course, given a module $N$ with flat q-connection in the above sense, $R$-module $\res N:=N/\xi$ inherits a flat connection $\res\nabla:\res N\to\res N\otimes_R\hat\Omega^1_{R/\roi}$ in the usual sense; thus q-connections are deformations of connections over $R$.

Given $N\in\Rep_\Gamma^\mu(A_\inf^\bx(R))$, we may convert the action of $\Gamma\cong\bb Z_p^d$ on $N$ into a flat q-connection by setting $\nabla_i^\sub{log}:=\tfrac{\gamma_i-1}{\mu}:N\to N$. This defines a functor $\Rep_\Gamma^\mu(A_\inf^\bx(R))\to\textrm{qMIC}(A_\inf^\bx(R))$, which is easily checked to be an equivalence of categories (see Corollary \ref{corollary_reps_vs_q_connections} for details); similarly with Frobenius structures.

Next we discuss the functor $(F,F_\Omega)^*$ in diagram (\ref{eqn_intro}), which is a local q-deformation of the Simpson correspondence. Let $A_\inf^{\bx}(R)^{(1)}:=A_\inf\otimes_{\phi,A_\inf}A_\inf(R)$ be the Frobenius twist of $A_\inf(R)$; it inherits an action by $\Gamma$, which is the identity modulo $\phi(\mu)$. To offer an alternative description, note that the relative Frobenius $F:A_\inf^\bx(R)^{(1)}\to A_\inf^\bx(R)$ is a morphism of $A_\inf$-algebras identifying the Frobenius twist with the sub $A_\inf$-algebra $\phi(A_\inf^\bx(R))\subseteq A_\inf^\bx(R)$. Descending along this Frobenius twist, q-connections become q-Higgs fields in the following sense:

\begin{definition}
A {\em flat q-Higgs field} on a $A_\inf^\bx(R)^{(1)}$-module $H$ is the data of $d$ commuting $A_\inf$-linear endomorphisms $\Theta_1^\sub{log},\dots,\Theta_d^\sub{log}$ which satisfy \[\Theta_i^\sub{log}(fh)=\gamma_i(f)\Theta_i^\sub{log}(h)+\frac{\gamma_i(f)-f}{\mu}h\] for all $f\in A_\inf^\bx(R)^{(1)}$, $h\in H$.

The category of finite projective $A_\inf^\bx(R)^{(1)}$-modules with q-connection is denoted by $\textrm{qHIG}(A_\inf^\bx(R)^{(1)})$, and similarly with Frobenius structures (whose unsurprising definition in this context may be found in Definition \ref{definition_phi_on_Higgs}).
\end{definition}

Given a flat q-Higgs field on $H$, one may naturally equip the base change $H\otimes_{A_\inf^\bx(R)^{(1)},F}A_\inf^\bx(R)$ with a flat q-connection, defined by \[\nabla_i^\sub{log}(h\otimes f):= \Theta_i^\sub{log}(h)\otimes \gamma_i(f)+h\otimes \frac{\gamma_i(f)-f}\mu.\] This defines the functor $(F,F_\Omega)^*$ appearing in diagram (\ref{eqn_intro}); the hardest part of the following theorem is essential surjectivity, which amounts to the existence of a good $\phi(A_\inf^\bx(R))$-lattice inside any $A_\inf^\bx(R)$-module with connection. We note that a similar functor, albeit in a rather different context, has also been shown to be an equivalence by Gros--Le~Stum--Quir\'os \cite{GrosLeStumQuiros2010, Gros2019}.

\begin{theorem}[Local q-Simpson correspondence: Corol.~\ref{corollary_q_Simp}]\label{thm_intro_q}
The functor $(F,F_\Omega)^*$ in diagram (\ref{eqn_intro}) is an equivalence of categories.
\end{theorem}

In Section \ref{section_prismatic_crystals} we relate the previous notions to crystals on the prismatic site of Bhatt--Scholze \cite{BhattScholze2019}. Viewing $R^{(1)}:=A_\inf/\tilde\xi\otimes_{A_\inf/\xi,\phi}R$ as a formally smooth $\roi=A_\inf/\tilde\xi$-algebra, let $\CR_{\Prism}(R^{(1)}/(A_{\inf},\tilde\xi))$ be the category of finite locally free crystals of $\roi_\Prism$-modules on the prismatic site of $\Spf R$ over the prism $(A_\inf,\tilde\xi)$, and let $\op{F-CR}_{\Prism}(R^{(1)}/(A_{\inf},\tilde\xi))$ be the category of such crystals equipped with the Frobenius structure. See Definition \ref{definiton_prismatic_crystal} and the paragraphs before Corollary \ref{corollary_main_thm_with_phi} for further details. Evaluating a crystal on the object $(A_\inf^\bx(R)^{(1)},\tilde\xi)$ of the prismatic site produces a finite projective $A_\inf^\bx(R)^{(1)}$-module; we show that the stratification data of the crystal corresponds  to a q-Higgs field, thereby defining a functor \begin{equation}\op{ev}_{A_\inf^\bx(R)^{(1)}}:\CR_{\Prism}(R^{(1)}/(A_{\inf},\tilde\xi))\To \mathrm{qHIG}(A_\inf^{\bx}(R)^{(1)})\label{eqn:ev}.\end{equation}  The main theorem of Section \ref{section_prismatic_crystals} identifies prismatic crystals in terms of modules with quasi-nilpotent q-connection as follows:

\begin{theorem}[Thm.~\ref{thm:main}]
The functor (\ref{eqn:ev}) is fully faithful, with essential image given by the objects of $\mathrm{qHIG}(A_\inf^{\bx}(R)^{(1)})$ for which the endomorphisms $\Theta_1^\sub{log},\dots,\Theta_d^\sub{log}$ are $(p,\mu)$-adically quasi-nilpotent.
\end{theorem}

We remark that the quasi-nilpotence condition is automatically satisfied in the presence of a Frobenius structure, whence it follows that the functor $\op{ev}_{A_\inf^\bx(R)^{(1)}}$ of diagram (\ref{eqn_intro}) is also an equivalence.

\begin{remark}\label{intro_bk}
We prove Theorem \ref{thm_intro_q} in greater generality than only for $A_\inf^\bx(R)$: the results in Section~\ref{section_q_connections} are established in a general framework for q-de Rham cohomology which applies in particular if $A_\inf$ is replaced by a flat q-PD pair over $\bb Z_p[[q-1]]$. The precise axiomatic hypotheses may be found in Remark~\ref{remark_q_via_deformation}, while a comparison to Bhatt--Scholze's framed q-PD pairs is discussed in Remark \ref{examples_qBK}.

Many of the arguments of Section \ref{section_prismatic_crystals} relating crystals, stratifications, and modules with q-connection or q-Higgs fields seem to work in a similar degree of generality, though we have not checked all the details. In the context of derived prismatic cohomology we refer to forthcoming work of Z.~Mao \cite{Mao2021} for more general results. Meanwhile, in the q-crystalline context, see the papers of A.~Chatzistamatiou~\cite{Chatzistamatiou2020} and K.~Li \cite{Li2021}. See Remark \ref{remark_general_prisms_base} for further discussion and references.

\end{remark}

Section \ref{section_crystalline} explores generalised representations and connections over crystalline period rings, in particular an analogue of Faltings' associatedness. In order to work with geometric Galois representations independently of the framing, we phrase the results in terms of $\res R$ rather than $R_\infty$. Letting $\CR(R/A_\sub{crys})$ denote the category of locally finite free crystals on the big crystalline site of $R/A_\crys$, evaluation defines a functor \[\ev_{A_\crys^\bx(R)}:\CR(R/A_\crys)\To \Rep_\Delta^\mu(A_\crys(\res R))\] (the target being defined similarly to Definition \ref{definition_intro_gen_rep}), and we say that a crystal $\cal F\in \CR(R/A_\crys)$ is {\em associated} to some $M\in \Rep^\mu_\Delta(A_\inf(\res R))$ if there is an isomorphism $\ev_{A_\crys(\res R)}(\cal F)\cong M\otimes_{A_\inf(\res R)}A_\crys(\res R)$ in $\Rep^\mu_\Delta(A_\crys(\res R))$. The main goal of Section \ref{section_crystalline} is the following, stating in particular that $\cal F$ is uniquely determined by $M$ by the chosen isomorphism:

\begin{theorem}[Thms.~\ref{theorem_crystal_genrep} \& \ref{theorem_Admissibility}]\label{thm_crystalline_intro}
The functor $\ev_{A_\crys^\bx(R)}$ is fully faithful. Moreover, given any $M\in \Rep^\mu_\Delta(A_\inf(\res R))$ admitting a Frobenius structure, then there exists a unique crystal associated to $M$.
\end{theorem}


The remaining sections of the paper globalise aspects of the above local theory of generalised representations over $A_\inf$. Letting $X$ denote the adic generic fibre of $\frak X$ and $\bb A_{\inf, X}=W(\roi^+_{X^\flat})$ the infinitesimal period ring sheaf on its pro-\'etale site, we introduce the following class of coefficient systems:

\begin{definition}
A {\em relative Breuil--Kisin--Fargues module} (without Frobenius) on $\frak X$ is a locally finite free sheaf of $\bb A_{\inf,X}$-modules on $X$ which is ``trivial mod $<\mu$'' (to be explained below -- this triviality is the key part of the definition which depends on the model $\frak X$ rather than only on the generic fibre $X$); the category of these is denoted by $\BKF(\frak X)$.

A Frobenius structure is a Frobenius-semi-linear isomorphism $\phi_{\bb M}:\bb M[\tfrac1\xi]\isoto \bb M[\tfrac1{\tilde\xi}]$, and we denote by $\BKF(\frak X,\phi)$ the category of relative Breuil--Kisin--Fargues modules equipped with a Frobenius structure.
\end{definition}

As a natural modification of Definition \ref{definition_intro_gen_rep}, we say that a locally finite free sheaf $\bb M$ of $\bb A_{\inf, X}$-modules is trivial modulo $\al\in A_\inf$ if the quotient $\bb M/\al$ is completely determined by its restriction $\nu_*(\bb M/\al)$ to $\frak X$, where $\nu_*:X_\sub{pro\'et}\to\frak X_\sub{Zar}$ is the projection map of sites; more precisely, we ask that $\nu_*(\bb M/\al)$ be a finite locally free sheaf of $\nu_*(\bb A_{\inf,X}/\al)$-modules and that the counit map $\nu^{-1}\nu_*(\bb M/\al)\otimes_{\nu^{-1}\nu_*(\bb A_{\inf,X}/\al)}\bb A_{\inf,X}/\al\to \bb M/\al$ be an isomorphism. We say that $\bb M$ is {\em trivial modulo $<\mu$} if it is trivial modulo $\mu/\phi^{-r}(\mu)$ for all $r\ge1$; because of non-zero almost-zero contributions arising from the almost purity theorem, this seems to be strictly weaker than asking that it be trivial modulo $\mu$.

Under internal hom and tensor product of locally finite free sheaves, the category of relative Breuil--Kisin--Fargues modules $\BKF(\frak X)$ becomes a rigid symmetric monoidal $A_\inf$-linear category. Moreover, the association $\frak X\mapsto \BKF(\frak X)$ is stacky for the Zariski topology. That is, given a Zariski cover $\{\frak U_1,\frak U_2\}$ of $\frak X$, then a relative BKF module $\bb M$ on $\frak X$ is determined by its restrictions $\bb M_1$, $\bb M_2$ to $\frak U_1,\frak U_2$ and by the glueing isomorphism $\bb M_1|_{\frak U_1\times_{\frak X}\frak U_2}\cong \bb M_2|_{\frak U_1\times_{\frak X}\frak U_2}$. This permits relative Breuil--Kisin--Fargues modules to be described locally on $\frak X$, in which case they are given by the generalised representations of Definition~\ref{definition_intro_gen_rep}. Indeed, fixing an affine open $\Spf R\subseteq \frak X$ where $R$ is a small, $p$-adically complete, formally smooth $\roi$-algebra with a fixed choice of framing, then the tower of finite \'etale morphisms defining $R_\infty$ corresponds to an affinoid perfectoid cover $U_\infty$ of the rigid analytic generic fibre $\op{Spa}(R[\tfrac1p],R)$ of $\Spf R$ (see Example~\ref{example_affinoid_perfectoid_covers}), and we prove:

\begin{theorem}[Thms.~\ref{theorem_Ainf_reps_vs_pro_etale} \& \ref{theorem_local_with_phi}]\label{theorem_intro1}
The global sections functor \[\Gamma(U_\infty,-):\BKF(\Spf R)\To \Rep_\Gamma^{<\mu}(A_\inf(R_\infty))\] is a well-defined equivalence of categories.
\end{theorem}

The theorem should not appear surprising, but its proof is rather technical. In particular, there is no good theory of vector bundles over the integral structure sheaf on adic spaces: without the smallness condition in the definition of a relative Breuil--Kisin--Fargues module, it would not be true (as far as we are aware) that $\Gamma(U_\infty,\bb M)$ is a finite projective $A_\inf(R_\infty)$-module.

\begin{remark}[``q-Riemann--Hilbert'']\label{remark_qRH}
Summarising the first four equivalences in diagram (\ref{eqn_intro}) shows that \[\BKF(\Spf R)\simeq \textrm{qMIC}(A_\inf^\bx(R)).\]  Since relative Breuil--Kisin--Fargues modules may be seen as deformations of local systems via the forthcoming specialisation functor $\sigma_\sub{\'et}^*$, one might view this equivalence as a Riemann--Hilbert style correspondence.

This equivalence also shows that the right side is independent of the chosen framing. This is a categorical analogue of the independence of $A\Omega_{R}$, which may be locally represented as a q-de Rham complex, on coordinates. As already mentioned in \cite[Rmk.~7.4]{Scholze2017}, the categorical independence does not seem to follow from the latter statement. 
Combined with the specialisation functor $\sigma_\sub{\'et}^*$, which may now be used to associate a local system on $\Spa(R[\tfrac1p],R)$ to each module with q-connection, this establishes $A_\inf$-analogues of \cite[Conj.~7.5 \& Rmk.~7.6]{Scholze2017}.

A Riemann--Hilbert functor associating local systems to q-connections, in the context of formally smooth $\bb Z_p$-algebras, has been constructed and studied by L.~Mann \cite{Mann_masters}.
\end{remark}

Incorporating Frobenius structures, the overall equivalence $\BKF(\Spf R,\phi)\simeq \op{F-CR}_{\Prism}(R^{(1)}/(A_{\inf},\tilde\xi))$ of diagram (\ref{eqn_intro}) turns out not to depend on the chosen framing (even though all other categories of the diagram are defined in terms of the framing) and glues over small opens of $\frak X$ to produce a global comparison equivalence to prismatic crystals:

\begin{theorem}[{Thm.~\ref{theorem_local_with_phi}(ii)}]\label{thm_intro_BKF_prism}
For any separated, smooth, $p$-adic formal $\roi$-scheme $\frak X$, there is a natural equivalence of categories \[\op{F-CR}_{\Prism}(\frak X^{(1)}/(A_{\inf},\tilde\xi))\quis \op{BKF}(\frak X,\phi).\]
\end{theorem}

Although the functor of the theorem can be easily defined (see the proof of Theorem \ref{theorem_prism_to_BKF}), our proof that it really takes values in relative Breuil--Kisin--Fargues modules and is an equivalence ultimately passes through all the intermediate equivalencies of diagram (\ref{eqn_intro}). Using the equivalence of the theorem, the specialisation functors $\sigma^*_\sub{dR},\,\sigma^*_\sub{crys}$ below should then correspond to pulling back along morphisms of sites to $(\frak X^{(1)}/(A_\inf,\tilde\xi))_\Prism$ respectively from the crystalline site of $\frak X/\roi$ and the crystalline site of $\frak X_k/W(k)$; however, we have not made any attempt to verify these expectations.

\begin{remark}
Without Frobenius structures, it seems that $\BKF(\frak X)$ does not admit any naive prismatic interpretation. There is still a well-defined fully faithful functor $\op{CR}_{\Prism}(\frak X^{(1)}/(A_{\inf},\tilde\xi))\to \op{BKF}(\frak X)$ in the global context, but already for small affines the functor (\ref{eqn:ev}) is not essentially surjective. It remains to be seen whether there are any interesting examples of relative Breuil--Kisin--Fargues modules which do not satisfy the necessary quasi-nilpotence condition to come from prismatic crystals.
\end{remark}

\begin{example}[Global Dieudonn\'e theory: see Ex.~\ref{example_global_Dieudonne}]
Example \ref{intro_examples}(i) may be globalised to a fully faithfully embedding of the category of $p$-divisible groups over $\frak X$ into $\BKF(\frak X,\phi)$, whose essential image consists of those $(\bb M,\phi_{\bb M})$ such that $\bb M\subseteq\phi_{\bb M}(\bb M)\subseteq\tfrac{1}{\tilde\xi}\bb M$.
\end{example}

Having introduced relative Breuil--Kisin--Fargues modules as a natural globalisation of generalised representations and modules with q-connection, we show that they serve as coefficient systems for the $A_\inf$-cohomology theory developed in \cite{BhattMorrowScholze1} and reconstructed as prismatic cohomology in \cite{BhattScholze2019}.
%
%
More precisely, to each $\bb M\in \BKF(\frak X,\phi)$ we functorially associate a sheaf of complexes  $A\Omega_{\frak X}(\bb M)$ of $A_\inf$-modules on $\frak X_\sub{\'et}$, equipped with a semi-linear Frobenius. This is an analogue of $\bb A\Omega_\frak X$ from \cite{BhattMorrowScholze1} (in particular, in the case of the trivial coefficient system $\bb M=\bb A_{\sub{inf},X}$ we recover $\bb A\Omega_{\frak X}(\bb A_{\sub{inf},X})=\bb A\Omega_\frak X$) and is constructed via the same nearby cycle formula $A\Omega_\frak X(\bb M):=L\eta_\mu\big(\hat{R\nu_*\bb M}\big)$, where $L\eta_\mu$ is the d\'eclage functor \cite[\S6]{BhattMorrowScholze1} and the hat denotes derived $p$-adic completion. There exist moreover three specialisation functors \[\sigma^*_\sub{dR},\,\sigma^*_\sub{crys},\,\sigma^*_\sub{\'et}:\BKF(\frak X,\phi)\To \categ{3cm}{vector bundles with $\roi$-linear flat connection on $\frak X$},\,\categ{2.5cm}{locally finite free $F$-crystals on $(\frak X_k/W(k))_\sub{crys}$},\,\categ{3cm}{locally free lisse $\bb Z_p$-sheaves on $X$},\] where $\frak X_k$ denotes the special fibre of $\frak X$, whence a relative Breuil--Kisin--Fargues module should be viewed as a means of encoding a vector bundle, F-crystal, and lisse $\bb Z_p$-sheaf which are necessarily related to each other in some fashion. Theorem \ref{thm_crystalline_intro} is a local manifestation of this relation, and we prove moreover that their cohomologies are intertwined by $A\Omega_\frak X(\bb M)$, analogously to \cite[Thm.~1.10]{BhattMorrowScholze1}:

\begin{theorem}[See Thm.~\ref{theorem_specialisations}]
There exist natural, for $\bb M\in\BKF(\frak X,\phi)$, equivalences
\begin{enumerate}
\item (de Rham) $A\Omega_{\frak X}(\bb M)\dotimes_{A_\inf,\theta}\roi\simeq \sigma^*_\sub{dR}\bb M\otimes_{\roi_\frak X}\Omega^\blob_{\frak X/\roi}$,
\item (crystalline) $\hat{A\Omega_{\frak X}(\bb M)\dotimes_{A_\inf}W(k)}\simeq Ru_*\sigma^*_\sub{crys}(\bb M)$, where $u:(\frak X_k/W(k))_\sub{crys}\to\frak X_{k\,\sub{Zar}}=\frak X_\sub{Zar}$ is the usual map of sites, and
\item (\'etale) $\hat{A\Omega_\frak X(\bb M)[\tfrac1\mu]}^{\phi=1}=R\nu_*(\sigma^*_\sub{\'et}\bb M)$,
\end{enumerate}
where the hats denote derived $p$-adic completions (taken after inversion of $\mu$, but before Frobenius-fixed points, in the \'etale comparison).
\end{theorem}




\subsection*{Acknowledgements}
We thank Ahmed Abbes, Bhargav Bhatt, Ofer Gabber, Michel Gros, Zhouhang Mao, Arthur Ogus, and Peter Scholze for helpful discussions. The first author would like to thank the University of Tokyo for its support and hospitality during enjoyable visits in February 2017 and September 2018 during which this project was worked on. The second author was financially supported by JSPS Grants-in-Aid for Scientific Research, Grant Number 15H02050.

\subsection*{Notation}
We adopt the same notation as \cite{BhattMorrowScholze1}. In particular, $\roi$ denotes throughout the ring of integers of a mixed characteristic perfectoid field $C$ containing all $p$-power roots of unity, and $A_\inf=W(\roi^\flat)$ is Fontaine's ring. There are the maximal ideals $\frak m\subseteq\roi$ and $\frak m^\flat\subseteq\roi^\flat$, and the resulting ideal $W(\frak m^\flat):=\op{ker}(A_\inf\to W(k))$, where $k$ is the residue field of $\roi^\flat$ (equivalently, of $\roi$).

We fix throughout the paper a compatible sequence of $p$-power roots of unity (which will also be used to trivialise certain twists) and write $\ep:=(1,\zeta_p,\zeta_{p^2},\dots)\in\roi^\flat$, $\mu:=[\ep]-1\in A_\inf$, $\xi_r:=\tfrac{[\ep]-1}{[\ep^{1/p^r}]-1}\in A_\inf$, $\xi=\xi_1$, $\tilde\xi=\phi(\xi)$. The surjections $\theta_r,\tilde\theta_r:A_\inf\to W_r(\roi)$ are defined as in \cite[\S3]{BhattMorrowScholze1}.

\newpage
\section{Small generalised representations over $A_\sub{inf}(R_\infty)$}\label{section_small_reps}
Let $R$ be a $p$-adically complete, formally smooth $\roi$-algebra, which we assume is small in the following sense: there exists a formally \'etale\footnote{When we speak of ``formally \'etale'' maps we implicitly include the condition of being of topological finite presentation; in the cases in which we are interested, such maps are always completions of \'etale maps by \cite{Elkik1973}. The topology (always $p$-adic, $\pi$-adic, or $(p,\xi)$-adic) should be clear from the context.} map $\square:\roi\pid{\ul T^{\pm1}}=\roi\pid{T_1^{\pm1},\dots,T_d^{\pm1}}\to R$. Throughout \S\ref{section_small_reps} we moreover fix such a map, known as a {\em framing}, and let $R_\infty$ be the $p$-adic completion of \[\indlim_jR\otimes_{\roi\pid{\ul T^{\pm1}}}\roi\pid{\ul T^{\pm1/p^j}}.\] The $R$-algebra $R_\infty$ is equipped with the usual $R$-linear continuous action by $\Gamma:=\bb Z_p(1)^d$, characterised by \[\gamma\cdot T_1^{k_1}\cdots T_d^{k_d}:=\gamma(k_1,\dots,k_d)T_1^{k_1}\cdots T_d^{k_d}\text{ where }\gamma\in\Gamma=\Hom_{\bb Z_p}((\bb Q_p/\bb Z_p)^d,\mu_{p^\infty}) \text{ and }k_1,\dots,k_d\in\bb Z[\tfrac1p],\] which satisfies $R_\infty^\Gamma=R$. Our fixed choice of $p$-power roots of unity defines a preferred generator $\ep\in\bb Z_p(1)\subseteq\roi^{\flat\times}$, and we let $\gamma_1,\dots,\gamma_d\in \Gamma=\bb Z_p(1)^d$ be the corresponding basis; that is, for $i,j=1,\dots,d$, and $r\ge0$, we have
\[\gamma_i(T_j^{1/p^r})=\begin{cases} \zeta_{p^r}T_i^{1/p^r}&\text{if }i=j\\ T_j^{1/p^r} & \text{if }i\neq j\end{cases}.\]

The ring $R_\infty$ is perfectoid in the sense of \cite[\S3]{BhattMorrowScholze1}; we denote by $R_\infty^\flat$ its tilt and by $A_\inf(R_\infty):=W(R_\infty^\flat)$ Fontaine's infinitesimal period ring construction. We view $A_\inf(R_\infty)$, and modules over it, equipped with the $(p,\xi)$-adic $=(p,\mu)$-adic $=(p,[\pi])$-adic topology (where $\pi$ is any non-zero element of the maximal ideal $\frak m^\flat\subseteq\roi^\flat$).

We begin our discussion of generalised representations with a general definition:

\begin{definition}\label{definition_trivial}
Let $B$ be a ring equipped with an adic topology and a continuous action (via ring homomorphisms) by a topological group $G$. We write $\Rep_G(B)$ for the category of finite projective $B$-modules $M$ equipped with a semi-linear, continuous (wrt.~the adic topology on $M$) action by $G$. Following Faltings \cite{Faltings2005}, such $M$ are known as {\em generalised representations}.

Given $b\in B$, we say that a generalised representation $M$ is {\em trivial modulo $b$} if the $(B/b)^G$-module $(M/b)^G$ is finite projective and the canonical map $(M/b)^G\otimes_{(B/b)^G}B/b\to M/b$ is an isomorphism. We write $\Rep_G^{b}(B)\subseteq \Rep_G(B)$ for the full subcategory of such generalised representations which are trivial modulo $b$. This category has internal hom (see Remark \ref{remark_internal_hom}) and tensor product.

Observe that if $b\in b'B$ then ``trivial mod $b$'' implies ``trivial mod $b'$'', i.e., $\Rep_G^{b}(B)\subseteq \Rep_G^{b'}(B)$.
\end{definition}

The goal of this section is to study $\Rep_\Gamma(A_\inf(R_\infty))$, i.e., generalised representations in the case $B=A_\inf(R_\infty)$ and $G=\Gamma$. We will be particularly interested in those representations which are trivial modulo $\mu$, as well as those which are trivial modulo $\xi_r$ for all $r\ge1$; we say that the latter representations are {\em trivial modulo $<\mu$} and denote the category of them by $\Rep_\Gamma^{<\mu}(A_\inf(R_\infty))\supseteq \Rep_\Gamma^{\mu}(A_\inf(R_\infty))$. The main goals of this section are Theorem \ref{theorem_<mu_implies_mu} stating that the previous inclusion of categories is in fact an equality, and Theorem \ref{theorem_descent_to_framed} stating that such generalised representations descend uniquely to the framed period ring $A_\inf^\bx(R)$ (defined in \S\ref{ss_framed}).

\begin{remark}[$\res R$ in place of $R_\infty$]\label{remark_Delta}
Assuming $\Spec(R/pR)$ is connected, let $\res R$ denote (as often) the $p$-adic completion of the integral closure of $R$ inside the maximal ind-\'etale extension of $R[\tfrac1p]$ inside some fixed algebraic closure of $\Frac R$. Thus $\res R$ extends $R_\infty$ (the extension depends on a compatible choice of $p$-power roots of the $T_i$ inside $\res R$) and is equipped with a continuous $R$-action by $\Delta:=\pi_1^\sub{\'et}(R[\tfrac1p])$ (where we omit from the notation the point corresponding to the fixed algebraic closure of $\Frac R$) extending the $\Gamma$-action on $R_\infty$. The extension $\res R\supseteq R_\infty$ is the usual one to which Faltings almost purity theorem is applied, a consequence of which is that the base change functor \[-\otimes_{A_\inf(R_\infty)}A_\inf(\res R):\Rep_\Gamma^{<\mu}(A_\inf(R_\infty))\To \Rep_\Delta^{<\mu}(A_\inf(\res R))\] is an equivalence of categories: using the framework of the pro-\'etale site and Scholze's reformulation of the almost purity theorem, a complete proof of this equivalence is obtained from applying Theorem \ref{theorem_Ainf_reps_vs_pro_etale} simultaneously to both examples of Examples \ref{example_affinoid_perfectoid_covers}.

The reader who prefers framing-independent statements may therefore replace $\Rep_\Gamma^{<\mu}(A_\inf(R_\infty))$ by $\Rep_\Delta^{<\mu}(A_\inf(\res R))$ in the remainder of this section (nevertheless, the proofs rely on a choice of framing and the above equivalence). In particular, since the above base change functor sends $\Rep_\Gamma^{\mu}(A_\inf(R_\infty))$ to $\Rep_\Delta^{\mu}(A_\inf(\res R))$, it follows from Theorem~\ref{theorem_descent_to_framed} that also $\Rep_\Delta^{<\mu}(A_\inf(\res R))=\Rep_\Delta^{\mu}(A_\inf(\res R))$ (this was pointed out to us by Scholze) and therefore the base change equivalence may be rewritten \[\Rep_\Gamma^{\mu}(A_\inf(R_\infty))\quis \Rep_\Delta^{\mu}(A_\inf(\res R))\]
\end{remark}

We finish this preliminary material by recording that continuity of the action is often automatic for the group $\Gamma$:

\begin{lemma}[Automatic continuity]\label{lemma_automatic_continuity}
In the set-up of Definition \ref{definition_trivial}, suppose that $G$ is isomorphic as a topological group to $\bb Z_p^d$, and that the adic topology on $B$ is defined by an ideal $I$ containing $p$. Let $M$ be a finite projective $B$-module equipped with a semi-linear $G$-action which is trivial modulo $I$, in the sense that $(M/IM)^G$ is a finite $(B/I)^G$-module and the canonical map $(M/IM)^G\otimes_{(B/I)^G}B/I\to M/IM$ is an isomorphism. Then the $G$-action on $M$ is continuous with respect to the $I$-adic topology.
\end{lemma}
\begin{proof}
It is equivalent to prove that the $G$-action on $M/I^NM$ (with the discrete topology) is continuous for all $N\ge1$, so we replace $B$ and $B/I^N$ and henceforth assume $I^N=0$.

We first treat the case that $(M/IM)^G$ is finite free module; lifting a basis to $M$ shows that $M$ is a finite free $B$-module with a basis $e_1,\dots,e_n$ such that $\gamma(e_i)\equiv e_i$ mod $IM$ for all $i=1,\dots,n$ and $\gamma\in G$. Let $c:\Gamma\to 1+IM_n(B)$ be the associated $1$-cocycle (see Remark \ref{remark_1cocycles}), so that $(\gamma(e_1),\dots,\gamma(e_n))=(e_1,\dots,e_n)c(\gamma)$ for all $\gamma\in G$.

By the continuity of the $G$-action on the discretely topologised $B$, there exists $m\ge1$ such that $\gamma^{p^m}(c(\gamma_i))=c(\gamma_i)$ for all $i=1,\dots,n$ and $\gamma\in G$. Using the $1$-cocycle identity, a simple induction shows that $\gamma_i^{p^m}(c(\gamma_i^a))=c(\gamma_i^a)$ for all $a\ge1$, and then a similar induction shows that $c(\gamma_i^{p^ma})=c(\gamma_i^{p^m})^a$ for all $a\ge 1$. Taking $a=p^{N-1}$ obtains $c(\gamma_i^{p^{m+N-1}})=c(\gamma_i^{p^m})^{p^{N-1}}=1$, where the final equality follows from the facts that $c(\gamma_i^{p^m})\in 1+IM_n(B)$ and $I^N=0$.

Set $\ell:=m+N-1$ and let $\gamma\in G$. The final conclusion of the previous paragraph shows that \[\gamma_i^{p^\ell}(c(\gamma))=c(\gamma_i^{p^\ell}\gamma)=c(\gamma\gamma_i^{p^\ell})=c(\gamma)\] for all $i=1,\dots,d$. But the stabiliser of the $G$-action on $c(\gamma)$ is an open subgroup of $G$, so this stabiliser therefore contains $G^{p^\ell}$. In particular $\gamma^{p^\ell}(c(\gamma))=c(\gamma)$, and then the same argument as the previous paragraph shows that $c(\gamma^{p^{\ell+N-1}})=1$; i.e., $\gamma^{p^{\ell+N-1}}$ fixes each element $e_1,\dots,e_n$.

This shows that the action of the open subgroup $G^{p^{\ell+N-1}}\subseteq G$ on $M\cong B^n$ is given by the coordinate-wise action; but this is continuous since the $G$ action on $B$ is continuous, and so completes the proof in the finite free case.

We now consider the general case that the $(B/I)^G$-module $(M/IM)^G$ is merely finite projective. Note that, given any $f\in (B/I)^G$ and lift $\tilde f\in B$, then the $G$-action on $B$ naturally extends to the localisation $B_{\tilde f}$ (and so the action on $M$ extends to $M_{\tilde f}$): indeed, for any $\gamma\in G$ one has $\gamma(\tilde f)\equiv \tilde f$ modulo the nilpotent ideal $I$, whence $\gamma(\tilde f)$ is also a unit in $B_{\tilde f}$. If moreover $f$ is chosen so that the localisation of $(M/IM)^G$ at $f$ is finite free over the localisation $(B/I)^G_{f}$, then the $G$-action on the $B_{\tilde f}$-module $M_{\tilde f}$ is trivial modulo $I$ in the finite free sense which we have already treated; so then the $G$-action on $M_{\tilde f}$ is continuous. Now pick $f_1,\dots,f\in (B/I)^G$, generating the unit ideal of this ring, such that each corresponding localisation of $(M/IM)^G$ is finite free; let $\tilde f_1,\dots,\tilde f_n\in B$ be arbitrary lifts, and note that these lifts generate the unit ideal of $B$ (since they generate it modulo a nilpotent ideal). Therefore $M$ embeds into $\prod_{i=1}^nM_{\tilde f_i}$, where we have seen that the $G$-action is continuous.
\end{proof}

\subsection{The framed period ring}\label{ss_framed}
Our study of generalised representations over $A_\inf(R_\infty)$ will heavily use the framed period ring, which we recall in this subsection and for which we refer to \cite[\S9.2]{BhattMorrowScholze1} \cite[\S12]{Tsuji_simons} for further details. Let $A_\inf\pid{\ul U^{\pm1}}$ denote the $(p,\xi)$-adic completion of $A_\inf[U_1^{\pm1},\dots,U_d^{\pm1}]$, and view $A_\inf(R_\infty)$ as an $A_\inf\pid{\ul U^{\pm1}}$-algebra via $U_i\mapsto [T_i^\flat]=$ the Teichm\"uller lift of $T^\flat_i:=(T_i,T_i^{1/p},\dots)\in R_\infty^\flat$. By formal \'etaleness of the framing $\square:\roi\pid{\ul T^{\pm1}}=A_\inf\pid{\ul U^{\pm1}}/\xi\to R$, there exists a unique $(p,\xi)$-adically complete, formally \'etale $A_\inf\pid{\ul U^{\pm1}}$-algebra $A_\inf^\bx(R)$ and a morphism of $A_\inf\pid{\ul U^{\pm1}}$-algebras $A_\inf^\bx(R)\to A_\inf(R_\infty)$ such that the square
\[\xymatrix{
A_\inf(\roi\pid{\ul T^{\pm1/p^\infty}})=A_\inf\pid{\ul U^{\pm1/p^\infty}}\ar[r] & A_\inf(R_\infty)\\
A_\inf\pid{\ul U^{\pm1}}\ar[u]\ar[r] & A_\inf^\bx(R)\ar[u]
}\] (here $\roi\pid{\ul T^{\pm1/p^\infty}}$, resp.~$A_\inf\pid{\ul U^{\pm1/p^\infty}}$, denotes the $p$-adic, resp.~$(p,\xi)$-adic, completion of $\roi[\ul T^{\pm1/p^\infty}]$, resp.~$A_\inf[\ul U^{\pm1/p^\infty}]$) identifies modulo $\xi$ with 
\[\xymatrix{
\roi\pid{\ul T^{\pm1/p^\infty}}\ar[r] & R_\infty\\
\roi\pid{\ul T^{\pm1}}\ar[u]\ar[r] & \,R\,.\ar[u]
}\] The first square is a pushout modulo any power of the ideal $(p,\xi)$, and therefore $A_\inf(R_\infty)$ may be identified with the $(p,\xi)$-adic completion of \[\bigoplus_{k_1,\dots,k_d\in\bb Z[\tfrac1p]\cap[0,1)}A_\inf^\bx(R)\,U_1^{k_1}\cdots U_d^{k_d}.\]

Regarding actions, we let $\Gamma=\bb Z_p(1)^d$ act via $A_\inf$-algebra homomorphisms on $A_\inf\pid{\ul U^{\pm1}}$ as \[(\ep_1,\dots,\ep_d)\cdot U_1^{k_1}\dots U_d^{k_d}:=[\ep_1^{k_1}\cdots\ep_d^{k_d}]U_1^{k_1}\dots U_d^{k_d}\]
(here we view $\bb Z_p(1)$ as a subgroup of the units of $\roi^\flat$). This is compatible, via the maps of the previous paragraph, with the trivial action on $R$ and the existing $\Gamma$-action on $A_\inf(R_\infty)$; in particular, by formal \'etaleness, it extends to a continuous action via $A_\inf$-algebra automorphisms on $A_\inf^\bx(R)$, still compatibly with all maps above. It is important to note that $\Gamma$ acts as the identity on the quotient $A_\inf^\bx(R)/\mu$ (unlike on $A_\inf(R_\infty)/\mu$). Indeed, since the action is by the identity on $R$, it is enough by formal \'etaleness to check that it is also the identity on $A_\inf\pid{\ul U^{\pm1}}/\mu$; but this is immediate from the definitions: each $\ep_1,\dots,\ep_d$ is a $\bb Z_p$-linear multiple of the fixed generator $\ep\in\bb Z_p(1)$, and then $[\ep_1^{k_1}\cdots\ep_d^{k_d}]-1$ is divisible by $\mu=[\ep]-1$.

Many of our arguments will involve passing between $A_\inf(R_\infty)$ and $A^\bx_\inf(R_\infty)$. To do this we note that there is a $\Gamma$-equivariant direct sum decomposition \[A_\inf(R_\infty)=A^\bx_\inf(R)\oplus A_\inf^\sub{n-i}(R_\infty)\] of $A^\bx_\inf(R)$-modules, where $A_\inf^\sub{n-i}(R_\infty)$ (``n-i'' stands for ``non-integral'') is the $(p,\xi)$-adic completion of \[\bigoplus_{\substack{k_1,\dots,k_d\in\bb Z[\tfrac1p]\cap[0,1)\\\sub{not all $0$}}}A_\inf^\bx(R)U_1^{k_1}\cdots U_d^{k_d}\] Our goal will often to be to show that $A_\inf^\sub{n-i}(R_\infty)$ can be discounted in some sense.

For convenience we collect here some topological properties of these rings which will occasionally be used without mention:

\begin{lemma}\label{lemma_topology}
Let $B$ denote $A_\inf(R_\infty)$, $A_\inf^\bx(R)$, or $A_\inf^\sub{n-i}(R_\infty)$ (in the latter case note that $B$ is not a ring, merely an $A_\inf^\bx(R)$-module); let $c\in A_\inf$ be any element whose image in $A_\inf/p=\roi^\flat$ is a non-zero element of $\frak m^\flat$ (equivalently, $c\notin A_\inf^\times\cup pA_\inf$). Then:
\begin{enumerate}
\item $B$ is complete and separated for the $(p,c)$-adic topology, which is the same as the $(p,\xi)$-adic topology.
\item $B$ and $B/c$ are $p$-torsion-free and $p$-adically complete and separated.
\item $B$ and $B/p$ are $c$-torsion-free and $c$-adically complete and separated.
\item $cB$ is closed in $B$ and $c(B/p)$ is closed in $B/p$.
\item $B\to\projlim_r B/\tilde\xi_r$ is an isomorphism.
\item The set-theoretic identification $A_\inf(R_\infty)=A^\bx_\inf(R)\times A_\inf^\sub{n-i}(R_\infty)$ is a homeomorphism (in which the three $A_\inf$-modules are equipped with the $(p,\xi)$-adic topology and the right side is given the product topology).
\end{enumerate}
\end{lemma}
\begin{proof}
Properties (i)--(iii) and (v) only depend on $B$ as an $A_\inf$-module and are clearly inherited by any direct summand, so it is enough to treat the case $B=A_\inf(R_\infty)$. Write $\res c\in\frak m^\flat\setminus\{0\}$ for the image of $c$, and expand it as $\res c=(c_0,c_1,\dots)\in\roi^\flat=\projlim_{x\mapsto x^p}\roi$. Since $R_\infty/p$ is a flat $\roi/p$-module, its $c_n$-torsion is given by multiples of $p/c_n$ (assuming that $p\in c_n\roi$, which is true for $n\gg0$), which is annihilated by $\phi:R_\infty/p\to R_\infty/p$. Passing to the limit over $\phi$ therefore shows that multiplication by $\res c$ is injective on $R_\infty^\flat=\projlim_\phi R_\infty/p$. Since each $c_n$ is nilpotent in $R_\infty/p$, the limit $R_\infty^\flat$ is moreover $\res c$-adically complete and separated.

The ring $A_\inf(R_\infty)=W(R_\infty^\flat)$ is $p$-torsion-free and $p$-adically complete and separated since it is the ring of Witt vectors of a perfect ring of characteristic $p$. From the short exact sequences \[0\to A_\inf(R_\infty)/p^{n-1}\to A_\inf(R_\infty)/p^n\to R_\infty^\flat\to 0\] an induction implies that each $A_\inf(R_\infty)/p^n$ is $c$-torsion-free and $c$-adically complete and separated, whence taking the limit as $n\to\infty$ shows that the same is true of $A_\inf(R_\infty)$, showing also that this ring is $(p,c)$-adically complete and separated. It now follows formally that $A_\inf(R_\infty)/c$ is $p$-torsion-free and $p$-adically complete and separated, which completes the proof of (i)--(iii). For (v), we refer the reader to \cite[Lems.~3.2 \& 3.12]{BhattMorrowScholze1}.

Part (vi) is completely formal: it holds for any finite direct sum of modules equipped with the adic topologies associated to an ideal. Part (iv) is a consequence of $B/c$ and $B/(c,p)$ being complete and separated for the $(c,p)$-adic topology.
\end{proof}

\subsection{Lemmas on generalised representations trivial mod $<\mu$}\label{subsection_mu_implies_phi-1mu}
In this subsection we establish two technical consequences of triviality modulo $<\mu$ for generalised representations over $A_\inf(R_\infty)$; see Proposition \ref{proposition_reduction_to_mod_p} and Corollary \ref{corollary_mu_smallness}.

\begin{lemma}\label{lemma_no_non_zero}
Let $B$ be a perfectoid ring (in the sense of \cite[\S3]{BhattMorrowScholze1}) over $\roi$ which is $p$-torsion-free; then the following conditions are equivalent:
\begin{enumerate}
\item The $\roi$-module $B/p$ contains no non-zero almost-zero elements.
\item $\bigcap_{r\ge 1}\tfrac{\zeta_p-1}{\zeta_{p^r}-1}B=(\zeta_p-1)B$;
\item $\bigcap_{r\ge 1}\tfrac{\ep-1}{\ep^{1/p^r}-1}B^\flat=(\ep-1)B^\flat$;
\end{enumerate}
Moreover, under these equivalent conditions, the isomorphisms $\theta_r:A_\inf(B)/\xi_r\isoto W_r(B)$ induce, upon passage to the limit as $r\to\infty$, a short exact sequence of $A_\sub{inf}(B)$-modules \[0\To A_\sub{inf}(B)/\mu\To W(B)\To \Xi(B)\To 0,\] where $\Xi(B)$ is an $A_\sub{inf}(B)$-module which is $p$-torsion-free but killed by $W(\frak m^\flat)$.
\end{lemma}
\begin{proof}
Rescaling by $p/(\zeta_p-1)$, conditions (i) and (ii) are simply reformulations of each other; we could similarly replace (i) by the condition that the $\roi$-module $B/\pi$ contains no non-zero almost-zero elements, where $\pi$ is any chosen pseudo-uniformiser of $\roi$. Similarly, (iii) is equivalent to the assertion that the $\roi^\flat$-module $B^\flat/\pi^\flat$ contains no non-zero almost-zero elements, where $\pi^\flat$ is a chosen pseudo-uniformiser of $\roi^\flat$; but we may chose $\pi$ and $\pi^\flat$ so that the untilting map induces an identification $B^{\flat}/\pi^\flat=B/\pi$, which completes the proof that (ii) and (iii) are equivalent.

The short exact sequence was constructed in the case $B=\roi$ in \cite[Lem.~3.23]{BhattMorrowScholze1}; the proof there works verbatim for general $B$ satisfying (ii).
\end{proof}

\begin{example}\label{example_R_infty_no_almost_zero}
Here we check that the equivalent conditions of Lemma \ref{lemma_no_non_zero} are satisfied for all the perfectoid rings which we will encounter.
\begin{enumerate}
\item Firstly, the equivalent conditions are satisfied for $B=R_\infty$: we will show that all almost-zero elements of $R_\infty/p=\indlim_jR/p\otimes_{\roi[\ul T^{\pm1}]/p}\roi[\ul T^{\pm1/p^j}]/p$ are zero. Since the transition maps in this filtered colimit are injective (this is clear without the base change $R/p\otimes_{\roi[\ul T^{\pm1}]/p}-$, but this base change is flat), it is enough to show for each fixed $j\ge1$ that $R/p\otimes_{\roi[\ul T^{\pm1}]/p}\roi[\ul T^{\pm1/p^j}]/p$ contains no non-zero almost-zero elements; but this is smooth over $\roi/p$, hence is free as an $\roi/p$-module \cite[Lem.~8.10]{BhattMorrowScholze1}, so the claim reduces to the fact that $\roi/p$ contains no non-zero almost-zero elements (which follows from the usual valuation argument).

\item Suppose that $A$ is a perfectoid Tate ring over $C$; then we claim that $B=A^\circ$ (the subring of power bounded elements) satisfies the equivalent conditions. Indeed, it is equivalent to check that the inclusion $A':=\{f\in A:\frak mf\subseteq A^\circ\}\supseteq A^\circ$ is an equality, since $pA'/pA^\circ$ is precisely the set of almost zero elements of $A^\circ/p$. But $\frak mf\subseteq A^\circ$ if and only if $\frak mf\subseteq \frak mA^\circ$, whence $A'$ is a subring; since also $A'\subseteq\tfrac1pA^\circ$ and $A^\circ$ is bounded (as $A$ is perfectoid Tate), it follows that all elements of $A'$ are power bounded, i.e., $A'\subseteq A^\circ$.

\item In particular, suppose that $X$ is a reduced rigid analytic variety over $C$ and that $\cal U\in X_\sub{pro\'et}$ is an affinoid perfectoid with associated perfectoid Huber pair $(A,A^+)=(\Gamma(\cal U,\hat\roi_X),\Gamma(\cal U,\hat\roi_X^+))$. Then we claim that $B=A^+$ satisfies the equivalent conditions of the previous lemma. Thanks to the previous example, it is enough to check that $A^+=A^\circ$. Since $A$ is a perfectoid Tate ring, the inclusion $A^+\subseteq A^\circ$ is almost an equality \cite[Lem.~5.6]{Scholze2012}, and therefore $pA^\circ/pA^+$ consists of almost-zero elements of $A^+/p$; so it is enough to show that $A^+/pA^+$ contains no non-zero almost-zero elements. To do this we write $\cal U=\projlimf_i\Spa(A_i,A_i^+)$ as a cofiltered limit of affinoids along finite \'etale transition maps, so that $A^+$ is the $p$-adic completion of $\indlim_i A_i^+$ \cite[Lem.~4.10(iv)]{Scholze2013} and therefore $A^+/p=\indlim_i A_i^+/p$. The transition maps in the latter filtered colimit are injective (since each $A_i^+\to A_{i'}^+$ is an integral extension and $A_i^+$ is integrally closed in $A_i$), so the problem reduces to checking that each $A_i^+/p$ contains no non-zero almost-zero elements. Since $\Spa(A_i,A_i^+)$ is \'etale over the reduced rigid analytic variety $X$, it is the affinoid adic space associated to a reduced rigid affinoid, i.e., $A_i$ is topologically of finite type over $C$ and reduced, whence $A_i^+=A_i^\circ$ is the full subring of power bounded elements \cite[\S6.2.3 Prop.~1 \& \S6.2.4 Thm.~1]{BoschLutkebohmert1993}. But now we may complete the proof as in the previous example: $A'_i:=\{f\in A_i:\frak mf\subseteq A_i^\circ\}$ is again a subring of power bounded elements, hence equal to $A_i^\circ$, whence the almost-zero elements $pA'_i/pA^\circ\subseteq A^\circ/p$ are all zero.
\end{enumerate}
\end{example}

Thanks to the previous example and lemma we have a short exact sequence \[0\To A_\sub{inf}(R_\infty)/\mu\To W(R_\infty)\To \Xi(R_\infty)\To 0\] on which $\Gamma$ acts. We wish to study the resulting group cohomology, but the actions are not continuous in the strictest sense: although $A_\inf(R_\infty)/(\mu,p^N)$ is a discrete $\Gamma$-module for any $N\ge1$ ({\em Proof}: it is isomorphic via $\tilde\theta_N$ to $W_N(R_\infty)/[\zeta_{p^N}]-1$, which equals $W_N(R_\infty/p^M)/[\zeta_{p^N}]-1$ for $M\gg0$.), this is not true of the quotients $W(R_\infty)/p^N$. Therefore we let $\gamma_1,\dots\gamma_d\in \Gamma=\bb Z_p(1)^d$ be the basis elements corresponding to our fixed compatible sequence of $p$-power roots of unity $\ep\in\bb Z_p(1)$ (see the Notation) and let $\bb Z^d\subseteq\Gamma$ be the dense subgroup they generate; for any abelian group $M$ equipped with an action by $\Gamma$, we will study the non-topological group cohomology $R\Gamma(\bb Z^d,M)$, which may be calculated by the Koszul complex $K(\gamma_1-1,\dots,\gamma_d-1;M)$. Note however that if $M$ can be expressed as an inverse limit $M=\projlim_r M_r$ along surjective transition maps of discrete $\Gamma$-modules, each of which is killed by a power of $p$, then $R\Gamma(\bb Z^d,M)\simeq R\Gamma_\sub{cont}(\Gamma,M)$ where $M$ is equipped with the inverse limit topology; we generally prefer to state our results in terms of continuous group cohomology when possible and will occasionally use this identification without explicit mention. See \cite[Lem.~7.3]{BhattMorrowScholze1} for further discussion of these matters.

\begin{lemma}\phantomsection\label{lemma_mu_1}
\begin{enumerate}
\item The maps \[H^i(\bb Z^d,A_\sub{inf}(R_\infty)/\mu)\To H^i(\bb Z^d,W(R_\infty))\] and \[H^i(\bb Z^d,A_\sub{inf}(R_\infty)/(\mu,p^N))\To H^i(\bb Z^d,W(R_\infty)/p^N)\] are injective for all $i,N\ge0$.
\item $H^i(\bb Z^d, A_\sub{inf}(R_\infty)/\phi^{-s}(\mu))$ is $p$-torsion-free for all $i\ge0$ and $s\in\bb Z$.
\item $H^1(\bb Z^d,W(R_\infty))$ is $p$-torsion-free.
\end{enumerate}
\end{lemma}
\begin{proof}
Taking group cohomology of the short exact sequence in the above paragraph (and of its reduction modulo $p^N$: this is still short exact since $A_\sub{inf}(R_\infty)/(\mu,p^N)$ has no non-zero elements killed by $W(\frak m^\flat)$ by an induction over $N$ and Example \ref{example_R_infty_no_almost_zero}(i), whereas $\Xi(R_\infty)/p^N$ is killed by $W(\frak m^\flat)$ by Lemma \ref{lemma_no_non_zero}), the injectivity claims will follow if we can show that the resulting boundary maps are zero; since $\Xi(R_\infty)$ is killed by $W(\frak m^\flat)$, it is therefore enough to show that $H^i(\bb Z^d,A_\sub{inf}(R_\infty)/\mu)$ and $H^i(\bb Z^d,A_\sub{inf}(R_\infty)/(\mu,p^N))$ have no non-zero elements killed by $W(\frak m^\flat)$. This is a standard type of Koszul complex calculation which we include for convenience of the reader; the calculation will also establish part (ii) (which, by applying $\phi^s$, reduces to the case $s=0$).

As we recalled in \S\ref{ss_framed}, $A_\inf(R_\infty)$ is the $(p,\xi)$-adic $=(p,\mu)$-adic completion of $\bigoplus_{\ul k}A^\bx_\inf(R)U_1^{k_1}\cdots U_d^{k_d}$. Going modulo $\mu$, we may therefore identify $A_\inf(R_\infty)/\mu$ with the $p$-adic completion of \[\bigoplus_{k_1,\dots,k_d\in\bb Z[\tfrac1p]\cap[0,1)}(A_\inf^\bx(R)/\mu)\,U_1^{k_1}\cdots U_d^{k_d},\] where the generators $\gamma_1,\dots,\gamma_d$ of $\Gamma$ act $A_\inf^\bx(R)/\mu$-linearly on the summands on the right as multiplication by $[\ep^{k_1}],\dots,[\ep^{k_d}]$ respectively.

A standard calculation of Koszul cohomology shows that each cohomology group of \[R\Gamma(\bb Z^d,(A_\inf^\bx(R)/\mu)\,U_1^{k_1}\cdots U_d^{k_d})\simeq K([\ep^{k_1}]-1,\dots,[\ep^{k_d}]-1;A^\bx_\inf(R)/\mu)\] is a finite direct sum of copies of torsion and quotients, namely \[(A_\inf^\bx(R)/\mu)[[\ep^k]-1]\qquad\text{and}\qquad A_\inf^\bx(R)/[\ep^k]-1\] for a certain value of $k\in\bb Z[\tfrac1p]\cap(0,1]$ (more precisely, this calculation follows from Lemma \ref{lemma_Koszul}(ii), with $k$ given by whichever of $k_1,\dots,k_d$ has the smallest $p$-adic valuation, or by $k=1$ if all the $k_i$ are zero). These modules are isomorphic via $\times\tfrac{[\ep]-1}{[\ep^k]-1}:A_\inf^\bx(R)/[\ep^k]-1\isoto (A_\inf^\bx(R)/\mu)[[\ep^k]-1]$ 
and are $p$-torsion-free by Lemma \ref{lemma_topology}(ii). Writing \[R\Gamma(\bb Z^d,A_\inf(R_\infty)/\mu)\simeq\op{Rlim}_N\bigoplus_{k_1,\dots,k_d\in\bb Z[\tfrac1p]\cap[0,1)}K([\ep^{k_1}]-1,\dots,[\ep^{k_d}]-1;A^\bx_\inf(R)/\mu)/p^N,\] one obtains the following conclusions:
\begin{enumerate}
\item $H^i(\bb Z^d,A_\inf(R_\infty)/\mu)$ is $p$-torsion-free;
\item $\displaystyle H^i(\bb Z^d,A_\inf(R_\infty)/\mu)/p^N=H^i(\bb Z^d,A_\inf(R_\infty)/(\mu,p^N))\cong \bigoplus_{\sub{various }k\in\bb Z[\tfrac1p]\cap(0,1]}A_\inf^\bx(R)/([\ep^k]-1,p^N)$.
\item $H^i(\bb Z^d,A_\inf(R_\infty)/\mu)\isoto\projlim_NH^i(\bb Z^d,A_\inf(R_\infty)/\mu)/p^N$.
\end{enumerate}

To complete the proof of (i) it is therefore enough to show that each summand $A_\inf^\bx(R)/([\ep^k]-1,p^N)$ contains no non-zero almost-zero elements. By induction on $N$ this reduces to the case $N=1$, in which case $A_\inf^\bx(R)/([\ep^k]-1,p)$ is \'etale over $(\roi^\flat/(\ep^k-1))[\ul U^{\pm1}]$ and in particular smooth over $\roi^\flat/(\ep^k-1)$; it follows that $A_\inf^\bx(R)/([\ep^k]-1,p)$ is free as a $\roi^\flat/(\ep^k-1)$-module (see \cite[Lem.~8.10]{BhattMorrowScholze1} and its proof). But $\roi^\flat/(\ep^k-1)$ itself has no non-zero almost-zero elements by a valuation argument, thereby completing the proof of the injectivity claims (and of the first $p$-torsion-freeness claim).

It remains to show that $H^1(\bb Z^d,W(R_\infty))$ is $p$-torsion-free. From the already established injectivities and $p$-torsion-freeness, we obtain an isomorphism $H^1(\bb Z^d,W(R_\infty))[p]\cong H^1(\bb Z^d,\Xi(R_\infty))[p]$. Since the latter module is killed by $W(\frak m^\flat)$, it is sufficient to show that $H^1(\bb Z^d,W(R_\infty))$ has no non-zero elements killed by $W(\frak m)$. Since the projective system $W(R_\infty)=\projlim_rW_r(R_\infty)$ has surjective transition maps we may identify $R\Gamma(\bb Z^d,W(R_\infty))$ with $\op{Rlim}_rR\Gamma(\bb Z^d,W_r(R_\infty))$ and thus obtain a short exact sequence \[0\to{\projlim_r}^1(W_r(R_\infty)^\Gamma)\To H^1(\bb Z^d,W(R_\infty))\To \projlim_rH^1(\bb Z^d,W_r(R_\infty))\To 0.\] The initial $\lim^1$ term vanishes since the transition maps are surjective: indeed, we have $W_r(R_\infty)^\Gamma=W_r(R)$. Therefore it is enough to check that each $H^1(\bb Z^d,W_r(R_\infty))$ has no non-zero elements killed by $W_r(\frak m)$; this may be proved in the same way as in the first part of the proof, and in any case may be found in \cite[Lem.~9.7(iii)]{BhattMorrowScholze1}.
\end{proof}

\begin{corollary}\label{corollary_mod_p_mu}
The canonical map $(A_\inf(R_\infty)/\mu)^\Gamma/p\to (A_\inf(R_\infty)/(\mu,p))^\Gamma$ is an isomorphism.
\end{corollary}
\begin{proof}
It is equivalent to show that $H^1_\sub{cont}(\Gamma,A_\inf(R_\infty)/\mu)$ is $p$-torsion-free, which is Lemma \ref{lemma_mu_1}(ii).
\end{proof}

In order to pass to the limit in the proof of Proposition \ref{proposition_reduction_to_mod_p} we will need the following lemma, which we presume is well-known but for which we could not find any reference:

\begin{lemma}\label{lemma_limits_of_finite_projective}
Let $B$ be a commutative ring and $B\supseteq I_1\supseteq I_2\supseteq\cdots$ a decreasing chain of ideals such that $B\isoto\projlim_rB/I_r$; assume also that $I_r/I_{r+1}$ is contained in the Jacobson radical of $B/I_{r+1}$ for each $r\ge1$. Then:
\begin{enumerate}
\item Given a $B$-module $M$ such that $M\to\projlim_rM/I_rM$ is an isomorphism, then $M$ is finite projective if and only each $M/I_rM$ is a finite projective $B/I_r$-module for each $r\ge1$.
\item There is an equivalence of categories \[\categ{3cm}{Finite projective $B$-modules $M$}\simeq\categ{8cm}{compatible inverse systems $(M_r)_{r\ge1}$ of finite projective $B/I_r$-modules such that $M_{r+1}\otimes_{B/I_{r+1}}B/I_r\isoto M_r$ for all $r\ge1$}\] given by $M\mapsto(M/I_rM)_{r\ge1}$ and $(M_r)_{r\ge1}\mapsto\projlim_rM_r$.
\end{enumerate}
\end{lemma}
\begin{proof}
For part (i) the necessity is clear and the sufficiency will follow from part (ii). In order to prove (ii), the only difficulty is showing that, given $(M_r)_{r\ge1}$ in the right category, then $M:=\projlim_rM_r$ is a finite projective $B$-module and the canonical maps $M/I_rM\to M_r$ are isomorphism for all $r\ge1$; we will now check this.

Let $F$ be a finite free $B$-module equipped with a surjection $F/I_1\to M_1$ (by picking a finite set of generators of $M_1$). By freeness of $F$, we may lift the morphism $F/I_1\to M_1$ to a morphism $f:F\to M$ (note that $M\to M_1$ is surjective since all the transition maps $M_{r+1}\to M_r$ are surjective). Each induced map $f_r:F/I_r\to M_r$ is surjective by Nakayama's lemma, since it is surjective mod $I_1/I_r$ and this ideal belongs to the Jacobson radical of $B/I_r$.

Let $\op{Sec}(f_r)$ denote the set of $B$-linear sections of $f_r$. We claim that the reduction map $\op{Sec}(f_{r+1})\to \op{Sec}(f_r)$ is surjective. So suppose $s:M_r\to F/I_r$ is a section of $f_r$, and use projectivity of $M_{r+1}$ to lift the composition $M_{r+1}\to M_r\xto{s}F/I_r$ to a morphism $M_{r+1}\xto{s'}F/I_{r+1}$. Then $s'$ is not necessarily a section of $f_{r+1}$, but it is modulo $I_r/I_{r+1}$, i.e., the morphism $\op{id}-f_{r+1}s':M_{r+1}\to M_{r+1}$ has image in $I_r M_{r+1}$; this latter module is a quotient of $I_rF/I_{r+1}F$ via $f_{r+1}$, so projectivity of $M_{r+1}$ yields a morphism $h:M_{r+1}\to I_rF/I_{r+1}F$ satisfying $f_{r+1}h=\op{id}-f_{r+1}s'$. Therefore $s'+h$ is a section of $f_{r+1}$ lifting $s$; this completes the proof of the claim.

Thanks to the claim (and the fact that $\op{Sec}(f_1)$ is non-empty since $M_1$ is finite projective), there exists a compatible family of sections $s_r:M_r\to F/I_r$ for all $r\ge1$. Upon taking the limit this produces a morphism $s:M:=\projlim_rM_r\to F$ which splits $f$ (since it splits $f$ modulo $I_r$ for all $r$). This proves that $M$ is a finite projective $B$-module. It also shows that the surjection $M/I_r\to M_r$ is injective (since, after composing with $s_r$, it identifies with the split injection ``$s$ mod $I_r$''), whence is an isomorphism.
\end{proof}

\begin{example}
Let $B$ be a commutative ring which is $p$-adically complete and separated. We claim that the lemma applies to $W(B)$ and its chain of ideals $VW(B)\supseteq V^2W(B)\supseteq\cdots$. Since $W(B)=\projlim_rW_r(B)$, we just need to check the Jacobson radical condition.

For $r\ge 1$ fixed, the ring $W_r(B)$ is $p$-adically complete and separated, and $W_r(B)/p=W_r(B/p^N)/p$ for $N\gg0$ \cite[Lem.~3.2(ii)]{BhattMorrowScholze1}. To show that $V^{r-1}(B)$ is contained in the Jacobson radical of $W_r(B)$, we may therefore assume that $p$ is nilpotent in $B$; but in that case the ideal $V^{r-1}(B)$ is even nilpotent.
\end{example}

In Theorem \ref{theorem_<mu_implies_mu} we will show that triviality modulo $<\mu$ actually implies triviality modulo $\mu$; the proof will require the following weaker statement:

\begin{proposition}\label{proposition_reduction_to_mod_p}
Let $M\in\Rep_\Gamma(A_\inf(R_\infty))$ be a generalised representation which is trivial modulo $<\mu$ and also trivial modulo $(p,\mu)$.\footnote{Technically this has not been defined though the definition should be obvious; it means that $(M/(p,\mu))^\Gamma$ is a finite projective $(A_\inf(R_\infty)/(p,\mu))^\Gamma$-module and that the canonical map $(M/(p,\mu))^\Gamma\otimes_{(A_\inf(R_\infty)/(p,\mu))^\Gamma}A_\inf(R_\infty)/(p,\mu)\to M/(p,\mu)$ is an isomorphism.} Then $M$ is trivial modulo $\mu$.
\end{proposition}
\begin{proof}
We claim first that $M\otimes_{A_\inf(R_\infty)}W(R_\infty)$ is a trivial representation, i.e., that $(M\otimes_{A_\inf}W(R_\infty))^\Gamma$ is a finite projective $W(R)=W(R_\infty)^\Gamma$-module and that
\[(M\otimes_{A_\inf}W(R_\infty))^\Gamma\otimes_{W(R)}W(R_\infty)\To M\otimes_{A_\inf}W(R_\infty)\] is an isomorphism. For each $r\ge1$, the triviality of $M$ modulo $\xi_r$ means that $(M/\xi_r)^\Gamma$ is a finite projective $W_r(R)=(A_\inf(R_\infty)/\xi_r)^\Gamma$-module and that $(M/\xi_r)^\Gamma\otimes_{W_r(R)}W_r(R_\infty)\isoto M/\xi_r$. Considering this isomorphism for $r+1$ instead of $r$ and then going modulo $\xi_r$ also shows that $(M/\xi_{r+1})^\Gamma\otimes_{W_{r+1}(R)}W_r(R)\isoto (M/\xi_r)^\Gamma$. We may therefore apply the previous lemma and example to see that $(M\otimes_{A_\inf(R)}W(R_\infty))^\Gamma$, which equals $\projlim_r(M/\xi_{r+1})^\Gamma$, is indeed a finite projective $W(R)$-module and that $(M\otimes_{A_\inf(R)}W(R_\infty))^\Gamma\otimes_{W(R)}W_r(R)\isoto (M/\xi_r)^\Gamma$ for each $r\ge 1$. Combining this isomorphism with the triviality of $M$ modulo $\xi_r$ completes the proof of the claim.

The $p$-torsion-freeness of $H^1(\bb Z^d,W(R_\infty))$ from Lemma \ref{lemma_mu_1}(iii) means that $(W(R_\infty)/p^{N+1})^{\bb Z^d}\to (W(R_\infty)/p^N)^{\bb Z^d}$ is surjective for any $N\ge1$. Combined with the triviality claim proved in the previous paragraph, it follows that
 \[ (M\otimes_{\bb A_\inf(R_\infty)}W(R_\infty)/p^{N+1})^{\bb Z^d}\To ( M\otimes_{\bb A_\inf(R_\infty)}W(R_\infty)/p^N)^{\bb Z^d}\] is also surjective, whence $p^N:H^1(\bb Z^d, M\otimes_{A_\inf(R_\infty)}W(R_\infty)/p)\to H^1(\bb Z^d, M\otimes_{A_\inf(R_\infty)}W(R_\infty)/p^{N+1})$ is injective. But $M/(\mu,p)$ is trivial by assumption so we know from Lemma \ref{lemma_mu_1} that $H^1(\bb Z^d, M/(\mu,p))\into H^1(\bb Z^d, M\otimes_{A_\inf(R_\infty)}W(R_\infty)/p)$, whence we deduce that $p^N:H^1(\bb Z^d, M/(\mu,p))\to H^1(\bb Z^d, M/(\mu,p^{N+1}))$ is also injective, or equivalently that $(M/(\mu,p^{N+1}))^{\bb Z^d}\to(M/(\mu,p^N))^{\bb Z^d}$ is surjective.

By taking the limit as $N\to \infty$ it follows that $(M/\mu)^\Gamma\to (M/(\mu,p))^\Gamma$ is surjective, whence $(M/\mu)^\Gamma/p$ identifies with the finite projective module $(M/(\mu,p))^\Gamma$ over $(A_\inf(R_\infty)/(\mu,p))^\Gamma=((A_\inf(R_\infty)/\mu)^\Gamma/p$ (this equality was proved in Corollary \ref{corollary_mod_p_mu}). By $p$-completeness and $p$-torsion-freeness of both $(M/\mu)^\Gamma$ and $(A_\inf(R_\infty)/\mu)^\Gamma$ this implies that $(M/\mu)^\Gamma$ is finite projective over $(A_\inf(R_\infty)/\mu)^\Gamma$, and also that the canonical map \[(M/\mu)^\Gamma\otimes_{(A_\inf(R_\infty)/\mu)^\Gamma}A_\inf(R_\infty)/\mu\To M/\mu\] is an isomorphism since it is an isomorphism modulo $p$.
\end{proof}

We note the following consequence, which in any case will later be improved when we show that trivial modulo $<\mu$ implies trivial modulo $\mu$:

\begin{corollary}\label{corollary_mu_smallness}
Let $M\in\Rep_\Gamma(A_\inf(R_\infty))$ be a generalised representation which is trivial modulo $<\mu$. Then it is trivial modulo $\phi^{-1}(\mu)$.
\end{corollary}
\begin{proof}
Replacing $ M$ by a Frobenius twist, we may instead assume that it is trivial modulo $\phi(\xi_r)$ for all $r\ge1$ and we must prove it is trivial modulo $\mu$. 

Triviality modulo $\phi(\xi_r)$ implies triviality modulo $\xi_{r-1}$, and triviality modulo $\tilde\xi$ implies triviality modulo $(\mu,\tilde\xi)=(\mu,p)$; we then apply Proposition \ref{proposition_reduction_to_mod_p}.
\end{proof}

\subsection{Descent to the framed period ring}\label{subsection_descent}
The first main result of this subsection asserts that generalised representations over $A_\inf(R_\infty)$ which are trivial modulo $\mu$ descend uniquely to the framed period ring:

\begin{theorem}\label{theorem_descent_to_framed}
The base change functor \[-\otimes_{A_\sub{inf}^\bx(R)}A_\sub{inf}(R_\infty):\Rep_\Gamma^\mu(A_\inf^\bx(R))\To \Rep_\Gamma^\mu(A_\inf(R_\infty))\] is an equivalence of categories.
\end{theorem}

The second main goal is the following:

\begin{theorem}\label{theorem_<mu_implies_mu}
Let $M\in\Rep_\Gamma(A_\inf(R_\infty))$. Then $M$ is trivial modulo $\mu$ if and only if it is trivial modulo~$<\mu$, i.e., $\Rep_\Gamma^{<\mu}(A_\inf(R_\infty))=\Rep_\Gamma^\mu(A_\inf(R_\infty))$.
\end{theorem}

The proofs of Theorems \ref{theorem_descent_to_framed} and \ref{theorem_<mu_implies_mu} will require simultaneously establishing analogous results modulo $p$. Therefore for the remainder of the subsection we adopt one of the following two situations:
\[A:=A_\inf,\qquad A_\infty:=A_\inf(R_\infty),\qquad A^\square:=A^\bx_\inf(R),\qquad A_\infty^\sub{n-i}:=A_\inf^\sub{n-i}(R_\infty),\qquad \qquad\tag{``integral case''}\]
or
\[A:=A_\inf/p=\roi^\flat,\quad A_\infty:=A_\inf(R_\infty)/p=R_\infty^\flat,\quad A^\square:=A^\bx_\inf(R)/p,\quad A_\infty^\sub{n-i}:=A_\inf^\sub{n-i}(R_\infty)/p\quad\tag{``mod $p$ case''}\]
We moreover fix an element $c_1\in A$ such that $\mu A\subseteq c_1A\subseteq\phi^{-1}(\mu)A$; then we set $c_0:=c_1/\phi^{-1}(\mu)$, which divides the element $\xi\in A$, and we assume that $c_0$ is not a unit of $A$. In writing $\mu$, $\xi$ here we bring to the reader's attention a slight abuse of notation which we will use throughout the rest of the subsection in order to give uniform proofs: in the mod $p$ case we sometimes write $[\ep]$ to mean $\ep\in \roi^\flat$ and similarly $\mu$, $\xi$ in place of $\ep-1,(\ep-1)/(\ep^{1/p}-1) \in \roi^\flat$; this should not cause any confusion. Our uniform proofs will also sometimes include redundancies in the mod $p$ case, for example the appearance of $p$-adic completions.

\begin{example}
In the mod $p$ case the above hypotheses are quite flexible: $c_0$ may be any element of $\roi^\flat$ with valuation satisfying $0<\nu(c_0)\le\nu((\ep-1)/(\ep^{1/p}-1))$. In the integral case the hypotheses are more restrictive, but in fact the only case of interest will be $c_0=\xi$ and $c_1=\mu$.
\end{example}

We will prove the following in several steps:

\begin{theorem}\label{theorem_descent_framed_general}
The base change functor \[-\otimes_{A^\bx}A_\infty:\Rep_\Gamma^{c_1}(A^\bx)\To \Rep_\Gamma^{c_1}(A_\infty)\] is an equivalence of categories.
\end{theorem}

Theorem \ref{theorem_descent_framed_general} easily implies the main theorems:

\begin{proof}[Proof of Theorem \ref{theorem_descent_to_framed}]
This is simply Theorem \ref{theorem_descent_framed_general} in the integral case with $c_0=\xi$ and $c_1=\mu$.
\end{proof}

\begin{proof}[Proof of Theorem \ref{theorem_<mu_implies_mu}]
Suppose that $M\in\Rep_\Gamma(A_\inf(R_\infty))$ is trivial modulo $<\mu$, i.e., trivial modulo $\xi_r$ for all $r\ge 1$. Therefore the generalised representation $M/p \in\Rep_\Gamma(R_\infty^\flat)$ is trivial modulo $(\ep-1)/(\ep^{1/p^r}-1)$ for all $r\ge 1$, and so Theorem \ref{theorem_descent_framed_general} (in the mod $p$ case, with $c_1=(\ep-1)/(\ep^{1/p^r}-1)$ for each $r\ge 1$) shows that $M/p=N\otimes_{A^\bx_\inf(R)/p}R_\infty^\flat$ for some $N\in\Rep_\Gamma(A^\bx_\inf(R)/p)$ which is trivial modulo $(\ep-1)/(\ep^{1/p^r}-1)$ for all $r\ge 1$.

We claim that $N$ is automatically trivial modulo $\ep-1$. Indeed, since $\Gamma$ acts as the identity on $A^\bx_\inf(R)/p$ modulo $\ep-1$ (hence also modulo $(\ep-1)/(\ep^{1/p^r}-1)$ for any $r\ge 1$), triviality modulo $(\ep-1)/(\ep^{1/p^r}-1)$ of the $\Gamma$-action on $N$ means that $\gamma(n)-n\in \tfrac{\ep-1}{\ep^{1/p^r}-1}N$ for all $\gamma\in \Gamma$ and $n\in N$. Letting $r\to\infty$ shows that $\gamma(n)-n\in (\ep-1)N$ for all $\gamma\in \Gamma$ and $n\in N$ (here we use that the $\roi^\flat$-algebra $R':=A^\bx_\inf(R)/p$ has the property $\bigcap_{r\ge1}\tfrac{\ep-1}{\ep^{1/p^r}-1}R'=(\ep-1)R'$; but $R'$ is formally smooth over $\roi^\flat$, hence $\ep-1$-torsion-free, so it is equivalent to check that the smooth $\roi^\flat/(\ep-1)$-algebra $R'/(\ep-1)$ has no non-zero almost-zero elements; as in Example \ref{example_R_infty_no_almost_zero}(i) this reduces to the analogous assertion for $\roi^\flat/(\ep-1)$ itself), thereby establishing triviality of the $\Gamma$-action on $N/(\ep-1)$.

We have shown that $M$ is trivial modulo $(p,\mu)$; but therefore it is trivial modulo $\mu$ by Proposition~\ref{proposition_reduction_to_mod_p}, completing the proof.
\end{proof}

We remind the reader that we have fixed $A$, $A_\infty$, $A^\bx$, $c_0$, and $c_1$. The rest of this subsection is dedicated to proving Theorem \ref{theorem_descent_framed_general}, first by establishing fully faithfulness then essential surjectivity.

\begin{proposition}\label{proposition_ff}
The base change functor of Theorem \ref{theorem_descent_framed_general} is fully faithful.
\end{proposition}
\begin{proof}
The functor is faithful since $A^\bx\to A_\infty$ is injective. To prove fullness, let $N_1, N_2\in\Rep_\Gamma^{c_1}(A^\bx)$ and let $f:N_1\otimes_{A^\bx}A_\infty\to N_2\otimes_{A^\bx}A_\infty$ be an $A_\infty$-linear, $\Gamma$-equivariant map. It is enough to show that $f(N_1)\subseteq N_2$: indeed, by $A_\infty$-linearity this forces $f$ to agree with $f|_{N_1}\otimes_{A^\bx}A_\infty$.

By writing $N_2\otimes_{A^\bx}A_\infty$ as the $(p,\mu)$-adic completion of $\bigoplus_{k_1,\dots,k_d\in\bb Z[\tfrac1p]\cap[0,1)}N_2\,U_1^{k_1}\cdots U_d^{k_d}$, we may express $f(n)$, for each $n\in N_1$, as a $(p,\mu)$-adically convergent sum \[f(n)=\sum_{\ul k}f_{\ul k}(n)U_1^{k_1}\cdots U_d^{k_d}\] where $f_{\ul k}(n)\in N_2$ (in fact, the convergence of the sum will be irrelevant, we are merely using that $N_2\otimes_{A^\bx}A_\infty\into \prod_{\ul k}N_2\,U_1^{k_1}\cdots U_d^{k_d}$). For each $\ul k$, the morphism $f_{\ul k}:N_1\to N_2$ is $A^\bx$-linear, and the $\Gamma$-equivariance of $f$ implies that \[f_{\ul k}(\gamma_i n)=[\ep^{k_i}]\gamma_i(f_{\ul k}(n))\] for all $n\in N_1$ and $i=1,\dots,d$.

In other words, setting $N:=\Hom_{A^\bx}(N_1,N_2)\in\Rep_\Gamma^{c_1}(A^\bx)$ (see Remark \ref{remark_internal_hom}), the elements $f_{\ul k}\in N$ satisfy $([\ep^{k_i}]\gamma_i-1)f_{\ul k}=0$ for $i=1,\dots,d$. The $\Gamma$-action on $N$ is trivial modulo $c_1$, so we may write $\gamma_i=1+c_1\delta_i$ for some unique $\delta_i\in\End_{A}(N)$ for each $i$; here we crucially use that $c_1|\mu$ so that $\Gamma$ acts as the identity on $A^\bx/c_1$. But then, assuming that $k_i\neq 0$ for some $i$ (whence $[\ep^{k_i}]-1$ divides $[\ep^{1/p}]-1$), we deduce that \[[\ep^{k_i}]\gamma_i-1=([\ep^{k_i}]-1)(1+\tfrac{[\ep^{1/p}]-1}{[\ep^{k_i}]-1}[\ep^{k_i}]c_0\delta_i)\] is injective since $1+\tfrac{[\ep^{1/p}]-1}{[\ep^{k_i}]-1}[\ep^{k_i}]c_0\delta_i$ is an automorphism by $c_0$-adic completeness of $N$; therefore $f_{\ul k}=0$.

We have shown that the only possibly non-zero $f_{\ul k}$ occurs when $\ul k=0$; but this exactly means that $f(N_1)\subseteq N_2$, as required.
\end{proof}

\begin{remark}\label{remark_internal_hom}
In the context of Definition \ref{definition_trivial}, note that $\Rep_G(B)$ has an internal hom: given $M_1,M_2\in\Rep_G(B)$, we equip $\Hom_B(M_1,M_2)$ with action $\gamma f:m\mapsto \gamma(f(\gamma^{-1}(m)))$, so that $\Hom_B(M_1,M_2)^G=\Hom_{B,G}(M_1,M_2)$. Furthermore, elementary manipulations of Hom modules show that if $M_1,M_2$ are trivial mod $b$, for some fixed $b\in B$, then so is $\Hom_B(M_1,M_2)$.
\end{remark}

It remains to prove that the functor in Theorem \ref{theorem_descent_framed_general} is essentially surjective; this proceeds via explicit manipulations of $1$-cocycles, so we being with a reminder on such issues.

\begin{remark}[Reminder on semi-linear actions and $1$-cocycles]\label{remark_1cocycles}
Let $B$ be a ring equipped with an action (via ring homomorphisms) by a group $G$. Let $M$ be a $B$-module equipped with a semi-linear $G$-action, and denote by $G\times\Aut_B(M)\to\Aut_B(M)$, $(\gamma,X)\mapsto {}^\gamma\! X=\gamma X\gamma^{-1}$ the resulting conjugation action.

Recall that a $1$-cocycle $\al:G\to\Aut_B(M)$ is a function satisfying $\al(\gamma\delta)=\al(\gamma)\circ{}^\gamma\!\al(\delta)$ for all $\gamma,\delta\in G$; the set of all $1$-cocycles (pointed by the $1$-cocycle $\gamma\mapsto\op{id}$) is denoted by $Z^1(G,\Aut_B(M))$. There is a right action on $Z^1(G,\Aut_B(M))$ by $\Aut_B(M)$ given by $(\al X)(\gamma):=X^{-1}\circ \al(\gamma)\circ {}^\gamma\!X$, where $\al$ is a $1$-cocycle and $X\in \Aut_B(M)$. The resulting quotient set is $H^1(G,\Aut_B(M))=Z^1(G,\Aut_B(M))/\Aut_B(M)$.

Any $1$-cocycle $\al$ defines a new semi-linear $G$-action $\star_\al$ on $B$ via the rule $\gamma\star_\al b:=\al(\gamma)(\gamma(b))$; conversely, any semi-linear $G$-action $\star$ on $B$ defines a $1$-cocycle \[\al_\star:G\to\Aut_B(M),\qquad \gamma\mapsto \langle b\mapsto\gamma\star(\gamma^{-1}b)\rangle,\] and these two processes invert each other. Given $1$-cocycles $\al$, $\al'$ with associated semi-linear actions $\star$,~$\star'$, then $\al$ and $\al'$ are cohomologous (i.e., define the same class in $H^1(G,\Aut_B(M))$) if and only if there exists a $G$-equivariant isomorphism of $B$-modules $(M,\star)\cong (M,\star')$.

The above formalism will be applied in the case $M=B^n$ equipped with the coordinate-wise $G$-action, where $G$ will be a topological group and its action on $B$ will be continuous for the adic topology associated to some ideal. Then the conjugate action of $G$ on $\Aut_B(M)=\GL_n(B)$ is of course given by applying $G$ to the terms of the matrix, and the continuous $1$-cocycles $\al:Z\to\GL_n(B)$ classify continuous semi-linear $G$-actions on $B^n$.
\end{remark}


\begin{lemma}\label{lemma_vanishing2}
\begin{enumerate}
\item In the integral case, the group cohomologies $H^i_\sub{cont}(\Gamma,A_\inf^\sub{n-i}(R_\infty)/c_1)$ are $p$-torsion-free for all $i\ge0$; also $H^1_\sub{cont}(\Gamma, A_\inf(R_\infty)/c_1)$ and $H^1_\sub{cont}(\Gamma, A_\inf^\bx(R)/c_1)$ are $p$-torsion-free.
\item In both the integral and mod $p$ cases, the group cohomologies $H^i_\sub{cont}(\Gamma,A_\inf^\sub{n-i}/c_1)$ are killed by $\phi^{-1}(\mu)$ for all $i\ge0$. 
\end{enumerate}
\end{lemma}
\begin{proof}

We begin by proving (ii) and the first part of (i) in the integral case. The second and third paragraphs of the proof of Lemma \ref{lemma_mu_1}, replacing $\mu$ by $c_1$ and using that $\phi^{-1}(\mu)|c_1|\mu$, show that each group cohomology $H^i_\sub{cont}(\Gamma,A_\inf^\sub{n-i}(R_\infty)/c_1)$ is the $p$-adic completion of a direct sum of copies of  $A^\bx_\inf(R)/([\ep^k]-1)$ for various $k\in\bb Z[\tfrac1p]\cap(0,1)$. Such a module is indeed $p$-torsion-free and killed by $\phi^{-1}(\mu)=[\ep^{1/p}]-1$.

This argument may be repeated in the mod $p$ case, replacing Lemma \ref{lemma_Koszul} by the more general calculation when the module $N$ has $g$-torsion \cite[Lem.~7.10(ii)]{BhattMorrowScholze1}, to show that $H^i_\sub{cont}(\Gamma,(A_\inf^\sub{n-i}(R_\infty)/p)/c_1)$ is a direct sum of copies of $(A^\bx_\inf(R)/(p,\ep^k-1)$ and $(A^\bx_\inf(R)/p)[\ep^k-1] $ for various $k\in\bb Z[\tfrac1p]\cap(0,1)$; again, this is killed by $\ep^{1/p}-1$.

It remains to prove the statements in (i) about $H^1$; since $A_\inf(R_\infty)=A_\inf^\bx(R)\oplus A_\inf^\sub{n-i}(R_\infty)$, it is enough to show that $H^1_\sub{cont}(\Gamma, A_\inf^\bx(R)/c_1)$ is $p$-torsion-free, i.e., that $(A_\inf^\bx(R)/c_1)^{\Gamma}\to (A_\inf^\bx(R)/(c_1,p))^\Gamma$ is surjective. But $\Gamma$ acts as the identity on both sides since $c_1|\mu$, so this is trivial.
\end{proof}

The following two lemmas show that any $1$-cocycle representing a semi-linear action over $A_\infty$ is homologous to one over $A^\bx$:

\begin{lemma}\label{lemma_ess_surj1}
Let $m\ge 0$ and let $\al:\Gamma\to \GL_n(A_\infty)$ be a continuous $1$-cocycle such that \[\al(\gamma)\in 1+c_1 M_n(A^\bx)+c_1c_0^m M_n(A_\infty)\] for all $\gamma\in\Gamma$. Then there exists $X\in 1+c_0^{m+1}M_n(A_\infty)$ such that the cohomologous $1$-cocycle $\al':\Gamma\to \GL_n(A_\infty)$, $\gamma\mapsto X^{-1}\al(\gamma)\gamma(X)$ satisfies \[\al'(\gamma)\in 1+c_1 M_n(A^\bx)+c_1c_0^{m+1} M_n(A_\infty)\] 
for all $\gamma\in\Gamma$.
\end{lemma}
\begin{proof}
Recall the decomposition $A_\infty=A^\bx\oplus A^\sub{n-i}$. One sees without difficulty there is a well-defined $\Gamma$-equivariant map of groups \begin{align*}\rho:1+c_1 M_n(A^\bx)+c_1c_0^m M_n(A_\infty)&\To M_n(A_\infty/c_1)/M_n(A^\bx/c_1)=M_n(A^\sub{n-i}/c_1).\\1+c_1 P+c_1c_0^m Q&\longmapsto Q\end{align*}
The $1$-cocycle $\rho \al:\Gamma\to M_n(A^\sub{n-i}/c_1) $ defines a class in $H^1_\sub{cont}(\Gamma, M_n(A^\sub{n-i}/c_1))$, which is killed by $\phi^{-1}(\mu)$ by Lemma \ref{lemma_vanishing2}(ii). That is, there exists $Y\in M_n(A^\sub{n-i})$ such that $\gamma(Y)-Y\equiv-\phi^{-1}(\mu)\rho(\al(\gamma))$ mod $c_1 M_n(A^\sub{n-i})$ for all $\gamma\in\Gamma$.

Set $X:=1+c_0^{m+1} Y$. Fixing $\gamma\in \Gamma$ we may write \[\al(\gamma)=1+c_1P+c_1c_0^mQ,\qquad \gamma(Y)-Y=\phi^{-1}(\mu)Z\] for some $P\in M_n(A^\bx)$, $Q\in M_n(A^\sub{n-i})$, $Z\in M_n(A^\sub{n-i})$, so that $Z\equiv -Q$ mod $c_0M_n(A^\sub{n-i})$; then
\begin{align*}
X^{-1}\al(\gamma)\gamma(X)&= (1+c_0^{m+1} Y)^{-1}(1+c_1 P+c_1c_0^m Q)(1+c_0^{m+1}Y+c_1c_0^m Z)\\
&=(1+c_0^{m+1} Y)^{-1}\big(1+c_0^{m+1}Y+c_1c_0^mZ+(c_1P+c_1c_0^mQ)(1+c_0^{m+1}Y+c_1c_0^mZ)\big)\\
&=1+(1+c_0^{m+1} Y)^{-1}\big(c_1c_0^mZ+(c_1P+c_1c_0^mQ)(1+c_0^{m+1}Y+c_1c_0^mZ)\big)\\
&\equiv 1+\big(c_1c_0^mZ+(c_1P+c_1c_0^mQ)(1+c_0^{m+1}Y+c_1c_0^mY)\big)\mod{c_1c_0^{m+1}M_n(A_\infty)}\\
&\equiv1+c_1c_0^mZ+c_1P+c_1c_0^mQ\mod{c_1c_0^{m+1}M_n(A_\infty)}\\
&\equiv1+c_1P\mod{c_1c_0^{m+1}M_n(A_\infty)}
\end{align*}
as required.
\end{proof}

\begin{lemma}\label{lemma_ess_surj2}
Let $m\ge 1$ and let $\al:\Gamma\to\GL_n(A_\infty)$ be a continuous $1$-cocycle such that $\al(\gamma)\in 1+c_1M_n(A_\infty)$ for all $\gamma\in\Gamma$. Then there exists $X\in 1+c_0M_n(A_\infty)$ such that the cohomologous $1$-cocycle $\al':\Gamma\to\GL_n(A_\infty)$, $\gamma\mapsto X^{-1}\al(\gamma)\gamma(X)$ satisfies $\al'(\gamma)\in 1+c_1M_n(A^\bx)$ for all $\gamma\in\Gamma$.
\end{lemma}
\begin{proof}
By applying Lemma \ref{lemma_ess_surj1} repeatedly we obtain a sequence of elements $X_m\in1+c_0^{m+1}M_n(A_\infty)$, $m\ge0$, such that the elements $Y_m:=X_0\cdots X_m$ satisfy the following: \[Y_m^{-1}\al(\gamma)\gamma(Y_m)\in 1+c_1M_n(A^\bx)+c_1c_0^{m+1}M_n(A_\infty)\] for all $\gamma\in\Gamma$ and $m\ge0$.

Since $Y_{m+1}-Y_m=X_0\cdots X_m(X_{m+1}-1)\in c_0^{m+1}M_n(A_\infty)$, the sequence $Y_m$ converges as $m\to\infty$ to some $X\in M_n(A_\infty)$; note that $X$ lies in $1+c_0M_n(A_\infty)$ (since all the $Y_m$ do and $c_0A_\infty$ is closed in $A_\infty$ by Lemma \ref{lemma_topology}) and hence is invertible. Moreover, the same argument shows that the sequence $Y_m^{-1}$ also converges, with limit necessarily given by $X^{-1}$ since multiplication is continuous.

 Taking the limit of the previous identity shows that \[X^{-1}\al(\gamma)\gamma(X)\in 1+c_1M_n(A^\bx)+c_1c_0^{m+1}M_n(A_\infty)\] for all $m\ge0$, since this subset is closed in $M_n(A_\infty)$ by Lemma \ref{lemma_topology}. Intersecting over all $m\ge0$ and using $\bigcap_{m\ge1}(c_1A^\bx+c_1c_0^mA_\infty)=c_1A^\bx$ (write $A_\infty=A^\bx\oplus A^\sub{n-i}$ and use $\bigcap_{m\ge1}c_1c_0^mA^\sub{n-i}=0$) then shows that in fact $X^{-1}\al(\gamma)\gamma(X)\in 1+c_1M_n(A^\bx)$, as desired.
\end{proof}

To prove essential surjectivity we need the following additional lemmas, which will allow us glue from the free case to obtain the general finite projective case. Let us say that a homomorphism of commutative rings $B\to C$ is a ``proj-isomorphism'' if and only if $P\mapsto P\otimes_BC$ induces a bijection from isomorphism classes of finite projective $B$-modules to isomorphism classes of finite projective $C$-modules. It is well-known that $B\to B_\sub{red}$ is a proj-isomorphism, and that if $B$ is $I$-adically complete and separated with respect to some ideal $I$ then $B\to B/I$ is a proj-isomorphism.

\begin{lemma}\label{lemma_proj-isoms}
\begin{enumerate}
\item The canonical map $A^\bx/c_1=(A^\bx/c_1)^\Gamma\to (A_\infty/c_1)^\Gamma$ is a proj-isomorphism.
\item The maps $A^\bx/(c_1,p)\to (A_\infty/c_1)^\Gamma/p\to A_\infty/(c_1,p)$ induce homeomorphisms on $\Spec$.
\end{enumerate}
\end{lemma}
\begin{proof}
$(A_\inf(R_\infty)/c_1)^\Gamma$ is $p$-adically complete and separated, and $(A_\inf(R_\infty)/c_1)^\Gamma/p\isoto (A_\inf(R_\infty)/(c_1,p))^\Gamma$ by Lemma \ref{lemma_vanishing2}(i), so the map $(A_\inf(R_\infty)/c_1)^\Gamma\to (A_\inf(R_\infty)/(c_1,p))^\Gamma$ is a proj-isomorphism. Similarly $A^\bx_\inf(R)/c_1\to A^\bx_\inf(R)/(c_1,p)$ is a proj-isomorphism by Lemma \ref{lemma_topology}(ii). Therefore part (i) reduces to the mod $p$ case, which we prove by checking that the inclusion $(A^\bx_\inf(R)/p)/c_1\subseteq ((A_\inf(R_\infty)/p)/c_1)^\Gamma$ becomes an isomorphism when we mod out by the nilradicals on both sides (and so is a proj-isomorphism). Indeed, ignoring contributions from the left side, any element on the right side is a finite sum of elements of the form $fU_1^{k_1}\cdots U_d^{k_d}$, where $k_1,\dots,k_d\in \bb Z[\tfrac 1p]\cap [0,1)$ are not all zero, $f\in (A^\bx_\inf(R)/p)/c_1$, and $(\ep^{k_i}-1)f=0$ in $(A^\bx_\inf(R)/p)/c_1$. Picking some $i$ for which $k_i\neq 0$, then the annihilator of $\ep^{k_i}-1$ in $(A^\bx_\inf(R)/p)/c_1$ is given by multiples of the element $\tfrac{c_1}{\ep^{k_i}-1}$, which is nilpotent in $\roi^\flat/c_1$ as desired.

We have also proved that the first arrow in part (ii) is a homeomorphism on $\Spec$, since the two sides have the same underlying reduced subschemes. But $A^\bx/(c_1,p)\to A_\infty/(c_1,p)$ is an extension of rings in characteristic $p$ such that any element $f\in A_\infty/(c_1,p)$ satisfies $f^{p^m}\in A^\bx/(c_1,p)$ for $m\gg0$; therefore it also induces a homeomorphism on $\Spec$.
\end{proof}

\begin{lemma}\label{lemma_localisation}
Let $R'$ be the $p$-completion of a localisation of $R$, and set $R_\infty':=R'\hat\otimes_RR_\infty$, i.e., the usual extension of $R'$ induced by the framing $\square:\roi\pid{\ul T^{\pm1}}\to R\to R'$. Write $A_\infty'=A_\inf(R_\infty')$ (resp.~$A_\inf(R_\infty')/p$) and $A^{\bx\prime}=A_\inf^\bx(R')$ (resp.~$A_\inf^\bx(R')/p$) in the integral (resp.~mod $p$) case. Then:
\begin{enumerate}
\item the canonical map \[\big(A_\infty/(c_1,p)\big)^\Gamma \otimes_{A^\bx/(c_1,p)} A^{\bx\prime}/(c_1,p)\To \big(A_\infty'/(c_1,p)\big)^\Gamma\]
is an isomorphism;
\item the canonical maps $A^\bx/(c_1,p)\to A^{\bx\prime}/(c_1,p)$ and $\big(A_\infty/(c_1,p)\big)^\Gamma\to \big(A_\infty'/(c_1,p)\big)^\Gamma$ are localisations.
\end{enumerate}
\end{lemma}
\begin{proof}
It suffices to treat the mod $p$ case (indeed, the assertions in fact only concern the mod $p$ case). 

(i): The $A^{\bx}/c_1$-module $A_{\infty}/c_1$ is the
direct sum of the free $A^{\bx}/c_1$-submodules $(A^{\bx}/c_1)U_1^{k_1}\cdots U_d^{k_d}$, for $k_1,\ldots,k_d\in \bb Z[\frac{1}{p}]\cap [0,1)$, and each direct summand is $\Gamma$-stable with the action of $\gamma_i$ given by multiplication by $[\varepsilon^{k_i}]$ for $i=1,\dots,d$. As $c_1A\subset \varphi^{-1}(\mu) A$, we see that the $\Gamma$-invariant part of each direct summand is $A^{\bx}/c_1$ if all $k_i$ are $0$, and $\frac{c_1}{[\varepsilon^k]-1}(A^{\bx}/c_1)U_1^{k_1}\cdots U_d^{k_d}$ otherwise, where $k$ is one of $k_1,\dots,k_d$ with the smallest $p$-adic valuation. The same computation applies to $A'_{\infty}/c_1$, and so  claim (i) follows from the flatness of $A^{\bx}/c_1\to A^{\bx\prime}/c_1$.

(ii): The map $A^{\bx}/c_1\to A^{\bx\prime}/c_1$ is \'etale by construction, and modulo the nilpotent element $\xi$ it is a localisation (namely, a quotient of the localisation $R/p\to R'/p$). Therefore it is a localisation; the same follows for the second map in claim (ii) by using (i).
\end{proof}

We may now prove essential surjectivity of the base change functor:

\begin{proposition}\label{proposition_ess_surj}
The base change functor of Theorem \ref{theorem_descent_framed_general} is essentially surjective.
\end{proposition}
\begin{proof}
Let $M\in\Rep_\Gamma^{c_1}(A_\infty)$, and assume initially that the finite projective $(A_\infty/c_1)^\Gamma$-module $(M/c_1)^\Gamma$ is actually finite free. Lifting a choice of a basis (which is also a basis of the finite free $A_\infty/c_1$-module $M/c_1=(M/c_1)^\Gamma\otimes_{(A_\infty/c_1)^\Gamma}A_\infty/c_1$) shows that $M$ is a finite free $A_\infty$-module; in fact, it lets us identify $M$ with $A_\infty^n$ in such a way that $(M/c_1)^\Gamma$ identifies compatibly with $((A_\infty/c_1)^\Gamma)^n\subseteq (A_\infty/c_1)^n$.


Letting $\Gamma$ act via the standard basis of $A_\infty^n$, the induced $1$-cocycle $\al:\Gamma\to \GL_n(A_\infty)$ representing the given $\Gamma$-action on $M$ therefore satisfies $\al(\gamma)\in1+c_1M_n(A_\infty)$ for all $\gamma\in\Gamma$.
%
%
By Lemma \ref{lemma_ess_surj2}, $\al$ is therefore cohomologous to a continuous $1$-cocycle $\al':\Gamma\to\GL_n(A_\infty)$ such that $\al'(\gamma)\in 1+c_1 M_n(A^\bx)$ for all $\gamma\in\Gamma$. In other words, $M$ is $A_\infty$-linearly and $\Gamma$-equivariantly isomorphic to $A_\infty^n$ equipped with new $\Gamma$-action $\gamma\star\ul x:=\al'(\gamma)(\gamma(\ul x))$ (where $\gamma(\ul x)$ refers to the coordinate-wise action of $\Gamma$ on $A_\infty^n$); but this is the base change of $N:=A^{\bx n}$ equipped with the analogous $\Gamma$-action. In conclusion, we have found $N\in \Rep_\Gamma^{c_1}(A^\bx)$ such that $M\cong N\otimes_{A^\bx}A_\infty$, which completes the proof of essential surjectivity under the extra assumption that $(M/c_1)^\Gamma$ is finite free.


We now proceed by a glueing argument to show that the base change functor is essentially surjective in general. The maps
\[\xymatrix{
A^\bx/c_1\ar@{=}[r] & (A^\bx/c_1)^\Gamma\ar[r] & (A_\infty/c_1)^\Gamma\ar[r] & A_\infty/c_1\\
A^\bx\ar[u]\ar[rrr] &&& A_\infty\ar[u]
}\]
all induce homeomorphisms on $\Spf$ by Lemma \ref{lemma_proj-isoms}(ii) (in the integral case, we take formal spectra on the bottom row with respect to the $(p,c_1)$-adic topology and on the top row with respect to the $p$-adic topology; in the mod $p$-case, the formal spectra on the bottom row are with respect to the $c_1$-adic topology and on the top row the formal spectra are actually usual spectra); the common topological space is moreover isomorphic to $\frak X:=\Spf(R)$, since the $c_1$-adic and $\xi$-adic topologies on $A^\bx/p$ are the same and $A^\bx/(\xi,p) =R/p$. Therefore $\frak X$ admits various structure sheaves : \[\roi_{A^\bx},\quad\roi_{A^\bx/c_1}=\roi_{A^\bx}/c_1,\quad \roi_{A_\infty},\quad \roi_{A_\infty}/c_1=\roi_{A_\infty/c_1},\quad \roi_{(A_\infty/c_1)^\Gamma}\] Moreover, sheafifiying the construction of the framed period ring from \S\ref{ss_framed} shows that, for any $p$-completed localisation $R\to R'$, with corresponding open $\frak U=\Spf R'\subseteq\frak X$, the sections on $\frak U$ of these sheaves are the corresponding period ring construction for $R'$ in place of $R$ with respect to the induced framing $\square:\roi\pid{\ul T^{\pm1}}\to R\to R'$. This also shows that the first four sheaves are equipped with $\Gamma$-actions (though the action on $\roi_{A^\bx/c_1}$ is trivial) and, using Lemma \ref{lemma_localisation}, that $(\roi_{A_\infty/c_1})^\Gamma=\roi_{(A_\infty/c_1)^\Gamma}$.

Glueing is now a formal process in which the main possible source of confusion is the notation. Let $M\in\Rep_\Gamma^{c_1}(A_\infty)$, and let $\cal M:=M\otimes_{A_\infty}\roi_{A_\infty}$ be the corresponding locally free sheaf of $\roi_{A_\infty}$-modules. By writing $M/c_1=(M/c_1)^\Gamma\otimes_{(A_\infty/c_1)^\Gamma}A_\infty/c_1$ we see that $(\cal M/c_1)^\Gamma=(M/c_1)^\Gamma\otimes_{(A_\infty/c_1)^\Gamma}\roi_{(A_\infty/c_1)^\Gamma}$ is the locally free sheaf of $\roi_{(A_\infty/c_1)^\Gamma}$-modules associated to the finite projective module $(M/c_1)^\Gamma$. There exists a finite open cover $\{\frak U_\lambda\}_\lambda$ of $\frak X$, where each $\frak U_\lambda$ corresponds to some $p$-completed localisation $R_\lambda$ of $R$, such that each $\Gamma(\frak U_\lambda, (\cal M/c_1)^\Gamma)$ is a finite free $\Gamma(\frak U_\lambda,\roi_{(A_\infty/c_1)^\Gamma})$-module. Therefore we may apply the finite free case of essential surjectivity to $M\otimes_{A_\infty}\Gamma(\frak U_\lambda,\roi_{A_\infty})\in\Rep_\Gamma^{c_1}(\Gamma(\frak U_\lambda,\roi_{A_\infty}))$ to obtain $N_\lambda\in \Rep_\Gamma^{c_1}(\Gamma(\frak U_\lambda,\roi_{A^\bx}))$ satisfying $N_\lambda\otimes_{\Gamma(\frak U_\lambda,\roi_{A^\bx})}\Gamma(\frak U_\lambda,\roi_{A_\infty})=M\otimes_{A_\infty}\Gamma(\frak U_\lambda,\roi_{A_\infty})$. Furthermore, since the base change functor has already been shown to be fully faithful in Proposition \ref{proposition_ff} (not only for $R$, but also for all the $R_\lambda$ and the $R_\lambda\hat\otimes_R R_{\lambda'}$), the $N_\lambda$ are unique and may be glued: the result is a locally free sheaf $\cal N$ of $\roi_{A^\bx}$-modules, equipped with a $\Gamma$-action which is the identity modulo $c_1$ (since it is the identity modulo $c_1$ after restriction to each $\frak U_\lambda$) and an isomorphism $\cal N\otimes_{\roi_{A^\bx}}\roi_{A_\infty}\cong \cal M$.

Then $N:=\Gamma(\frak X,\cal N)$ is a finite projective $A^\bx$-module equipped with semi-linear $\Gamma$-action and satisfying $N\otimes_{A^\bx}\roi_{A^\bx}=\cal N$. Taking global sections of the identities of previous paragraph shows moreover that $N\otimes_{A^\bx}A_\infty\cong M$ and that the $\Gamma$-action on $N$ is the identity modulo $c_1$. It remains only to check that the $\Gamma$-action on $N$ is continuous, as then we will have $N\in\Rep_\Gamma^{c_1}(A^\bx)$ and $N\otimes_{A^\bx}A_\infty\cong M$ as required; but the continuity is automatic by Lemma \ref{lemma_automatic_continuity}.
\end{proof}

\subsection{Additional lemmas on small generalised representations}
In this subsection we state some variants on results from \cite{BhattMorrowScholze1}, firstly that $L\eta$ can transform almost quasi-isomorphisms into quasi-isomorphisms:

\begin{lemma}\label{lemma_good}
Let $\frak M\subseteq A$ be an ideal of a ring and $f\in \frak M$ a non-zero-divisor. If $C\to D$ is a morphism of complexes of $A$-modules such that all cohomology groups of the cone are killed by $\frak M$ and all cohomology groups of $C\dotimes_AA/fA$ contain no non-zero elements killed by $\frak M^2$, then $L\eta_f C\to L\eta_fD$ is a quasi-isomorphism.
\end{lemma}
\begin{proof}
This follows by applying \cite[Lem.~5.14]{Bhatt2016} to $C\to D\to\op{cone}(C\to D)$.
\end{proof}

\begin{lemma}\label{lemma_Leta_isom}
\begin{enumerate}
\item Let $M\in\Rep_\Gamma^b(W_r(R_\infty))$ for some $b\in W_r(\frak m)$ and let $D\in D(W_r(\roi))$. Then any almost quasi-isomorphism $R\Gamma_\sub{cont}(\Gamma,M)\to D$ becomes a quasi-isomorphism after applying $L\eta_\pi$ for any non-zero-divisor $\pi\in W_r(\frak m)$.
\item Let $M\in\Rep_\Gamma^{<\mu}(A_\inf(R_\infty))$ and let $D\in D(A_\inf)$. Then any morphism $R\Gamma_\sub{cont}(\Gamma,M)\to D$, such that all cohomology groups of the cone are killed by $W(\frak m^\flat)$, becomes a quasi-isomorphism after applying $L\eta_\pi$ for any non-zero-divisor $\pi\in\bigcup_{j\ge0}\phi^{-j}(\mu)A_\inf$.
\end{enumerate}
\end{lemma}
\begin{proof}
(i) We begin by considering the $W_r(R_\infty)$ case. We are assuming that the generalised representation $M$ is trivial modulo $b$ for some $b\in W_r(\frak m)$; possibly after replacing $b$ by a factor of the form $[\varpi]$ for some small $\varpi\in\frak m\setminus\{0\}$, we may assume that the arbitrary element $\pi\in W_r(\frak m)$ is divisible by $b$. Since $L\eta_{\pi}=L\eta_{\pi/b}L\eta_{b}$ we may therefore assume $\pi=b$, i.e., henceforth $M$ is trivial modulo $\pi\in W_r(\frak m)$.

By Lemma \ref{lemma_good}, it is enough to show that the cohomology groups $H^i_\sub{cont}(\Gamma,M/\pi)$ contain no non-zero elements killed by $W_r(\frak m)$. But $M/\pi$ is now a trivial representation, thereby reducing the problem to case of the trivial representation $M=W_r(R_\infty)$, where it is a known calculation which was already used at the end of the proof of Lemma \ref{lemma_mu_1}; again see \cite[Lem.~9.7(iii)]{BhattMorrowScholze1}, bearing in mind that $\pi$ is a non-zero divisor.

(ii) Now we consider the case of $A_\sub{inf}(R_\infty)$. Since $M$ is trivial modulo $\phi^{-1}(\mu)$ by Corollary \ref{corollary_mu_smallness} (of course, it is even trivially modulo $\mu$ by Theorem \ref{theorem_<mu_implies_mu}, but we prefer to quote the weaker result), the same argument as in the $W_r(R_\infty)$-case allows us to fix some $j\ge1$, assume $M$ is trivial modulo $\phi^{-j}(\mu)$, and show that applying $L\eta_{\phi^{-j}}(\mu)$ yields a quasi-isomorphism. Replacing $M$ by $\phi^j_*M$ we may even suppose that $j=0$, though this has no effect on the proof other than simplifying notation. The same argument as in the previous case then reduces the problem to showing that the cohomology groups $H^i_\sub{cont}(\Gamma, A_\sub{inf}(R_\infty)/\mu)$ have no non-zero elements killed by $W(\frak m^\flat)^2$. Since these cohomology groups are $p$-adically complete and separated (by step (iii) in the proof of Lemma \ref{lemma_mu_1}), it is equivalently to check that they have no non-zero elements killed by $W(\frak m^\flat)$; but this is exactly what we showed in the proof of Lemma \ref{lemma_mu_1}.
\end{proof}

Secondly, we will need to know that $L\eta_\mu$ of certain group cohomologies can be computed in terms of the Koszul complexes:

\begin{lemma}\label{lemma_Koszul}
Let $g\in A$ be an element of a commutative ring, $N$ a $g$-torsion-free $A$-module, and $\delta_1,\dots,\delta_d$ pairwise commuting, $A$-linear endomorphisms of $N$; assume that $\delta_i\in g\End_A(N)$ for all $i$
\begin{enumerate}
\item There is a natural quasi-isomorphism $L\eta_g K(\delta_1,\dots,\delta_d;N)\simeq K(\tfrac{\delta_1}g,\dots,\tfrac{\delta_d}g;N)$.
\item If $\delta_i\in g\Aut_A(N)$ for at least one $i$, then there exists an isomorphism of $A$-modules \[H^n(K(\delta_1,\dots,\delta_d;N))\cong (N/gN)^{{d-1}\choose{n-1}}\] for each $n\ge 1$.
\end{enumerate}
\end{lemma}
\begin{proof}
Write $\delta_i=g\delta_i'$ for $i=1,\dots,d$, where $\delta_i\in\End_A(N)$, and note that the endomorphisms $\delta_i'$ commute (since $N$ is $g$-torsion-free, this follows at once from the commutativity of the $\delta_i$). Let $A'$ be the $A$-subalgebra of $\End_A(N)$ generated by $\delta_1',\dots,\delta_d'$; in case (ii) we also add as a generator the inverse of whichever $\delta_i'$ is invertible. Then the $A$-module structure on $N$ factors through the commutative $A$-algebra $A'$ so we may view $N$ as an $A'$-module and apply \cite[Lem.~7.9--7.10]{BhattMorrowScholze1}.
\end{proof}

\begin{proposition}\label{proposition_Leta_framed}
Let $N\in \Rep_\Gamma^\mu(A_\inf^\bx(R))$ and set $M:=N\otimes_{A_\sub{inf}^\bx(R)}A_\sub{inf}(R_\infty) \in\Rep_\Gamma^\mu(A_\inf(R_\infty))$; fix $s\ge0$ and set $N_s:=N\otimes_{A_\inf\pid{\ul U^{\pm1}}}A_\inf\pid{\ul U^{\pm1/p^s}}$. Then:
\begin{enumerate}
\item The canonical map \[R\Gamma_\sub{cont}(\Gamma,N_s)\To R\Gamma_\sub{cont}(\Gamma,M)\] becomes a quasi-isomorphism after applying $L\eta_{\phi^{-s-1}(\mu)}$ (hence also after applying $L\eta_{\phi^{-s}(\mu)}=L\eta_{\phi^{-s}(\xi)}L\eta_{\phi^{-s-1}(\mu)}$).
\item There is a natural quasi-isomorphism \[L\eta_{\phi^{-s}(\mu)} R\Gamma_\sub{cont}(\Gamma,N_s)\simeq K(\tfrac{\gamma_1-1}{\phi^{-s}(\mu)},\dots,\tfrac{\gamma_d-1}{\phi^{-s}(\mu)};N_s).\]
\end{enumerate}
\end{proposition}
\begin{proof}
(i): As we recalled in \S\ref{ss_framed}, we may identify $A_\inf(R_\infty)$ with the $(p,\xi)=(p,\phi^{-s-1}(\mu))$-adic completion of $A_\inf^\bx(R)\otimes_{A_\inf\pid{\ul U^{\pm1}}}A_\inf\pid{\ul U^{\pm1/p^\infty}}$; therefore $N_s\to M$ is split injective (compatibly with the $\Gamma$-actions) with complement given by the $(p,\phi^{-s-1}(\mu))$-adic completion of \[P:=\bigoplus_{\substack{k_1,\dots,k_d\in\bb Z[\tfrac1p]\cap[0,1)\\\sub{not all in $p^{-s}\bb Z$}}}N_sU_1^{k_1}\cdots U_d^{k_d}.\] We now proceed differently to the homotopy argument of \cite[Lem.~9.6]{BhattMorrowScholze1}, though that argument can probably be adapted to the current situation.

We must show that $L\eta_{\phi^{-s-1}(\mu)}$ kills the derived $(p,\phi^{-s-1}(\mu))$-adic completion of $C:=R\Gamma(\bb Z^d,P)$. Since derived $(p,\phi^{-s-1}(\mu))$-adic completion is the same as derived $p$-adic completion followed by derived $\phi^{-s-1}(\mu)$-adic completion, and the latter commutes with $L\eta_{\phi^{-s-1}(\mu)}$ by \cite[Lem.~6.20]{BhattMorrowScholze1}, it is enough to show that $L\eta_{\phi^{-s-1}(\mu)}$ kills the derived $p$-adic completion of $C$. In other words, since $L\eta_{\phi^{-s-1}(\mu)}$ kills the $\phi^{-s-1}(\mu)$-torsion in each cohomology group \cite[Lem.~6.4]{BhattMorrowScholze1}, we should show that all cohomology groups of the derived $p$-adic completion $\hat C$ of $C$ are killed by $\phi^{-s-1}(\mu)$.

The group cohomology of each direct summand is computed by a Koszul complex \[R\Gamma(\bb Z^d,NU_1^{k_1}\cdots U_d^{k_d})\simeq K([\ep^{k_1}]\gamma_1-1,\dots,[\ep^{k_d}]\gamma_d-1;N)\tag{\dag}\] We claim that all cohomologies of this complex are $p$-torsion-free and killed by $\phi^{-s-1}(\mu)$. The $p$-torsion-freeness will imply that $H^i(C)=\hat\bigoplus_{k_1,\dots,k_d}H^i(\bb Z^d,NU_1^{k_1}\cdots U_d^{k_d})$ (where the hat denotes usual $p$-adic completion), which we then see is killed by $\phi^{-s-1}(\mu)$ since the same is true of all the summands. The claim is proved by a standard Koszul complex calculation, but we include the argument anyway for the sake of completeness.

Up to reordering, we may suppose that $k_1$ has the smallest $p$-adic valuation of $k_1,\dots,k_d$; since at least one of these terms is not in $p^{-s}\bb Z$, it follows that $\nu_p(k_1)\le-s-1$. Since our representation $N$ is assumed to be trivial mod $\mu$, we may write $\gamma_i=1+\mu\delta_i$ for some $\delta_i\in\End_{A_\sub{inf}}(N)$; thus we see that $[\ep^{k_i}]\gamma_i-1=[\ep^{k_i}]-1+\mu[\ep^{k_i}]\delta_i$ is divisible by $[\ep^{k_1}]-1$, for each $i$, and moreover that \[[\ep^{k_1}]\gamma_1-1=([\ep^{k_1}]-1)(1+\tfrac{\mu}{[\ep^{k_1}]-1}[\ep^{k_1}]\delta_1)\] differs from $[\ep^{k_1}]-1$ by an automorphism of $N$ (since $\tfrac{\mu}{[\ep^{k_1}]-1}$ is equal to $\xi_r$ up to a unit of $A_\inf$, where $r=-\nu_p(k_1)$, and $N$ is $\xi_r$-adically complete). It therefore follows from Lemma \ref{lemma_Koszul}(ii) that each cohomology group of the Koszul complex (\dag) is a finite direct sum of copies of $N/([\ep^{k_1}]-1)N$; this is indeed $p$-torsion-free and killed by $[\ep^{1/p^{s+1}}]-1=\phi^{-s-1}(\mu)$, as required.

(ii): Since the group cohomology is given by $K({\gamma_1-1},\dots,{\gamma_d-1};N_s)$ and $\Gamma$ acts as the identity on $A_\inf^\bx(R)\otimes_{A_\inf\pid{\ul U^{\pm1}}}A_\inf\pid{\ul U^{\pm1/p^s}}/\phi^{-s}(\mu)$ hence also as the identity on $N_s/\phi^{-s}(\mu)$, this is an instance of Lemma \ref{lemma_Koszul}(i).
\end{proof}

\begin{corollary}\label{corollary_completeness}
Let $M\in\Rep^\mu_\Gamma(A_\inf(R_\infty))$ and $s\ge0$. Then the canonical map \[L\eta_{\phi^{-s}(\mu)} R\Gamma_\sub{cont}(\Gamma,M)\To\op{Rlim}_r(L\eta_{\phi^{-s}(\mu)} R\Gamma_\sub{cont}(\Gamma,M))/\tilde\xi_r\] is an equivalence.
\end{corollary}
\begin{proof}
By Theorem \ref{theorem_descent_to_framed} there exists $N\in \Rep^\mu_\Gamma(A_\inf^\bx(R))$ such that $M=N\otimes_{A_\sub{inf}^\bx(R)}A_\sub{inf}(R_\infty)$. Then Proposition \ref{proposition_Leta_framed} describes $L\eta_{\phi^{-s}(\mu)} R\Gamma_\sub{cont}(\Gamma,M)$ in terms of a Koszul complex, each term of which is a finite direct sum of copies of $N_s=N\otimes_{A_\inf\pid{\ul U^{\pm1}}}A_\inf\pid{\ul U^{\pm1/p^s}}$; therefore it is enough to show that $N_s\isoto\projlim_rN_s/\tilde\xi_r$. Since $N_s$ is a finite projective $A_\inf^\bx(R)$-module, this reduces to the isomorphism $A_\inf^\bx(R)\isoto\projlim_r A_\inf^\bx(R)/\tilde\xi_r$ of Lemma \ref{lemma_topology}(v).
\end{proof}

\subsection{Frobenius structures on generalised representations}\label{ss_Frobenius}
The period rings $A_\inf$ and $A_\inf(R_\infty)$ are equipped with their usual Witt vector Frobenius automorphisms $\phi$; furthermore, elementary deformation arguments show that $\phi$ extends from $A_\inf$ to a unique ring endomorphism of $A_\inf^\bx(R)$ satisfying $\phi(U_i)=U_i^p$ for $i=1,\dots,d$ and lifting the absolute Frobenius modulo $p$. These Frobenii all commute with the $\Gamma$-actions and are compatible with the inclusions $A_\inf\subseteq A_\inf^\bx(R)\subseteq A_\inf(R_\infty)$. In this subsection we first prove an analogue of Theorem \ref{theorem_descent_to_framed} taking them into account, then establish moreover the compatibility with certain induced filtrations. The relevant categories are defined as follows:

\begin{definition}\label{definition_phi}
{\em Generalised $A_\inf(R_\infty)$-representations with Frobenius}: $\Rep_\Gamma^\mu(A_\inf(R_\infty),\phi)$ consists of pairs $(M,\phi_M)$ where $M\in \Rep_\Gamma^\mu(A_\inf(R_\infty))$ and $\phi_M:M[\tfrac1\xi]\to M[\tfrac1{\tilde\xi}]$ is a $\Gamma$-equivariant, $\phi$-semi-linear isomorphism (equivalently, $\phi_M:(\phi^*M)[\tfrac1{\tilde\xi}]\to M[\tfrac1{\tilde\xi}]$ is a $\Gamma$-equivariant, $A_\inf(R_\infty)[\tfrac1{\tilde\xi}]$-linear isomorphism, where $\phi^*M:=M\otimes_{A_\inf(R_\infty),\phi}A_\inf(R_\infty)$).

{\em Generalised $A_\inf^\bx(R)$-representations with Frobenius}: $\Rep_\Gamma^\mu(A_\inf^\bx(R),\phi)$ consists of pairs $(N,\phi_N)$ where $N\in \Rep_\Gamma^\mu(A_\inf^\bx(R))$ and $\phi_N:(\phi^*N)[\tfrac1{\tilde\xi}]\to N[\tfrac1{\tilde\xi}]$ is a $\Gamma$-equivariant isomorphism of $A_\inf^\bx(R)[\tfrac1{\tilde\xi}]$-modules, where $\phi^*N:=N\otimes_{A_\inf^\bx(R),\phi}A_\inf^\bx(R)$.
\end{definition}

The following is the desired analogue of Theorem \ref{theorem_descent_to_framed}:

\begin{theorem}\label{theorem_descent_to_framed_with_phi}:
The base change functor \[-\otimes_{A_\sub{inf}^\bx(R)}A_\sub{inf}(R_\infty):\Rep_\Gamma^\mu(A_\inf^\bx(R),\phi)\To \Rep_\Gamma^\mu(A_\inf(R_\infty),\phi)\] is an equivalence of categories.
\end{theorem}
\begin{proof}

To check fully faithfulness we fix $(N_i,\phi_{N_i})\in \Rep_\Gamma^\mu(A_\inf^\bx(R),\phi)$, $i=1,2$, and a morphism $f:N_1\to N_2$ in $\Rep_\Gamma^\mu(A_\inf^\bx(R))$; since we already know that the functor is fully faithful without Frobenius structure, we must show that $f$ respects the $\phi_{N_i}$ if it does so after the base change $-\otimes_{A_\inf^\bx(R)}A_\inf(R_\infty)$. This is an immediate consequence of the fact that $A_\inf^\bx(R)\to A_\inf(R_\infty)$ is injective.

Next we check essential surjectivity. Given $(M,\phi_M)\in \Rep_\Gamma^\mu(A_\inf(R_\infty),\phi)$, we already know that $M$ maybe written as $N\otimes_{A_\inf^\bx(R)}A_\inf(R_\infty)$ for some unique $N\in\Rep_\Gamma^\mu(A_\inf^\bx(R))$; we must show that $\phi_M$ also descends to $N$. We may replace $\phi_M$ by ${\tilde{\xi}}^r\phi_M$ for $r\gg0$ so that it is induced by an $A_\inf(R_\infty)$-linear map $\phi_M:\phi^*M=(\phi^*N)\otimes_{A_\inf^\bx(R)}A_\inf(R_\infty)\to M=N\otimes_{A_\inf^\bx(R)}A_\inf(R_\infty)$, i.e., so that the Frobenius is ``effective''. But $\phi^*N$, equipped with its diagonal $\Gamma$-action,  belongs to $\Rep_\Gamma^\mu(A_\inf^\bx(R))$ since the $\Gamma$-action on it is trivial modulo $\phi(\mu)=\tilde\xi\mu$, hence also trivial modulo $\mu$. So the known equivalence without Frobenius structures indeed implies that $\phi_M$ is induced by a unique morphism $\phi_N:\phi^*N\to N$; note that $\phi_N$ is necessarily an isomorphism after inverting $\tilde\xi$, since there exists $\psi_M:M\to\phi^*M$ such that $\psi_M\circ\phi_M$ and $\phi_M\circ\psi_M$ are multiplication by $\tilde\xi^s$ for some $s\gg0$, and we may descend $\psi_M$ to $N$ by arguing as we did for $\phi_M$.
\end{proof}

\begin{remark}\label{remark_cats_with_non-fixed_phi}
Whenever possible we prefer to establish our results first without Frobenius and then add the Frobenius structure; the second step is typically straightforward, as the previous proof illustrates. For example, in the context of Remark \ref{remark_Delta}, the base change equivalence $\Rep_\Gamma^{\mu}(A_\inf(R_\infty))\quis \Rep_\Delta^{\mu}(A_\inf(\res R))$ (by combining Theorem \ref{theorem_<mu_implies_mu} with the results of that remark) is easily checked to admit an analogue with Frobenius structures \begin{equation}\Rep_\Gamma^{\mu}(A_\inf(R_\infty),\phi)\quis \Rep_\Delta^{\mu}(A_\inf(\res R),\phi),\label{eqn_Gamma_vs_Delta_with_phi}\end{equation} where $\Rep_\Delta^{\mu}(A_\inf(\res R),\phi)$ is defined analogously to $\Rep_\Delta^{\mu}(A_\inf(R_\infty),\phi)$.

Nevertheless, several results will require additional properties of the representation which follow from the existence of some Frobenius, without actually requiring any fixed choice of it. Therefore we denote by $\Rep_\Gamma^\mu(A_\inf(R_\infty),\exists\phi)$ the essential image of the forgetful functor \[\Rep_\Gamma^\mu(A_\inf(R_\infty),\phi)\To \Rep_\Gamma^\mu(A_\inf(R_\infty)),\qquad (M,\phi)\mapsto M,\] i.e., it is the full subcategory consisting of $M\in \Rep_\Gamma^\mu(A_\inf(R_\infty))$ for which there exists some $\phi_M:M[\tfrac1\xi]\isoto M[\tfrac1{\tilde\xi}]$ (which we do not fix). The category $\Rep_\Gamma^\mu(A_\inf^\bx(R),\exists\phi)$ is defined analogously, and clearly the equivalence of the previous theorem also holds for these categories.

For example, Proposition \ref{proposition_reps_vs_q_connections_with_phi} and Lemma \ref{lemma_phi_implies_xi_nilp}(iii) show that if $N\in \Rep_\Gamma^\mu(A_\inf^\bx(R),\exists\phi)$ then the operators $\tfrac1\mu(\gamma_i-1)$ are automatically $\tilde\xi$-adically quasi-nilpotent on $N$; this property is intrinsic to the $\Gamma$-action but is not satisfied for arbitrary $N\in \Rep_\Gamma^\mu(A_\inf^\bx(R))$.
\end{remark}

\subsection{Examples of relative Breuil--Kisin--Fargues modules}\label{ss_examples}
We continue to let $R$ denote a $p$-adically complete, small, formally smooth $\roi$-algebra with a fixed choice of framing. In Section \ref{section_BKF_on_proetale} we will reinterpret $\Rep_\Gamma^\mu(A_\inf(R_\infty),\phi)$ as {\em relative Breuil--Kisin--Fargues modules on $\Spf R$} and show in particular that it does not depend on the choice of framing (assuming $\Spec(R/pR)$ is connected, this already follows up to the ambiguity of an algebraic closure of $\Frac R$ from the equivalence~(\ref{eqn_Gamma_vs_Delta_with_phi})). Here we present examples of such relative Breuil--Kisin--Fargues modules from the point of view of generalised representations.

\subsubsection{Dieudonn\'e theory}\label{sss_Dieudonne}
Our first examples of small generalised representations over $A_\inf(R_\infty)$ come from $p$-divisible groups via the following Dieudonn\'e classification theorem where $\op{BT}(-)$ denotes the category of $p$-divisible groups over a ring:

\begin{theorem}\label{theorem_p_div}
There is a fully faithful embedding \[\Phi^\sub{BKF}_R:\op{BT}(R)\To\Rep_\Gamma^\mu(A_\inf(R_\infty),\phi)\] whose essential image consists of those $(M,\phi_M)$ such that $M\subseteq\phi_M(M)\subseteq\tfrac{1}{\tilde\xi}M$.
\end{theorem}

We begin by explaining a prismatic proof of the theorem using J.~Anschutz and A.-C.~Le Bras' \cite{AnschutzLeBras2019} general classification theorem for $p$-divisible groups in terms of certain filtered prismatic crystals:

\begin{proof}[Proof via prismatic crystals]
Appealing to the equivalences of diagram (\ref{eqn_intro}) of the introduction (most of which remain to be proved in Sections \ref{section_q_connections} and \ref{section_prismatic_crystals}), the goal is to show that $\op{BT}(R)$ is equivalent to the category of those locally finite free $F$-crystals $(\cal F,\phi_\cal F)\in \op{CR}_{\Prism}(R/(A_\inf,\xi), \phi)$ such that $\cal F\subseteq\phi_\cal F(\cal F)\subseteq\tfrac1{\xi}\cal F$ (we refer ahead to Corollary \ref{corollary_main_thm_with_phi} for the notation, and use the same Frobenius-identification as in the first line of the proof of Proposition \ref{lem:PrismFinObjCov}). 

Since the absolute prismatic site of $\roi$ has initial object $(A_\inf,\xi)$, we see that the prismatic site of $R$ over $(A_\inf,\xi)$ is the same as its absolute prismatic site \cite[Lem.~4.7]{BhattScholze2019}. Therefore the category of crystals in the previous paragraph is antiequivalent to Anschutz--Le Bras' prismatic Dieudonn\'e crystals $\op{DM}(R)$, and our desired equivalence is their antiequivalence $\cal M_\Prism:\op{BT}(R)\quis \op{DM}(R)$ \cite[Thm.~4.6.9]{AnschutzLeBras2019} (more precisely, Anschutz--Le Bras' general result is formulated in terms of filtered crystals, but in this case the filtration is irrelevant, i.e., $\op{DF}(R)\quis \op{DM}(R)$, by \cite[Prop.~5.2.3]{AnschutzLeBras2019} since $R$ admits the perfectoid quasi-syntomic cover $R_\infty$).
\end{proof}

In the remainder of this subsection we sketch a more classical, non-prismatic proof of Theorem~\ref{theorem_p_div} in the case $p\neq 2$, as it was part of our original motivation for studying the category of generalised representations $\Rep_\Gamma^\mu(A_\inf(R_\infty),\phi)$.

\begin{remark}\label{remark_BT_without_framing}
We remark that we may also state Theorem \ref{theorem_p_div} independently of any choice of framing: there exists a fully faithful embedding $\Phi^\sub{BKF}_{\res R/R}:\op{BT}(R)\to\Rep_\Delta^\mu(A_\inf(\res R),\phi)$ whose essential image consists of those $(M,\phi_M)$ such that $M\subseteq\phi_M(M)\subseteq\tfrac{1}{\tilde\xi}M$. Here $\res R$ and $\Delta$ are as in Remark \ref{remark_Delta} (and so we implicitly assume that $\Spec(R/pR)$ is connected). Indeed, the naturality of Theorem \ref{theorem_p_div_perfectoid} below refers to the fact that, given a morphism $A\to B$ of perfectoid rings and a $p$-divisible group $G$ over $A$, there is a functorial identification $\Phi_{B}^{\sub{inf}}(G\otimes_AB)\cong\Phi_A^\sub{inf}(A)\otimes_{A_\inf(A)}A_\inf(B)$. Therefore, using the definition $\Phi_R^\sub{BKF}:=\Phi^\sub{inf}_{R_\infty}(-\otimes_RR_\infty)$ below, we see that the composition \[\op{BT}(R)\xto{\Phi_R^\sub{BKF}}\Rep_\Gamma^\mu(A_\inf(R_\infty),\phi)\stackrel{\sub{(\ref{eqn_Gamma_vs_Delta_with_phi})}}\quis \Rep_\Delta^\mu(A_\inf(\res R),\phi)\] is the framing-independent functor $\Phi^\sub{BKF}_{\res R/R}$. Nevertheless, we must use the choice of framing to show that the functor $\Phi_R^\sub{BKF}$ is indeed an equivalence.
\end{remark}

The case $R=\roi$ of the theorem (in which case the codomain is the category of usual Breuil--Kisin--Fargues modules) is due to Fargues (see \cite[Thm.~14.4.1]{ScholzeWeinstein2020}). The following already known generalisation of Fargues' result are an input to the proof of Theorem \ref{theorem_p_div}:

\begin{theorem}[Lau et al.]\label{theorem_p_div_perfectoid}
Let $A$ be a perfectoid ring. Then there is a natural equivalence of categories \[\Phi_{A}^{\sub{inf}}:\op{BT}(A)\quis\categ{6cm}{finite projective $A_\inf(A)$-modules $M$ equipped with a Frobenius semi-linear isomorphism $\phi_M:M[\tfrac1\xi]\isoto M[\tfrac1{\tilde\xi}]$ such that $M\subseteq\phi_M(M)\subseteq\tfrac{1}{\tilde\xi}M$}\] (where $\xi\in A_\inf(A)$ is a generator of the kernel of Fontaine's map $A_\inf(A)\to A$). Moreover, this equivalence is compatible with crystalline Dieudonn\'e theory, as will be explained in the proof of Lemma~\ref{lemma_PhiBKF_defined}.
\end{theorem}
\begin{proof}[Remarks on the proof]
The case in which $A$ is a perfect field of characteristic $p$ is classical. Berthelot \cite{Berthelot1980} treated the case of perfect valuation rings of characteristic $p$, from which Gabber then deduced the case of a general perfect $\bb F_p$-algebra $A$ using descent methods. In mixed characteristic, we have already mentioned that the theorem is due to Fargues when $A$ is the ring of integers of a complete, non-archimedean, algebraically closed field, and in general when $p\neq2$ it is due to Lau \cite{Lau2018}.

Meanwhile, Scholze \cite[Thm.~17.5.2]{ScholzeWeinstein2020} has shown that the theorem follows in general from Berthelot's result, using v-descent in a similar fashion to Gabber and using  the results of \cite{BhattMorrowScholze1} to control the relationship to usual crystalline Dieudonn\'e theory.\phantom\qedhere
\end{proof}

\begin{definition}\label{definition_BKF_of_perfectoid}
For a perfectoid ring $A$, we denote by $\BKF(A,\phi)$ the category of Breuil--Kisin--Fargues modules over $A$, i.e., finite projective $A_\inf(A)$-modules $M$ equipped with a Frobenius semi-linear isomorphism $\phi_M:M[\tfrac1\xi]\isoto M[\tfrac1{\tilde\xi}]$. Also let $\BKF(A,\phi,[-1,0])$ denote its full subcategory where the height condition $M\subseteq\phi_M(M)\subseteq\tfrac{1}{\tilde\xi}M$ holds, i.e., the target category in Theorem \ref{theorem_p_div_perfectoid}.
\end{definition}

We now begin to explain the non-prismatic proof of Theorem \ref{theorem_p_div}; we let
\[\Rep^\mu_\Gamma(A_\inf(R_\infty),\phi,[-1,0])\subseteq \Rep^\mu_\Gamma(A_\inf(R_\infty),\phi)\] denote the full subcategory consisting of those $(M,\phi_M)$ such that $M\subseteq\phi_M(M)\subseteq\tfrac{1}{\tilde\xi}M$, i.e., the desired image in Theorem \ref{theorem_p_div}.
Given a $p$-divisible group $G$ over $R$, the corresponding $p$-divisible group $G\otimes_RR_\infty$ inherits a $\Gamma$-action, which in turn implies by functoriality that the $A_\inf(R_\infty)$-module $\Phi_{R_\infty}^\inf(G\otimes_RR_\infty)$ is equipped with a semi-linear $\Gamma$-action; here we are applying Theorem \ref{theorem_p_div_perfectoid} to the perfectoid ring $A=R_\infty$. In order to show that this defines a functor \[\Phi^\sub{BKF}_R:=\Phi_{R_\infty}^\inf(-\otimes_RR_\infty):\op{BT}(R)\To \Rep^\mu_\Gamma(A_\inf(R_\infty),\phi,[-1,0]),\] we must check the following:

\begin{lemma}\label{lemma_PhiBKF_defined}
For any $p$-divisible group $G$ over $R$, the induced $\Gamma$-action on $\Phi_{R_\infty}^\inf(G\otimes_RR_\infty)$ is continuous and trivial modulo $\mu$.
\end{lemma}
\begin{proof}
The compatibility of the perfectoid Dieudonn\'e theory with the classical crystalline theory of \cite{BerthelotBreenMessing1982} amounts to the following statement: for any $p$-divisible group $H$ over $R_\infty$, the base change $\Phi_{R_\infty}^\inf(H)\otimes_{A_\inf(R_\infty)}A_\crys(R_\infty)/p^n$ is given by evaluating $\bb D(H\otimes_{R_\infty}R_\infty/p^n)$ (namely, the crystalline Dieudonn\'e module of the $p$-divisible group $H\otimes_{R_\infty}R_\infty/p^n$ over $R_\infty/p^n$; this is a crystal on the crystalline site of $R_\infty/p^n$ over $\bb Z_p$) on the pd-thickening $A_\crys(R_\infty)/p^n$ of $R_\infty/p^n$. That is,
\[\Phi_{R_\infty}^\inf(H)\otimes_{A_\inf(R_\infty)}A_\crys(R_\infty)/p^n=\Gamma\big(\Spec R_\infty/p^n\into \Spec A_\crys(R_\infty)/p^n,\, \bb D(H\otimes_{R_\infty}R_\infty/p^n)\big).\] But the kernel of the map $A_\inf(R_\infty)/\mu\to R_\infty$ has divided powers (since $A_\inf(R_\infty)/\mu$ and $R_\infty$ are $p$-torsion-free, an examination of the the $p$-adic valuations of factorials, e.g., \cite[Lem.~2.35]{BhattScholze2019}, reduces this claim to checking that $\xi^p/p!\in A_\inf(R_\infty)/\mu$; but in $A_\inf$ we have $\xi^p\equiv \tilde\xi\equiv\mu^{p-1}$ mod $p$, so $\xi^p$ is divisible by $p$ in $A_\inf/\mu$), whence the quotient map $A_\inf(R_\infty)\to A_\inf(R_\infty)/\mu$ factors through $A_\sub{crys}(R_\infty)$; base changing the above expression through this factorisation and using the the crystal property, we see that \[\Phi_{R_\infty}^\inf(H)/(p^n,\mu)=\Gamma\big(\Spec R_\infty/p^n\into \Spec A_\inf(R_\infty)/(p^n,\mu),\, \bb D(H\otimes_{R_\infty}R_\infty/p^n)\big).\]

We will apply this to $H=G\otimes_RR_\infty$, in which case we can also evaluate the Dieudonn\'e crystal of $G\otimes_RR/p^n$ on the pd-thickening $A_\inf^\bx(R)/(p^n,\mu)\to R/p^n$; from base change one deduces that \[\Phi_{R_\infty}^\inf(G\otimes_RR_\infty)/(p^n,\mu)=\Gamma\big(\Spec R/p^n\into \Spec A_\inf^\bx(R)/(p^n,\mu),\, \bb D(G\otimes_RR/p^n)\big)\otimes_{A_\inf^\bx(R)/(p^n,\mu)}A_\inf(R_\infty)/(p^n,\mu).\] By naturality this identification is compatible with the $\Gamma$-actions, where we recall from the discussion of \S\ref{ss_framed} that $\Gamma$ acts trivially on $A^\bx_\sub{inf}(R)/\mu$. Base change also shows that such identifications are compatible as $n$ increases, whence one obtains \[\Phi_{R_\infty}^\inf(G\otimes_RR_\infty)/\mu=\bigg(\projlim_n\Gamma\big(\Spec R/p^n\into \Spec A_\inf^\bx(R)/(p^n,\mu),\, \bb D(G\otimes_RR/p^n)\big)\bigg)\otimes_{A_\inf^\bx(R)/\mu}A_\inf(R_\infty)/\mu,\] witnessing that the $\Gamma$-action on $\Phi_{R_\infty}^\inf(G\otimes_RR_\infty)$ is trivial modulo $\mu$.



Finally, the continuity of the $\Gamma$-action on $\Phi_{R_\infty}^\inf(G\otimes_RR_\infty)$ is automatic by Lemma \ref{lemma_automatic_continuity}.
\end{proof}

Now that our desired functor $\Phi^\sub{BKF}_R$ has been shown to be well-defined, we may easily check that it is fully faithful:

\begin{proof}[Proof that $\Phi^\sub{BKF}_R$ is fully faithful]
By construction there is a commutative diagram
\[\xymatrix@C=2cm{
\op{BT}(R_\infty)\ar[r]^-{\Phi_{R_\infty}^\inf} &\BKF(R_\infty,\phi,[-1,0])\\
\op{BT}(R)\ar[r]_-{\Phi_R^\sub{BKF}}\ar[u]^{-\otimes_RR_\infty} & \Rep^\mu_\Gamma(A_\inf(R_\infty),\phi,[-1,0])\ar[u]_{\text{forget $\Gamma$ action}}
}\]
Since $-\otimes_RR_\infty$ is faithful, so is $\Phi_R^\sub{BKF}$.

Next, since $\Phi_{R_\infty}^\inf$ is an equivalence of categories by Theorem \ref{theorem_p_div_perfectoid} one sees that, given $G_1,G_2\in\op{BT}(R)$, then any morphism $\Phi_R^\sub{BKF}(G_1)\to \Phi_R^\sub{BKF}(G_2)$ naturally induces a $\Gamma$-equivariant morphism $G_1\otimes_RR_\infty\to G_2\otimes_RR_\infty$; passing to $\Gamma$-invariants this descends to the desired morphism $G_1\to G_2$ and so establishes fullness.
\end{proof}

\begin{proof}[Sketch that $\Phi_R^\sub{BKF}$ is essentially surjective]
Given $M\in\Rep^\mu_\Gamma(A_\inf(R_\infty),\phi,[-1,0])$ and forgetting the $\Gamma$-action on $M$, Theorem \ref{theorem_p_div_perfectoid} shows that there exists a unique $p$-divisible group $G$ over $R_\infty$ such that $\Phi^\sub{inf}_{R_\infty}(G)=M$. To prove essential surjectivity it must be shown in particular that $G$ descends from $R_\infty$ to $R$; that is, we must construct a descent isomorphism $G\otimes_{R_\infty,p_1}R_\infty(1)\cong G\otimes_{R_\infty,p_2}R_\infty(1)$, where $R_\infty(1)$ denotes the $p$-adic completion of $R_\infty\otimes_RR_\infty$. There are three key ingredients to this construction of this descent: Grothendieck--Messing's deformation theory \cite{Messing1972} to partly reduce the problem to $R_\infty(1)/p$; Lau's Dieudonn\'e theory \cite{Lau2018} over complete intersection semiperfect rings such as $R_\infty(1)/p$ (which is where we require $p\neq 2$); and the existence of a filtered crystal $\cal F$ on the big crystalline site of $R$ associated to $M$ (in the terminology of Definition \ref{definition_associated} and appealing to Theorems \ref{theorem_crystal_genrep} and  \ref{theorem_Admissibility}, $\cal F\in \CR(R/A_\crys,\phi,\op{SatFil})$ is the unique F-crystal with saturated filtration which is associated to $M$). The detailed proof appears in \S\ref{ss_DieudonneII}.
\end{proof}

We finish our discussion of Dieudonn\'e theory by explaining its compatibility with Tate modules. We adopt the framing-independent point of view of Remark \ref{remark_BT_without_framing} (and so assume that $\Spec(R/pR)$ is connected), which provides us with an equivalence
of categories
$$\Phi_{\res R/R}^{\BKF}\colon \op{BT}(R)\xrightarrow{\sim} \Rep_{\Delta}^{\mu}(A_\inf(\res R),\varphi,[-1,0])$$
defined by $\Phi_{\res R}^\inf(-\otimes_{R}\res R)$.


\begin{proposition}\label{prop:TateModule}
Let $G$ be a $p$-divisible group over $R$, let $M:=\Phi_{\res R/R}^{\BKF}(G)$ be its image under the functor immediately above, and let $TG:=\varprojlim_n G[p^n](\res R)$ be its 
Tate module. Then there is a natural (in~$G$) $\Delta$-equivariant isomorphism of $\mathbb Z_p$-modules $$TG\cong M^{\varphi=1}=(M[\tfrac{1}{\mu}])^{\varphi=1}.$$ 
\end{proposition}

\begin{proof}
The second equality follows from $\varphi_M(M)\supset M$
and the proof of Proposition \ref{proposition_etale}(ii). Let $H$ be the constant \'etale $p$-divisible
group over $R$ associated to $\bb Q/\bb Z$, and let $H_{\res R}$ and $G_{\res R}$
be the base changes of $H$ and $G$ to $\res R$. Then, by the
fully faithfulness of $\Phi_{\res R}^{\inf}$, we have
isomorphisms of $\bb Z_p$-modules
\begin{multline*}
TG\cong\Hom_{\op{BT}(\res R)}(H_{\res R},G_{\res R})
\cong \Hom_{\BKF(\res R,\varphi, [-1,0])}(\Phi_{\res R}^{\inf}(H_{\res R}),
\Phi_{\res R}^{\inf}(G_{\res R}))\\
=\Hom_{\BKF(\res R,\varphi,[-1,0])}(A_\inf(\res R),M)\cong M^{\varphi=1}.
\end{multline*}
where the final isomorphism is given by $g\mapsto g(1)$.
We must prove that this isomorphism is $\Delta$-equivariant.
For an element $t\in TG$, let $f_t$ and $g_t$ denote the morphisms
$H_{\res R}\to G_{\res R}$ and $A_\inf(\res R)\to M$ corresponding to 
$t$ under the above isomorphisms. Then, for $\delta\in \Delta$, 
the base change of $f_t$ under the morphism 
$\Spec(\delta)\colon \Spec(\res R)\to \Spec(\res R)$
coincides with $f_{\delta(t)}$.
(Here we use the following fact: for $P\in G[p^n](\res R)$, letting
$f_P\colon \Spec(\res R)\to G[p^n]_{\res R}$ denote the morphism 
of schemes over $\Spec(\res R)$ corresponding to $P$, 
then the base change of $f_P$ under $\Spec(\delta)\colon \Spec(\res R)
\to \Spec(\res R)$ coincides with $f_{\delta(P)}$.)
By the compatibility of the functor $\Phi_{\res R}^{\inf}$ with
the base change under $\delta$, we obtain
$g_{\delta(t)}\circ \delta=\delta\circ g_t$, which gives
$g_{\delta(t)}(1)=\delta(g_t(1))$ as required.
\end{proof}

\subsubsection{Relative prismatic cohmology}\label{sssection_relative}
Throughout this subsection we fix a proper smooth $p$-adic formal $R$-scheme $f:\cal Y\to \Spf R$. Our goal is to construct natural ``relative $A_\inf$-cohomology groups'' $H^*_{A_\inf}(\cal Y/R)\in\Rep_\Gamma^\mu(A_\inf(R_\infty),\phi)$, under suitable flatness hypotheses on the relative de Rham cohomology of $\cal Y$ over $R$.

I.~Gaisin and T.~Koshikawa have also studied relative $A_\inf$-cohomology, as well as the relative Hodge--Tate spectral sequence, via products of topoi; for example they also prove Lemma \ref{lemma_de_rham_tf_implies_prism_tf}(i), by a different method.

\begin{remark}\label{remark_untwist}
From Section \ref{section_prismatic_crystals} onwards when we compare generalised representations and relative Breuil--Kisin--Fargues modules to prismatic crystals, it will be clearer to distinguish the two $A_\inf$-algebra structures on $\roi$ (and more generally on any $\roi$-algebra), namely the default structure via $\theta$ or the Frobenius-twised structure via $\tilde\theta$; the latter, namely $A_\inf/\tilde\xi$, should be more correctly denoted by $\roi^{(1)}$. See the start of Section \ref{section_prismatic_crystals}.

However, to simplify notation in this self-contained subsection, here we instead make the convention that the default structure is the Frobenius-twisted one. For compatibility with our later notation in prismatic cohomology, the expression (\ref{equation_prismatic_cohoml}) should be replaced by $R\Gamma_\Prism(\cal Y^{(1)}\times_{ R^{(1)}} R_\infty^{(1)}/(A_\inf(R_\infty),\tilde\xi))$ and similarly throughout the rest of the subsection.
\end{remark}

The shortest definition of these generalised representations will be via the prismatic cohomology \cite{BhattScholze2019}
\begin{equation}R\Gamma_\Prism(\cal Y\times_{ R} R_\infty/(A_\inf(R_\infty),\tilde\xi)),\label{equation_prismatic_cohoml}\end{equation} where the base prism is the perfect prism $(A_\inf(R_\infty),(\tilde\xi))$ equipped with the $\delta$-structure induced by its Frobenius. This perfect complex of $A_\inf(R_\infty)$-modules should be viewed as the prototypical example of a ``derived relative Breuil--Kisin--Fargues module''; this notion can presumably be made precise via crystals of perfect complexes on the prismatic site, but here we content ourselves with the following consequences:
\begin{equation}R\Gamma_\Prism(\cal Y\times_{ R} R_\infty/(A_\inf(R_\infty),\tilde\xi))=R\Gamma_\Prism(\cal Y/(A_\inf^\bx(R),\xi))\otimes_{A_\inf^\bx(R),\phi}A_\inf(R_\infty)\label{equation_framed_prismatic}\end{equation}
\begin{equation}R\Gamma_\Prism(\cal Y\times_{ R} R_\infty/(A_\inf(R_\infty),\tilde\xi))/\mu=R\Gamma_\Prism(\cal Y/(A_\inf^\bx(R)/\phi^{-1}(\mu),p))\otimes_{A_\inf^\bx(R)/\phi^{-1}(\mu),\phi}A_\inf(R_\infty)/\mu\label{equation_prismatic_trivial}\end{equation}
These are proved by applying base change in prismatic cohomology \cite[Thm.~1.8(5)]{BhattScholze2019} along the maps of prisms:
\[\xymatrix{
(A_\inf^\bx(R),\xi)\ar[r]^\phi\ar[d]&(A_\inf(R_\infty),\tilde\xi)\ar[d] \\
(A_\inf^\bx(R)/\phi^{-1}(\mu),p)\ar[r]_\phi&(A_\inf(R_\infty)/\mu,p)
}\]
Identity (\ref{equation_framed_prismatic}) is a derived version of our forthcoming descent of generalised representations to the level of q-Higgs bundles (Corollary \ref{corollary_descent_of_reps_to_phi_twist}), while (\ref{equation_prismatic_trivial}) is a derived version of a generalised representation being trivial modulo $\mu$.

Now we impose conditions on the relative de Rham cohomology to get actual (non-derived) relative Breuill--Kisin--Fargues modules:

\begin{lemma}\label{lemma_de_rham_tf_implies_prism_tf}
Assume that the relative de Rham cohomologies $H^i_\sub{dR}(\cal Y/R)$ are finite projective $R$-modules for all $i\ge0$. Then:
\begin{enumerate}
\item the cohomology groups of (\ref{equation_prismatic_cohoml}) are all finite projective $A_\inf(R_\infty)$-modules;
\item the $\Gamma$-action on each of these cohomology groups (induced by functoriality from the $\Gamma$-action on the base prism $(A_\inf(R_\infty),\tilde\xi)$) is trivial modulo $\mu$.
\end{enumerate}
\end{lemma}
\begin{proof}
We assume, without any loss of generality, that $\Spec (R/pR)$ is connected so that the various finite projective modules which appear will have constant rank.

(i): We will actually prove that the cohomology groups of $R\Gamma_\Prism(\cal Y/(A_\inf^\bx(R),\xi))$ are all finite projective $A_\inf^\bx(R)$-modules; claim (i) then follows from  (\ref{equation_framed_prismatic}).

We begin with a base change comment. Let $S$ be a perfectoid ring and $R_\infty\to S$ a homomorphism, thereby inducing a map of prisms $(A_\inf(R_\infty),\tilde\xi)\to (A_\inf(S),\tilde\xi)$. Then the base change of $R\Gamma_\Prism(\cal Y/(A_\inf^\bx(R),\xi))$ along  $A_\inf^\bx(R)\xto{\phi}A_\inf(R_\infty)\to A_\inf(S)$ is the prismatic cohomology $R\Gamma_\Prism(\cal Y\times_RS/(A_\inf(S),(\tilde\xi)))$, which modulo $\xi$ is given by $R\Gamma_\sub{dR}(\cal Y/R)\dotimes_RS$ by the de Rham comparison \cite[Thm.~1.8(3)]{BhattScholze2019}. Note that the cohomologies of $R\Gamma_\sub{dR}(\cal Y/R)\dotimes_RS$ are the finite projective $S$-modules $H^*_\sub{dR}(\cal Y/R)\otimes_RS$.

By a descending induction we may suppose that all degree $> j$ cohomologies of $\Gamma_\Prism(\cal Y/(A_\inf^\bx(R),\xi))$ are finite projective $A_\inf^\bx(R)$-modules, and we must prove it in degree $j$, that is for $M:=H^j_\Prism(\cal Y/(A_\inf^\bx(R),\xi))$. Thanks to the inductive hypothesis, we see that (1) $M$ is a finitely presented module (it is the top degree of a perfect complex), that (2) $M/\phi^{-1}(\xi)\otimes_{A_\inf^\bx(R)/\phi^{-1}(\xi),\phi}A_\inf^\bx(R)/\xi=H^j_\sub{dR}(\cal Y/R)$ (by the de Rham comparison and vanishing of the higher Tors), whence $M/\phi^{-1}(\xi)$ is a finite projective $A_\inf^\bx(R)/\phi^{-1}(\xi)$-modules (as the Frobenius endomophism of $A_\inf^\bx(R)$ is finite flat), and that (3) for any $R_\infty\to S$ as in the first paragraph, we have $M\otimes_{A_\inf^\bx(R),\phi}A_\inf(S)=H^j_\Prism(\cal Y\times_{ R} S/A_\inf(S))$ (again since there are no higher Tors).

Since $A_\inf^\bx(R)$ is $\phi^{-1}(\xi)$-adically complete, there exists a finite projective $A_\inf(R_\infty)$-module $M'$ lifting $M/\phi^{-1}(\xi)$, whence there exists a map $f:M'\to M$ which is an isomorphism mod $\phi^{-1}(\xi)$: this map is surjective by Nakayama's lemma. 
To prove that it is injective, we reduce to the case of a valuation ring as follows (the argument is inspired by v or arc descent, but easier in this special case): let $W$ be the localisation of $R_\infty$ outside its prime ideal $\frak m R_\infty$, so that $W$ is a height one valuation ring. Its $p$-adic completion $\hat W$ is a rank one perfectoid valuation ring containing $R_\infty$.


To prove that $f$ is injective, it is now enough to check that it is injective after base change along the inclusion $A_\inf^\bx(R)\xto{\phi}A_\inf(R_\infty)\to A_\inf(W)$. By our base change observations above, we must therefore show that \[M'\otimes_{A_\inf^\bx(R),\phi}A_\inf(W)\To H^j_\Prism(\cal Y\times_{ R} W/A_\inf(W))\] is injective (or equivalently, an isomorphism); but since this is an isomorphism modulo $\xi$, the problem finally reduces to showing that the target has no $\xi$-torsion. But the target identifies with the $A_\inf$-cohomology of the proper smooth formal $W$-scheme $\cal Y\times_{\Spf R}\Spf W$, thanks to the comparison of \cite[\S17]{BhattScholze2019}, and we have assumed that this scheme has torsion-free de Rham cohomologies; so the $A_\inf$-cohomologies are all finite free $A_\inf(W)$-modules by \cite[Thm.~1.8, Corol.~4.17 \& Lem.~4.18]{BhattMorrowScholze1}.

(ii): Combined with the finite projectivity established in the proof of part (i), base change along the map of prisms $(A_\inf^\bx(R),\xi)\to (A_\inf^\bx(R)/\phi^{-1}(\mu),p)$ shows that the prismatic cohomologies of $\cal Y$ with respect to the latter prism are also finite projective, and that \[H^i_\Prism(\cal Y\times_RR_\infty/A_\inf(R_\infty))/\mu=H^i_\Prism(\cal Y/A_\inf^\bx(R)/\phi^{-1}(\mu))\otimes_{A_\inf^\bx(R)/\phi^{-1}(\mu),\phi}A_\inf(R_\infty)/\mu\] (this is obtained by taking cohomology of (\ref{equation_prismatic_trivial})). But $\Gamma$ acts as the identity on $A_\inf^\bx(R)/\phi^{-1}(\mu)$, hence also on the prismatic cohomology group on the right; this shows that the $\Gamma$-action on $H^i_\Prism(\cal Y\times_RR_\infty/A_\inf(R_\infty))$ is indeed trivial mod $\mu$, as desired.

\end{proof}

Continuing to assume that the relative de Rham cohomologies $H^*_\sub{dR}(\cal Y/R)$ are finite projective $R$-modules, we may now define our desired generalised representations \[H^i_{A_\inf}(\cal Y/R):=H^i_\Prism(\cal Y\times_{ R} R_\infty/(A_\inf(R_\infty),\tilde\xi))\] for each $i\ge0$. These are finite projective $A_\inf(R_\infty)$-modules (by Lemma \ref{lemma_de_rham_tf_implies_prism_tf}(i)) equipped with an action by $\Gamma$ (induced by functoriality via the action on the base $(A_\inf(R_\infty),\tilde\xi)$) which is trivial modulo $\mu$ (by Lemma \ref{lemma_de_rham_tf_implies_prism_tf}(ii)) and automatically continuous (Lemma \ref{lemma_automatic_continuity}), and with a Frobenius structure (induced by the Frobenius on prismatic cohomology \cite[Thm.~1.8(6)]{BhattScholze2019}). In short, we have constructed
\[H^i_{A_\inf}(\cal Y/R)\in \Rep^\mu_\Gamma(A_\inf(R_\infty),\phi),\]
which fulfils the goal of this subsection.

\begin{remark}
%
Similarly to Remark \ref{remark_BT_without_framing}, we could alternatively appeal to the prismatic cohomology of $\cal Y\times_R\res R$ over $(A_\inf(\res R),\tilde\xi)$ to directly construct $H_{A_\inf}^*(\cal Y/R)\in\Rep^{\mu}_\Delta(A_\inf(\res R),\phi)$ in a framing-independent manner.
\end{remark}

\subsubsection{Faltings' relative Fontaine--Laffaille modules}\label{sssection_Fontaine-Laffaille}
In this short subsection we fix a $p$-adically complete, formally smooth $W(k)$-algebra $R_W$-algebra and an isomorphism $R_W\hat\otimes_W\roi\cong R$. Note that such an $R_W$ always exists up to non-unique isomorphism: pick a formally smooth lift $R_W$ of the smooth $k$-algebra $R\otimes_\roi k$, and use formal smoothness to construct an isomorphism lifting the identification $(R_W\hat\otimes_W\roi)\otimes_\roi k=R\otimes_\roi k$.

Let $\op{MF}^\nabla_{[0,p-2],\sub{fr}}(R_W,\Phi)$ denote the category of Faltings' relative Fontaine--Laffaille modules with underlying free module \cite[\S II.d]{Faltings1989}; an object of this category is a quadruple $M=(M,\Fil^\blob M,\nabla,\Phi)$ where:
\begin{itemize}
\item $M$ is a finite free $R_W$-module equipped with a decreasing filtration, whose graded pieces are free modules, such that $\Fil^0M=M$ and $\Fil^{p-1}M=0$;
\item $\nabla:M\to M\otimes_{R_W}\hat\Omega^1_{R_W/W}$ is a $p$-adically quasi-nilpotent connection on $M$ which satisfies Griffiths transversality with respect to the filtration.
\item $\Phi$ is a horizontal Frobenius whose precise definition we do not recall here.
\end{itemize}
For further details we refer the reader to \cite[\S4]{Tsuji_simons}.

Relative Fontaine--Laffaille modules were shown by Faltings to provide a class of coefficients in relative integral $p$-adic Hodge theory. They may be seen as special cases of relative Breuil--Kisin--Fargues modules thanks to the following theorem, which for the sake of naturality we state using $\res R$ rather than $R_\infty$ (see Remark \ref{remark_Delta}). We must also take into account the arithmetic part of the fundamental group of $R_W$, so set $\cal G:=\Aut(\res R/R_W)=\pi_1^\sub{\'et}(R_W[\tfrac1p])$. Given a relative Fontaine--Laffaille module $M$, let $T_\sub{crys}(M)$ denote the (dual of the) ``crystalline'' Galois representation associated to it by Faltings; this is a finite free $\bb Z_p$-module equipped with a continuous action by $\cal G$; we again refer to \cite[\S4]{Tsuji_simons} for further details.

\begin{theorem}[Tsuji]\label{theorem_FL}
Let $M=(M,\Fil^\blob M,\nabla,\Phi)\in \op{MF}^\nabla_{[0,p-2],\sub{fr}}(R_W,\Phi)$; then there exists a unique $\cal G$-stable, finite projective $A_\inf(\res R)$-submodule $\TA_\inf(M)\subseteq T_\sub{crys}(M)\otimes_{\bb Z_p}A_\inf(\res R)$ such that the $\cal G$-action on $\TA_\inf(M)$ is trivial modulo $\mu$. In particular, restricting to the action of $\Delta\subseteq \cal G$ we obtain $\TA_\inf(M)\in\Rep^\mu_\Delta(A_\inf(\res R))$.
\end{theorem}

In fact, more is true: firstly, the submodule $\TA_\inf(M)$ is known to be sandwiched as follows \[\mu^{p-2}T_\sub{crys}(M)\otimes_{\bb Z_p}A_\inf(\res R)\subseteq \TA_\inf(M)\subseteq T_\sub{crys}(M)\otimes_{\bb Z_p}A_\inf(\res R).\] Secondly, the triviality modulo $\mu$ can be explicitly witnessed by a $\cal G$-equivariant isomorphism \[\TA_\inf(M)/\mu\cong M\otimes_{R_W,\al}A_\inf(\res R)/\mu.\] Here the map $\al:R_W\to A_\inf(\res R)$, which depends on the framing, is the composition $R_W\to A_\inf^\bx(R)\to A_\inf(\res R)$, where the second arrow is the canonical one from \S\ref{ss_framed} and the first arrow is the unique $W$-algebra homomorphism sending $T_i$ to $[T_i^\flat]$ (again, see \S\ref{ss_framed} for a reminder on the notation) \cite[Lem.~64]{Tsuji_simons}.

In \cite{Tsuji_simons}, the second author uses the point of view of generalised representations to re-establish certain aspects of the theory of relative Fontaine--Laffaille modules, notably fully faithfulness of the functor $T_\sub{crys}$. 

Combining Theorem \ref{theorem_FL} with Theorems \ref{theorem_descent_to_framed}, \ref{theorem_<mu_implies_mu} and Remark \ref{remark_Delta}, we see that any relative Fontaine--Laffaille module gives rise to a module with $q$-connection over $A_\inf^\bx(R)$; it might be interesting to have an explicit description of this process.

\newpage
\section{Generalised representations as q-connections}\label{section_q_connections}
In Section \ref{section_small_reps} we studied generalised representations over $A_\inf(R_\infty)$ and $A_\inf^\bx(R)$, culminating in the equivalence of categories of Theorem \ref{theorem_descent_to_framed} that $\Rep_\Gamma^\mu(A_\inf^\bx(R))\isoto \Rep_\Gamma^\mu(A_\inf(R_\infty))$; in particular, the generalised representations $\Rep_\Gamma^\mu(A_\inf^\bx(R))$ offer a framed approach to our relative Breuil--Kisin--Fargues modules.

The goal of this section is to interpret $\Rep_\Gamma^\mu(A_\inf^\bx(R))$ as modules with q-connection, in the sense of \cite{Scholze2017}, by observing in Corollary \ref{corollary_reps_vs_q_connections} that there is an equivalence of categories \[\Rep_\Gamma^\mu(A^\bx_\inf(R))\simeq \op{qMIC}(A_\inf^\bx(R))\] where $\op{qMIC}(A_\inf^\bx(R))$ is the category of finite projective $A_\inf^\bx(R)$-modules with flat q-connection; this equivalence will be a relatively formal consequence of the definitions. However, in the presence of a Frobenius structure we then descend further from $A_\inf^\bx(R)$ to its Frobenius twist $A_\inf^\bx(R)^{(1)}:=A_\inf\otimes_{\phi,A_\inf}A_\inf^\bx(R)$, by establishing in Corollary \ref{corollary_q_Simp} an equivalence \[(F,F_\Omega)^*:\text{\rm qHIG}( A^{\bx(1)},\phi)\quis\text{\rm qMIC}(A^\bx,\phi),\] where $\text{\rm qHIG}( A^{\bx(1)},\phi)$ is the category of finite projective modules over the Frobenius twist $ A^{\bx(1)}$ equipped with q-Higgs field and Frobenius structure; this equivalence may be viewed as a q-deformed Simpson correspondence.

We work throughout under an axiomatic set-up for q-de Rham cohomology in the presence of a framing which is sufficiently general to encompass the smooth cases of Bhatt--Scholze's framed q-PD data (see Remark~\ref{examples_qBK}).

\subsection{q-connections}\label{ss_modules_with_q}
In this subsection we introduce the basic notation and terminology of q-de Rham cohomology from \cite[\S9.2 \& \S12.1]{BhattMorrowScholze1} \cite{Scholze2017} \cite{BhattScholze2019}, often adopting the axiomatic point of view from Andr\'e's general theory of non-commutative differential rings \cite{Andre2001}. We initially adopt the following set-up:

\begin{quote} (qDR1) Let $A$ be a commutative base ring and $A^\bx$ a commutative $A$-algebra equipped with $d$ commuting $A$-algebra automorphisms $\gamma_1,\dots,\gamma_d$, i.e., an action of $\bb Z^d$. Also fix an element $q\in A$ such that $q-1$ is a non-zero-divisor of $A^\bx$, and assume that $\gamma_i\equiv \op{id}$ mod $q-1$ for each $i=1,\dots,d$.
\end{quote}

\begin{definition}[The $q$-de Rham complex]
Let $\q\Omega^\blob_{A^\bx/A}:=\bigoplus_{n=0}^d \q\Omega^n_{A^\bx/A}$ be the differential graded $A$-algebra defined as follows:
\begin{itemize}
\item $\q\Omega^0_{A^\bx/A}:=A^\bx$;
\item $\q\Omega^1_{A^\bx/A}$ is the free left $A^\bx$-module on formal basis elements $\dlog(U_1),\dots,\dlog(U_d)$;
\item the right $A^\bx$-module structure on $\q\Omega^1_{A^\bx/A}$ is twisted by the rule $\dlog(U_i)\cdot f:=\gamma_i(f)\dlog(U_i)$ for all $f\in A^\bx$ and $i=1,\dots,d$;
\item $\dlog(U_i)\dlog(U_j)=-\dlog(U_j)\dlog(U_i)$ if $i\neq j$ and $=0$ if $i=j$;
\item The map $\bigoplus_{1\le i_1<\cdots<i_n\le d}A^\bx\to\q\Omega^n_{A^\bx/A}$, $(f_{\ul i})\mapsto\sum_{1\le i_1<\cdots<i_n\le d}f_{\ul i}\dlog(U_{i_1})\cdots\dlog(U_{i_n})$ is an isomorphism of left $A^\bx$-modules.

\item The $0^\sub{th}$ differential is given by $d_q:A\to \q\Omega^1_{A^\bx/A}$, $f\mapsto \sum_{i=1}^d\frac{\gamma_i(f)-f}{q-1}\dlog(U_i)$;
\item The elements $\dlog(U_i)\in\q\Omega^1_{A^\bx/A}$ are cocycles, for $i=1,\dots,d$.
\end{itemize}
In other words, as a graded left $A$-module, $\q\Omega^\blob_{A^\bx/A}$ is the same as the exterior algebra on $\dlog(U_1),\dots,\dlog(U_d)$ over $A^\bx$, but multiplication $\cdot$ is given by \begin{align*}f\dlog(U_{i_1})\wedge\cdots\wedge&\dlog(U_{i_n})\cdot g\dlog(U_{j_1})\wedge\cdots\wedge\dlog(U_{j_m})\\&=f\gamma_{i_1}\circ\cdots\circ\gamma_{i_n}(g)\dlog(U_{i_1})\wedge\cdots\wedge\dlog(U_{i_n})\wedge \dlog(U_{j_1})\wedge\cdots\wedge\dlog(U_{j_m})\end{align*}
and the differential is induced by $d_q$ in the necessary way to produce a differential graded algebra such that the $\dlog(U_i)$ are cocycles.

The data $d_q:A^\bx\to\q\Omega^1_{A^\bx/A}$ forms a {\em differential ring} (over $A$) in the sense of Andr\'e \cite[2.1.2.1]{Andre2001}; that is, $\q\Omega^1_{A^\bx/A}$ is an $A^\bx$-bimodule (on which the $A$ action is symmetric) and the Leibniz rule $d_q(fg)=d_q(f)g+fd_q(g)$ is satisfied (and $d_q$ is $A$-linear).
\end{definition}

We next recall the formalism of modules with connection in this set-up (contrary to \cite{Andre2001}, we work with right rather than left modules):

\begin{definition}\label{definition_q_connections}
A {\em module with \q-connection} over $A^\bx$ is a right $A^\bx$-module $N$ equipped with an $A$-linear map $\nabla:N\to N\otimes_{A^\bx}\q\Omega^1_{A^\bx/A}$ satisfying the Leibniz rule $\nabla(nf)=\nabla(n)f+n\otimes d_q(f)$ for all $n\in N$ and $f\in A^\bx$. (To lighten notation we will tend to denote $q$-connections by the familiar connection symbol $\nabla$, rather than $\nabla_q$ as found in \cite{Scholze2017}.)
\end{definition}

In Andr\'e's framework \cite[2.2.1]{Andre2001}, a module with q-connection $(N,\nabla)$ over $A^\bx$ is precisely a module with connection over the differential ring $d_q:A^\bx\to\q\Omega^1_{A^\bx/A}$. Therefore \cite[Lem.~2.2.1.3]{Andre2001} shows that $\nabla$ extends uniquely to a map of graded $A$-modules $\nabla:N\otimes_{A^\bx}\q\Omega^\blob_{A^\bx/A}\to N\otimes_{A^\bx}\q\Omega^{\blob+1}_{A^\bx/A}$ satisfying \[\nabla((n\otimes\omega)\cdot\omega')=\nabla(n\otimes\omega)\cdot\omega'+(-1)^{\deg\omega}(n\otimes\omega)\cdot d_q(\omega').\]  One says that the $q$-connection $\nabla$ is {\em flat} or {\em integrable} if $\nabla^2=0$ (as in the classical case, it is enough to check that $\nabla^2:N\to N\otimes_{A^\bx}\q\Omega^2_{A^\bx/A}$ vanishes \cite[Lem.~2.2.3.1]{Andre2001}), in which case \[N\otimes_{A^\bx}\q\Omega^\blob_{A^\bx/A}=[N\xto{\nabla}\q\Omega^1_{A^\bx/A}\xto{\nabla}\q\Omega^2_{A^\bx/A}\xto{\nabla}\cdots]\] is a complex of $A$-modules known as the associated {\em q-de Rham complex}. We write $\textrm{qMIC}(A^\bx)$ for the category of finite projective modules\footnote{Many of the results in this section, even those which we only state for the category $\textrm{qMIC}(A^\bx)$, do work in greater generality; but for our ultimate aim of studying generalised representations the finite projective case is sufficient.} with flat q-connection. (The notation is not ideal, as it depends on the data of $q\in A$, the $\bb Z^d$-action, and the chosen basis $\dlog U_1,\dots,\dlog U_d$ of $q\Omega^1_{A^\bx/A}$.)

\begin{lemma}\label{lemma_q_coordinates}
Let $(N,\nabla)$ be a module with q-connection over $A^\bx$ and denote by $\nabla_{i}^\sub{log}:N\to N$ the $A$-linear maps\footnote{Henceforth called the ``logarithmic coordinates'' of the q-connection; in particular, we write $d_{q,i}^\sub{log}:=\tfrac{\gamma_i-1}{q-1}:A^\bx\to A^\bx$.} characterised by $\nabla(n)=\sum_{i=1}^d\nabla_{i}^\sub{log}(n)\otimes\dlog(U_i)$ for all $n\in N$. Then the q-connection $\nabla$ is flat if and only if the maps $\nabla_{1}^\sub{log},\dots,\nabla_{d}^\sub{log}$ pairwise commute.
\end{lemma}
\begin{proof}
One checks directly that $\nabla^2:N\to N\otimes_{A^\bx} \q\Omega^2_{A^\bx/A}$ is given by \[n\mapsto \sum_{i,j=1}^d\nabla_{j}^\sub{log}\nabla_{i}^\sub{log}(n)\otimes\dlog(U_j)\wedge\dlog(U_i),\] which vanishes if and only if $\nabla_{i}^\sub{log}\nabla_{j}^\sub{log}(n)=\nabla_{j}^\sub{log}\nabla_{i}^\sub{log}(n)$ for all $i,j$, since $\q\Omega^2_{A^\bx/A}$ is the exterior square of the left free $A^\bx$-module on $\dlog(U_1),\dots,\dlog(U_d)$.
\end{proof}

\begin{remark}[Reduction mod $q-1$]\label{remark_connection_mod_mu}
In the main case of interest, namely when $A^\bx$ arises by deformation theory as will be explained in Remark \ref{remark_q_via_deformation}, the reduction mod $q-1$ of the differential ring $d_q:A^\bx\to\Omega^1_{A^\bx/A}$ is precisely the usual de Rham derivative $d:R\to \hat\Omega^1_{R/\res A}$ taking value in the $p$-adically completed module of K\"ahler differentials of the $p$-adic formally smooth $\res A:=A/(q-1)$-algebra $R:=A^\bx/(q-1)$. Similarly, the reduction mod $q-1$ of a module with q-connection over $A^\bx$ is a module with connection (with respect to the $p$-adically complete module of differentials) over the $\res A$-algebra $R$.
\end{remark}

\begin{remark}[Tensor product and enriched Hom]\label{remark_tensor}
In this remark we add the following hypothesis, which will be true in our cases of interest, to the set-up (qDR1):

\begin{quote} (qDR2) Suppose we are given units $U_1,\dots,U_d\in A^\bx$ such that $d_q(U_i)=U_i\dlog(U_i)$ for $i=1,\dots,d$; in other words, \[\gamma_i(U_j)=\begin{cases}qU_j& \text{if }i=j\\ U_j &\text{if }i\neq j.\end{cases}\]
\end{quote}
Then, in Andr\'e's terminology, the differential ring $d_q:A^\bx\to\q\Omega^1_{A^\bx/A}$ is  {\em semi-classical} \cite[2.4.1]{Andre2001}: that is, $A^\bx$ is commutative and the differential ring is {\em reduced} \cite[2.1.2.2]{Andre2001} in the sense that $\q\Omega^1_{A^\bx/A}$ is generated as a right $A^\bx$-module by $d_q(A^\bx)$. The latter condition implies that the differential ring $d_q:A^\bx\to\q\Omega^1_{A^\bx/A}$ extends universally to a differential graded $A$-algebra \cite[2.1.2.4]{Andre2001}, which is easily seen to be nothing other than the q-de Rham complex $\q\Omega^\blob_{A^\bx/A}$ (when checking this identification, it is helpful to note that $U_id_q(U_j) = d_q(U_j) U_i$ for $i\neq j$ in $\q\Omega^1_{A^\bx/A}$, whence differentiating shows that the commutativity relation $d_q(U_i) d_q(U_j) = d_q(U_j) d_q(U_i)$ also holds in the universal case).

It follows from the semi-classicality and reducedness that the category of modules with $q$-connection forms a symmetric monoidal $A$-linear abelian category, and that the subcategory of finite projective modules with q-connection and having invertible ``volte'' is exact and self-dual. We now explicitly spell out these notions in our case.

Given a module with q-connection $(N,\nabla)$, its {\em volte} is the map of $A^\bx$-bimodules \[\frak r(\nabla):\q\Omega^1_{A^\bx/A}\otimes_{A^\bx}N\To N\otimes_{A^\bx}\q\Omega^1_{A^\bx/A},\qquad fd_q(g)\otimes m\mapsto f\nabla(gm)-fg\nabla(m)\]
(here we view $N$, which is a priori a right $A^\bx$-module, as a symmetric $A^\bx$-bimodule; secondly, we will denote the volte by $\frak r(\nabla)$ rather than Andr\'e's notation $\phi(\nabla)$ to avoid confusion with Frobenius maps); this is well-defined by \cite[2.2.4.1, 2.2.4.2, 2.4.1.1]{Andre2001}. In terms of the logarithmic coordinates of Lemma \ref{lemma_q_coordinates}, this is easily checked to be given explicitly by \[\frak r(\nabla):\dlog(U_i)\otimes n\mapsto ((q-1)\nabla_{i}^\sub{log}(n)+n)\otimes\dlog(U_i)\] for $i=1,\dots,d$. Therefore the volte is an isomorphism if and only if each endomorphism $\op{id}+(q-1)\nabla_{i}^\sub{log}$ of $N$ is an automorphism, which in particular holds if $N$ is $q-1$-adically complete.

For the next two items we fix modules with q-connections $(N,\nabla)$, $(N',\nabla')$.
\begin{enumerate}
\item Their tensor product \cite[2.3.1.1]{Andre2001} has underlying $A^\bx$-module $N\otimes_{A^\bx}N'$ with $q$-connection given by \[\op{id}\otimes\nabla'+(\op{id}\otimes\phi(\nabla'))\circ(\nabla\otimes\op{id}):N\otimes_{A^\bx}N'\To N\otimes_{A^\bx}N'\otimes_{A^\bx}\q\Omega^1_{A^\bx/A}.\] In terms of logarithmic coordinates, this is \[n\otimes n'\mapsto\sum_{i=1}^d\left(n\otimes\nabla_{i}'^{\sub{log}}(n')+\nabla_{i}^{\sub{log}}(n)\otimes \big((q-1)\nabla_{i}'^{\sub{log}}(n')+n'\big)\right)\otimes\dlog(U_i).\]
\item Next we construct the enriched Hom structure \cite[2.4.4]{Andre2001}; we assume that the volte $\frak r(\nabla)$ is invertible (i.e., each $\op{id}+(q-1)\nabla_{i}^\sub{log}$ is an isomorphism). Then we may equip $\Hom_{A^\bx}(N,N')$ (note that $N$ and $N'$ are viewed as symmetric $A^\bx$-bimodules, so this $\Hom$ is unambiguous) with a q-connection $\nabla_\sub{Hom}=\sum_{i=1}^d\nabla_{\sub{Hom},i}(\cdot)\otimes\dlog(U_i)$ characterised by \[\nabla_{\sub{Hom},i}^\sub{log}(f)\circ(\op{id}+(q-1)\nabla_{i}^\sub{log})=\nabla_{i}^{\sub{log}'}\circ f-f\circ\nabla_{i}^\sub{log}\] for $f\in \Hom_{A^\bx}(N,N')$ and $i=1,\dots,d$. This is easily seen to be the unique q-connection on $\Hom_{A^\bx}(N,N')$ for which the evaluation map $N\otimes_{A^\bx}\Hom_{A^\bx}(N,N')\to N'$ is compatible with q-connections, equipping the domain with the tensor product $q$-connection of the previous item.

Note that the categorical hom $\Hom((N,\nabla_q),(N',\nabla_q))$, i.e., morphisms of $A^\bx$-modules $N\to N'$ which commute with the q-connections, is then equal to the flat sections $\{f\in \Hom_{A^\bx}(N,N'):\nabla_\sub{Hom}(f)=0\}$.
\end{enumerate}
\end{remark}

We have not yet used in a very essential way that the differential ring $d_q:A^\bx\to\q\Omega^1_{A^\bx/A}$ was constructed from an action of $\bb Z^d$ on $A^\bx$, but it is crucial to the following correspondence:

\begin{proposition}\label{proposition_reps_vs_q_connections}
Adopt hypotheses (qDR1) and (qDR2). Then, given a generalised representation $N\in\Rep_{\bb Z^q}^{q-1}(A^\bx)$, the map \[\nabla:N\to N\otimes_{A^\bx}\q\Omega^1_{A^\bx/A},\quad n\mapsto\sum_{i=1}^d\frac{\gamma_i-1}{q-1}(n)\otimes\dlog(U_i)\] is a flat \q-connection whose associated \q-de Rham complex is naturally (in $N$) quasi-isomorphic to $L\eta_{q-1}R\Gamma(\bb Z^d,N)$. The resulting functor \[\Rep_{\bb Z^d}^{q-1}(A^\bx)\To\categ{5cm}{finite projective modules over $A^\bx$ with flat \q-connection and invertible volte}\] is an equivalence of symmetric monoidal categories which is compatible with internal homs.

If $A^\bx$ is $(p,q-1)$-adically complete and separated, then the action of $\bb Z^d$ extends uniquely to a semi-linear continuous action of $\bb Z_p^d$ on $A^\bx$ and the aforementioned equivalence may be rewritten as \[\Rep_{\bb Z_p^d}^{q-1}(A^\bx)\quis \text{\rm qMIC}(A^\bx).\]
\end{proposition}
\begin{proof}
In terms of logarithmic coordinates, $\nabla$ is given by the rule $\nabla_{i}^\sub{log}=\tfrac{\gamma_i-1}{q-1}$; it is easy to check that $\nabla$ being a q-connection (i.e., satisfying the \q-Leibniz rule) corresponds exactly to the $\gamma_i$ being semi-linear, and then that flatness of $\nabla$ corresponds exactly to the $\gamma_i$ commuting (by Lemma \ref{lemma_q_coordinates}). By construction the volte of $\nabla$ is given by $\sum_{i=1}^d \dlog(U_i)\otimes n_i\mapsto \sum_{i=1}^d \gamma_i(n_i)\otimes \dlog(U_i)$, whose invertibility corresponds exactly to the $\gamma_i$ being automorphisms. This shows that the functor is an equivalence of categories.

To check that the functor is compatible with the tensor products on both sides, let $N,N'\in\Rep_{\bb Z^d}^{q-1}(A^\bx)$. In view of the explicit formula for the tensor product \q-connection on $N\otimes_{A^\bx}N'$ from Remark \ref{remark_tensor}(i), it must be shown that \[\frac{1}{q-1}(\gamma_i(n)\otimes\gamma_i(n')-n\otimes n')=n\otimes\nabla_{i}^{\sub{log}'}(n')+\nabla_{i}^{\sub{log}}(n)\otimes \big((q-1)\nabla_{i}^{\sub{log}'}(n')+n'\big)\]
for all $n\in N$, $n'\in N'$, and $i=1,\dots,d$. This follows easily by writing $\gamma_i(n)=n+(q-1)\nabla_{i}^\sub{log}(n)$, $\gamma_i(n')=n'+(q-1)\nabla_{i}^{\sub{log}'}(n)$, and expanding the left side.

For internal homs, Remark \ref{remark_tensor}(ii) shows that the $\Gamma$-action associated to the hom $q$-connection on $\Hom_{A^\bx}(N,N')$ satisfies \[\frac1{q-1}(\gamma_i(f)-f)(\gamma_i(n))=\nabla_{i}^{\sub{log}'}(f(n))-f(\nabla_{i}^\sub{log}(n))\] for all $f\in \Hom_{A^\bx}(N,N')$, $n\in N$, and $i=1,\dots,d$. Expanding shows that $\gamma_i(f)(\gamma_i(n))=\gamma_i(f(n))$, i.e., the $\bb Z^d$-action on $\Hom_{A^\bx}(N,N')$ is the usual one of Remark \ref{remark_internal_hom} given by $\gamma(f):=\gamma f\gamma^{-1}$, as desired.

Representing the group cohomology by the Koszul complex $K(\gamma_1-1,\dots,\gamma_d-1;N)$, Lemma \ref{lemma_Koszul}(i) shows that $L\eta_{q-1}R\Gamma(\bb Z^d,N)\simeq K(\tfrac{\gamma_1-1}{q-1},\dots,\tfrac{\gamma_d-1}{q-1};N)$, which is given by \[N\To\bigoplus_{1\le i\le d}N\To\cdots\To\bigoplus_{1\le i_1<\cdots<i_n\le d}N\To\cdots\] with certain explicit differentials involving alternating sums of the maps $\tfrac{\gamma_i-1}{q-1}$ \cite[Def.~7.1]{BhattMorrowScholze1}. Identifying the copy of $N$ in slot $1\le i_1<\cdots<i_n\le d$ with $N\otimes\dlog(U_{i_1})\cdots \dlog(U_{i_d})$ produces the desired isomorphism with the q-de Rham complex $N\otimes_{A^\bx}\q\Omega^\blob_{A^\bx/A}$.

Finally, assume that $A$ is $(p,q-1)$-adically complete and separated. As already mentioned above, the $q-1$-adic completeness implies that the volte is always invertible. Secondly, any $A$-linear automorphism $\gamma$ of a $(p,q-1)$-adically complete module $N$ which is $\equiv\id$ mod $q-1$ automatically and uniquely extends to a continuous action of $\bb Z_p$: indeed the fact that $\gamma\equiv\op{id}$ mod $q-1$ implies that $\gamma^{p^m}\equiv\op{id}$ mod $(p,q-1)^m$ for any $m\ge1$.
\end{proof}

We finish the subsection by summarising the result in the case which interests us, namely $A=A_\inf$, $A^\bx=A_\inf^\bx(R)$, $q=[\ep]$, where we follow the notation of \S\ref{ss_framed}. Since $A_\inf^\bx(R)$ is formally \'etale over $A_\inf\pid{U_1^{\pm1},\dots,U_d^{\pm1}}$, the completed module of K\"ahler differentials $\hat\Omega^1_{A_\inf^\bx(R)/A_\inf}$ is the free $A_\inf^\bx(R)$-module with basis elements $\dlog(U_i)$, $i=1,\dots,d$, and so identifies with $\q\Omega^1_{A_\inf^\bx(R)/A_\inf}$ as defined above. This set-up clearly satisfies the hypotheses (qDR1) and (qDR2), and Proposition \ref{proposition_reps_vs_q_connections} specialises to the following:

\begin{corollary}\label{corollary_reps_vs_q_connections}
Given a generalised representation $N\in\Rep_\Gamma^\mu(A_\inf^\bx(R))$, the map \[\nabla:N\to N\otimes_{A^\bx_\inf(R)}\hat\Omega^1_{A^\bx_\inf(R)/A_\inf},\quad n\mapsto\sum_{i=1}^d\frac1\mu(\gamma_i(n)-n)\otimes\dlog(U_i)\] is a flat $q$-connection whose associated $q$-de Rham complex is naturally (in $N$) quasi-isomorphic to $L\eta_{\mu}R\Gamma_\sub{cont}(\Gamma,N)$. The resulting functor \[\Rep_{\Gamma}^\mu(A_\inf^\bx(R))\To\text{\rm qMIC}(A_\inf^\bx(R))\]
is an equivalence of symmetric monoidal categories.
\end{corollary}
\begin{proof}
As already mentioned, this is a special case of Proposition \ref{proposition_reps_vs_q_connections}, identifying $\Gamma$ with $\bb Z_p^d$.
\end{proof}

\begin{remark}\label{remark_dependence_on_ep}
Both sides of the equivalence $\Rep_{\Gamma}^\mu(A_\inf^\bx(R))\simeq\text{\rm qMIC}(A_\inf^\bx(R))$ depend on the chosen framing by definition, and the equivalence depends on our chosen compatible sequence of $p$-power roots of unity. But this latter dependence is mild, in the following sense. Changing the sequence of $p$-power roots of unity has the effect of changing $\ep\in\bb Z_p(1)\subseteq\roi^{\flat\times}$ into $\ep^a$ for some $a\in\bb Z_p^\times$, and similarly $\gamma_i$ into $\gamma_i^a$. Then $\gamma_i^a=1+a(\gamma_i-1)+\mu^2 \beta_i$ for some $\beta_i\in\End_{A_\inf}(N)$ belonging to the $p$-adic closure of the submodule generated by $\gamma_i$. (To prove this consider the map $\bb Z_p[[T]]\to \End_{A_\inf}(N)$, $T\mapsto\gamma_i-1$, which makes sense since $N$ is $(p,\mu)$-adically complete, and note that $(1+T)^a\equiv 1+aT$ mod $T^2$.) It follows that \[\frac{\gamma_i^a-1}{[\ep^a]-1}=\al_i\circ \frac{\gamma_i-1}{[\ep]-1}\] for some automorphism $\al_i\in \End_{A_\inf}(N)$ in the $p$-adic closure of the submodule generated by $\gamma_i$.
\end{remark}

\subsection{Frobenius structure on modules with q-connection}\label{ss_Frobenius}
We now incorporate a Frobenius into the theory of modules with $q$-connection. So, in addition to hypotheses (qDR1) and (qDR2), we assume that:
\begin{quote}
(qDR$\phi$) We are given compatible ring endomorphisms $\phi:A\to A$ and $\phi:A^\bx\to A^\bx$ (where compatible means that the obvious diagram commutes) such that $\phi(q)=q^p$ in $A$ and such that $\phi$ commutes with $\gamma_1,\dots,\gamma_d$ on $A^\bx$.
\end{quote}
We hope that the simultaneous use of $\phi$ for the two endomorphisms does not cause confusion.

\begin{remark}
In practice the endomorphism of $A$ will be part of the data (e.g., the usual Frobenius automorphism of $A_\inf$) whereas that of $A^\bx$ will involve a choice (e.g., a ring automorphism of $A_\inf^\bx(R)$ lifting the absolute Frobenius); it is not necessary to assume that $\phi(U_i)=U_i^p$ for $i=1,\dots,d$, though in practice this will often be true and thereby simplify some of the formulae which appear.
\end{remark}

\begin{remark}[Via deformation from a framing]\label{remark_q_via_deformation}
The standard way in which to fulfil all our hypothesis is to ``\q-deform'' a smooth algebra equipped with a framing by adopting the following set-up:

\begin{quote}
Let $A$ be a ring, $q\in A$ a unit such that $q-1$ is a non-zero-divisor, and assume that $A$ is $(p,q-1)$-adically complete and separated. Let $A\pid{\ul U^{\pm1}}=A\pid{U_1^{\pm1},\dots,U_d^{\pm1}}$ denote the $(p,q-1)$-adic completion of the Laurent polynomial algebra $A[U_1^{\pm1},\dots,U_d^{\pm1}]$, and fix a $(p,q-1)$-formally \'etale $A\pid{\ul U^{\pm1}}$-algebra $A^\bx$. (In practice $A^\bx$ arises by lifting an \'etale algebra over $A/(p,q-1)[U_1^{\pm1},\dots,U_d^{\pm1}]$, c.f., Remark \ref{examples_qBK}). Assume further that $A$ is equipped with an endomorphism $\phi$ which lifts the absolute Frobenius on $A/p$ and which satisfies $\phi(q)=q^p$.
\end{quote}
Then, by the same deformation arguments as in \S\ref{ss_framed}, there is a unique continuous action of $\bb Z_p^d$ on $A^\bx$ via $A$-algebra homomorphisms under which the usual basis elements $\gamma_1,\dots,\gamma_d\in\bb Z_p^d$ act as \[\gamma_i(U_j)=\begin{cases}qU_j & i=j \\ U_j & i\neq j\end{cases}.\] The data \[A\to A^\bx\circlearrowleft \gamma_1,\dots,\gamma_d,\qquad q\in A,\qquad U_1,\dots,U_d\in A^\bx\] is thus a set-up for \q-de Rham cohomology satisfying hypotheses (qDR1) and (qDR2). Moreover, in this case $\q\Omega^1_{A^\bx/A}$ identifies as a left $A^\bx$-modules with the $(p,q-1)$-adic completion $\hat\Omega^1_{A^\bx/A}$ of $\Omega^1_{A^\bx/A}$.


Similar deformation arguments show that $\phi$ extends uniquely to an endomorphism $\phi$ of $A^\bx$ which lifts the absolute Frobenius on $A^\bx/p$ and which satisfies $\phi(U_i)=U_i^p$ for $i=1,\dots,d$, and in this way we also satisfy condition (qDR$\phi$).

\end{remark}

\begin{remark}[Relation to framed algebras over q-PD pairs]\label{examples_qBK}
We explain the relation to Bhatt--Scholze's framed q-PD data. Let $(A,I)$ be a q-PD pair in the sense of \cite[Def.~16.2]{BhattScholze2019} (minor notational warning: our $A$ is Bhatt--Scholze's $D$, whereas they use $A$ to denote $\bb Z_p[[q-1]]$), and $R$ a $p$-completely smooth $A/I$-algebra which is small in the sense that there exists an \'etale map $A/(I,p)[\ul U^{\pm1}]:=A/(I,p)[U_1^{\pm1},\dots,U_d^{\pm1}]\to R/p$, which we fix. Assume in addition that $A$ is $q-1$-torsion-free (which we note is also imposed in the theory of q-de Rham cohomology \cite[\S16.3]{BhattScholze2019}). Then:
\begin{enumerate}
\item Firstly, $(A,[p]_q)$ is a bounded prism, by definition of a q-PD pair, and so $A$ is $(p,[p]_q)$-adically$=(p,q-1)$-adically complete and separated by \cite[Lem.~3.7]{BhattScholze2019}.
\item Next recall that $q-1\in I$ and $\phi(I)\subseteq ([p]_q)$, where $\phi$ is the Frobenius on $A$ induced by its $\delta$-structure, so $(q^p-1,p)\subseteq (I^{(p)},p)\subseteq ([p]_q,p)\subseteq (q-1,p)\subseteq (I,p)$, where $I^{(p)}=\{a^p:a\in I\}$; the maps $A/(p,[p]_q)\to A/(p,q-1)\to A/(p,I)$ are therefore quotients by nil-ideals, whence base changing along the maps $A/(p,[p]_q)[\ul U^{\pm1}]\to A/(p,q-1)[\ul U^{\pm1}]\to A/(p,I)[\ul U^{\pm1}]$ induce equivalences on the categories of \'etale algebras.
\end{enumerate}
So the fixed framing may be used to construct a $(p,q-1)$-formally \'etale $A\pid{\ul U^{\pm1}}$-algebra $A^\bx$ (denoted by $P$ in \cite[Const.~16.19]{BhattScholze2019}) such that $A^\bx/(p,I)\isoto R/p$. The deformation arguments of Remark \ref{remark_q_via_deformation} or \cite[Const.~16.18]{BhattScholze2019} show that $A\to A^\bx$ (equipped with their Frobenii and the framing) satisfy the set-up of Remark \ref{remark_q_via_deformation}. Assuming in addition that $A$ is flat over $\bb Z_p[[q-1]]$(equivalently $(p,q-1)$-completely flat by the argument of \cite[Lem.~4.30]{BhattMorrowScholze1}), this data also forms a framed q-PD datum in the sense of \cite[Const.~16.19]{BhattScholze2019}.

More precisely, this argument shows any framed q-PD datum $(P,S,J)$ for which $J=IP$ (i.e., $P/IP=R$) provides an example of the set-up of Remark \ref{remark_q_via_deformation}, and so is covered by our main results.



\end{remark}

Returning to the axiomatic set-up (qDR$\phi$), the endomorphism $\phi$ induces an endomorphism of $\Omega^1_{A^\bx/A}$, which we denote by $\phi_\Omega$ to distinguish it from $\phi$, given by definition by \[\phi_\Omega(\sum_{i=1}^df_id_q(U_i)):=\sum_{i=1}^d\phi(f_i)d_q(\phi(U_i))=[p]_q \sum_{i=1}^d\phi(f_i)\phi(U_i)U_i^{-1}d_q(U_i)\] where $[p]_q=1+q+\cdots+q^{p-1}\in A$ is the ``q-analogue of $p$''. Alternatively, in terms of logarithmic coordinates, $\phi_\Omega$ is given more cleanly by $\phi_\Omega(\sum_{i=1}^df_i\dlog(U_i))=[p]_q \sum_{i=1}^d\phi(f_i)\dlog(U_i)$.

The precise meaning of the following lemma will be explained in the proof:

\begin{lemma}\label{lemma_phi_on_qdR}
Adopt the set-up (qDR1), (qDR2), (qDR$\phi$). Then the pair $(\phi,\phi_\Omega)$ is an endomorphism of the differential ring $d_q:A^\bx\to\q\Omega^1_{A^\bx/A}$.
\end{lemma}
\begin{proof}
According to the general theory of differential rings, this means that the following diagram commutes
\[\xymatrix{
A^\bx\ar[r]^-{d_q}\ar[d]_\phi & \q\Omega^1_{A^\bx/A}\ar[d]^{\phi_\Omega}\\
A^\bx\ar[r]^-{d_q} &\q\Omega^1_{A^\bx/A}
}\]
and that $\phi_\Omega$ is compatible with the bimodule structure on $\q\Omega^1_{A^\bx/A}$. Both these assertions are clear from the definition of $\phi_\Omega$ and the assumption that $\phi$ commutes with $\gamma_1,\dots,\gamma_d$.
\end{proof}

It follows formally from Lemma \ref{lemma_phi_on_qdR} that there is a well-defined base change functor \[\categ{3.4cm}{modules with q-connection over $A^\bx$}\xto{(\phi,\phi_\Omega)^*}\categ{3.4cm}{modules with q-connection over $A^\bx$},\qquad (N,\nabla)\mapsto (\phi^*N,\phi^*\nabla)\] defined as follows \cite[2.2.2]{Andre2001}: a q-connection $\nabla$ on an $A^\bx$-module $N$ gives a q-connection $\phi^*\nabla$ on $\phi^*N=N\otimes_{A^\bx,\phi}A^\bx$ by 
\begin{align*}
\phi^*\nabla:\phi^*N&\To N\otimes_{A^\bx,\phi}\q\Omega^1_{A^\bx/A}\;(=\phi^*N\otimes_{A^\bx}\q\Omega^1_{A^\bx/A})\\ n\otimes f&\mapsto (\op{id}_N\otimes\phi_\Omega)(\nabla n)\cdot f+n\otimes d_q(f)
\end{align*}
(where $n\in N$, $f\in A^\bx$). More explicitly, writing $\nabla=\sum_{i=1}^d\nabla_{i}(-)\otimes d_q(U_i)$ in terms of non-logarithmic coordinates, then the analogous coordinates of $\phi^*\nabla$ are \[(\phi^*\nabla)_{i}:\phi^*N\To\phi^*N,\qquad n\otimes f\mapsto \nabla_{i}(n)\otimes [p]_q \phi(U_i)U_i^{-1} \gamma_i(f)+n\otimes d_{q,i}(f),\] where $d_{q,i}(f):=\tfrac{\gamma_i(f)-f}{U_i(q-1)}$.
Alternatively, in terms of the logarithmic coordinates $\nabla_{i}^\sub{log}(-)=\nabla_{i}(-)U_i$, the logarithmic coordinates of $\phi^*\nabla$ are given by \[(\phi^*\nabla)^\sub{log}_i:\phi^*N\To\phi^*N,\qquad n\otimes f\mapsto \nabla_{i}^\sub{log}(n)\otimes [p]_q\gamma_i(f)+n\otimes d_{q,i}^\sub{log}(f).\] Moreover, since the pair  $(\phi,\phi_\Omega)$ extends to an endomorphism of the q-de Rham complex $\q\Omega^\blob_{A^\bx/A}$ (as can either be checked explicitly, or by appealing to the universality mentioned just after hypothesis (qDR2)), the functor $(\phi,\phi_\Omega)^*$ preserves flatness (either by direct inspection or by \cite[Lem.~2.2.3.2]{Andre2001}), the tensor product of Remark \ref{remark_tensor}(i), and the enriched Hom of Remark \ref{remark_tensor}(ii) when $N$ is finite projective (to ensure that $\phi^*\Hom_{A^\bx}(N,N')\isoto\Hom_{A^\bx}(\phi^*N,\phi^*N')$).

An $A^\bx$-linear morphism $\phi_N:\phi^*N\to N$ is a morphism of modules with q-connections, or said to be ``horizontal'' with respect to the q-connection, if the obvious diagram
\[\xymatrix{
\phi^*N\ar[r]^-{\phi^*\nabla}\ar[d]_{\phi_N} & \phi^*N\otimes_{A^\bx}\q\Omega^1_{A^\bx/A}\ar[d]^{\phi_N\otimes\op{id}}\\
N\ar[r]^-{\nabla} & N\otimes_{A^\bx}\q\Omega^1_{A^\bx/A}
}\]
commutes. Explicitly, this means that $\nabla_{i}\phi_N(n\otimes 1)=\phi_N(\nabla_{i}(n)\otimes 1)[p]_q \phi(U_i)U_i^{-1}$ for all $i=1,\dots,d$ and $n\in N$, or in terms of logarithmic coordinates $\nabla_{i}^\sub{log}\phi_N(n\otimes 1)=\phi_N(\nabla_{i}^\sub{log}(n)\otimes 1)[p]_q$. More generally, these assertions about morphisms of modules with q-connection remain true if we invert any element of $A$, as we shall do in the next definition.

\begin{definition}
Let $\textrm{qMIC}(A^\bx,\phi)$ be the category consisting of pairs $(N,\phi_N)$, where $N\in \textrm{qMIC}(A^\bx)$ and $\phi_N:(\phi^*N)[\tfrac1{[p]_q}]\to N[\tfrac1{[p]_q}]$ is an isomorphism of modules with $q$-connection over $A^\bx[\tfrac1{[p]_q}]$. (Such a pair is really a triple $(N,\nabla,\phi_N)$, but then the notation risks becoming heavy.)
\end{definition}

We similarly consider the category of generalised representations with Frobenius $\Rep_{\bb Z^d}^{q-1}(A^\bx,\phi)$, consisting of pairs $(N,\phi_N)$ where $N\in \Rep_{\bb Z^d}^{q-1}(A^\bx)$ and $\phi_N:(\phi^*N)[\tfrac1{[p]_q}]\to N[\tfrac1{[p]_q}]$ is a $\bb Z^d$-equivariant isomorphism of $A^\bx[\tfrac1{[p]_q}]$-modules; when $A^\bx$ is $(p,q-1)$-adically complete we use the analogous notions for $\bb Z_p^d$ in place of $\bb Z^d$. It is immediate that the equivalence of Proposition \ref{proposition_reps_vs_q_connections} may be modified to incorporate the additional Frobenius structures; we state the resulting equivalence in the complete case, as that is what interests us:

\begin{proposition}\label{proposition_reps_vs_q_connections_with_phi}
Adopt the set-up (qDR1), (qDR2), (qDR$\phi$) and assume that $A^\bx$ is $(p,q-1)$-adically complete and separated. Then the functor of Proposition \ref{proposition_reps_vs_q_connections} induces an equivalences of categories \[\textrm{\rm qMIC}(A^\bx,\phi)\Isoto\Rep_{\bb Z_p^d}^{q-1}(A^\bx,\phi).\]
\end{proposition}
\begin{proof}
The functor does not change the underlying module $N$ but simply replaces the $\bb Z_p^d$-action by the q-connection $\nabla$ having logarithmic coordinates $\nabla_{i}^\sub{log}=\tfrac{\gamma_i-1}{q-1}$, as in Proposition \ref{proposition_reps_vs_q_connections}. Moreover, the final equivalence of Proposition \ref{proposition_reps_vs_q_connections} is compatible with the two pullback functors $\phi^*$. So it remains only to notice that $\phi_N$ commutes with the $\gamma_i$ if and only if it is horizontal with respect to $\nabla$.
\end{proof}

The previous proposition implies in particular to our main case of interest, namely $A=A_\inf$, $A^\bx=A_\inf^\bx(R)$, $q=[\ep]$, (in which case $[p]_q=\tilde\xi$) thereby providing an analogue of Corollary \ref{corollary_reps_vs_q_connections} with Frobenii; the precise statement may be found later as the final equivalence of Theorem \ref{theorem_local_with_phi}.

\begin{remark}[Tensor product and enriched Hom]
Let $(N,\phi_N),(N',\phi_{N'})\in \op{qMIC}(A^\bx,\phi)$. The q-connections on $N\otimes_{A^\bx}N'$ and $\Hom_{A^\bx}(N,N')$ were defined in Remark \ref{remark_tensor}. It is clear that they also admit Frobenii \[\phi_N\otimes\phi_{N'}:\phi^*(N\otimes_{A^\bx}N')[\tfrac{1}{[p]_q}]=(\phi^*N\otimes_{A^\bx}\phi^*N')[\tfrac{1}{[p]_q}]\To N\otimes_{A^\bx}N'[\tfrac{1}{[p]_q}]\] and \begin{align*}\phi_\sub{Hom}:\phi^*(\Hom_{A^\bx}(N,N'))[\tfrac{1}{[p]_q}]=\Hom_{A^\bx}((\phi^*N)[\tfrac{1}{[p]_q}],(\phi^* N')[\tfrac{1}{[p]_q}])&\To \Hom_{A^\bx}(N,N')[\tfrac1{[p]_q}]\\  f&\mapsto \phi_{N'}\circ \phi^*f\circ\phi_N^{-1},\end{align*} thereby upgrading them to objects of $\op{qMIC}(A^\bx,\phi)$. The categorical $\Hom((N,\phi_N),(N',\phi_{N'}))$ then equals the flat sections fixed by the Frobenius in the enriched Hom, namely \[\{f\in\Hom_{A^\bx}(N,N'):\nabla_\sub{Hom}(f)=0\textrm{ and }\phi_\sub{Hom}(f)=f\}.\]
\end{remark}

\subsection{q-Higgs fields and a q-Simpson correspondence}\label{ss_q_Higgs}
We continue to work under the hypotheses (qDR1), (qDR2), (qDR$\phi$); in this subsection we introduce q-Higgs fields over the Frobenius twist of $A^\bx$, which will then be used to establish a q-Simpson correspondence.

We begin by introducing the Frobenius twist $A^{\bx(1)}$, which is another example of a differential ring. Let $ A^{\bx(1)}=A\otimes_{\phi,A}A^\bx$ denote the Frobenius pullback of $A^\bx$, fitting into the usual commutative diagram
\[\xymatrix{
A^\bx\ar[r]^{W}\ar@/^10mm/[rr]^\phi &  A^{\bx(1)}\ar[r]^F& A^\bx\\
A\ar[u]\ar[r]_\phi & A\ar[u]\ar[ur] &
}\]
(the notation $W,F$ being standard in the theory of the relative Frobenius); the $A$-algebra automorphisms $\gamma_1,\dots,\gamma_d$ of $A^\bx$ induce $A$-algebra automorphisms $\op{id}\otimes\gamma_1,\dots,\op{id}\otimes\gamma_d$ of $ A^{\bx(1)}$, which are the identity modulo $q^p-1$. Also let $\q\Omega^1_{ A^{\bx(1)}/A}$ be the $ A^{\bx(1)}$-bimodule obtained by base changing $\q\Omega^1_{A^\bx/A}$ along $W:A^\bx\to A^{\bx(1)}$; in other words, $\q\Omega^1_{ A^{\bx(1)}/A}$ is the free left $ A^{\bx(1)}$-module on generators $\dlog(U_1), \dots, \dlog(U_d)$, with right action twisted by the $\op{id}\otimes\gamma_i$, just as we defined the right module structure on $\q\Omega^1_{A^\bx/A}$. Thus the data
\begin{equation}A\to A^{\bx(1)}:=A\otimes_{\phi,A}A^\bx\circlearrowleft \op{id}\otimes\gamma_1,\dots,\op{id}\otimes\gamma_d,\qquad q^p\in A\qquad 1\otimes U_1,\dots,1\otimes U_d\in A^{\bx(1)}\label{eqn_Higgs}\end{equation} is a new set-up for q-de Rham cohomology satisfying (qDR1)\footnote{To be precise, it might not be true that $q^p-1$ is a non-zero-divisor of $ A^{\bx(1)}$, though it always will be in our cases of interest. However this does not matter, since the denominators can be avoided when defining the q-derivative $d_{q^p}$: instead define it as $\op{id}\otimes d_q:A\otimes_{\phi,A}A^\bx=A^{\bx(1)}\to A\otimes_{\phi,A}\Omega^1_{A^\bx/A}=\Omega^1_{ A^{\bx(1)}/A}$.} and (qDR2), where the q-parameter is $q^p$. 

Therefore, according to the general theory presented in \S\ref{ss_modules_with_q}, there is an associated differential ring \begin{equation}d_{q^p}: A^{\bx(1)}\To\q\Omega^1_{ A^{\bx(1)}/A},\qquad f\mapsto\sum_{i=1}^d\frac{\op{id}\otimes\gamma_i-1}{q^p-1}(f)\dlog(U_i),\label{eqn_qdr_twist}\end{equation} over which modules with connection are precisely $ A^{\bx(1)}$-modules with q-connection. However, we are now more interested in the differential ring  \begin{equation}[p]_q d_{q^p}: A^{\bx(1)}\To\q\Omega^1_{ A^{\bx(1)}/A},\qquad f\mapsto\sum_{i=1}^d\frac{\op{id}\otimes\gamma_i-1}{q-1}(f)\dlog(U_i).\label{eqn_Higg}\end{equation} Modules with connection over this differential ring are precisely the objects of the following definition:

\begin{definition}
A {\em module with q-Higgs field} over $ A^{\bx(1)}$ is a right $ A^{\bx(1)}$-module $H$ equipped with an $A$-linear map $\Theta:H\to H\otimes_{ A^{\bx(1)}}\q\Omega^1_{ A^{\bx(1)}/A}$ satisfying $\Theta(hf)=\Theta(h)f+h\otimes[p]_q d_{q^p}(f)$ for $h\in H$ and $f\in A^{\bx(1)}$.

We write $\op{qHIG}( A^{\bx(1)})$ for the category of finite projective $ A^{\bx(1)}$-modules with flat q-Higgs field, where ``flat'' is explained immediately below.
\end{definition}

As in the case of q-connections, the general theory of differential rings shows that a q-Higgs field $\Theta$ extends uniquely to a map of graded $A$-modules $\Theta:H\otimes_{ A^{\bx(1)}}\q\Omega^\blob_{ A^{\bx(1)}/A}\to H\otimes_{ A^{\bx(1)}}\q\Omega^{\blob+1}_{ A^{\bx(1)}/A}$ satisfying \[\Theta((h\otimes\omega)\cdot\omega')=\Theta(h\otimes\omega)\cdot\omega'+(-1)^{\deg\omega}(h\otimes\omega)\cdot [p]_q d_{q^p}(\omega').\]  One says that the $q$-Higgs field $\Theta$ is {\em flat} or {\em integrable} if $\Theta^2=0$ (again, it is enough to check that $\Theta^2:H\to H\otimes_{ A^{\bx(1)}}\q\Omega^2_{ A^{\bx(1)}/A}$ vanishes), in which case \[H\otimes_{ A^{\bx(1)}}\q\Omega^\blob_{ A^{\bx(1)}/A}=[H\xto{\Theta}H\otimes_{ A^{\bx(1)}}\q\Omega^1_{A^{\bx(1)}/A}\xto{\Theta}H\otimes_{ A^{\bx(1)}}\q\Omega^2_{ A^{\bx(1)}/A}\xto{\Theta}\cdots]\] is a complex of $A$-modules which we will call the associated {\em $q$-Higgs complex}. In particular, we have the $q$-Higgs complex $A^{\bx(1)}\to \q\Omega^1_{ A^{\bx(1)}/A}\to \q\Omega^2_{ A^{\bx(1)}/A}\to\cdots$ of $A^{\bx(1)}$ itself, which may be explicitly described by base changing the q-de Rham complex $q\Omega^\blob_{A^\bx/A}$ along $\phi:A\to A$ and then multiplying the $n^\sub{th}$ differential by $[p]_q^n$.

The logarithmic coordinates $\Theta_i^\sub{log}:H\to H$ of the q-Higgs field are given by the formula \[\Theta(h)=\sum_{i=1}^d\Theta_i^\sub{log}(h)\otimes\dlog(U_i),\] and as in Lemma \ref{lemma_q_coordinates} one sees that $\Theta$ is flat if and only if the $\Theta_i$ pairwise commute.

\begin{remark}[{Specialisations modulo $q-1$ and $[p]_q$}]\label{remark_specialisation_mod}
Assume in this remark that our set-up for q-de Rham cohomology arises by deformation theory as in Remark \ref{remark_q_via_deformation}, and let $(H,\Theta)$ be a $ A^{\bx(1)}$-module with q-Higgs field. Then, modulo $q-1$ the q-Higgs field induces an $A/(q-1)$-linear $p$-connection on the $A/(q-1)\otimes_{\phi,A}A^\bx$-module $H/(q-1)$. Similarly, modulo $[p]_q$ it induces an $A/[p]_q$-linear Higgs field on the $A/[p]_q\otimes_{\phi,A}A^\bx$-module $H/[p]_q$.
\end{remark}

Next we explain how modules with q-Higgs fields over $ A^{\bx(1)}$ are related to modules with q-connection over $A^\bx$, ultimately leading to a q-Simpson correspondence. The key initial observation is that the Frobenius endomorphism $(\phi,\phi_\Omega)$ of the differential ring $d_q:A^\bx\to\q\Omega^1_{A^\bx/A}$ (see Lemma \ref{lemma_phi_on_qdR}) factors through the differential ring $[p]_q d_{q^p}: A^{\bx(1)}\To\q\Omega^1_{ A^{\bx(1)}/A}$ governing q-Higgs fields. Indeed, there is a commutative diagram
\begin{equation}\xymatrix{
A^\bx\ar[r]^-{d_{q}}\ar[d]_{W} \ar@/_1cm/[dd]_{\phi} & \q\Omega^1_{A^\bx/A}\ar[d]^{[p]_q W_\Omega}\ar@/^1cm/[dd]^{\phi_\Omega}\\
 A^{\bx(1)}\ar[r]^-{[p]_q d_{q^p}}\ar[d]_F & \q\Omega^1_{ A^{\bx(1)}/A}\ar[d]^{F_\Omega}\\
A^\bx\ar[r]^-{d_{q}} &\q\Omega^1_{A^\bx/A}
}\label{eqn_WF0}\end{equation}
where \[W_\Omega:\Omega^1_{A^\bx/A}\to \Omega^1_{ A^{\bx(1)}/A},\qquad \sum_{i=1}^df_i\dlog(U_i)\mapsto \sum_{i=1}^dW(f_i)\dlog(U_i)\]
and \[F_\Omega:\Omega^1_{ A^{\bx(1)}/A}\to \Omega^1_{A^\bx/A},\qquad \sum_{i=1}^df_i\dlog(U_i)\mapsto \sum_{i=1}^dF(f_i)\dlog(U_i)\] are the obvious maps induced by $W$ and $F$. The vertical maps in (\ref{eqn_WF0}) define morphisms of differential rings, as is easily checked in the same way as Lemma \ref{lemma_phi_on_qdR}, which extend uniquely to morphisms of differential graded $A$-algebras between the q-de Rham and q-Higgs complexes; there are therefore induced base change functors \begin{equation}\categ{3.3cm}{modules with flat q-connection over $A^\bx$}\xto{(W,[p]_q W_\Omega)^*}\categ{3.3cm}{modules with flat q-Higgs field over $ A^{\bx(1)}$}\xto{(F,F_\Omega)^*} \categ{3.3cm}{modules with flat q-connection over $A^\bx$}\label{eqn_WF}\end{equation} whose composition is precisely the functor $(N,\nabla)\mapsto (\phi^*N,\phi^*\nabla)$ which was studied in \S\ref{ss_Frobenius}.

A putative ``q-Simpson correspondence'' would ideally predict, at least under favourable circumstances, that the base change functor $(F,F_\Omega)^*$ establishes an equivalence between the full subcategories consisting of objects whose q-Higgs field, resp.\ q-connection, is $[p]_q$-adically quasi-nilpotent; we will prove this in the presence of Frobenius structures (see Corollary \ref{corollary_q_Simp}). We first define this nilpotence condition and then, for the sake of clarity, explicitly spell out the functor $(F,F_\Omega)^*$ in terms of coordinates.

\begin{definition}\label{definition_nilpotent_connection}
Let $\Theta=\sum_{i=1}^d\Theta_i^\sub{log}(-)\otimes\dlog(U_i):H\to H\otimes_{ A^{\bx(1)}}\q\Omega^1_{ A^{\bx(1)}/A}$ be a q-Higgs field on a $ A^{\bx(1)}$-module $H$; we say that $\Theta$ is {\em $[p]_q$-adically quasi-nilpotent} if the operator $\Theta_i:=U_i^{-1}\Theta_i^\sub{log}:H\to H$ is $[p]_q$-adically quasi-nilpotent in the usual sense for each $i=1,\dots,d$, i.e., since the operator is $A$-linear this means that for each $h\in H$ there exists $m\ge0$ such that $\Theta_i^m(h)\in[p]_q H$. (Remark: since $\Theta_i^\sub{log}$ modulo $[p]_q$ is $( A^{\bx(1)})/[p]_q$-linear, it is  equivalent to ask that $\Theta_i^\sub{log}$ itself be $[p]_q$-adically quasi-nilpotent.)

Similarly, a q-connection $\nabla=\sum_{i=1}^d\nabla_i^\sub{log}(-)\otimes\dlog(U_i):N\to N\otimes_{A^\bx}\Omega^1_{A^\bx/A}$ on an $A^\bx$-module $N$ is said to be {\em $[p]_q$-adically quasi-nilpotent} if each $A$-linear endomorphism $\nabla_i:=U_i^{-1}\nabla_i^\sub{log}$ is {\em $[p]_q$-adically quasi-nilpotent}. (In this case it is essential that the nilpotence condition is imposed on the non-logarithmic coordinate $\nabla_i$ rather than $\nabla_i^\sub{log}$.)
\end{definition}

\begin{remark}\label{remark_nilpotent_connection}
In Section \ref{section_prismatic_crystals} we will be interested in the weaker condition that the q-Higgs field $\Theta$ or q-connection $\nabla$ is {\em $(p,[p]_q)$-adically quasi-nilpotent} (equivalently, $(p,q-1)$-adically quasi-nilpotent). In other words, this means that the Higgs field $\Theta$ mod $[p]_q$ or connection $\nabla$ mod $q-1$ is $p$-adically quasi-nipotent.
\end{remark}

In terms of logarithmic coordinates, the functor $(F,F_\Omega)^*$ is given explicitly by \[(H,\Theta=\sum_{i=1}^d\Theta_i^\sub{log}(-)\otimes\dlog(U_i))\mapsto (N=H\otimes_{ A^{\bx(1)}}A^\bx,\nabla=\sum_{i=1}^d\nabla_i^\sub{log}(-)\otimes\dlog(U_i))\] where the components are given by \begin{equation}\nabla_i^\sub{log}:H\otimes_{ A^{\bx(1)}}A^\bx\To H\otimes_{ A^{\bx(1)}}A^\bx,\qquad h\otimes f\mapsto \Theta_i^\sub{log}(h)\otimes \gamma_i(f)+h\otimes d_{q,i}^\sub{log}(f).\label{eqn_Higg_to_conn}\end{equation} The main technical result of this subsection is the following analysis of the functor $(F,F_\Omega)^*$:

\begin{theorem}\label{theorem_Simpsons}
Assume that our set-up for q-de Rham cohomology satisfying (qDR1), (qDR2), (qDR$\phi$) arises by deformation theory as in Remark \ref{remark_q_via_deformation}. Then base change along $F$ defines a fully faithful functor
\[(F,F_\Omega)^*:\categ{6.2cm}{$[p]_q$-adically complete and separated modules with $[p]_q$-adically quasi-nilpotent, flat q-Higgs field over $ A^{\bx(1)}$}\To \categ{6cm}{$[p]_q$-adically complete and separated modules with $[p]_q$-adically quasi-nilpotent, flat q-connection over $A^\bx$}.\]
Moreover:
\begin{enumerate}
\item Given $(H,\Theta)$ in the domain category and its image $(N,\nabla)$ in the codomain, there is a natural injective quasi-isomorphism of complexes of $A$-modules \[H\otimes_{ A^{\bx(1)}}\mathrm{q}\Omega^\blob_{ A^{\bx(1)}/A}\quis N\otimes_{A^\bx}\mathrm{q}\Omega^\blob_{A^\bx/A}\] from the q-Higgs complex of $H$ to the q-de Rham complex of $N$.
\item The essential image of $(F,F_\Omega)^*$ contains any object $(N,\nabla)$ of the codomain category satisfying the following condition: $N$ is $[p]_q$-torsion-free and there exists a morphism of $A^\bx$-modules with q-connection $\phi_N:\phi^*N\to N$ such that $\op{coker}\phi_N$ is killed by a power of $[p]_q$.
\end{enumerate}
\end{theorem}

\begin{remark}
\begin{enumerate}
\item It seems possible to weaken the hypotheses of the theorem slightly, in particular to drop the requirement that the Frobenius lift $\phi$ satisfies $\phi(U_i)=U_i^p$ for $i=1,\dots,d$, but it complicates the proof and we have no use for it.
\item A similar such functor has been shown to be an equivalence by Gros--Le~Stum--Quir\'os \cite{GrosLeStumQuiros2010, Gros2019}, in the case that $q^p=1$ and $d=1$, by proving that the non-commutative algebras governing the two categories are Morita equivalent. It is natural to ask whether their method could be extended to our situation; our current proof of the surjectivity assertion is completely different, using the Frobenius to directly construct the descent of the q-connection to a q-Higgs field.
\end{enumerate}
\end{remark}

\begin{proof}[Proof of Theorem \ref{theorem_Simpsons}, part I]
We begin by fixing some notation which will be used throughout the proof. Let $(H,\Theta=\sum_{i=1}^d\Theta_i^\sub{log}(-)\dlog(U_i))$ be a module with $[p]_q$-adic quasi-nilpotent q-Higgs field over $ A^{\bx(1)}$ and let $\Theta_i:=\Theta_i^\sub{log}(-)U_i^{-1}:H\to H$ be its non-logarithmic coordinates. Letting $\nabla_i:H\otimes_{ A^{\bx(1)}}A^\bx\to H\otimes_{ A^{\bx(1)}}A^\bx$ denote the non-logarithmic coordinates for the associated q-connection, the above explicit formula (\ref{eqn_Higg_to_conn}) for $\nabla_i$ shows that \[\nabla_i(h\otimes \ul U^{\ul k})=\begin{cases}(q^{k_i}\Theta_i^\sub{log}(h)+[k_i]_qh)\otimes \ul U^{\ul{k}}U_i^{-1} & 0<k_i<p\\ \Theta_i(h)\otimes \ul U^{\ul{k}}U_i^{p-1}&k_i=0\end{cases}\] where $0\le k_1,\dots,k_d<p$ and $h\in H$; here $[k]_q:=1+q+\cdots+q^{k-1}$ is the $q$-analogue of $k$, and we write $\ul U^{\ul k}:=U_1^{k_1}\cdots U_d^{k_d}$. For each $i=1,\dots,d$ and $0<k<p$, we will denote the operator appearing in the first case by \[\al_i^{(k)}:=q^k\Theta_i^\sub{log}(-)+[k]_q\op{id}:H\to H.\] We note that $\al_i^{(k)}$ is an $A$-linear automorphism of $H$: firstly, $[k]_q$ is invertible in $A$ since it is $\equiv k$ mod $q-1$ and $A$ is $(p,q-1)$-adically complete; secondly, $\Theta_i^\sub{log}$ is $[p]_q$-adically quasi-nilpotent and $H$ is $[p]_q$-adically complete.

We are now prepared to prove that the functor in the theorem is well-defined, namely that it preserves $[p]_q$-adic quasi-nilpotence. This holds vacuously in the context of Corollary \ref{corollary_q_Simp} since $[p]_q$-adic quasi-nilpotence is automatic in the presence of Frobenius structures (by Lemma \ref{lemma_phi_implies_xi_nilp}), so for simplicity of notation we will only treat the case $d=1$ (the general case being treated by examining a block matrix built out of $p^{d-1}$ copies of the one below. The above formula for $\nabla_1$ shows that the matrix representing $\nabla_1$ as an element of $\End(H\otimes_{ A^{\bx(1)}}A^\bx)=\End (\bigoplus_{k=0}^{p-1}HU_1^k)=\op{M}_p(\End(H))$ is
\[\begin{pmatrix}
0 & \al_1^{(1)} & 0 & \cdots &0\\
0 & 0 & \al_1^{(2)} & \cdots &0\\
0 & 0 & \ddots & 0 &0\\
0 & 0 & \cdots & \al_1^{(p-2)} &0\\
0 & 0 & \cdots & 0 & \al_1^{(p-1)}\\
\Theta_1 & 0&\cdots &0&0
\end{pmatrix}\]
The $p$-fold composition $\nabla_1^p$ is therefore represented by the diagonal matrix with entries \[\al_1^{(1)}\cdots\al_1^{(p-1)}\Theta_1,\quad \al_1^{(2)}\cdots\al_1^{(p-1)}\Theta_1\al_1^{(1)},\quad\dots,\quad \al_1^{(p-1)}\Theta_1\al_1^{(1)}\cdots\al_1^{(p-2)},\quad\Theta_1\al_1^{(1)}\cdots\al_1^{(p-1)},\] and so $\nabla_1^{pm}$ is represented by the diagonal matrix with entries given by the $m$-fold compositions of these operators. The key is now to recall that $\Theta_1$ is $ A^{\bx(1)}$-linear modulo $[p]_q$, whence $\Theta_1$ and $\Theta_1^\sub{log}=U_1\Theta_1$ commute modulo $[p]_q$; therefore the operators $\al_1^{(1)},\dots,\al_1^{(p-1)},\Theta_1$ are pairwise commuting modulo $[p]_q$ and so $\nabla_1^{pm}$ is congruent modulo $[p]_q$ to the operator \[\bigoplus_{k=0}^{p-1}HU_1^k\to \bigoplus_{k=0}^{p-1}HU_1^k,\qquad\sum_{k=0}^{p-1}h_kU_1^k\mapsto\sum_{k=0}^{p-1}(\al_1^{(1)}\cdots\al_1^{(p-1)})^m\Theta_1^m(h_k) U_1^k.\] But the right side vanishes modulo $[p]_q$ for $m\gg0$, since $\Theta_1$ is $[p]_q$-adically quasi-nilpotent by hypothesis; therefore the same is true of $\nabla$, completing the proof that the functor is well-defined.

We now return to the case of general $d\ge1$ and show that the functor is fully faithful; so let $(H,\Theta)$, $(H',\Theta')$ belong to the domain category and let $f:H\otimes_{ A^{\bx(1)}}A^\bx\to H'\otimes_{ A^{\bx(1)}}A^\bx$ be an $A^\bx$-linear morphism compatible with the corresponding q-connections $\nabla$, $\nabla'$. Write $f|_{H}=\bigoplus_{0\le k_1,\dots,k_d<p}f_{\ul k}(-)U_{\ul k}$, i.e., $f(h)=\bigoplus_{\ul k}f_{\ul k}(h)U_{\ul k}$ for each $h\in H$, where the maps $f_{\ul k}:H\to H'$ are $ A^{\bx(1)}$-linear. Then, given $h\in H$, we have
\begin{align*}
f(\nabla_i(h))&=f(\Theta_i(h)\otimes U_i^{p-1})\\
&=\sum_{\ul k}f_{\ul k}(\Theta_i(h))\otimes U_{\ul k}U_i^{p-1}\\
&=\sum_{\substack{\ul k\\\sub{s.t. }k_i=0}}f_{\ul k}(\Theta_i(h))\otimes U_{\ul k}U_i^{p-1}
+\sum_{\substack{\ul k\\\sub{s.t. }k_i\neq 0}}f_{\ul k}(\Theta_i(h))U_i\otimes U_{\ul k}U_i^{-1}
\end{align*} and similarly
\begin{align*}
\nabla_i'(f(h))&=\nabla_i'(\sum_{\ul k}f_{\ul k}(h)\otimes \ul U^{\ul k})\\
&=\sum_{\substack{\ul k\\\sub{s.t. }k_i=0}}\Theta_i'(f_{\ul k}(h))\otimes U_{\ul k}U_i^{p-1}
+\sum_{\substack{\ul k\\\sub{s.t. }k_i\neq0}}\al_i^{(k_i)}(f_{\ul k}(h))\otimes U_{\ul k}U_i^{-1}.
\end{align*}

But $f$ commutes with each $\nabla_i$ so we deduce at once that $f_{\ul k}$ commutes with $\Theta_i$ whenever $k_i=0$; in particular $f_0:H\to H'$ commutes with $\Theta_1,\dots,\Theta_d$ and hence is a morphism of $ A^{\bx(1)}$-modules with q-Higgs fields. We similarly deduce that $U_if_{\ul k}\Theta_i=\al_i^{(k_i)}f_{\ul k}$ whenever $k_i\neq0$. Fixing a non-zero multi-index $\ul k$ (i.e., $k_i\neq0$ for some $i$, which we fix for the remainder of the paragraph) our goal is to show that $f_{\ul k}=0$; indeed, this will imply that $f$ is simply the base change of the morphism $f_0$. By the $ A^{\bx(1)}$-linearity of $f_{\ul k}$ we may rewrite the previous relation as $f_{\ul k}\circ (U_i\Theta_i)=\al_i^{(k)}\circ f_{\ul k}$, whence we deduce $f_{\ul k}\circ(U_i\Theta_i)^m=\al_i^{(k_i)\,m}\circ f_{\ul k}$ for any $m\ge0$; the $[p]_q$-adic quasi-nilpotence of $U_i\Theta_i$ now implies that for any $h\in H$ and $m\ge0$ there exists $m'\gg0$ such that $\al_i^{(k)\,m'}(f_{\ul k}(h))\in[p]_q^{m}H'$. But we have explained above that $\al_i^{(k)}$ is an $A$-linear automorphism of $H'$, so it follows that $f_{\ul k}(h)\in [p]_q^mH'$; since $h\in H$ and $m\ge0$ were arbitrary (and $H'$ is assumed to be $[p]_q$-adically separated), we have indeed shown that $f_{\ul k}=0$, as desired.

Next we establish part (i) of the theorem, namely the relation between the q-Higgs and q-de Rham complexes; so let $(H,\Theta)$ be in the domain category, with image $(N,\nabla)$ under the base change functor $(F,F_\Omega)^*$. Write $F_H:H\into N=H\otimes_{ A^{\bx(1)},F}A^\bx$, $h\mapsto h\otimes 1$ for the canonical inclusion; then it follows formally from the fact that $(F,F_\Omega)$ is a morphism of differential rings that $F_H$ extends from an embedding of the q-Higgs complex $H\otimes_{ A^{\bx(1)}}\q\Omega^\blob_{ A^{\bx(1)}/A}$ into the q-de Rham complex:

\[\xymatrix{
H\ar@{^(->}[d]_{F_H}\ar[r]^-{\Theta}&N\otimes_{ A^{\bx(1)}}\q\Omega^1_{ A^{\bx(1)}/A}\ar@{^(->}[d]_{F_H\otimes F_\Omega}\ar[r]^-{\Theta}&H\otimes_{A^{\square(1)}}\q\Omega^2_{ A^{\bx(1)}/A}\ar@{^(->}[d]_{F_H\otimes F_\Omega}\ar[r]^-\Theta&\cdots\\
N\ar[r]_-{\nabla} &N\otimes_{A^\bx}\q\Omega^1_{A^\bx/A}\ar[r]_-{\nabla} &N\otimes_{A^\bx}\q\Omega^2_{A^\bx/A}\ar[r]_-\nabla&\cdots
}\]
(by a slight abuse of notation we continue to write $F_\Omega:\q\Omega^n_{ A^{\bx(1)}/A}\to \q\Omega^n_{A^\bx/A}$ for the map induced by $F_\Omega$ on the higher wedge powers, characterised by $f\dlog(U_{i_1})\wedge\cdots\wedge\dlog(U_{i_n})\mapsto F(f)\dlog(U_{i_1})\wedge\cdots\wedge\dlog(U_{i_n})$). We want to prove that this canonical inclusion is a quasi-isomorphism. In fact, as in Proposition \ref{proposition_reps_vs_q_connections}, the q-de Rham complex identifies with the Koszul complex $K(\nabla_1^\sub{log},\dots,\nabla_d^\sub{log};N)$; similarly, the q-Higgs complex identifies with $K(\Theta_1^\sub{log},\dots,\Theta_d^\sub{log};H)$, and the above inclusion is the canonical inclusion of Koszul complexes induced by $F_H$, which makes sense since $F_H\Theta_i^\sub{log}=\nabla_i^\sub{log}F_H$ for each $i=1,\dots,n$. But this inclusion of Koszul complexes is split, with complement \[\bigoplus_{\substack{0\le k_1,\dots,k_d<p\\\sub{not all zero}}}K(\nabla_1^\sub{log},\dots,\nabla_d^\sub{log};H\ul U^{\ul k}),\] so it is enough to check that each of these latter Koszul complexes is acyclic: picking $i$ for which $k_i\neq 0$, then the action of $\nabla_i^\sub{log}$ on $H\ul U^{\ul k}\subseteq N$ is given by $\al_i^{(k_i)}$ on $H$, which we already know is an isomorphism and which therefore implies the desired acyclicity.
\end{proof}

It remains to prove the surjectivity part of Theorem \ref{theorem_Simpsons}; for this we will need the following rigidity result:

\begin{lemma}\label{lemma_injective_rigidity}
Let $(H,\Theta)$ be a $ A^{\bx(1)}$-module with $[p]_q$-adically quasi-nilpotent q-Higgs field, $(N,\nabla)$ an $A^\bx$-module with $[p]_q$-adically quasi-nilpotent q-connection, and $f:H\otimes_{ A^{\bx(1)},F}A^\bx\to N$ a morphism of $A^\bx$-modules with q-connection; assume also that the module $H$ is killed by a power of $[p]_q$. If the restriction $f|_H:H\to N$ is injective, then so is $f$.
\end{lemma}
\begin{proof}
First we note that, in terms of logarithmic coordinates, the compatibility of $f$ with the q-connections on each side amounts to the following explicit formulae:
\[\nabla_i^\sub{log}(f(h)\ul U^{\ul k})=\begin{cases}f(\al_i^{(k_i)}(h))\ul U^{\ul k} & 0<k_i<p\\f(\Theta_i^\sub{log}(h))\ul U^{\ul k} & k_i=0\end{cases}\] for $0\le k_1,\dots,k_d<p$ and $h\in H$. Here the $\al_i^{(k)}$ are as defined in the first part of the proof of Theorem \ref{theorem_Simpsons}.

Now, among all linear relations of the form $\sum_{\ul k}f(h_{\ul k})\ul U^{\ul k}=0$ in $N$, for $h_{\ul k}\in H$, let us consider one which minimises the cardinality of the support $\{\ul k:h_{\ul k}\neq 0\}$. Furthermore, for each $i$ we may assume that $h_{\ul k}\neq 0$ for some $\ul k$ satisfying $k_i=0$: otherwise we may just divide out the relation by a positive power of $U_i$. But it is also true that the linear relation is not simply given by $f(h_{\ul 0})=0$, as $f|_H$ is injective; therefore there exists some $i$ (which we now fix) for which there is a multi-index $\ul k$ such that $k_i\neq 0$ and $h_{\ul k}\neq 0$. In conclusion, we arrive at a linear relation $\sum_{\ul k}f(h_{\ul k})\ul U^{\ul k}=0$ of shortest possible length and a value of $i$ such that, rewriting the relation as
\[\sum_{\substack{\ul k\\\sub{s.t.}k_i=0}}f(h_{\ul k})\ul U^{\ul k}+\sum_{\substack{\ul k\\\sub{s.t.}k_i\neq0}}f(h_{\ul k})\ul U^{\ul k}=0,\] each summation is non-empty (hence strictly shorter).

Applying $(\nabla_i^\sub{log})^m$ and using the formulae from the first paragraph, we deduce that
\[\sum_{\substack{\ul k\\\sub{s.t.}k_i=0}}f((\Theta_i^\sub{log})^m(h_{\ul k}))\ul U^{\ul k}+\sum_{\substack{\ul k\\\sub{s.t.}k_i\neq0}}f((\al_i^{(k_i)})^m(h_{\ul k}))\ul U^{\ul k}=0\]
By the hypotheses (namely, $\Theta_i^\sub{log}$ is $[p]_q$-adically quasi-nilpotent and $H$ is killed by a power of $[p]_q$) we may choose $m\gg0$ such that all terms in the first summation vanish. We therefore obtain a strictly shorter relation $\sum_{\ul k\sub{ s.t. }k_i\neq0}f((\al_i^{(k_i)})^m(h_{\ul k}))\ul U^{\ul k}=0$ and so deduce that $(\al_i^{(k_i)})^m(h_{\ul k})=0$ for all multi-indexes $\ul k$ appearing in this sum; but each $(\al_i^{(k_i)})^m$ is an automorphism (as in the first part of the proof of Theorem \ref{theorem_Simpsons}), so in fact these $h_{\ul k}$ all vanish. But this contradicts the choice of our linear relation and so completes the proof.
\end{proof}

We can now complete the proof of Theorem \ref{theorem_Simpsons}:

\begin{proof}[Proof of Theorem \ref{theorem_Simpsons}, part II]
We now turn to the sujectivity assertion, so let $(N,\nabla)$ belong to the codomain category and assume that $N$ is $[p]_q$-torsion-free and there exists a morphism of $A^\bx$-modules with q-connection $\phi_N:\phi^*N\to N$ such that $\op{coker}\phi_N$ is killed by a power of $[p]_q$. Our goal is to construct an object $(H,\Theta)$ in the domain category such that $H\otimes_{ A^{\bx(1)},F}A^\bx=N$; the fully faithfulness which we have already proved shows that $H$ depends only on $N$ itself, but we will use the Frobenius $\phi_N$ in the construction.

We first introduce filtrations on $W^*N=N\otimes_{A^\bx,W} A^{\bx(1)}$ and on $\phi^*N=N\otimes_{A^\bx,\phi}A^\bx=W^*N\otimes_{ A^{\bx(1)},F}A^\bx$ by pulling back the $[p]_q$-adic filtration on $N$ along $\phi_N$, i.e., \[\Fil^j(W^*N)=\{h\in W^*N:\phi_N(h\otimes1)\in[p]_q^jN\}\] and \[\Fil^j(\phi^*N)=\{n\in \phi^*N:\phi_N(n)\in[p]_q^jN\}.\] Since the submodules $[p]_q^jN\subseteq N$ are clearly preserved by the $\nabla_i^\sub{log}$ and $\phi_N:(\phi^*N,\phi^*\nabla)\to(N,\nabla)$ is a morphism of $A^\bx$-modules with q-connection, the submodules $\Fil^j(\phi^*N)$ are preserved by the $\nabla_i^\sub{log}$. Recall moreover that the $ A^{\bx(1)}$-module $W^*N$ may be equipped with a q-Higgs field $\Theta$, by base changing the q-connection on $N$ along the morphism of differential rings $(W,[p]_q W_\Omega)$ as in diagram (\ref{eqn_WF0}); since base changing this further along $(F,F_\Omega)$ gives the q-connection $\phi^*\nabla$ on $\phi^*N$, the explicit formula (\ref{eqn_Higg_to_conn}) shows that the coordinates of $\Theta$ are given by \[\Theta_i^\sub{log}=(\phi^*\nabla)_i^\sub{log}|_{W^*N},\] where we view $W^*N$ as a submodule of $\phi^*N=W^*N\otimes_{A^{\bx(1)},F}A^\bx$ via the canonical inclusion. Intersecting $\Fil^j(\phi^*N)$ along this canonical inclusion gives $\Fil^j(W^*N)$ by definition, and so we conclude that $\Fil^j(W^*N)$ is preserved by the $\Theta_i^\sub{log}$; in short, each $\Fil^j(W^*N)$ is a sub q-Higgs module of $W^*N$. Furthermore, the proof of Lemma \ref{lemma_phi_implies_xi_nilp}(i) shows that $\Theta_i^\sub{log}(W^*N)\subseteq[p]_q W^*N$

The key to the proof is establishing the following claim:\footnote{This claim is not merely a technical step in the proof of the theorem, but of independent interest: namely, it states that the Nygaard-style filtration $\Fil^j(\phi^*N)$ on the module with q-connection $\phi^*N$ is compatible with its descent to the module with q-Higgs field $W^*N$.} the canonical map \[\Fil^j(W^*N)\otimes_{ A^{\bx(1)},F}A^\bx\To \Fil^j(\phi^*N)\] is an isomorphism of $A^\bx$-modules for each $j\in\bb Z$. Before proving this claim we explain how it completes the proof. Let $b\gg0$ be large enough so that $\phi_N(\phi^*N)\supseteq[p]_q^bN$; then it is clear from the definition of the filtration that $\phi_N$ restricts to an isomorphism $\phi_N:\Fil^b(\phi^*N)\isoto [p]_q^bN$. The claim allows us to rewrite the domain of this latter map as $\Fil^b(W^*N)\otimes_{ A^{\bx(1)},F}A^\bx,$ whence \begin{equation}\Fil^b(W^*N)\otimes_{A^{\bx(1)},F}A\isoto [p]_q^bN\cong N,\label{eqn_descent_of_filtration}\end{equation} thereby writing $N$ as the base change of a q-Higgs module along $(F,F_\Omega)$, as desired.

It remains to prove the claim, which we do by induction on $j$ noting that it is trivial when $j\le0$ as then $\Fil^j(W^*N)=W^*N$ and $\Fil^j(\phi^*N)=\phi^*N$. Assuming that the claim is true for some $j\ge0$, we consider the diagram
\[\xymatrix{
0\ar[r] & \Fil^{j+1}(W^*N)\otimes_{ A^{\bx(1)},F}A^\bx \ar[r]\ar[d]& \Fil^{j}(W^*N)\otimes_{ A^{\bx(1)},F}A^\bx\ar[r]\ar[d]_{\cong} &  \op{gr}^{j}(W^*N)\otimes_{ A^{\bx(1)},F}A^\bx\ar[r]\ar[d]&0\\
0\ar[r] & \Fil^{j+1}(\phi^*N) \ar[r]& \Fil^{j}(\phi^*N)\ar[r] &  \op{gr}^{j}(\phi^*N)\ar[r]&0
}\]
 and so deduce that it is enough to prove injectivity of the right vertical arrow. Since the filtration on $W^*N$ is the intersection of the filtration on $\phi^*N$ along the canonical inclusion $W^*N\into\phi^*N$, the induced map on gradeds $\op{gr}^j(W^*N)\to \op{gr}^j(\phi^*N)$ is certainly injective. But now we note that the conditions of Lemma \ref{lemma_injective_rigidity} are all satisfied: the q-Higgs field $\Theta$ on $W^*N$ induces a q-Higgs field on $\op{gr}^j(W^*N)$, and similarly the q-connection $\phi^*\nabla$ on $\phi^*N$ induces a q-connection on $\op{gr}^j(\phi^*N)$; the morphism $\op{gr}^{j}(W^*N)\otimes_{ A^{\bx(1)},F}A^\bx\to \op{gr}^{j}(\phi^*N)$ is compatible with the q-connections on each side because the same is formally true of the identification $W^*N\otimes_{ A^{\bx(1)},F}A^\bx=\phi^*N$; finally, $\op{gr}^j(\phi^*N)$ is killed by $[p]_q$ by definition. So that lemma implies that the right vertical arrow is indeed injective, thereby completing the proof of the claim and hence of part (ii) of the theorem.
%
\end{proof}

To obtain a clear q-Simpson correspondence from the theorem we will consider Frobenius structures on q-Higgs bundles. So let $\phi:=W\circ F$ be the induced Frobenius endomorphism of $ A^{\bx(1)}$; a q-Higgs field $\Theta$ on a $ A^{\bx(1)}$-module $H$ induces a q-Higgs field $\phi^*\Theta$ on the Frobenius pullback $\phi^*H=H\otimes_{ A^{\bx(1)},\phi} A^{\bx(1)}$, by base changing along the morphism of differential rings
\[\xymatrix{
 A^{\bx(1)}\ar[r]^-{d_q}\ar[d]_\phi & \Omega^1_{ A^{\bx(1)}/A}\ar[d]^{\phi_\Omega:\sum_{i=1}^df_i\dlog(U_i)\mapsto[p]_q \sum_{i=1}^d\phi(f_i)\dlog(U_i)}\\
 A^{\bx(1)}\ar[r]^-{d_q} &\Omega^1_{ A^{\bx(1)}/A}
}\]
(this is simply the composition $(W,[p]_qW_\Omega)\circ(F,F_\Omega)$ of the two morphisms from (\ref{eqn_WF0})). We can then formulate the obvious analogue for q-Higgs modules which we have already considered for modules with q-connection:

\begin{definition}\label{definition_phi_on_Higgs}
Let $\textrm{qHIG}(A^\bx,\phi)$ denote the category of pairs $(H,\phi_H)$, where $H\in \textrm{qHIG}( A^{\bx(1)})$ and $\phi_H:(\phi^*H)[\tfrac1{[p]_q}]\to H[\tfrac1{[p]_q}]$ is an isomorphism of modules with q-Higgs field over $ A^{\bx(1)}[\tfrac1{[p]_q}]$.
\end{definition}

In order for the previous definition to be reasonable we need $H$ to be included in $H[\tfrac1{[p]_q}]$; therefore we will assume in Lemma \ref{lemma_phi_implies_xi_nilp} and Corollary \ref{corollary_q_Simp} that $[p]_q$ is a non-zero-divisor of $ A^{\bx(1)}$. This assumption could be avoided by restricting to modules with an effective Frobenius (i.e., already defined as $\phi_H:\phi^*H\to H$, rather than inverting $[\tfrac1{[p]_q}]$), but in any case it is satisfied in our cases of interest, especially that of Corollary \ref{corollary_descent_of_reps_to_phi_twist}.

In the presence of a Frobenius structure, $[p]_q$-adic quasi-nilpotence is automatic:

\begin{lemma}\label{lemma_phi_implies_xi_nilp}
Assume in parts (ii) and (iii) that $[p]_q$ is a non-zero-divisor of $ A^{\bx(1)}$.
\begin{enumerate}
\item Given a module with flat q-connection $(N,\nabla)$ over $A^\bx$, then the q-Higgs field on its base change along $(W,[p]_qW_\Omega)$ is $[p]_q$-adically quasi-nilpotent.
\item Given $(H,\phi_H)\in \textrm{\rm qHIG}( A^{\bx(1)},\phi)$, then the q-Higgs field on $H$ is $[p]_q$-adically quasi-nilpotent.
\item Assume that $A^\bx$ is generated as a module over $ A^{\bx(1)}$ (via the morphism $F$) by the elements $U_1^{k_1}\cdots U_d^{k_1}$ for $0\le k_1,\dots,k_d<p$ (e.g., this holds if our set-up arises as in Remark \ref{remark_q_via_deformation}); then, given $(N,\phi_N)\in\textrm{\rm qMIC}(A^\bx,\phi)$, the q-connection on $N$ is $[p]_q$-adically quasi-nilpotent.
\end{enumerate}
\end{lemma}
\begin{proof}
(i): Let $(N,\nabla)$ be an $A^\bx$-module with q-connection and let $(H,\Theta)$ be the q-Higgs field over $ A^{\bx(1)}$ obtained by base changing it along $(W,[p]_qW_\Omega)$. Identifying $H=N\otimes_{A^\bx,W} A^{\bx(1)}$ with $N\otimes_{A,\phi}A$, the components of the q-Higgs field on $H$ are simply given by \[\Theta_i^\sub{log}(n\otimes\lambda)=\nabla_i^\sub{log}(n)\otimes[p]_q\lambda\] for $n\in N$ and $\lambda\in A$. In particular, $\Theta_i^\sub{log}(H)\subseteq[p]_q H$, which is a much stronger condition on $\Theta$ than $[p]_q$-adic quasi-nilpotence.

(ii): The q-Higgs field $\phi^*\Theta$ on $\phi^*H$ has logarithmic coordinates given by \[(\phi^*\Theta)_i^\sub{log}(h\otimes f)=\Theta_i^\sub{log}(h)\otimes[p]_q\gamma_i(f)+h\otimes[p]_q d_{q^p,i}^\sub{log}(f),\] for $h\in H$ and $f\in A^{\bx(1)}$; in particular, $(\phi^*\Theta)_i^\sub{log}$ has image contained in $[p]_q \phi^*H$ (and so is $[p]_q$-adically quasi-nilpotent, though this is not relevant at the moment). By iterating we see that $((\phi^*\Theta)_i^\sub{log})^m(\phi^*H)\subseteq[p]_q^m \phi^*H$ for any $m\ge1$, whence applying $\phi_H$ shows that \[(\Theta_i^\sub{log})^m(\phi_H(\phi^*H))\subseteq[p]_q^m\phi_H(\phi^*H).\] Now pick integers $a\le b$ satisfying $[p]_q^bH\subseteq\phi_H(\phi^*H)\subseteq [p]_q^aH$, and apply the previous inclusion with $m=b-a+1$ to deduce that $(\Theta_i^\sub{log})^{b-a+1}(H)\subseteq[p]_qH$, as desired.

(iii): The q-connection on $\phi^*N$ has non-logarithmic components given by \[(\phi^*\nabla)_{i}(n\otimes f)=\nabla_{i}(n)\otimes [p]_q \phi(U_i)U_i^{-1}\gamma_i(f)+ n\otimes d_{q,i}(f)\] for $n\in N$ and $f\in A^\bx$. We will check in a moment that $d_{q,i}^{p}(A^\bx)\subseteq A^\bx[p]_q$ whence $(\phi^*\nabla)_i^p(\phi^*N)\subseteq (\phi^*N)[p]_q$. As in part (ii), we now pick integers $a\le b$ such that $[p]_q^bN\subseteq\phi_N(\phi^*N)\subseteq [p]_q^aN$ and deduce that $\nabla_i^{p(b-a+1)}(N)\subseteq[p]_qN$.

It just remains to check that $d_{q,i}^{p}(A^\bx)\subseteq A^\bx[p]_q$. Using the extra hypothesis on the structure of $F: A^{\bx(1)}\to A^\bx$ and the fact that $d_{q,i}(\phi(f))\in A^\bx[p]_q$ for any $f\in A^\bx$, this reduces easily to checking that $d_{q,i}^{p}(U_1^{k_1}\cdots U_d^{k_d})\in A^\bx[p]_q$ for all $0\le k_1,\dots,k_d<p$. But a direct calculation shows that $d_{q,i}^{p}(U_1^{k_1}\cdots U_d^{k_d})$ even vanishes.
\end{proof}

The following is our q-Simpson correspondence:

\begin{corollary}[$q$-Simpson correspondence]\label{corollary_q_Simp}
Assume that our set-up for q-de Rham cohomology satisfying (qDR1), (qDR2), (qDR$\phi$) arises by deformation theory as in Remark \ref{remark_q_via_deformation}; assume also that $[p]_q$ is a non-zero-divisor of $ A^{\bx(1)}$. Then the functor $(F,F_\Omega)^*$ induces an equivalence of categories \[(F,F_\Omega)^*:\text{\rm qHIG}( A^{\bx(1)},\phi)\quis\text{\rm qMIC}(A^\bx,\phi).\] Moreover, given $(H,\phi_H)\in\text{\rm qHIG}( A^{\bx(1)},\phi)$ with image $(N,\phi_N)\in\text{\rm qMIC}(A^\bx,\phi)$, there is a natural injective quasi-isomorphism of complexes of $A$-modules \[H\otimes_{ A^{\bx(1)}}\q\Omega^\blob_{ A^{\bx(1)}/A}\quis N\otimes_{A^\bx}\q\Omega^\blob_{A^\bx/A}.\]
\end{corollary}
\begin{proof}
We begin by remarking that the functor is clearly well-defined (even without appealing to the well-definedness part of Theorem \ref{theorem_Simpsons}) since the Frobenii on the differential rings $[p]_q d_{q^p}: A^{\bx(1)}\to \Omega^1_{ A^{\bx(1)}/A}$ and $d_q:A^\bx\to\Omega^1_{A^\bx/A}$ are compatible with the morphism of differential rings $(F,F_\Omega)$.

To check fully faithfulness we fix $(H_i,\phi_{H_i})\in \text{\rm qHIG}( A^{\bx(1)},\phi)$ for $i=1,2$, and a morphism $f:H_1\otimes_{ A^{\bx(1)},F}A^\bx\to H_2\otimes_{ A^{\bx(1)},F}A^\bx$ of $A^\bx$-modules which is compatible with the q-connections and the induced Frobenii. Since the q-Higgs fields on $H_1$ and $H_2$ are $[p]_q$-adically quasi-nilpotent by Lemma \ref{lemma_phi_implies_xi_nilp}(ii), the fully faithfulness in Theorem \ref{theorem_Simpsons} implies that $f$ is induced by a unique morphism of q-Higgs bundles $H_1\to H_2$. We only need to check that this latter morphism is compatible with the $\phi_{H_i}$, but this trivially follows from the compatibility after base change to $A^\bx$ thanks to the inclusions $H_i\subseteq H_i\otimes_{ A^{\bx(1)},F}A^\bx$ and $\phi^*H_i\subseteq\phi^*(H_i\otimes_{ A^{\bx(1)},F}A^\bx)$.

To check essentially surjectivity, let $(N,\phi_N)\in\text{\rm qMIC}(A^\bx,\phi)$; by twisting we reduce to the case that $\phi_N$ is effective, i.e., restricts to $\phi_N:\phi^*N\to N$. Since the q-connection is $[p]_q$-adically quasi-nilpotent by Lemma \ref{lemma_phi_implies_xi_nilp}(iii), Theorem \ref{theorem_Simpsons} shows that $N$ has the form $N=H\otimes_{ A^{\bx(1)},F}A^\bx$ for some $H\in\q\HIG( A^{\bx(1)})$ on which the q-Higgs field is $[p]_q$-adically quasi-nilpotent. We must show that the Frobenius structure $\phi_N:\phi^*N\to N$ is induced by a Frobenius $\phi_H$ on $H$; but the identification of modules with q-connection $\phi^*N=(\phi^*H)\otimes_{ A^{\bx(1)},F}A^\bx$ and the fully faithfullness of Theorem \ref{theorem_Simpsons} (note that the theorem applies, since the q-Higgs field on $\phi^*H$ was seen to be $[p]_q$-adically quasi-nilpotent in the proof of Lemma \ref{lemma_phi_implies_xi_nilp}(ii)) imply that $\phi_N$ is indeed induced by some $\phi_H$, as required.

Finally, the quasi-isomorphism between the q-Higgs and q-de Rham complexes follows from the same statement in Theorem \ref{theorem_Simpsons}, again using the q-Higgs field on any object of $\text{\rm qHIG}( A^{\bx(1)},\phi)$ is necessarily $[p]_q$-adically quasi-nilpotent.
\end{proof}

We finish the subsection by returning to our main case of interest, namely $A=A_\inf$, $A^\bx=A_\inf^\bx(R)$, $q=[\ep]$ from \S\ref{ss_framed}; note that in this case $[p]_q=\tilde\xi$. The Frobenius structure in this set-up was explained at the start of \S\ref{ss_Frobenius}. Corollary \ref{corollary_q_Simp} takes the form of an equivalence of categories \begin{equation}(F,F_\Omega)^*:\text{\rm qHIG}(A_\inf^\bx(R)^{(1)},\phi)\quis\text{\rm qMIC}(A^\bx_\inf(R),\phi),\label{eqn_q_Simpsons_for_A_inf}\end{equation} compatible with the q-Higgs and q-de Rham complexes. This may be rephrased in terms of representations as follows:

\begin{corollary}\label{corollary_descent_of_reps_to_phi_twist}
The base change functor \[-\otimes_{A_\inf^\bx(R^{(1)}),F}A_\inf^\bx(R):\Rep_\Gamma^\mu(A_\inf^\bx(R)^{(1)},\phi)\To \Rep_\Gamma^\mu(A_\inf^\bx(R),\phi)\] is an equivalence of categories compatible with group cohomology.
\end{corollary}
\begin{proof}
Although the domain category has not been defined, it should be clear: it consists of finite projective $ A^{\bx}_\inf(R)^{(1)}$-modules $H$ equipped both with a $(p,\xi)$-adically continuous $\Gamma$-action which is the identity modulo $\mu$ and with an isomorphism $\phi_H:(\phi^*H)[\tfrac1{\tilde\xi}]\isoto H[\tfrac1{\tilde\xi}]$. Arguing just as in Propositions \ref{proposition_reps_vs_q_connections} and \ref{proposition_reps_vs_q_connections_with_phi}, one checks the following: given a generalised representation $H\in\Rep^{\mu}_{\Gamma}(A_\inf^\bx(R)^{(1)})$, the map \[\Theta:H\to H\otimes_{A_\inf^\bx(R)^{(1)}}\q\Omega^1_{ A^{\bx}_\inf(R)^{(1)}/A_\inf},\quad h\mapsto\sum_{i=1}^d\frac{\gamma_i-1}{q-1}(h)\otimes\dlog(U_i)\] is a flat q-Higgs field, and this defines an equivalence of categories $\textrm{qHIG}(A_\inf^\bx(R))^{(1)}\quis \Rep_{\Gamma}^\mu(A_\inf^\bx(R)^{(1)})$ (and similarly with Frobenius structures).

In this way the functor in the statement of the corollary may be identified with the equivalence~(\ref{eqn_q_Simpsons_for_A_inf}).
\end{proof}

\begin{remark}\label{remark_descent_of_reps_to_phi_twist}
To understand the previous corollary, it may be helpful to identify the relative Frobenius $F:A_\inf^\bx(R)^{(1)}\to A^\bx_\inf(R)$ as the inclusion of $A_\inf$-algebras $\phi(A_\inf^\bx(R))\subseteq A_\inf^\bx(R)$ (this is allowed since the endomorphism $\phi$ is an automorphism of $A_\inf$ and $\phi$ is injective on $A_\inf^\bx(R)$). From this point of view, the corollary states that a generalised representation in $\Rep_\Gamma^\mu(A_\inf^\bx(R),\phi)$ admits a unique descent to the subring $\phi(A_\inf^\bx(R))$. It is likely that this can be proved by using the $\tilde\xi$-adic quasi-nilpotence to modify the 1-cocycle arguments of \S\ref{subsection_descent}, but the reformulation in terms of a q-Simpson correspondence is more general and seems more intuitive.
\end{remark}

\newpage
\section{Generalised representations as prismatic crystals}\label{section_prismatic_crystals}
In this section we reinterpret our generalised representations of interest in terms of prismatic crystals, in the following sense:

\begin{definition}\label{definiton_prismatic_crystal}
Let $(A,I)$ be a bounded prism, and let $\frak X$ be a
smooth $p$-adic formal scheme over $A/I$. A {\em locally finite free crystal} on the prismatic site\footnote{When $\frak X=\Spf R$ is affine, we will adopt affine notation as in \cite{BhattScholze2019}, writing $(R/(A,I))_\Prism$:=$(\Spf R/(A,I))_\Prism^\sub{op}$ and denoting a typical object by $(B\to B/IB\ot R)$.} $(\frak X/(A,I))_{\Prism}$ is a sheaf of $\cal O_{\Prism}$-modules $\cal F$ such that $\cal F(\frak B)$ is a finite projective $B=\cal O_\Prism(\frak B)$-module for every object $\frak B=(\Spf(B)\leftarrow \Spf(B/IB) \to \frak X)$ of  $(\frak X/(A,I))_{\Prism}$, and such that the pull-back homomorphism $\cal F(\frak B)\otimes_BB'\to \cal F(\frak B')$ is an isomorphism for every morphism $\frak B'=(\Spf(B')\leftarrow \Spf(B'/IB')\to \frak X)) \to \frak B$ in $(\frak X/(A,I))_{\Prism}$. We write $\CR_{\Prism}(\frak X/(A,I))$ for the category of locally finite free crystals on $(\frak X/(A,I))_{\Prism}$.
\end{definition}

Throughout this section we adopt the local framework of Section \ref{section_small_reps}: so let $R$ be a $p$-adically complete, small, formally smooth $\roi$-algebra with a fixed choice of framing. We will freely use the objects and notation introduced at the start of \S\ref{ss_framed}, though we write $A^\bx=A_{\inf}^{\bx}(R)$ and $A^\bx_{\infty}=A_{\inf}(R_{\infty})$ for the sake of brevity.

Given an $A_{\inf}$-algebra $B$, we write $B^{(1)}:=A_{\inf}\otimes_{\phi,A_{\inf}}B$ for its base change along the Frobenius $\phi\colon A_{\inf}\xrightarrow{\cong}A_{\inf}$. If $B$ is itself equipped with a lifting $\phi$ of the absolute Frobenius on $B/p$, compatibly with the Frobenius of $A_{\inf}$, then there is an induced map of $A_\inf$-algebras $F:=1\otimes \phi:B^{(1)}\to B$ which is called the relative Frobenius (c.f., the beginning of \S\ref{ss_q_Higgs}). By default we regard any $\cal O$-algebra $C$ as an $A_{\inf}$-algebra via the morphism $\theta:A_{\inf}\to A_{\inf}/\xi \cong\cal O$, and we thus have $C^{(1)}= A_\inf/\tilde\xi\otimes_{\phi,A_{\inf}/\xi}C $.
In particular, if $C$ is a perfectoid ring over $\cal O$, then 
we regard $C^{(1)}$ as a quotient of $A_{\inf}(C)$ by the $A_{\inf}$-homomorphism
$A_{\inf}(C)\xto{\tilde\theta} A_{\inf}(C)/\tilde\xi \xto{\phi^{-1}\,\cong}(A_{\inf}(C)/\xi)^{(1)}= C^{(1)}$.

We now introduce some objects of the prismatic site $(R^{(1)}/(A_{\inf},\tilde\xi))_{\Prism}$, on which we will then evaluate our crystals; diagram (\ref{eqn_crystal_diagram}) provides a summary. Firstly, the Frobenius lift on $A^\bx$ induces one on $A^{\bx(1)}$, whence we may view these rings as $\delta$-rings. In particular, we regard $(A^{\bx(1)} \to R^{(1)}\xot= R^{(1)})$ as an object of $(R^{(1)}/(A_{\inf},\tilde\xi))_{\Prism}$ equipped with a right action of $\Gamma$. Its reduction modulo $\mu$, namely $(A^{\bx(1)}/\mu\to R^{(1)}/(\zeta_p-1)\ot  R^{(1)})$ is also an object of $(R^{(1)}/(A_{\inf},\tilde\xi))_{\Prism}$, on which the induced 
$\Gamma$-action is by the identity. Therefore evaluation on the first of these objects defines a functor $$\op{ev}_{A^{\bx(1)}}\colon \CR_{\Prism}(R^{(1)}/(A_{\inf},\tilde\xi))
\to \Rep^{\mu}_{\Gamma}(A^{\bx(1)}),\qquad \cal F\mapsto \cal F(A^{\bx(1)}\to R^{(1)}\xot= R^{(1)}).$$ We also introduce the composite functor 
$$\op{ev}_{A^{\bx(1)}}^\phi\colon \CR_{\Prism}(R^{(1)}/(A_{\inf},\tilde\xi)) \xto{\op{ev}_{A^{\bx(1)}}} \Rep^{\mu}_{\Gamma}(A^{\bx(1)})\xto{-\otimes_{A^{\bx(1)},F}A^\bx}\Rep^{\mu}_{\Gamma}(A^\bx).$$

Similarly, the Frobenius lift on $A^\bx_{\infty}$ allows us to view it as a $\delta$-ring, and we regard $(A^\bx_{\infty}\to R_{\infty}^{(1)}\ot R^{(1)})$ as an object of $(R^{(1)}/(A_{\inf},\tilde\xi))_{\Prism}$ equipped with a right action of $\Gamma$. Now note that the composition of the relative Frobenius $F:A^{\bx(1)}\to A^\bx$ with the canonical injection $A^\bx\hookrightarrow A^\bx_{\infty}$ is a $\Gamma$-equivariant map of $\delta$-rings over $A_{\inf}$, whose reduction modulo $\tilde\xi$ gives the map $R^{(1)}\to R_{\infty}^{(1)}$ induced by the inclusion map $R\hookrightarrow R_{\infty}$. Hence it defines a natural $\Gamma$-equivariant morphism 
\begin{equation}\label{eq:PrismAinfAinfframe}
(A^{\bx(1)}\to R^{(1)}\xot= R^{(1)})\To (A^\bx_{\infty}\to R_{\infty}^{(1)}\ot R^{(1)})
\end{equation}
in $(R^{(1)}/(A_{\inf},\tilde\xi))_{\Prism}$, and evaluation on the right object in (\ref{eq:PrismAinfAinfframe}) defines a functor
$$\op{ev}_{A^\bx_{\infty}}\colon \CR_{\Prism}(R^{(1)}/(A_{\inf},\tilde\xi))
\to \Rep^{\mu}_{\Gamma}(A^\bx_{\infty}),\qquad 
\cal F\mapsto \cal F(A^\bx_{\infty}\to R_{\infty}^{(1)}\ot R^{(1)})
$$
such that  the left part of the following diagram is commutative up to canonical isomorphism:
\begin{equation}\xymatrix@C=2cm{
&\Rep_\Gamma^\mu(A^{\bx(1)})\ar[d]^{-\otimes_{A^{\bx(1)},F}A^\bx}\ar[r]^\simeq&\op{qHIG}(A^{\bx(1)})\ar[d]^{(F,F_\Omega)^*}\\
\CR_{\Prism}(R^{(1)}/(A_{\inf},\tilde\xi))\ar[r]^-{\op{ev}^\phi_{A^{\bx(1)}}}\ar@/^5mm/[ur]^-{\op{ev}_{A^{\bx(1)}}}
\ar@/_5mm/[dr]_{\op{ev}_{A^\bx_{\infty}}}& 
\Rep_\Gamma^\mu(A^\bx)\ar[d]_{\simeq}^{-\otimes_{A^\bx}A^\bx_{\infty}}\ar[r]^\simeq&\op{qMIC}(A^\bx)\\
&\Rep_\Gamma^\mu(A^\bx_{\infty})&
}\label{eqn_crystal_diagram}
\end{equation}
The right square of the diagram was explained just before, and in the proof of, Corollary \ref{corollary_descent_of_reps_to_phi_twist}. Recall also that the bottom vertical functor in the diagram is an equivalence of categories by Theorem \ref{theorem_descent_to_framed}.

The goal of this section is to identify prismatic crystals as follows:

\begin{theorem}\label{thm:main}
The functors $\op{ev}_{A^{\bx(1)}}$ and $\op{ev}_{A^{\bx(1)}}^\phi$ are fully faithful, with essential image given by those generalised representations for which the corresponding q-Higgs field (resp.~q-connection) is $(p,\mu)$-adically quasi-nilpotent.
\end{theorem}

\begin{remark}\label{remark_general_prisms_base}
Although we have not checked all details, the theorem (and bulk of proof) appears to hold more generally for any set-up for q-de Rham cohomology arising from deformation theory as in Remark~\ref{remark_q_via_deformation}, under the additional assumption that $A$ is flat over $\bb Z_p[[q-1]]$ (so, in particular, for the framed q-PD data of the sort mentioned at the end of Remark \ref{examples_qBK}). The paper of Chatzistamatiou \cite{Chatzistamatiou2020} works in this degree of generality, though it remains to compare his $\cal D$-module style definition of q-connections and his quasi-nilpotence condition \cite[Defs.~2.1.2 \& 2.1.3]{Chatzistamatiou2020} to those we adopt, namely Definition \ref{definition_q_connections} and Remark \ref{remark_nilpotent_connection}. For a correspondence between crystals on prismatic and q-crystalline sites of higher level, see Li \cite{Li2021}.

The main step of our argument which does not work in general as written is the proof of Proposition \ref{lem:PrismFinObjCov}, where we exploit the base prism being perfect to ultimately show in Corollary \ref{cor:FinObjCovqHiggs} that the smooth lift $A^{\bx(1)}$ defines a cover of the final object. But this can be overcome with alternative arguments: see forthcoming work of Z.~Mao \cite{Mao2021} for a general result about animated $\delta$-pairs in derived prismatic cohomology, \cite[Prop.~1.1.2]{Chatzistamatiou2020} in the q-crystalline context, or the recent talks by A.~Ogus and Y.~Tian at the 2021 IHES conference on arithmetic geometry in honor of Luc Illusie.

Recent work of Bhatt--Scholze \cite{BhattScholze2021} identifies absolute prismatic crystals in the arithmetic context with lattices in crystalline Galois representations.
\end{remark}

We already mentioned the notion of $(p,\mu)$-adic=$(p,\tilde\xi)$-adic quasi-nilpotence of q-Higgs fields and q-connections in Remark \ref{remark_nilpotent_connection}. Explicitly, given $H\in \Rep_\Gamma^\mu(A^{\bx(1)})$ (resp.~$N\in\Rep_\Gamma^\mu(A^\bx)$), recall that the non-logarithmic coordinates of the corresponding q-Higgs field $\Theta$ (resp.~q-connection $\nabla$) are given by $\Theta_i:=\tfrac{\gamma_i-1}{U_i\mu}$ (resp.~$\nabla_i:=\tfrac{\gamma_i-1}{U_i\mu}$); the nilpotence condition is asking for these operators to be $(p,\mu)$-adically quasi-nilpotent. We will write \[\op{qHIG}_{\sub{conv}}(A^\bx)\subseteq \op{qHIG}(A^\bx)\qquad\op{qMIC}_{\sub{conv}}(A^\bx)\subseteq \op{qMIC}(A^\bx)\] to denote the full subcategories of objects satisfying this quasi-nilpotence condition, and \[\Rep_{\Gamma,\sub{conv}}^\mu(A^{\bx(1)})\subseteq \Rep_{\Gamma}^\mu(A^{\bx(1)})\qquad \Rep_{\Gamma,\sub{conv}}^\mu(A^\bx)\subseteq \Rep_{\Gamma}^\mu(A^\bx)\] for the corresponding full subcategories on the generalised representation side.\footnote{The subcategory $\Rep_{\Gamma,\sub{conv}}^\mu(A_\inf^\bx(R))\subseteq \Rep_{\Gamma}^\mu(A_\inf^\bx(R))$ does not depend on the chosen compatible sequence of $p$-power roots of unity (and similarly for $\Rep_{\Gamma,\sub{conv}}^\mu(A_\inf^\bx(R)^{(1)})$). Indeed, we saw in Remark \ref{remark_dependence_on_ep} that changing the compatible sequence of $p$-power roots of unity has the effect of changing $\tfrac{\gamma_i-1}{\mu U_i}$ into $\tfrac{\gamma_i^a-1}{U_i([\ep^a]-1)}=\tfrac{a\mu}{[\ep^a]-1}\tfrac{\gamma_i-1}{U_i\mu}+\mu\tfrac{\mu}{U([\ep^a]-1)}\beta_i$, which is clearly $(p,\mu)$-adically quasi-nilpotent if and only if $\tfrac{\gamma_i-1}{U_i\mu}$ is. One can argue similarly for $\Rep_{\Gamma,\sub{conv}}^\mu(A_\inf^\bx(R)^{(1)})$.} Theorem \ref{thm:main} states that the evaluation functors define equivalences of categories
\[\op{ev}_{A^{\bx(1)}}\colon \CR_{\Prism}(R^{(1)}/(A_{\inf},\tilde\xi))
\xrightarrow{\sim} \Rep_{\Gamma,\sub{conv}}^\mu(A^{\bx(1)}),\qquad \op{ev}^\phi_{A^{\bx(1)}}\colon \CR_{\Prism}(R^{(1)}/(A_{\inf},\tilde\xi))
\xrightarrow{\sim} \Rep_{\Gamma,\sub{conv}}^\mu(A^\bx).\]

The faithfulness of $\op{ev}_{A^{\bx(1)}}$ is an immediate consequence of Corollary \ref{cor:FinObjCovqHiggs} below (an object $\frak B$ of a site is a {\em covering of the final object} precisely when, for any other object $\frak B'$, there is a cover $\frak B''\to\frak B'$ which admits a morphism $\frak B''\to\frak B$):

\begin{proposition}\label{lem:PrismFinObjCov}
The object $(A^\bx_{\infty}\to R_{\infty}^{(1)}\ot R^{(1)})$ of the prismatic site $(R^{(1)}/(A_{\inf},\tilde\xi))_{\Prism}^\sub{op}$ is a covering of the final object
\end{proposition}
\begin{proof}
We can identify $(R/(A_{\inf},\xi))_{\Prism}$ with $(R^{(1)}/(A_{\inf},\tilde\xi))_{\Prism}$ by the base change under the isomorphism $\phi\colon A_{\inf}\xrightarrow{\cong}A_{\inf}$, and the base change of the object $(A^\bx_{\infty}\to R_{\infty}\ot R)$ of $(R/(A_{\inf},\xi))_{\Prism}$ can be identified with $(A^\bx_{\infty}\to R_{\infty}^{(1)}\ot R^{(1)})$ via the isomorphism $\phi\colon A_\infty^{\bx(1)}\xrightarrow{\cong}A^\bx_{\infty}$ of $\delta$-rings over $A_{\inf}$. Hence it is equivalent to prove that $(A^\bx_{\infty}\to R_{\infty}\ot R)$  is a covering of the final object of the site $(R/(A_{\inf},\xi))_{\Prism}^\sub{op}$.

Let $A^\bx_{\delta}$ be the $(p,\xi)$-adic completion of $A^\bx\otimes_{\bb Z_p[U_1,\ldots, U_d]}\bb Z_p\{U_1,\ldots, U_d\}$, where the final ring denotes a free $\delta$-ring, and put $R_{\delta}=A_{\delta}^\bx/\xi$. Since $A_{\inf}[U_1,\ldots, U_d]/(p,\xi)^n\to A^\bx/(p,\xi)^n$ is \'etale, the $\delta$-ring structure on $A_{\inf}\{U_1,\ldots, U_d\}$ uniquely
extends to that on $A^\bx_{\delta}$ \cite[Lem.~2.18]{BhattScholze2019}.
By the formal smoothness of $A_{\inf}\to A^\bx$ and the universal property of free $\delta$-rings, we see that any object of $(R/(A_{\inf},\xi))_{\Prism}$ receives a map from $(A^\bx_{\delta}\to R_{\delta}\ot R)$ of $(R/(A_{\inf},\xi))_{\Prism}$.
Therefore $(A^\bx_{\delta}\to R_{\delta}\ot R)$ is a covering of the final object.

Let $A_{\delta\sub{perf}}^{\bx}$ be the perfection of the prism $(A^\bx_{\delta}, \xi)$ \cite[Lem.~3.9]{BhattScholze2019}, and put $R_{\delta\sub{perf}}:=A^\bx_{\delta\sub{perf}}/\xi$. Then $R_{\delta\sub{perf}}$ is a perfectoid ring, we have $A^\bx_{\delta\sub{perf}}=A_{\inf}(\cal R^{\sub{perf}}_{\delta})$ \cite[Thm.~3.10]{BhattScholze2019}, and the morphism $(A^\bx_{\delta}\to R_{\delta}\ot R)\to (A^\bx_{\delta\sub{perf}}\to R_{\delta\sub{perf}}\ot R)$ in $(R/(A_{\inf},\xi))_{\Prism}$ is a flat cover \cite[Lem.~2.11]{BhattScholze2019}. Therefore $(A^\bx_{\delta\sub{perf}}\to R_{\delta\sub{perf}}\ot R)$ is also a covering of the final object.

By Andr\'e's Lemma \cite[Thm.~7.12]{BhattScholze2019}, there exists a flat cover $(A^\bx_{\delta\sub{perf}}\to R_{\delta\sub{perf}}\ot R)\to (A_\inf(R')\to R'\ot R)$ in $(R/(A_{\inf},\xi))_{\Prism}$ such that $R'$ is a perfectoid ring and there exists a compatible system of $p$-power roots of $T_i$ in $R'$. Then $(A_\inf(R')\to R'\ot R)$ is also a covering of the final object. But since $R_{\infty}$ is the $p$-adic completion of $R\otimes_{\bb Z_p[T_1,\ldots,T_d]}\bb Z_p[T_1^{1/p^{\infty}},\ldots, T_d^{1/p^{\infty}}]$, the $A_{\inf}$-homomorphism $R\to R'$ extends to a homomorphism $R_{\infty}\to R'$, which induces a morphism 
$(A^\bx_{\infty}\to R_{\infty}\ot R)\to (A_\inf(R')\to R'\ot R)$. Thus we see that $(A^\bx_{\infty}\to R_{\infty}\ot R)$ is also a covering of the final object.
\end{proof}

Since we have a morphism \eqref{eq:PrismAinfAinfframe}
in $(R^{(1)}/(A_{\inf},\tilde\xi))_{\Prism}$, we obtain the following corollary:

\begin{corollary}\label{cor:FinObjCovqHiggs}
The object $(A^{\bx(1)}\to R^{(1)}\xot= R^{(1)})$ is a covering of the final object of the site $(R^{(1)}/(A_{\inf},\tilde\xi))_{\Prism}^\sub{op}$.
\end{corollary} 

Also let $R_1:=A^\bx/\tilde\xi$, regarded as an $R^{(1)}$-algebra by reducing $F:A^{\bx(1)}\to A^\bx$ modulo $\tilde\xi$. Note that the reduction modulo $\tilde\xi$ of the canonical injection $A^\bx\hookrightarrow A^\bx_{\infty}$ defines an $R^{(1)}$-homomorphism $R_1\to R_{\infty}^{(1)}$, whence the injection defines a morphism 
$$(A^\bx\to R_1\ot R^{(1)})\To (A^\bx_{\infty} \to R^{(1)}_{\infty}\ot R^{(1)})$$ in $(R^{(1)}/(A_{\inf},\tilde\xi))_{\Prism}$. So from Proposition \ref{lem:PrismFinObjCov} we also obtain the following corollary which will be needed later:

\begin{corollary}\label{corol:DphiFinObjCov}
The object $(A^\bx\to R_1\ot R^{(1)})$ is a covering of the final object of the site $(R^{(1)}/(A_{\inf},\tilde\xi))_{\Prism}^\sub{op}$
\end{corollary}

\subsection{Crystals as $A^\square$-modules with stratification}\label{ss_crystals_as_strats}
To prove Theorem \ref{thm:main} we need to reinterpret crystals in the usual way, namely in terms of modules with stratification.

We begin by verifying a couple of lemmas concerning base change and tensor product of prisms, since there is no dedicated reference in \cite{BhattScholze2019}. They yield Corollary \ref{cor:DescentFlatCover}, namely flat descent in the prismatic site, which is required to then identify crystals with stratifications.

\begin{lemma}\label{lem:CompletionOverPrism}
Let $(A,I) \to (A',I')$ be a map of bounded prisms and let $M$ be a $(p,I)$-completely flat $A$-module.

\begin{enumerate}
\item The natural morphism $M\otimes_A^LA'\to M\hat\otimes_AA':=\projlim_n (M\otimes_AA')/(p,I')^n$ witnesses the domain as the derived $(p,I')$-completion of the codomain. Moreover,  $M\hat\otimes_AA'$ is $(p,I')$-completely flat, and the natural homomorphism
$(M\hat\otimes_AA')/(p,I')^n\to (M\otimes_AA')/(p,I')^n$ is an isomorphism for any $n\ge0$.

\item Let $(A',I')\to (A'',I'')$ be another map of bounded prisms. Then the homomorphism $M\otimes_AA''\cong(M\otimes_AA')\otimes_AA''\to (M\hat\otimes_AA')\otimes_AA''$ induces an isomorphism $M\hat\otimes_AA''\xrightarrow{\cong} (M\hat\otimes_AA')\hat\otimes_AA''$.
\end{enumerate}
\end{lemma}
\begin{proof}
The claim (ii) immediately follows from the last claim of (i). 
Put $J'=pA'+I'$.  Let $M\hat\otimes^L_AA'$ be the derived $J'$-completion
of $M\otimes^L_AA'$. Then we have
\begin{equation}
(M\hat\otimes^L_AA')\otimes_A^L(A'/J^{\prime n})\xleftarrow{\cong}
(M\otimes_A^LA')\otimes^L_{A'}(A'/J^{\prime n})
\cong M\otimes_{A'}^L(A'/J^{\prime n}).
\label{eqn_complete}
\end{equation}
(We obtain the first isomorphism by considering the left adjoints of
$D(A'/J^{\prime n})\hookrightarrow D_{J'\text{-comp}}(A')
\hookrightarrow D(A')$.) Since $M$ is $(p,I)$-completely flat,
(\ref{eqn_complete}) for $n=1$ implies that $M\hat\otimes^L_AA'$ is $J'$-completely flat.
By \cite[Lem.~3.7(2)]{BhattScholze2019}, $M\hat\otimes^L_AA'$ is discrete
and $M':=H^0(M\hat\otimes^L_AA')$ is classically
$J'$-complete. The isomorphisms (\ref{eqn_complete}) imply
$M'/J^{\prime n}=M\otimes_{A}(A'/J^{\prime n})$. This completes the proof.
\end{proof}

\begin{lemma}[cf.~the proof of {\cite[Corol.~3.12]{BhattScholze2019}}]
\label{lem:PrismFibProd}
Let $(A_1,IA_1)\xleftarrow{f}(A,I)\xrightarrow{g}(A_2,IA_2)$ be maps of bounded prisms such that $f$ is flat (resp.~faithfully flat). Let $A_3:=\projlim_n(A_1\otimes_AA_2)/(p,I)^n$, equipped with the $\delta$-ring structure induced by those on $A,A_1, A_2$ as in \cite[Rem.~2.7 \& Lem.~2.17]{BhattScholze2019}. Then $(A_3,IA_3)$ is a bounded prism and represents the coproduct $(A_1,IA_1)\sqcup_{(A,I)}(A_2,IA_2)$ in the category of bounded prisms. Moreover, the canonical map $(A_2,IA_2)\to (A_3,IA_3)$ of bounded prisms is flat (resp.~faithfully flat), and $A_3/(p,I)^n \isoto A_1/(p,I)^n\otimes_{A/(p,I)^n}A_2/(p,I)^n$ for all $n\ge1$.
\end{lemma}
\begin{proof}
By applying Lemma \ref{lem:CompletionOverPrism} to the $A$-module $A_1$ and the map $(A,I)\to (A_2,IA_2)$, we see that $A_3$ is derived $(p,IA_2)$-complete, $(p,IA_2)$-completely flat, and the natural homomorphism  $A_3/(p,IA_2)^n\to (A_1\otimes_AA_2)/(p,IA_2)^n$ is an isomorphism. By applying \cite[Lem.~3.7(2)]{BhattScholze2019} to the $A_2$-module $A_3$, we obtain $A_3[IA_2]=0$ and the $p^{\infty}$-torsion boundedness of $A_3/IA_3$. This implies that $IA_3$ is a Cartier divisor of $A_3$, $A_3$ is derived $(p,IA_3)$-complete, and therefore $(A_3,IA_3)$ is a bounded prism. Since every bounded prism $(B,J)$ is classically $(p,J)$-complete \cite[Lem.~3.7(1)]{BhattScholze2019}, we see that $(A_3,IA_3)$ represents the coproduct in the category of bounded prisms. The last isomorphism in the lemma implies that $A_2/(p,I)\to A_3/(p,I)$ is faithfully flat if $A/(p,I)\to A_1/(p,I)$ is faithfully flat. 
\end{proof}

\begin{corollary}\label{cor:DescentFlatCover}
Let $(A,I)$ be a bounded prism, and let $\frak X$ be a $p$-adic smooth formal scheme over $A/I$.
\begin{enumerate}
\item If a presheaf $\cal F$ on $(\frak X/(A,I))_{\Prism}$ satisfies
the conditions of Definition \ref{definiton_prismatic_crystal} (except for a priori not being a sheaf), then in fact $\cal F$ is automatically a sheaf, hence a crystal locally free of finite type.
\item Let $\frak B_1=(\Spf(B_1)\leftarrow\Spf(B_1/IB_1)\to \frak X)\to
\frak B=(\Spf(B)\leftarrow \Spf(B/IB)\to \frak X)$ be a flat cover in the site $(\frak X/(A,I))_{\Prism}$; then the following descent of finite projective modules holds. Letting $\frak B_{\nu}=(\Spf(B_{\nu})\leftarrow\Spf(B_{\nu}/IB_{\nu})\to \frak X)$ be the fiber product over $\frak B$ of $\nu$ copies of $\frak B_1$, for $\nu=2,3$, then the category of finite projective $B$-modules is equivalent to the category of the following data: a finite projective $B_1$-module $M$ equipped with an isomorphism of $B_2$-modules $c\colon p_2^*M\cong p_1^*M$ satisfying both $\Delta^*(c)=\op{id}_M$ and the cocycle condition on $B_3$. 
\end{enumerate}
\end{corollary}
\begin{proof}
Lemma \ref{lem:PrismFibProd} implies that $B_\nu/(p,I)^n$ is the tensor product over $B/(p,I)^n$ of $\nu$ copies of $B_1/(p,I)^n$, for $\nu=2,3$. Since $B$, and $B_\nu$ are classically $(p,I)$-complete for $\nu=1,2,3$ \cite[Lem.~3.7(1)]{BhattScholze2019}, the claim follows from the faithfully flat descent of finite projective modules with respect to the faithfully flat homomorphism $B/(p,I)^n\to B'/(p,I)^n$ for each integer $n\geq 1$.
\end{proof}

\begin{remark}[Flat crystals]
For a bounded prism $(B,J)$, the category of derived $(p, J)$-complete, $(p,J)$-completely flat $B$-modules $M$ is equivalent to the category of inverse systems $(M_n)_{n\geq 1}$ of flat $B/(p,J)^n$-modules satisfying   $M_{n+1}\otimes_{B}B/(p,J)^n\xrightarrow{\cong} M_n$ $(n\geq 1)$; the equivalence is given by $M\mapsto (M/(p,I)^nM)_{n\geq 1}$
and $(M_n)_{n\geq 1}\mapsto \varprojlim_nM_n$ (we do not prove this equivalence, as we do not need it). Therefore, by Lemma \ref{lem:CompletionOverPrism}, one can define flat crystals on $(\frak X/(A,I))_{\Prism}$ by using $(p,I)$-derived complete, $(p,I)$-completely flat modules (instead of finite projective modules) and the classically $(p,I)$-completed base change $-\hat\otimes_BB'$, and then prove an analogue of Corollary \ref{cor:DescentFlatCover} for such crystals and modules.
\end{remark}

Now we turn to modules with stratification, beginning by recalling the appropriate \v{C}ech--Alexander complex in this setting:

\begin{definition}[{c.f., \cite[Cons.~16.13, Thm.~16.17]{BhattScholze2019}}]\label{def_q_env}
For each $\nu\ge0$, let \[A^\bx(\nu):=\hat{A^\bx\otimes_{A_\inf}\cdots\otimes_{A_\inf}A^\bx}\] be the $(p,\xi)$-completed tensor product of $\nu+1$ copies of $A^\bx$ over $A_{\inf}$. Noting that the kernel of the multiplication map $\Delta:A^\bx(\nu)\to A^\bx$ is generated modulo $(p,\xi)$ by a Koszul regular sequence (since $A^\bx/(p,\xi)$ is \'etale over a Laurent polynomial algebra over $A_\inf/(p,\xi)$), the composition $A^\bx(\nu)\xto{\Delta}A^\bx\xto{\sub{mod }\xi}R$ therefore admits a q-PD-envelope $A^\bx(\nu)\to\cal D(\nu)$ over the q-PD-pair $(A_{\inf},\xi)$ \cite[Lem.~16.10]{BhattScholze2019}. We will write $J_{\cal D(\nu)}$ for the kernel of the canonical map $\cal D(\nu)\to R$.
\end{definition}

\begin{remark}\label{rem:qPDenvChec}
We record that the following hold thanks to \cite[Lem.~16.10]{BhattScholze2019}:
\begin{enumerate}
\item The algebra $\cal D(\nu)$ is $\mu$-torsion free, and $\mu$-adically complete and separated.
\item $\cal D(\nu)/\mu$ is the $p$-adically completed PD-envelope of $A^\bx(\nu)/\mu\to  R$. 
\end{enumerate}
\end{remark}

Note that, by definition of a q-PD-pair, $\cal D(\nu)$ is equipped with a Frobenius satisfying $\phi(J_{\cal D(\nu)})\subseteq(\tilde\xi)$; it therefore induces a relative Frobenius map $F:R^{(1)}=(\cal D(\nu)/J_{\cal D(\nu)})^{(1)}\to \cal D(\nu)/\tilde\xi$. Let \begin{equation}\frak D(\nu):=\bigl(\cal D(\nu)\to \cal D(\nu)/\tilde\xi \xot{F} R^{(1)}\bigr)\label{eqn_qpd_envelope}\end{equation}
be the resulting object of $(R^{(1)}/(A_{\inf},\tilde\xi))_{\Prism}$. 

\begin{lemma}\label{lemma_Cech}
The object $\frak D(\nu)$ is the coproduct of $\nu+1$ copies of $\frak D(0)=(A^\bx\to R_1\ot R^{(1)})$ (which is the covering of the final object from Corollary \ref{corol:DphiFinObjCov}).
\end{lemma}
\begin{proof}
Let \[R_1(\nu):=R_1\otimes_{R^{(1)}}\cdots\otimes_{R^{(1)}}R_1\] be the (automatically $p$-complete, as the relative Frobenius is finite flat) tensor product of $\nu+1$ copies of $R_1$ over $R^{(1)}$. The surjection $A^\bx\to R_1=A^\bx/\tilde\xi$ of $A_\inf$-algebras induces a surjection $A^\bx(\nu)\to R_1(\nu)$, whose kernel we will denote by $J_{A^\bx(\nu)}$, which fits into a cocartesian diagram
\begin{equation}\xymatrix{
A^\bx(\nu)\ar[r] & R_1(\nu)\\
A^\bx(\nu)\ar[u]^\phi\ar[r] & R\ar[u]
}\label{eqn_q_vs_prism}\end{equation}
Here the left vertical arrow is the Frobenius induced by the Frobenius endomorphism on $A^\bx$.

Since coproducts in $(R^{(1)}/(A_{\inf},\tilde\xi))_{\Prism}$ are computed by prismatic envelopes (when they exist) as in \cite[Prop.~3.13]{BhattScholze2019}, we must show the following
\begin{enumerate}
\item the map of prisms $(A_\inf,\tilde\xi)\to (\cal D(\nu),\tilde\xi)$ is the prismatic envelope of the map of $\delta$-pairs $(A_\inf,\tilde\xi)\to (A^\bx(\nu),J_{A^\bx(\nu)})$;
\item the resulting canonical map $A^\bx(\nu)/J_{A^\bx(\nu)}\to D(\nu)/\tilde\xi$ fits into a commutative diagram
\[\xymatrix{
 & R_1(\nu)\ar@{=}[r] & A^\bx(\nu)/J_{A^\bx(\nu)}\ar[d]\\
R^{(1)}\ar@{=}[r]\ar[ur]&(\cal D(\nu)/J_{\cal D(\nu)})^{(1)}\ar[r]_-F & D(\nu)/\tilde\xi
}\]
\end{enumerate}

(i): As we have noted in Definition \ref{def_q_env}, there exists a sequence $x_1,\dots,x_r\in A^\bx(\nu)$ which is $(p,\xi)$-completely regular relative to $A_\inf$ (see \cite[\S2.6]{BhattScholze2019}) such that the kernel of the bottom horizontal arrow of (\ref{eqn_q_vs_prism}) is $(\xi,x_1,\dots,x_r)$. Since $\phi$ is flat (even finite flat), it follows that $\phi(x_1),\dots,\phi(x_r)\in A^\bx(\nu)$ is a $(p,\tilde\xi)$-complete regular sequence relative to $A_\inf$ such that the kernel of the top horizontal arrow is $(\tilde\xi,\phi(x_1),\dots,\phi(x_r))$. It follows from \cite[Prop.~3.13]{BhattScholze2019} that the map of $\delta$-pairs $(A_\inf,\tilde\xi)\to (A^\bx(\nu),J_{A^\bx(\nu)})$ admits a prismatic envelope; moreover, comparing the constructions of prismatic envelopes \cite[Prop.~3.13]{BhattScholze2019} of $q$-PD-envelopes \cite[Lem.~16.10]{BhattScholze2019} shows that this prismatic envelope is precisely $\cal D(\nu)$.

(ii): The commutativity of the diagram follows by noting that it receives a surjection from the obviously commutative diagram
\[\xymatrix{
& A^\bx(\nu)\ar[r] & A^\bx(\nu)\ar[d]\\
A^\bx(\nu)^{(1)}\ar@{=}[r]\ar[ur]^F&\cal D(\nu)^{(1)}\ar[r]_-F & D(\nu)
}\]
thanks to (\ref{eqn_q_vs_prism}).
\end{proof}

The $\frak D(\nu)$ of course assemble to form a cosimplicial object $\frak D(\blob)$ of $(R^{(1)}/(A_{\inf},\tilde\xi))_{\Prism}$, which Lemma~\ref{lemma_Cech} states is exactly the \v{C}ech construction associated to $\frak D(0)$ (similarly to \cite[Const.~4.16]{BhattScholze2019}, where a free $\delta$-$A_\inf$-algebra is used instead of $A^\bx$).

Regarding notation in this typical cosimplicial object
\[\xymatrix@C=3cm{
\frak D(\blob)=& \frak D(0) \ar@/_10mm/[r]_{p_1}\ar@/^10mm/[r]^{p_2}& \frak D(1)\ar[l]_\Delta 
\ar@/_10mm/[r]_{p_{12}}\ar@/^10mm/[r]^{p_{23}}\ar[r]^{p_{13}} & \frak D(2)\ar@{}[r]-{\cdots}\ar@/_6mm/[l]\ar@/^6mm/[l]&
}\]
we write $p_\nu:\frak D(0)\to\frak D(1)$, for $\nu=1,2$, for the canonical morphism corresponding to the map $[0]=\{0\}\to[1]=\{0,1\}$ with image $\{\nu-1\}$; similarly we write $p_{\nu\mu}\colon \frak D(1)\to \frak D(2)$, for $(\nu,\mu)=(1,2), (2,3), (1,3)$, for the canonical morphism corresponding to the increasing map $[1]=\{0,1\}\to [2]=\{0,1,2\}$ with image $\{\nu-1,\mu-1\}$; and $\Delta\colon \frak D(0)\to \frak D(1)$ for the morphism corresponding to the unique map $[1]=\{0,1\}\to [0]=\{0\}$; and $q_{\nu}\colon \frak D(0)\to \frak D(2)$, for $\nu=1,2,3$, for the morphism corresponding to the map $[0]\to [2]$ with image $\{\nu-1\}$.

We may now define stratifications in the usual way:

\begin{definition}\label{definition_stratification}
A {\em stratification} on an $A^\bx$-module $N$ with respect to $\cal D(\bullet)$ is a $\cal D(1)$-linear isomorphism $\varepsilon\colon N\otimes_{A^\bx,p_2}\cal D(1) \xrightarrow{\cong} N\otimes_{A^\bx,p_1}\cal D(1)$ satisfying the
following two conditions:
\begin{enumerate}
\item The scalar extension $\Delta^*(\varepsilon)$ of $\varepsilon$ by 
$\Delta\colon \cal D(1)\to A^\bx$
is the identity map of $N$.\par
\item We have $p_{12}^*(\varepsilon)
\circ p_{23}^*(\varepsilon)=p_{13}^*(\varepsilon)\colon 
N\otimes_{A^\bx,q_3}\cal D(2)
\xrightarrow{\cong} N\otimes_{A^\bx_,q_1}\cal D(2)$.
\end{enumerate}
Let $\Strat(\cal D(\bullet))$ be the category of finite projective
$A^\bx$-modules equipped with a stratification with
respect to $\cal D(\bullet)$.
\end{definition}

By Corollaries \ref{corol:DphiFinObjCov} \& \ref{cor:DescentFlatCover} and Lemma \ref{lemma_Cech}, the standard recipe describing crystals in terms of stratifications \cite[\S IV.1.6]{Berthelot1974} gives an equivalence of categories
\begin{equation}\label{eq:PrismCrysStratEquiv}
\op{ev}_{\frak D(\bullet)}\colon \CR_{\Prism}(R^{(1)}/(A_{\inf},\tilde\xi))
\xrightarrow{\sim}\\
\Strat(\cal D(\bullet))
\end{equation}
defined by sending a crystal $\cal F$ to the $A^\bx$-module $N=\cal F(\frak D(0))$ equipped with stratification defined as the composition
$$\varepsilon\colon N\otimes_{A^\bx,p_2}\cal D(1)\xrightarrow[p_2]{\cong}\cal  F(\frak D(1))\xrightarrow[p_1^{-1}]{\cong}N\otimes_{A^\bx_,p_1}\cal D(1).$$

Next we show that (\ref{eq:PrismCrysStratEquiv}) is suitably compatible with the functor $\op{ev}_{A^{\bx(1)}}^\phi$. To do this, we observe that the product $\Gamma^{\nu+1}$ of $\nu+1$ copies of the group $\Gamma$ naturally acts on the $\delta$-ring $A^\bx(\nu)$, and the homomorphism $A^\bx(\nu)\to R$ is invariant under the action. Taking the $q$-PD-envelope, we thus obtain an action of $\Gamma^{\nu+1}$ on the $q$-PD pair $(\cal D(\nu),J_{\cal D(\nu)})$. This action of $\Gamma^{\nu+1}$ on $\cal D(\nu)$ can be also constructed via the interpretation of $\cal D(\nu)$ as a prismatic envelope from the proof of Lemma \ref{lemma_Cech}: indeed, since $\phi\colon A^{\bx(1)}\to A^\bx$ is $\Gamma$-equivariant, the action of $\Gamma$ on $A^\bx$ induces an action on the $R^{(1)}$-algebra $R_1$, and then an action of $\Gamma^{\nu+1}$ on the $R^{(1)}$-algebra $R_1(\nu)$; the homomorphism $A^\bx(\nu)\to R_1(\nu)$ is  $\Gamma^{\nu+1}$-equivariant, and so the action extends to the prismatic envelope of this homomorphism over $(A_{\inf},\tilde\xi)$, which we saw was precisely $\cal D(\nu)$.

\begin{remark}\label{remark_Gamma_on_D}
\begin{enumerate}
\item The actions of $\Gamma^{\nu+1}$ on $\cal D(\nu)$ are simplicially compatible in $\nu$, i.e., there is an induced action of the simplicial group $E_\blob\Gamma$ on the simplicial formal scheme $\Spf \cal D(\blob)$.
\item Since the action of $\Gamma$ on $A^\bx/\mu$ is the identity, the action of $\Gamma^{\times (\nu+1)}$ on $A^\bx(\nu)/\mu$ is also the identity. Hence Remark \ref{rem:qPDenvChec}(ii) implies that the action of $\Gamma^{\nu+1}$ on $\cal D(\nu)/\mu$ is also the identity.
\end{enumerate}
\end{remark}


Given an $A^\bx$-module $N$ with stratification with respect to $\cal D(\blob)$, the $\Gamma^2$-action on $\cal D(1)$ induces a semilinear $\Gamma$-action on $N$ as follows. For each $\gamma\in\Gamma$, the base change of $\varepsilon$ along $\cal D(1)\xto{(1,\gamma)}\cal D(1)\xto{\Delta}A^\bx$ defines an isomorphism $N\otimes_{A^\bx,\gamma} A^\bx \isoto N$, i.e., a semi-linear action of the element $\gamma$ on $N$. Furthermore, for $\gamma'\in \Gamma$, the base change of this isomorphism along $\gamma'\colon A^\bx\xrightarrow{\cong}A^\bx$ is the base change of $\varepsilon$ along
$\gamma'\circ\Delta^*\circ(1,\gamma) =\Delta^*\circ (\gamma',\gamma'\gamma)$; so conditions (i) and (ii) in Definition \ref{definition_stratification} show that the isomorphisms $N\otimes_{A^\bx,\gamma}A^\bx\xrightarrow{\cong}N$, for $\gamma\in\Gamma$, define a semi-linear action of $\Gamma$ on $N$. Moreover, since the action of $\Gamma^{2}$ on $\cal D(1)/\mu$ is the identity, the induced action of $\Gamma$ on $N/\mu$ is also the identity (whence the action on $N$ is continuous by Lemma \ref{lemma_automatic_continuity}). We have thus defined a functor 
\begin{equation}\op{ev}_{A^\bx}^{\Strat}\colon \Strat(\cal D(\bullet))
\To \Rep_\Gamma^{\mu}(A^\bx),\label{eqn_stat_to_rep}\end{equation} which is compatible with $\op{ev}^\phi_{A^{\bx(1)}}$ in the following sense:

\begin{lemma}\label{lem:PrismStratEvaluation}
The following diagram is commutative up to canonical isomorphism:
$$\xymatrix@C=50pt{
\op{CR}_{\Prism}(R^{(1)}/(A_{\inf},\tilde\xi))\ar[r]_(.55){\op{ev}_{\frak D(\bullet)}}^(.55){\sim}
\ar[dr]_{\op{ev}^\phi_{A^{\bx(1)}}}&
\Strat(\cal D(\bullet))\ar[d]^{\op{ev}_{A^\bx}^{\Strat}}\\
&\Rep_{\Gamma}^{\mu}(A^\bx)
}
$$
\end{lemma}

\begin{proof}
Let $\cal F$ be an object of $\CR_{\Prism}(R^{(1)}/(A_{\inf},\tilde\xi))$ and put $(N,\varepsilon):=\op{ev}_{\frak D(\bullet)}(\cal F)$. Since the following two diagrams commute
\[\xymatrix{
\cal D(1)\ar[r]^{(1,\gamma)}&\cal D(1)\ar[r]^{\Delta}&\cal D(0)\\
\cal D(0)\ar[u]^{p_1\sub{, resp.~}p_2}\ar[urr]_{\op{id}\sub{, resp.~}\gamma}&&
}\]
we see that the action of $\gamma$ on $N=\cal F(\frak D(0))$ induced by the stratification $\varepsilon$ (as defined immediately before the lemma) coincides with the action of $\gamma$ on 
$\cal F(\frak D(0))$ induced by that on $\frak D(0)$.  

The $A_{\inf}$-algebra morphism $F: A^{\bx(1)}\to A^\bx$ induces a $\Gamma$-equivariant map $ (A^{\bx(1)}\to R^{(1)}\xot{=} R^{(1)})\to \frak D(0)$ in $(R^{(1)}/(A_{\inf},\tilde\xi))_{\Prism}$, and pulling back along this map induces a functorial isomorphism in $\Rep^\mu_{\Gamma}(A^\bx)$,
$$\op{ev}_{A^{\bx(1)}}^\phi(\cal F)=
\cal F(A^{\bx(1)}\to R^{(1)}\xot{=}R^{(1)})\otimes_{A^{\bx(1)},F}A^\bx
\xrightarrow{\cong} \cal F(\frak D(0))=\op{ev}_{A^\bx}^{\Strat}(N,\varepsilon),$$
as desired.
\end{proof}

\subsection{Generalised representations as $A^\square$-modules with stratification}\label{ss_gen_reps_as_strat}
In this subsection we associate a stratification to any $(p,\mu)$-adically quasi-nilpotent q-connection. In Proposition \ref{prop:BKFToStrat} this will be shown to invert the functor $\op{ev}_{A^\bx}^\sub{Strat}$ from (\ref{eqn_stat_to_rep}).

Given $N\in \Rep^{\mu}_\Gamma(A^\bx)$, we let $\Gamma$ act on $N\otimes_{A^\bx,p_1}\cal D(1)$ by tensoring the existing action on $N$ with the action on $\cal D(1)$ defined via the inclusion $\Gamma=\Gamma\times\{1\}\subseteq\Gamma^2$, where we recall that we defined a canonical action of $\Gamma^2$ on $\cal D(1)$ just before Remark \ref{remark_Gamma_on_D}. The inverse of the following isomorphism may be seen as a q-analogue of the Taylor expansion:


\begin{proposition}\label{prop:qPDTaylerExp}
Let $N\in \Rep^\mu_{\Gamma,\sub{conv}}(A^\bx)$. Then the homomorphism $\Delta_N:N\otimes_{A^\bx,p_1}\cal D(1)\to N$ induced by $\Delta\colon \cal D(1)\to A^\bx$ restricts to an isomorphism $$(N\otimes_{A^\bx,p_1}\cal D(1))^\Gamma\isoto N$$ of $A^\bx$-modules, with the $A^\bx$-module structure on the domain defined via $p_2:A^\bx\to\cal D(1)$.
\end{proposition}
\begin{proof}
To simplify notation, we put $L:=N\otimes_{A^\bx,p_1}\cal D(1)$, and $\ol L=L/\mu$  to simplify the notation. Since the action of $\Gamma^{ 2}$ on $\cal D(1)$ is 
the identity modulo $\mu$, as observed in Remark \ref{remark_Gamma_on_D}(ii), the same is true of the action of $\Gamma$ on $L$. Therefore, regarding $L$ as an $A^\bx$-module via its action on $N$ (so that the $\Gamma$-action we defined before the proof is semi-linear), there is an induced flat q-connection $\nabla_L$ on $L$ given by the construction of Corollary \ref{corollary_reps_vs_q_connections}, namely $\nabla_L\colon L\to L\otimes_{A^\bx}q\Omega^1_{A^\bx/A_{\inf}}$, $\ell\mapsto\sum_{i=1}^d\tfrac1\mu(\gamma_i(\ell)-\ell)\otimes \dlog U_i$. Note that $(N\otimes_{A^\bx,p_1}\cal D(1))^\Gamma=L^\Gamma=L^{\nabla_L=0}$, where the second equality implicitly uses continuity of the $\Gamma$-action (Lemma \ref{lemma_automatic_continuity}). The reduction modulo $\mu$ of $\nabla_L$ defines a flat connection $\nabla_{\ol L}\colon \ol{L}\to \ol L\otimes_{\ol{A}^\bx}\Omega_{\ol{A}^\bx/\ol{A}_{\inf}}$ on $\ol L$, where we write $\ol A_\inf:=A_\inf/\mu$ and similarly for other $A_\inf$-algebras/modules.

Recalling that $\cal D(1)$ is $\mu$-torsion-free as in Remark \ref{rem:qPDenvChec}, we have the following commutative diagram with exact rows
$$\xymatrix@C=30pt{
0\ar[r]& \ol L^{\nabla_{\ol L}=0}\ar[r]^-{\mu^n}
\ar[r]\ar[d]^{\Delta_N\sub{mod }\mu}
&(L/\mu^{n+1})^{\nabla_L=0}\ar[r]\ar[d]^{\Delta_N\sub{mod }\mu^{n+1}}
&(L/\mu^{n})^{\nabla_L=0}\ar[d]^{\Delta_N\sub{mod }\mu^n}
&\\
0\ar[r]&\ol{N}\ar[r]^-{\mu^n}
&N/\mu^{n+1}\ar[r]
&N/\mu^n\ar[r]&0
}$$
as well as an exact sequence of q-de Rham complexes (in fact, the left term is an honest de Rham complex)
$$0\To \ol L\otimes_{\ol A^\bx}\Omega^{\bullet}_{\ol A^\bx/\ol{A}_{\inf}}
\xrightarrow{\mu^n} L/\mu^{n+1}\otimes_{A^\bx}q\Omega^{\bullet}_{A^\bx/A_{\inf}}
\To L/\mu^n \otimes_{A^\bx}q\Omega^{\bullet}_{A^\bx/A_{\inf}}\To 0.$$
Since the $\cal D(1)$-module $L$ is $\mu$-adically complete and separated by Remark \ref{rem:qPDenvChec}(i), in order to establish the desired isomorphism of the proposition it suffices to show that $\Delta_N\colon \ol{L}^{\nabla_{\ol L}=0} \to \ol{N}$ is an isomorphism and that the de Rham complex $\ol L\otimes_{\ol A^\bx}\Omega^{\bullet}_{\ol{A}^\bx/\ol{A}_{\inf}}$ is acyclic in positive degree (in fact, we only need $H^1=0$). 

Put $\tau_i=p_2(U_i)-p_1(U_i)\in A^\bx(1)$ for $i=1,\dots,d$, and let $\ol{\tau}_i$ be its image in $\ol{A}^\bx(1):=A^\bx(1)/\mu$. By Remark \ref{rem:qPDenvChec} and the standard description of the PD-envelopes appearing in the \v{C}ech-Alexander complex in crystalline cohomology \cite[Corol.~I.4.5.3(ii)]{Berthelot1974}, the $\res A^\bx$-algebra $\res{\cal D}(1)$ (regarded as an algebra via $p_1$) is the $p$-adic completion of the PD-polynomial algebra over $\ol{A}^\bx$ in variables $\ol{\tau}_1,\dots,\ol{\tau}_d$. Therefore each element $\ell\in L$ may be uniquely written in the form \[\ell=\sum_{j_1,\dots,j_d\ge0}n_{\ul j}\otimes \ol{\tau}^{[\ul j]},\] where $n_{\ul j}$ is a sequence of elements of $\ol{N}$ which converges to $0$ $p$-adically as $\vert\ul j\vert:=\sum_{i=1}^dj_i$ tends to infinity, and $\res{\tau}^{[\ul j]}:=\prod_{1\leq i\leq d}\res \tau_i^{[j_i]}$.

Let $\nabla_{\res L,i}$ and $\nabla_{\res N,i}$, for $i=1,\dots,d$, be the non-logarithmic coordinates of the connections $\nabla_{\res L}$ and $\nabla_{\res N}$ respectively (where $\nabla_{\res N}$ denotes the reduction of $\nabla$ modulo $\mu$). Using the formulae in $\cal D(1)$
\[(\gamma_i,1)(\tau_j)=\begin{cases}
p_2(U_i)-[\ep]p_1(U_i)=\tau_i-\mu p_1(U_i) & \text{if }i=j,\\
\tau_j & \text{if }i\neq j,
\end{cases}
\]
we obtain for each $i=1,\dots,d$ that
\begin{equation}\label{eq:ND(1)NablaFormula}
\nabla_{\res L,i}(\ell)=\sum_{j_1,\dots,j_d\ge0}\nabla_{\res N,i}(n_{\ul j})\otimes\res\tau^{[\ul j]}
-\sum_{\substack{j_1,\dots,j_d\ge0\\j_i\neq 0}}n_{\ul j}\otimes \res\tau^{[\ul j-\ul 1_i]},
\end{equation}
where $\ul 1_i:=(0,\dots,1,\dots,0)$ denotes the multi-index with a single $1$ in the $i^\sub{th}$-place.

Let $\Omega_{\ol{A}^\bx(1)/\ol{A},2}^1=\ol A^\bx(1)\otimes_{p_1,\ol A^\bx}\Omega^1_{\ol{A}^\bx/\ol A_{\inf}}$ be the differential module of $p_2\colon \ol{A}^\bx\to \ol{A}^\bx(1)$. As mentioned in the last paragraph of \cite[IV \S1.3]{Berthelot1974}, there exists a  canonical integrable connection $\nabla_2\colon \ol{\cal D}(1)\to \ol{\cal D}(1)\otimes_{\ol{A}^\bx(1)} \Omega_{\ol{A}^\bx(1)/\ol{A},2}=\ol{\cal D}(1)\otimes_{p_1,\ol{A}^\bx} \Omega_{\ol{A}^\bx/\ol{A}_{\inf}}$. Since the corresponding HPD-stratification is a ring isomorphism \cite[IV Corol.~IV.1.3.5]{Berthelot1974}, this connection is a derivation. Furthermore, the formula \cite[IV (1.3.6)]{Berthelot1974} implies that $\nabla_2(p_2(a)\ol{\tau}^{[\ul{j}]})=p_2(a)\sum_{1\leq i\leq d, j_i\geq 1}\ol{\tau}^{[\ul{j}-\ul{1}_i]}\nabla_2(\ol{\tau}_i)=-p_2(a)\sum_{1\leq i\leq d,j_i\geq 1}\ol{\tau}^{[\ul{j}-\ul{1}_i]}\otimes dU_i$ for $i=1,2,\ldots, d$ and $n\geq 1$ (since $\ol{\cal D}(1)$ is $n!$-torsion free, this actually reduces to the claim for $\ul{j}=\ul{1}_i$). Therefore formula (\ref{eq:ND(1)NablaFormula}) implies that $\nabla_{\ol{L}}$ is the tensor product of the  flat connection $\nabla_{\ol{N}}$ on $\ol{N}$ and the flat  connection $\nabla_2$ on $\ol{\cal D}(1)$. Since the connection $\nabla_{\ol{N}}$ on $\ol{N}$ is $p$-adically quasi-nilpotent, it defines a stratification $\varepsilon\colon \ol{N}\otimes_{\ol{A}^\bx,p_2}\ol{\cal D}(1) \xrightarrow{\cong} \ol{N}\otimes_{\ol{A}^\bx,p_1}\ol{\cal D}(1)=\ol{L}$ on $\ol N$, and $\nabla_{\ol{L}}$ on $\ol{L}$ corresponds to $\op{id}\otimes \nabla_2$ on the source. (Since the connections on the both sides are compatible with the derivation $\nabla_2$ on $\ol{\cal D}(1)$, the claim is reduced to checking the inclusion $\varepsilon(\ol{N})\subset \ol{L}^{\nabla_{\ol{L}}=0}$, which follows from the explicit formula $\varepsilon(y)=\sum_{j_1,\dots,j_d\ge0}\prod_{1\leq i\leq d}\nabla_{\ol{N},i}^{j_i}(y)\otimes \ol \tau^{[\ul j]}$ and the explicit formula for $\nabla_2$ on $\ol{\cal D}(1)$ above.) The de Rham complex $\ol{\cal D}(1)\otimes_{p_1,\ol{A}^{\bx}}\Omega_{\ol{A}^\bx/\ol{A}_{\inf}}^{\bullet}$ regarded as a complex of $\ol{A}^\bx$-modules via $p_2\colon \ol{A}^\bx\to \ol A^\bx (1)$ is homotopic to $0$ by the proof of \cite[Lem.~V.2.1.2]{Berthelot1974} and the explicit description of $\nabla_2$ above. (We may apply the argument of the proof of  \cite[Lem.~V.2.1.2]{Berthelot1974} to $\ol{N}\otimes_{p_2,\ol{A}^{\bx}}(\ol{\cal D}(1)\otimes_{p_1,\ol{A}^{\bx}}\Omega_{\ol{A}^\bx/\ol{A}_{\inf}}^{\bullet})$ instead.) Hence the de Rham complex of $\ol{L}$ with respect to $\nabla_{\ol{L}}$ is acyclic in positive degree and $\varepsilon$ induces an isomorphism $\ol{N}\xrightarrow{\cong}\ol{L}^{\nabla_{\ol{L}}=0}$. This completes the proof because the composition of the last map with $\Delta_N\colon\ol{L}^{\nabla_{\ol{L}}=0}\to\ol{N}$ is the identity map.
\end{proof}

Thanks to Proposition \ref{prop:qPDTaylerExp} we may now state the existence of the stratification induced by the q-connection $\nabla$:

\begin{proposition}\label{prop:BKFToStrat}
Let $N\in \Rep^{\mu}_{\Gamma,\sub{conv}}(A^\bx)$, and let $\varepsilon:N\otimes_{A^\bx,p_2}\cal D(1)
\to N\otimes_{A^\bx,p_1}\cal D(1)$ be the $\cal D(1)$-linear homomorphism induced by the inverse of the isomorphism in Proposition \ref{prop:qPDTaylerExp}. Then:
\begin{enumerate}
\item The homomorphism $\varepsilon$ is a stratification on  $N$ with respect to $\cal D(\bullet)$.\par
\item The action of $\Gamma$ on $\op{ev}^{\Strat}_{A^\bx}(N,\varepsilon)$, whose underlying $A^\bx$-module is $N$, coincides with the original action of $\Gamma$ on $N$; i.e., $\op{ev}^{\Strat}_{A^\bx}(N,\varepsilon)=N$ in $\Rep^\mu_\Gamma(A^\bx)$.
\end{enumerate}
\end{proposition}
The construction of the stratification $\varepsilon$ in Proposition \ref{prop:BKFToStrat} is obviously functorial in $N$, and so we obtain a functor $$\Strat_{\cal D(\bullet)}\colon \Rep_{\Gamma,\sub{conv}}^{\mu}(A^\bx)\to \Strat \cal D(\bullet).$$
The goal of the rest of this subsection is to prove Proposition \ref{prop:BKFToStrat}.

Just before Proposition \ref{prop:qPDTaylerExp} we introduced an action of $\Gamma$ on $N\otimes_{A^\bx,p_1}\cal D(1)$; this may be extended along $\Gamma=\Gamma\times\{1\}\subseteq\Gamma^2$ to an action of $\Gamma^2$, by tensoring the canonical action of $\Gamma^2$ on $\cal D(1)$ by the action of $\Gamma^2$ on $N$ given by projection to the first coordinate. In terms of this extended action, the conclusion of Proposition \ref{prop:qPDTaylerExp} is an isomorphism $(N\otimes_{A^\bx,p_1}\cal D(1))^{\Gamma\times\{1\}}\isoto N$. More generally, for integers $1\le \mu\le \nu+1$, we need to introduce a semi-linear action of $\Gamma^{\nu+1}$ on the $\cal D(\nu)$-module $N\otimes_{A^\bx,p_\mu}\cal D(\nu)$, by tensoring its canonical action on $\cal D(\nu)$ with its action on $N$ given by projection to the $\mu^\sub{th}$ component.

\begin{lemma}\label{lem:GammaEquivStrat}
Under the notation and assumption in Proposition \ref{prop:BKFToStrat}, the $\cal D(1)$-linear homomorphism $\varepsilon$ is $\Gamma^{2}$-equivariant.
\end{lemma}
\begin{proof}
It suffices to prove that $\varepsilon(\gamma_2(n)\otimes 1)=(\gamma_1,\gamma_2)(\varepsilon(n\otimes 1))$ for all $n\in N$ and $\gamma_1,\gamma_2\in \Gamma$.

For $\gamma\in \Gamma$ and $f\in \cal D(1)$ we have $\Delta((\gamma,\gamma)(f))=\gamma(f)$, which implies  that $\Delta_N((\gamma,\gamma)(\ell))=\gamma(\Delta_N(\ell))$ for $\gamma\in \Gamma$ and $\ell\in N\otimes_{A^\bx,p_1}\cal D(1)$, where we recall that $\Delta_N:N\otimes_{A^\bx,p_1}\cal D(1)\to N$ denotes the map induced by $\Delta$. Therefore if $\ell$ is $\Gamma\times \{1\}$-invariant we have $\Delta_N((\gamma_1,\gamma_2)(\ell))=\Delta_N((\gamma_2,\gamma_2)(\ell))=\gamma_2(\Delta_N(\ell))$. Setting $n:=\Delta_N(\ell)$, this proves the necessary claim.
\end{proof}

To prove Proposition \ref{prop:BKFToStrat}(i), we will use the following variant
of Proposition \ref{prop:qPDTaylerExp}. 

\begin{proposition}\label{prop:qTaylorExpVariant}
Let $N\in \Rep^{\mu}_{\Gamma,\sub{conv}}(A^\bx)$. Then the homomorphism $N\otimes_{A^\bx, p_1^*}\cal D(2)\to N$ induced by $\Delta\colon \cal D(2)\to A^\bx$ restricts to an injection $$(N\otimes_{A^\bx,p_1}\cal D(2))^{\Gamma\times\Gamma\times \{1\}}\hookrightarrow N.$$
\end{proposition}
\begin{proof}
The strategy of the proof is the same as Proposition \ref{prop:qPDTaylerExp}.
Put $L=N\otimes_{A^\bx,p_1}\cal D(2)$, and 
$\ol{L}=L/\mu L$  to simplify the notation. 
Since the action of $\Gamma^3$ on $\cal D(2)$ is 
the identity modulo $\mu$ as observed in Remark \ref{remark_Gamma_on_D}(ii),
so is its action on $L$. Therefore, for $\nu=1,2$, regarding
$L$ as an $A^\bx$-module  via $p_{\nu}\colon A^\bx\to \cal D(2)$,
we may define a flat \txb $q$-connection $\nabla_L^{(\nu)}\colon L\to L\otimes_{p_{\nu},A^\bx}q\Omega^1_{A^\bx/A_{\inf}}$ by 
$\nabla^{(\nu)}_L(\ell)=\sum_{i=1}^d\frac{1}{\mu}(\iota_\nu(\gamma_i)-1)(\ell)\otimes d\log U_i$, where $\iota_{\nu}$ denotes the inclusion into the $\nu$th component $\Gamma\hookrightarrow \Gamma^{3}$.
We have $L^{\Gamma^{2}\times \{1\}}=L^{\nabla_{L}^{(1)}
=\nabla_L^{(2)}=0}$ because the $\Gamma^3$-action on $L$ is continuous
(Lemma \ref{lemma_automatic_continuity}). Hence, by the same argument as the proof
of Proposition~\ref {prop:qPDTaylerExp}, it suffices to prove that the homomorphism 
$\Delta_{\ol{N}}\colon \ol L^{\nabla_{\ol L}^{(1)}=\nabla_{\ol L}^{(2)}=0}\to \ol N$
induced by $\Delta\colon \ol{\cal D}(2)\to \ol A^{\bx}$ 
is an isomorphism, where $\nabla^{(\nu)}_{\ol L}$ is the flat connection 
$\ol L\to \ol L\otimes_{p_{\nu}, \ol{A}^\bx}\Omega^1_{\ol{A}^{\bx}/\ol{A}_{\inf}}$
obtained by taking the reduction modulo $\mu$ of $\nabla_{L}^{(\nu)}$.

Put $\tau_i^{(\nu)}=q_3^*(U_i)-q_{\nu}^*(U_i)\in A^\bx(2)$,
and let $\ol{\tau}_i^{(\nu)}$ be its image in $\ol{A}^\bx(2)$.
By Remark \ref{rem:qPDenvChec}(ii) and \cite[Corol. I.4.5.1 ii)]{Berthelot1980},
$\ol{\cal D}(2)$ regarded as an $\ol{A}^\bx$-algebra via $p_1$ is
the $p$-adic completion of the PD-polynomial algebra
over $\ol{A}^\bx$ in variables $\ol{\tau}_i^{(\nu)}$, for $1\leq i\leq d$ and $
\nu=1,2$. Hence each element $\ell$ of $\ol{L}$ may be uniquely written in the form 
$$\ell=\sum_{(\ul j,\ul j')\in \bb N^d\times \bb N^d}
n_{\ul j,\ul j'}\otimes \ol{\tau}^{(1)[\ul j]}\ol{\tau}^{(2)[\ul j']},$$
where $n_{\ul j,\ul j'}$ is an  element of $\ol{N}$ which converges to 
$0$ $p$-adically as $\vert\ul j \vert+\vert \ul j'\vert$ tends to infinity,
$\ol{\tau}^{(1)[\ul j]}
=\prod_{1\leq i\leq d}(\ol \tau_i^{(\nu)})^{[j_i]}$, and similarly for $\ol{\tau}^{(2)[\ul j']}$.
Let $\nabla_{\ol L,i}^{(\nu)}$ and $\nabla_{\ol N,i}$, for $i=1,\ldots, d$, be
the non-logarithmic coordinates of the connections $\nabla_{\ol L}^{(\nu)}$ and
$\nabla_{\ol N}:=(\nabla\mod \mu)$.
Then, similarly to (\ref{eq:ND(1)NablaFormula}),  we obtain 
\begin{align*}
\nabla_{\ol L,i}^{(1)}(\ell)&=\sum_{(\ul j,\ul j')\in \bb N^d\times\bb N^d}\nabla_{\ol N,i}(n_{\ul j,\ul j'})\otimes\ol\tau^{(1)[\ul j]}
\ol\tau^{(2)[\ul j']}
-\sum_{(\ul j,\ul j')\in\bb N^d\times \bb N^d, j_i\geq 1}n_{\ul j,\ul j'}\otimes \ol\tau^{(1)[\ul j-\ul 1_i]}
\ul\tau^{(2)[\ul j']},\\
\nabla_{\ol L,i}^{(2)}(\ell)&=
-\sum_{(\ul j,\ul j')\in\bb N^d\times\bb N^d,j_i'\geq 1}n_{\ul j,\ul j'}\otimes \ol\tau^{(1)[\ul j]}
\ol\tau^{(2)[\ul j'-\ul 1_i]}.
\end{align*}
Hence $\nabla^{(1)}_{\ol L}(\ell)=\nabla^{(2)}_{\ol L}(\ell)=0$ if and only if $n_{\ul j,\ul j'}=0$ for
all $\ul j'\in \bb N^d\backslash \{\ul 0\}$, and 
$n_{\ul j,\ul 0}=\prod_{1\leq i\leq d}\nabla_{\ol N,i}^{j_i}(n_{\ul 0,\ul 0})$ 
for all $\ul j\in \bb N^d$, where $\ul 0:=(0,0,\ldots, 0)$.
As $\Delta_{\ol N}(\ell)=n_{\ul 0,\ul 0}$ in $\ol N$, the $p$-adic
quasi-nilpotence of $\nabla_{\ol N}$ implies that 
$\Delta_{\ol N}\colon \ol L^{\nabla_{\ol L}^{(1)}=\nabla_{\ol L}^{(2)}=0}\to \ol N$
is an isomorphism with inverse given by 
$n\mapsto \sum_{\ul j\in \bb N^d}\prod_{1\leq i\leq d}\nabla_i^{j_i}(n)
\otimes \ol \tau^{(1)[\ul j]}$. 
\end{proof}

\begin{proof}[Proof of Proposition \ref{prop:BKFToStrat}]
Part (i): It is clear from the definition that the base change of $\varepsilon$ along $\Delta\colon \cal D(1)\to A^\bx$ is the identity map. By Lemma \ref{lem:GammaEquivStrat}, the base changes $p_{\nu\mu}^*(\varepsilon)\colon N\otimes_{A^\bx,q_{\mu}}\cal D(2)\to N\otimes_{A^\bx,q_{\nu}}\cal D(2)$, for $(\nu,\mu)=(1,2),(2,3), (1,3)$, are $\Gamma^{ 3}$-equivariant; therefore the restrictions of $p_{13}^*(\varepsilon)$ and $p_{12}^*(\varepsilon)\circ p_{23}^*(\varepsilon)$ to $N$ have images in the $\Gamma^{ 2}\times \{1\}$-invariant part of $N\otimes_{A^\bx,q_1}\cal D(2)$, and their compositions with the injective homomorphism of Proposition \ref{prop:qTaylorExpVariant}  are the identity map. Thus we see $p_{13}^*(\varepsilon)=p_{12}^*(\varepsilon)\circ p_{23}^*(\varepsilon)$ since both sides are $\cal D(2)$-linear. The two properties of $\varepsilon$ proven above imply that $\varepsilon$ is an isomorphism: indeed, its inverse is given by base changing it along the involution $\cal D(1)\xrightarrow{\cong}\cal D(1)$ swapping the two factors.

(ii): Let $\gamma\in \Gamma$. By Lemma \ref{lem:GammaEquivStrat} we have the following commutative diagram, in which the vertical isomorphisms are the $\cal D(1)$-linear maps given by the action of $(1,\gamma)$ on the source and target of $\varepsilon$:
$$\xymatrix{
(N\otimes_{A^\bx, p_2}\cal D(1))\otimes_{\cal D(1),(1,\gamma)}\cal D(1)
\ar[r]^{\varepsilon\otimes \op{id}}\ar[d]^{\cong}_{(1,\gamma)\otimes\op{id}}&
(N\otimes_{A^\bx, p_1}\cal D(1))\otimes_{\cal D(1),(1,\gamma)}\cal D(1)
\ar[d]^{\cong}_{(1,\gamma)\otimes\op{id}}\\
N\otimes_{A^\bx,p_2}\cal D(1)\ar[r]^{\varepsilon}&
N\otimes_{A^\bx, p_1}\cal D(1).
}
$$
The image of an element $n\in N$ under the left (resp.~right) vertical map is $\gamma(n)$ (resp.~$n$), by the definition of the $\Gamma^{ 2}$-action. By base changing this diagram along $\Delta\colon \cal D(1)\to A^\bx$, we obtain the following commutative diagram  whose left vertical isomorphism is given by the action of $\gamma$ on $N$.
$$\xymatrix@C=70pt{
N\otimes_{A^\bx,\gamma}A^\bx\ar[r]^-{\Delta^*((1,\gamma)^*(\varepsilon))}
\ar[d]^\cong_{\gamma\otimes\op{id}}&
N\ar[d]^{\op{id}}\\
N\ar[r]^{\op{id}}&N.
}
$$
This completes the proof.
\end{proof}

\subsection{Proof of Theorem \ref{thm:main} for $\op{ev}_{A^{\square(1)}}^\phi$, and variant with Frobenius structures}\label{ss_proof_of_main_thm1}
We are almost prepared to prove the first part of Theorem \ref{thm:main}. We just need the following final compatibiity:

\begin{lemma}\label{lem:StratGammaEquiv}
Let $(N,\varepsilon)$ be an object of $\Strat(\cal D(\blob))$; recall that $\ep$ induces a $\Gamma$-action on $N$, as encoded by the functor $\op{ev}_{A^\bx}^{\Strat}$ of (\ref{eqn_stat_to_rep}), so that we may also view $N$ as an object of $\Rep_\Gamma^\mu(A^\bx)$. Then $\varepsilon\colon N\otimes_{A^\bx,p_{2}}\cal D(1)\xrightarrow{\cong} N\otimes_{A^\bx,p_{1}}\cal D(1)$ is $\Gamma^{ 2}$-equivariant, where $\Gamma^2$ acts on each side as prescribed just before Lemma \ref{lem:GammaEquivStrat}.
\end{lemma}
\begin{proof}
Although it is possible to directly prove this, we will use prismatic crystals. So choose $\cal F\in \CR_{\Prism}(R^{(1)}/(A_{\inf},\tilde\xi))$
such that $\op{ev}_{\frak D(\bullet)}(\cal F)=(N,\varepsilon)$, by (\ref{eq:PrismCrysStratEquiv}). 
Then, by  Lemma \ref{lem:PrismStratEvaluation}, the action of $\Gamma$ on $N=\cal F(\frak D(0))$ induced by $\varepsilon$ coincides with the action induced by the action of $\Gamma$ on $\frak D(0)$. For each $\nu=1,2$, the map $p_{\nu}\colon \frak D(0) \to \frak D(1)$ is compatible with the action of $\Gamma^{ 2}$ and $\Gamma$ via the projection to the $\nu$-th component $\Gamma^{ 2}\to \Gamma$. Hence the same is true for the map $p_{\nu}\colon \cal F(\frak D(0))\to \cal F(\frak D(1))$, where $\cal F(\frak D(1))$  is equipped with the action of $\Gamma^{ 2}$ induced by that on $\frak D(1)$. Now the definition of the functor $\op{ev}_{\frak D(\bullet)}$ completes the proof.
\end{proof}

\begin{proof}[Proof of Theorem \ref{thm:main} for $\op{ev}_{A^{\bx(1)}}^\phi$]
By the commutative diagram in Lemma \ref{lem:PrismStratEvaluation}, it suffices to prove the following two statements:
\begin{enumerate}
\item The functor $\op{ev}^{\Strat}_{A^\bx}\colon \Strat(\cal D(\bullet))
\to \Rep^{\mu}_{\Gamma}(A^\bx)$ has essential image inside $\Rep^{\mu}_{\Gamma,\sub{conv}}(A^\bx)$.
\item The functors $\op{ev}^{\Strat}_{A^\bx}$ and $\Strat_{\cal D(\bullet)}$ give quasi-inverses equivalence of categories between $\Strat(\cal D(\bullet))$ and $\Rep_{\Gamma,\sub{conv}}^\mu(A^\bx)$.
\end{enumerate}
To prove both these statements, we fix an object $(N,\varepsilon)$ of $\Strat(\cal D(\bullet))$ and put $L=N\otimes_{A^\bx,p_1}\cal D(1)$ and $\ol L=L/\mu$; let $\Gamma$ act on $N$, and $\Gamma^2$ on $L$ and $N\otimes_{A^\bx,p_2}\cal D(1)$ as in the statement of  Lemma \ref{lem:StratGammaEquiv}. By Lemma~\ref{lem:StratGammaEquiv}, $\varepsilon$ restricts to a map $\ep|_N\colon N\to L^{\Gamma\times \{1\}}$, whose composition with $L^{\Gamma\times\{1\}}\subseteq L\xto{\Delta_N} N$ is seen to be the identity map by Definition \ref{definition_stratification}(i). Moreover, the $\Gamma\times\{1\}$-action on $L$ defines a flat q-connection $\nabla_L\colon L\to L\otimes_{A^\bx}q\Omega_{A^\bx/A_{\inf}}$, just as in the first paragraph of the proof of Proposition \ref{prop:qPDTaylerExp}, and we let $\nabla_{\res L}:\ol L\to \ol L\otimes_{\ol A^\bx}\Omega_{\ol A^\bx/\ol{A}_{\inf}}$ be the resulting connection by reducing modulo $\mu$. Then $L^{\Gamma\times\{1\}}=L^{\nabla_L=0}$, and $\ep|_N$ induces a homomorphism $\ol\ep|_N\colon \ol N\to \res L^{\nabla_{\res L}=0}$ whose composition with $\Delta_{\res N}\colon \ol L\to \ol N$ is the identity map. 

(i)  Letting $\ol \tau_i$ be as in the proof of Proposition \ref{prop:qPDTaylerExp} then, for each $n\in \ol N$, its image $\ol\ep|_N(n)\in\res L$ may be uniquely written in the form $\ol\ep|_N(n)=\sum_{j_1,\dots,j_d\ge0} n_{\ul j}\otimes\ol \tau^{[\ul j]}$, where $n_{\ul j}$ is a sequence of elements of $\res N$ satisfying $n_{\ul 0}=n$ and $n_{\ul j}\to 0$ $p$-adically as $\vert\ul j\vert\to\infty$. Applying formula \eqref{eq:ND(1)NablaFormula}\footnote{Although Proposition \ref{prop:qPDTaylerExp} was stated under the hypothesis that $\nabla$ was $(p,\mu)$-adically quasi-nilpotent, it was not required for this formula.} to $\ell=\res\ep|_N(x)$ and using the fact that $\nabla_{\res L}(\ell)=0$, we see that $n_{\ul j}=\prod_{1\leq i\leq d}\nabla_{\res N,i}^{j_i}(n)$. Therefore the connection $\nabla_{\res N}$ is $p$-adically quasi-nilpotent, as desired.

(ii) As $N$ belongs to $\Rep_{\Gamma,\sub{conv}}(A^\bx)$ by (i), we have an isomorphism $\Delta_N\colon L^{\Gamma\times\{1\}}\xrightarrow{\cong} N$ by Proposition~\ref{prop:qPDTaylerExp}, and the proof of (i) shows that the restriction of $\varepsilon$ to $N$ gives its inverse. By the construction of the functor $\Strat_{\cal D(\bullet)}$, this means the composition $\Strat_{\cal D(\bullet)}\circ \op{ev}_{A^\bx}^{\Strat}$ is the identity. Meanwhile the composition $\op{ev}_{A^\bx}^{\Strat}\circ\Strat_{\cal D(\bullet)}$ is also the identity by Proposition \ref{prop:BKFToStrat}.
\end{proof}

We briefly discuss the variant of Theorem \ref{thm:main} in which Frobenius structures are allowed. In the general context of Definition \ref{definiton_prismatic_crystal}, a Frobenius structure $\phi_\cal F$ on a crystal $\cal F$ locally free of finite type is by definition an isomorphism of sheaves of $\roi_\Prism[\tfrac1I]$-modules $\phi_\cal F:\cal F\otimes_{\roi_\Prism,\phi}\roi_\Prism[\tfrac1I]\isoto \cal F[\tfrac1I]$. An {\em F-crystal} $(\cal F,\phi_{\cal F})$ is a crystal locally free of finite type equipped with a Frobenius structure; write $\CR_{\Prism}(\frak X/(A,I), \phi)$ for the category of F-crystals.

The functors $\op{ev}_{A^{\bx(1)}}$ and $\op{ev}_{A^{\bx(1)}}^\phi$ are clearly compatible with Frobenius structures, thereby inducing an evident analogue of diagram (\ref{eqn_crystal_diagram}) in which Frobenius structures are incorporated everywhere.

\begin{corollary}\label{corollary_main_thm_with_phi}
The functors \[\op{ev}_{A^{\bx(1)}}:\op{CR}_{\Prism}(R^{(1)}/(A_\inf,\tilde\xi), \phi)\To\Rep_\Gamma^\mu(A^{\bx(1)},\phi)=\op{qHIG}(A^{\bx(1)},\phi)\] and \[\op{ev}_{A^{\bx(1)}}^\phi:\op{CR}_{\Prism}(R^{(1)}/(A_\inf,\tilde\xi), \phi)\To\Rep_\Gamma^\mu(A^\bx,\phi)=\op{qMIC}(A^\bx,\phi)\] are equivalences of categories.
\end{corollary}
\begin{proof}
Recall that if a generalised representation in $\Rep_\Gamma^\mu(A^\bx)$ admits a Frobenius structure, then the associated q-connection is automatically $(p,\mu)$-adically nipotent (even $\tilde\xi$-adically nipotent) by Lemma~\ref{lemma_phi_implies_xi_nilp}(iii). Similarly for the q-Higgs field associated to a generalised representation in $\Rep_\Gamma^\mu(A^{\bx(1)},\phi)$ by Lemma \ref{lemma_phi_implies_xi_nilp}(ii). So the fact that $\op{ev}_{A^{\bx(1)}}^\phi$ is an equivalence with Frobenius structures easily follows from the analogue without Frobenius structures, proved immediately above, by making a similar argument to, e.g., Theorem \ref{theorem_descent_to_framed_with_phi}.

It them follows from our q-Simpson correspondence with Frobenius, in the form of Corollary \ref{corollary_descent_of_reps_to_phi_twist}, that $\op{ev}_{A^{\bx(1)}}$ is necessarily also an equivalence with Frobenius structures.
\end{proof}

\subsection{Crystals as $A^{\square(1)}$-modules with stratification, and proof of Theorem \ref{thm:main} for $\op{ev}_{A^{\square(1)}}$}\label{ss_crystals_as_strats_Higgs}
This subsection is devoted to proving the first part of Theorem \ref{thm:main}, namely that $\op{ev}_{A^{\bx(1)}}$ induces an equivalence of categories \[\op{ev}_{A^{\bx(1)}}\colon \CR_{\Prism}(R^{(1)}/(A_{\inf},\tilde\xi))
\xrightarrow{\sim} \Rep_{\Gamma,\sub{conv}}^\mu(A^{\bx(1)}).\] (The reader who is only interested in crystals and generalised representations with Frobenius structure should note that the analogous equivalence has already been shown in Corollary \ref{corollary_main_thm_with_phi}.) To establish this equivalence, it in fact remains only (as we will explain just after the statement of Theorem \ref{thm:ConvqHiggStrat}) to show that the functor is well-defined: i.e., given a prismatic crystal $\cal F\in \CR_{\Prism}(R^{(1)}/(A_{\inf},\tilde\xi))$, we must show that the generalised representation $\op{ev}_{A^{\bx(1)}}(\cal F)\in\Rep^\mu_{\Gamma}(A_\inf^\bx(R)^{(1)})$ belongs to $\Rep^\mu_{\Gamma,\sub{conv}}(A_\inf^\bx(R)^{(1)})$. We will do this by developing an analogue over $A^{\bx(1)}$ of the stratification techniques of \S\ref{ss_crystals_as_strats}.

\begin{definition}\label{definition_Higgs_envelope}
For each $\nu\geq 0$, let $\cal E(\nu)$ be the prismatic envelope of $A^\bx(\nu)^{(1)}\to R^{(1)}$ over $(A_{\inf},\tilde\xi)$, and define the object \[\frak E(\nu):=(\cal E(\nu)\to\cal E(\nu)/\tilde\xi\ot R^{(1)})\] of $(R^{(1)}/(A_{\inf},\tilde\xi))_{\Prism}$. Thus $\frak E(\nu)$ is the coproduct of $\nu+1$ copies of $\frak E(0)=(A^{\bx(1)}\to R^{(1)}\xot{=} R^{(1)})$, and they assemble into a cosimplicial object in $(R^{(1)}/(A_{\inf},\tilde\xi))_{\Prism}$.
\end{definition}

\begin{remark}[Compatibility with $\cal D(\bullet)$]\label{remark_compatibilityII}
We further expand on Definition \ref{definition_Higgs_envelope} and explain how it relates to Definition \ref{def_q_env} and Lemma \ref{lemma_Cech}. The cocartesian diagram (\ref{eqn_q_vs_prism}) from the latter lemma may be broken into two cocartesian squares
\[\xymatrix{
A^\bx(\nu)\ar[rr] && R_1(\nu)\\
A^\bx(\nu)^{(1)}\ar[u]^F\ar[r]_{\Delta^{(1)}} & A^{\bx(1)} \ar[r]_{\sub{mod } \tilde\xi}& R^{(1)}\ar[u]\\
A^\bx(\nu)\ar[u]^W\ar[rr]&&R\ar[u]
}\]
where $F$ is the relative Frobenius induced by the Frobenius endomorphism of $A^\bx(\nu)$.

The kernel of the multiplication map $\Delta^{(1)}:A^\bx(\nu)^{(1)}\to A^{\bx(1)}$ is generated modulo $(p,\tilde\xi)$ by a Koszul regular sequence, namely the image under $W:A^\bx(\nu)\isoto A^\bx(\nu)^{(1)}$ of the Koszul regular sequence appearing in Definition \ref{def_q_env}. The map of $\delta$-pairs $(A_\inf,\tilde\xi)\to (A^\bx(\nu)^{(1)},(\tilde\xi,\ker\Delta^{(1)}))$ (the latter ideal is the kernel of the middle horizontal composition in the diagram) therefore admits a prismatic envelope $(\cal E(\nu),\tilde\xi)$ by \cite[Prop.~3.13]{BhattScholze2019}, i.e., $\cal E(\nu)$ is well-defined. 

Since the relative Frobenius $F$ is flat, loc.~cit.~states that the base change $(\cal E(\nu)\otimes_{A^\bx(\nu)^{(1)},F}A^\bx(\nu),\tilde\xi)$ is the prismatic envelope of $(A_\inf,\tilde\xi)\to (A^\bx(\nu),J_{A^\bx(\nu)})$, which is saw in the proof of Lemma \ref{lemma_Cech} is precisely $\cal D(\nu)$. There is therefore a natural isomorphism 
\begin{equation}\label{eq:EDCompareIsom}
\cal E(\nu)\otimes_{A^\bx(\nu)^{(1)},F}A^\bx(\nu)\xrightarrow{\cong}\cal D(\nu)\end{equation}
of $\delta$-rings over $A_{\inf}$. It induces a morphism $\frak E(\bullet)\to \frak D(\bullet)$ of cosimplicial objects of $(R^{(1)}/(A_{\inf},\tilde\xi))_{\Prism}$, which coincides with the morphism defined by taking the products of the morphism \[\frak E(0)=(A^{\bx(1)}\to R^{(1)})\xot{=} R^{(1)})\To \frak D(0)=(A^\bx\to R_1\ot R^{(1)})\] (using Lemma \ref{lemma_Cech}).

Therefore the study of $\cal E(\blob)$ will provide a Frobenius descent of the theory we have already developed over $\cal D(\blob)$. 
\end{remark}

The $A_\inf$-linear action of $\Gamma^{ \nu+1}$ on the $\delta$-ring $A^\bx(\nu)^{(1)}$  stabilizes the kernel of $A^\bx(\nu)^{(1)}\to R^{(1)}$ and so extends to an action on $\cal E(\nu)$ in such a way that the isomorphism \eqref{eq:EDCompareIsom} is $\Gamma^{\nu+1}$-equivariant. The induced action on $\cal E(\nu)/\mu$ is the identity, since we already saw that the same is true for $\cal D(\nu)/\mu$ just before Lemma \ref{lem:PrismStratEvaluation} and (\ref{eq:EDCompareIsom}) induces an injection $\cal E(\nu)/\mu\into \cal D(\nu)/\mu$ (note that $F:A^\bx(\nu)^{(1)}\to A^\bx(\nu)^{(1)}$ is finite free).
Just as on $\cal D(\blob)$, these actions are compatible in that $E_\blob\Gamma$ acts on the simplicial formal scheme $\Spf\cal E(\blob)$.


We define a stratification on an $A^{\bx(1)}$-module with
respect to $\cal E(\bullet)$ in exactly the same manner as in Definition \ref{definition_stratification} for an $A^\bx$-module with respect to $\cal D(\bullet)$. Then, similarly to \eqref{eq:PrismCrysStratEquiv}, but using Corollary \ref{cor:FinObjCovqHiggs} instead of Corollary \ref{corol:DphiFinObjCov}, we obtain an equivalence of categories
\begin{equation}
\op{ev}_{\frak E(\bullet)}\colon \CR_{\Prism}(R^{(1)}/(A_{\inf},\tilde\xi))
\xrightarrow{\sim}\Strat(\cal E(\bullet))
\end{equation}
sending a crystal $\cal F$ to the $A^{\bx(1)}$-module $H=\cal F(\frak E(0))$ equipped with stratification given by the composition
$$\varepsilon\colon H\otimes_{A^{\bx(1)},p_2}\cal E(1)\xrightarrow[p_2]{\cong}\cal F(\frak E(1))\xrightarrow[p_1^{-1}]{\cong}H\otimes_{A^{\bx(1)},p_1}\cal E(1).$$
Moreover, similarly to $\op{ev}_{A^\bx}^{\Strat}$ defined before Lemma \ref{lem:PrismStratEvaluation}, we can also construct a functor
\[\op{ev}_{A^{\bx(1)}}^{\Strat}\colon \Strat(\cal E(\bullet))\To \Rep^{\mu}_{\Gamma}(A^{\bx(1)})\]
by taking the base change of the stratification along $\cal E(1)\xto{(1,\gamma)}\cal E(1)\xto{\Delta}A^{\bx(1)}$, for each $\gamma \in \Gamma$.

The following is proved in exactly the same way as Lemma \ref{lem:PrismStratEvaluation}:
\begin{lemma}
The following diagram is commutative up to canonical isomorphism
$$\xymatrix@C=50pt{
\op{CR}_{\Prism}(R^{(1)}/(A_{\inf},\tilde\xi))\ar[r]_(.55){\op{ev}_{\frak E(\bullet)}}^-{\sim}
\ar[dr]_{\op{ev}_{A^{\bx(1)}}}&
\Strat(\cal E(\bullet))\ar[d]^{\op{ev}_{A^{\bx(1)}}^{\Strat}}\\
&\Rep_{\Gamma}^{\mu}(A^{\bx(1)}).
}
$$
\end{lemma}

\begin{remark}[Compatibility with $\cal D(\bullet)$, II]\label{remark_ED_compatability}
Stratifications with respect to $\cal E(\blob)$ are compatible, in the following way, with the stratifications with respect to $\cal D(\blob)$ which were studied in \S\ref{ss_crystals_as_strats}--\ref{ss_gen_reps_as_strat}.

Using the homomorphism $\cal E(\bullet)\to \cal D(\bullet)$ of cosimplical algebras over $A_{\inf}$ induced by \eqref{eq:EDCompareIsom}, one defines a functor $-\otimes_{A^{\bx(1)},F}A^\bx\colon \Strat(\cal E(\bullet))\to \Strat(\cal D(\bullet))$ in the obvious manner and checks that the following diagram is commutative up to canonical isomorphism:
\begin{equation}
\xymatrix@C=40pt{
\CR_{\Prism}(R^{(1)}/(A_{\inf},\tilde\xi))
\ar[r]^-{\sim}_-{\op{ev}_{\frak E(\bullet)}}
\ar[dr]^{\sim}_{\op{ev}_{\frak D(\bullet)}}
&
\Strat(\cal E(\bullet))\ar[r]^{\op{ev}^{\Strat}_{A^{\bx(1)}}}
\ar[d]^{-\otimes_{A^{\bx(1)},\phi}A^\bx}
&
\Rep^{\mu}_{\Gamma}(A^{\bx(1)})
\ar[d]^{-\otimes_{A^{\bx(1)},\phi}A^\bx}\\
&\Strat(\cal D(\bullet))\ar[r]_-{\op{ev}^{\Strat}_{A^\bx}}
& \Rep^{\mu}_{\Gamma}(A^\bx)
}\label{eqn_comp_of_strats}
\end{equation}
\end{remark}

The key to proving the remainder of Theorem \ref{thm:main} is the following:

\begin{theorem}\label{thm:ConvqHiggStrat}
The functor $\op{ev}_{A^{\bx(1)}}^{\Strat}\colon \Strat(\cal E(\bullet))
\to \Rep_{\Gamma}^{\mu}(A^{\bx(1)})$ has essential image inside
$\Rep_{\Gamma,\sub{conv}}^{\mu}(A^{\bx(1)})$.
\end{theorem}
\begin{proof}[Proof that Theorem \ref{thm:ConvqHiggStrat} completes the proof of Theorem \ref{thm:main}]
First we claim that the base change functor $-\otimes_{A^{\bx(1)},F}A^\bx:\Rep_\Gamma^\mu(A^{\bx(1)})\to \Rep_\Gamma^\mu(A^\bx)$, or equivalently $(F,F_\Omega)^*:\op{qHIG}(A^{\bx(1)})\to \op{qMIC}(A^{\bx})$, restricts to a fully faithful functor $\Rep_{\Gamma,\sub{conv}}^\mu(A^{\bx(1)})\to \Rep_{\Gamma,\sub{conv}}^\mu(A)$. Indeed, given $H\in \Rep_\Gamma^\mu(A^{\bx(1)})$, which we equivalently view as an object of $\op{qHIG}(A^{\bx(1)})$, applying Theorem \ref{theorem_Simpsons} to $H/p$ shows that the base change functor preserves $(p,[p]_q)$-adic=$(p,\mu)$-adic quasi-nilpotence. Fully faithfulness of the functor follows similarly by applying the the fully faithfulness of that theorem to $H/p^r$ for all $r\ge1$.

With the claim established, Theorem \ref{thm:ConvqHiggStrat} will allow us to replace the two $\Rep_\Gamma^\mu$ in diagram (\ref{eqn_comp_of_strats}) by $\Rep_{\Gamma,\sub{conv}}^\mu$, such that the right vertical functor is then fully faithfulful (by the above claim) and the bottom horizontal arrow is then an equivalence (by statement (ii) in the main proof in \S\ref{ss_proof_of_main_thm1}). So all the functors in the diagram will be equivalences, as desired.
\end{proof}

The remainder of this subsection is therefore devoted to proving Theorem \ref{thm:ConvqHiggStrat}, through some explicit but slightly technical calculations. The purpose of the next two lemmas is to explicitly describe $\ol{\cal E}(\nu)=\cal E(\nu)/\mu$ as a $p$-adically complete PD-polynomial algebra.

Let $U_i^{(1)}:=W(U_i)\in A^{\bx(1)}$, for $i=1,\dots,d$, be the image of $U_i\in A^\bx$ under $W:A^\bx\xrightarrow{\cong} A^{\bx(1)}=A_{\inf}\otimes_{\phi,A_{\inf}}A^\bx$, $a\mapsto a\otimes 1$; so the relative Frobenius $F\colon A^{\bx(1)}\to A^\bx$ sends  $U_i^{(1)}$ to $U_i^p$. Since the element $W(\tau_i)=p_2(U_i^{(1)})-p_1(U_i^{(1)})\in A^\bx(1)^{(1)}$ is contained in the kernel of $A^\bx(1)^{(1)}\to  R^{(1)}$, it becomes divisible by $\tilde\xi$ in the prismatic envelope $\cal E(1)$; we set \[\tau_i^{(1)}:=\tfrac1{\tilde\xi}(p_2(U_i^{(1)})-p_1(U_i^{(1)}))\in \cal E(1),\] and let $\ol\tau_i^{(1)}\in\ol{\cal E}(1)$ be its reduction modulo $\mu$.

\begin{lemma}\label{lem:E(1)Description}
\begin{enumerate}
\item We have $\frac{1}{n!}(\ol\tau_i^{(1)})^n\in \ol{\cal E}(1)$ for all integers $n\geq 1$
and $1\leq i\leq d$.
\item Let $1\leq \lambda\leq \nu+1$ be positive integers, and regard $\ol{\cal E}(\nu)$ as an $\ol A^{\bx(1)}:=A^{\bx(1)}/\mu$-algebra by the homomorphism $p_{\lambda}\colon \ol A^{\bx(1)}\to \ol{\cal E}(\nu)$. Then $\ol{\cal E}(\nu)$ is isomorphic to the $p$-adic completion of the PD-polynomial algebra over $\ol A^{\bx(1)}$ in the $\nu d$ variables $p_{\lambda j}(\ol\tau_i^{(1)})$, for $1\leq i\leq d$, $1\leq j\leq \nu+1,j\neq\lambda$. Here $p_{\lambda j}:\cal E(1)\to\cal E(\nu)$ denotes the cosimplicial structure map induced by the unique increasing map $[1]\to [\nu]$ with image $\{\lambda-1,j-1\}$.
\end{enumerate}
\end{lemma}
\begin{proof}
(i) As we commented parenthetically after Remark \ref{remark_compatibilityII}, the canonical map $\res{\cal E}(1)\to\res{\cal D}(1)$ induced by (\ref{eq:EDCompareIsom}) makes the latter into a finite free module over the former. Since the latter is also $p$-torsion-free thanks to Remark \ref{rem:qPDenvChec}, it suffices to check that the image of $(\ol\tau_i^{(1)})^n$ in $\res{\cal D}(1)$ is divisible by $n!$. But this follows from the fact that the image of $\res\tau_i^{(1)}$ in $\res{\cal D}(1)$ is $\tfrac1p(p_2(\ol U_i)^p-p_1(\ol U_i)^p)$, which belongs to the pd ideal $\ker(\res{\cal D}(1)\to R)$.

It will be useful to note for the proof of part (ii) that the latter element may be rewritten as \[\tfrac1p(p_2(\ol U_i)^p-p_1(\ol U_i)^p)=(p-1)!\ol\tau_i^{[p]}+\sum_{j=1}^{p-1}\tfrac1p\binom{p}{j}p_1(\res U_i)^{p-j}\ol\tau^j_i.\]


(ii) Since $\res{A}^{\bx(1)}$ and $\ol{\cal E}(\nu)$ are classically $p$-complete and $p$-torsion free it suffices to prove the claim after taking the reduction modulo $p$. Part (i) implies that $\ol{\cal E}(\nu)$ admits a unique PD-structure associated to the ideal generated by the elements $p_{\lambda j}(\tfrac1{n!}(\res\tau_i^{(1)})^n)\in \ol{\cal E}(\nu)$, for $1\le i\le d$, $1\le \lambda,j\le\nu+1$, $n\ge1$; reducing mod $p$ induces a PD-structure on  $\ol{\cal E}(\nu)/p$.

The proof will proceed by first constructing a certain commutative diagram:
\[\xymatrix{
\res A^{\bx(1)}/p\ar[d]\ar[r]^F\ar[dr]^{p_\lambda}&\res A^\bx/p\ar[r]\ar[dr]^{p_\lambda}&\cal B\ar[d]_\beta \\
\cal A\ar[r]_\al&\res{\cal E}(\nu)/p\ar[r]&\res{\cal D}(\nu)/p
}\]
Firstly, let $\cal A$ be the PD-polynomial algebra over $\ol A^{\bx(1)}/p$ in $\nu d$ variables
$V_{i j}$, where $1\leq i\leq d$, $1\leq  j\leq \nu+1$,  $j\neq \lambda$. There is a PD-homomorphism $\alpha\colon \cal A\to \ol{\cal E}(\nu)/p$ characterised by $p_\lambda:\res A^{\bx(1)}/p\to\cal E(\nu)/p$ and $V_{i j}\mapsto p_{\lambda j}(\ol\tau_i^{(1)})$, and our goal is to prove that $\al$ is an isomorphism.

Secondly, let \[\cal B:=\ol A^\bx/p[T_{i j}:1\leq i\leq d,1\leq  j\leq \nu+1, j\neq \lambda]/(T_{i j}^p)\] by the polynomial algebra over $\ol A^{\bx(1)}/p$ in $\nu d$ variables
$V_{i j}$, where $1\leq i\leq d$, $1\leq  j\leq \nu+1$,  $j\neq \lambda$, modulo the relations $T_{i j}^p=0$. The right vertical map $\beta$ is given by $p_{\lambda}\colon \ol A^\bx /p\to \ol{\cal D}(\nu)/p$ and $T_{i j}\mapsto p_{\lambda j}(\ol\tau_i)$.

We consider the composition
\begin{equation}\cal A\otimes_{\res A^{\bx(1)}/p}\cal B\xto{\al\otimes\op{id}_\cal B}\res{\cal E}(\nu)/p\otimes_{\res A^{\bx(1)}/p}\cal B\To \ol{\cal D}(\nu)/p,\label{pd_comp}\end{equation} where the second map, which is PD-homomorphism, is the canonical one induced by the diagram.

We claim that the second map in (\ref{pd_comp}) is an isomorphism.
Indeed, thanks to the \'etale framing $\ol A_{\inf}/p[\ol U_1,\ldots, \ol U_d]\to \ol A^\bx/p$, we know that $\ol A^\bx/p$ is a free $\ol A^{\bx(1)}/p$-module with basis $\prod_{1\leq i\leq d}\ol U_i^{n_i}$, for $0\le n_i<p$, and more generally that $\ol A^\bx(\nu)/p$ is a free $\ol A^{\bx}(\nu)^{(1)}/p$-module with basis $\prod_{1\leq i\leq d}p_{\lambda}(\ol U_i)^{n_i}\cdot\prod_{1\leq i\leq d,1\leq j\leq \nu+1, j\neq \lambda}p_{\lambda j}(\ol\tau_i)^{m_{i j}}$, for  $0\le n_i,m_{i j}<p$. So the claim follows from the isomorphism (\ref{eq:EDCompareIsom}) modulo $(p,\mu)$.

Next we claim that the composition (\ref{pd_comp}) is an isomorphism. Indeed, it is a PD-homomorphism of $\cal B$-algebras sending $V_{ij}$ to $p_{\lambda j}((p-1)!\ol\tau_i^{[p]}+\sum_{j=1}^{p-1}\tfrac1p\binom{p}{j}p_1(\res U_i)^{p-j}\ol\tau^j_i)$, by the observation at the end of part (i); this element can be rewritten as $(p-1)!\beta(T_{ij})^{[p]}+\beta(\sum_{i=1}^{p-1} \tfrac1p\binom{p}{j}\res U_i^{p-j}T_{ij}^j)$. The claimed isomorphism is therefore a special case of Lemma \ref{lem:PDPolyHiggs} below (in the notation of that lemma, we have $A=\res A^\bx/p$, $\cal B$ as in the current proof, $\cal C=\cal A\otimes_{\res A^{\bx(1)}/p}\cal B$, and $\cal D=\res{\cal D}(\nu)/p$; note that $\res{\cal D}(\nu)/p$ is indeed the PD-polynomial algebra over $A$ in the required variables, by the usual description of divided power envelopes in the smooth case, c.f., Prop.~3.32 and just after Def.~4.1 of \cite{BerthelotOgus1978}).

It therefore follows that $\al\otimes\op{id}_\cal B$ is an isomorphism; but the base change $\res A^{\bx(1)}/p\to\cal B$ (i.e., the top horizontal arrow in the diagram) is faithfully flat, so $\al$ is also an isomorphism, as desired.
\end{proof}

\begin{lemma}\label{lem:PDPolyHiggs}
Let $A$ be an $\bb F_p$-algebra and $n\ge0$. Let $\cal B:=A[T_1,\ldots, T_n]/(T_1^p,\ldots, T_n^p)$, let $\cal C$ be the PD-polynomial algebra over $\cal B$ in $n$ variables $V_1,\ldots, V_n$, and let $\cal D$ be the PD-envelope of the projection $\cal B\to A$ (in other words, $\cal D$ is the PD-polynomial algebra over $A$ in variables $T_1,\ldots, T_n$).

Given any elements $u_i\in \bb F_p^{\times}$ and $b_i\in (T_1,\dots,T_n)\subseteq \cal B$, for $i=1,\dots,n$, then the $A$-PD-homomorphism
$f\colon \cal C\to \cal D$ defined by \[T_i\mapsto T_i\qquad V_i\mapsto u_iT_i^{[p]}+b_i\] is an isomorphism.
\end{lemma}
\begin{proof}
We inductively define $\cal B$-algebras $\cal B^{(s)}$ and $A$-algebras $A^{(s)}$,
for $s\geq -1$, by $\cal B^{(-1)}:=\cal B$, $A^{(-1)}:=A$,
\[\cal B^{(s)}:=\cal B^{(s-1)}[V_1^{(s)},\ldots, V_n^{(s)}]/((V_i^{(s)})^p;1\leq i\leq n),\]
and \[A^{(s)}:=A^{(s-1)}[T_1^{(s)},\ldots, T_n^{(s)}]/((T_i^{(s)})^p;1\leq i\leq n).\]
Put $\cal B^{(\infty)}=\varinjlim_{s} \cal B^{(s)}$ and $A^{(\infty)}=\varinjlim_{s}A^{(s)}$.
Then we have an isomorphism of $\cal B$-algebras
$\cal B^{(\infty)}\xrightarrow{\cong} \cal C$ defined by $V_i^{(s)}\mapsto V_i^{[p^{s}]}$ and an isomorphism of $A$-algebras $A^{(\infty)}\xrightarrow{\cong} \cal D$ defined by
 $T_i^{(s)}\mapsto T_i^{[p^{s}]}$. Let $g\colon \cal B^{(\infty)}
\to A^{(\infty)}$ be the homomorphism induced by the map $f\colon \cal C\to \cal D$
via the previous isomorphisms. Since $g(T_i)=T_i^{(0)}$, we see that $g$
induces an isomorphism $\cal B^{(-1)}\xrightarrow{\cong} A^{(0)}$. 
By using $(T_i^{[p]})^{[p^{s}]}\in \bb F_p^{\times}\cdot
T_i^{[p^{s+1}]}$, we see that $g(V_i^{(s)})$ is of the form 
$v T_i^{(s+1)}+b$ for some $v\in \bb F_p^{\times}$ and $b\in (T_i^{(j)}; 1\leq i\leq n,0\leq j\leq s)\subseteq A^{(s)}$.
Hence, by induction on $s$, one can show that $g$ induces
an isomorphism $\cal B^{(s-1)}\xrightarrow{\cong} A^{(s)}$ for all
$s\ge0$. By taking the inductive limit over $s$, we
see that $g$ is an isomorphism.
\end{proof}

\begin{remark}
Lemma \ref{lem:E(1)Description} means that the algebra $\ol {\cal E}(1)/p$ (resp.~$\ol {\cal E}(1)/p^n$) is an analogue of the Hopf algebra used by H.~Oyama \cite{Oyama2017} (resp.~D.~Xu \cite{Xu2019}) in his study of Higgs modules (resp.~modules with $p$-connections) in characteristic $p$ (resp.~over $\bb Z/p^n$), improved by adding divided powers with the help of $\delta$-ring structure.
\end{remark}

The following analogue of Lemma \ref{lem:StratGammaEquiv} is proved in exactly the same way, so we omit the details. (The $\Gamma^2$-actions are analogous to those in Lemma \ref{lem:StratGammaEquiv}: namely $\Gamma^{ 2}$ acts on $H\otimes_{A^{\bx(1)},p_{\nu}^*}\cal E(1)$, for $\nu=1,2$, by tensoring the canonical action on $\cal E(1)$ by the action on $N$ given by projection to the $\nu^\sub{th}$ coordinate of $\Gamma^2$.)

\begin{lemma}\label{lem:qHiggsStratGammaEquiv}
Let $(H,\varepsilon)$ be an object of $\Strat(\cal E(\bullet))$; recall that $\ep$ induces a $\Gamma$-action on $N$, as encoded by the functor $\op{ev}_{A^{\bx(1)}}^{\Strat}$, so that we may also view $H$ as an object of $\Rep_\Gamma^\mu(A^{\bx(1)})$. Then $\varepsilon\colon H\otimes_{A^{\bx(1)},p_2^*}\cal E(1) \xrightarrow{\cong} H\otimes_{A^{\bx(1)},p_1^*}\cal E(1)$ is $\Gamma^2$-equivariant.\hfill\qed
\end{lemma}

Similarly to just before Proposition \ref{prop:qPDTaylerExp}, it is useful to restrict the $\Gamma^2$-action on $\cal E(1)$ to $\Gamma=\Gamma\times\{1\}\subset\Gamma^2$. This is then a semi-linear $\Gamma$-action on $\cal E(1)$, viewed as an $A^{\bx(1)}$-module via $p_1$, which is the identity modulo $\mu$ (as we commented just after Remark \ref{remark_compatibilityII}) and so corresponds to a q-Higgs field (as in the proof of Corollary \ref{corollary_descent_of_reps_to_phi_twist}) $\Theta:\cal E(1)\to\cal E(1)\otimes_{p_1,A^{\bx(1)}}\q\Omega^1_{A^{\bx(1)}/A}$. Its non-logarithmic coordinates are by definition \[\Theta_i:\cal E(1)\to\cal E(1),\quad f\mapsto \frac{\gamma_i(f)-f}{p_1(U_i^{(1)})\mu}.\] We record the following easy properties of this q-Higgs field on $\cal E(1)$:

\begin{lemma} \label{lem:qHiggsOnE(1)}
For each $i=1,\dots,d$, we have:
\begin{enumerate}
\item $\Theta_{i}(\tau_j^{(1)})=\begin{cases}-1 & j=i\\ 0&j\neq i.\end{cases}$
\item The mod $\mu$ induced endomorphism $\res\Theta_i$ of $\res{\cal E}(1)$ is a derivation.
\end{enumerate}
\end{lemma}
\begin{proof}
(i) The $\Gamma$-action on $\cal E(1)$ satisfies \[\gamma_i(p_1(U_j^{(1)}))=\begin{cases}[\ep]^pp_1(U_j^{(1)}) & i=j \\ p_1(U_j^{(1)})& i\neq j\end{cases}\qquad \gamma_i(p_2(U_j^{(1)}))=p_2(U_j^{(1)})),\] whence part (i) follows by a direct calculation.

(ii): This follows at once from the identity $\Theta_{i}(gf)=\Theta_{i}(g)\gamma(f)+g\Theta_{i}(f)$ for $g,f\in \cal E(1)$, as the $\Gamma$-action on $\cal E(1)$ is the identity modulo $\mu$.
\end{proof}

\begin{proof}[Proof of Theorem \ref{thm:ConvqHiggStrat}]
We prove the theorem by an argument similar to the that of statement (i) in the main proof in \S\ref{ss_proof_of_main_thm1}. Let $(H,\varepsilon)$ be an object of $\Strat(\cal E(\bullet))$, and put $G=H\otimes_{A^{\bx(1)},p_1}\cal E(1)$ and $\ol G=G/\mu$. By Lemma \ref{lem:qHiggsStratGammaEquiv}, the stratification $\varepsilon$ restricts a map $\ep|_H:H\to G^{\Gamma\times\{1\}}$, whose composition with $\Delta^*\colon G^{\Gamma\times\{1\}}\subseteq G\xto{\Delta} H$ is the identity map. The $\Gamma=\Gamma\times\{1\}$-action on the $A^{\bx(1)}$-module $G$ (via structure map $p_1$) defines a flat q-Higgs field $\Theta_G:G\to G\otimes_{p_1,A^{\bx(1)}}\q\Omega^1_{A^{\bx(1)}/A}$, whose $i^\sub{th}$ non-logarithmic component is by definition $\Theta_{G,i}=\tfrac{\gamma_i-1}{p_1(U_i^{(1)})\mu}$. Similarly, the $\Gamma$-action on $H$ corresponds to a q-Higgs field $\Theta_H$ whose non-logarithmetic components are $\Theta_{H,i}:=\tfrac{\gamma_i-1}{U_i^{(1)}\mu}$. Since $G^{\Gamma\times\{1\}}=G^{\Theta_G=0}$, we may reduce $\ep|_H$ modulo $\mu$ to obtain $\res\ep|_H:\res H\to \res G^{\Theta_{\res G}=0}$, whose composition with $\Delta_{\res H}:\res G\to\res H$ is the identity map; here $\Theta_{\res G}$ denotes the mod $\mu$ reduction of $\Theta_G$, and we will use similar notation for the reductions of the endomorphisms $\Theta_{G,i}$, $\Theta_{H,i}$.

We also let $\Theta_i$ and $\res\Theta_i$ be the non-logarithmetic components of the q-Higgs fields on $\cal E(1)$ and $\res{\cal E}(1)$, which were studied in Lemma \ref{lem:qHiggsOnE(1)}.

Then $\Theta_{G,i}(h\otimes f)=\Theta_{H,i}(h)\otimes \gamma(f)+h\otimes \Theta_{i}(f)$ for $h\in H$ and $f\in \cal E(1)$; by taking the reduction modulo $\mu$ we see that $\Theta_{\ol G,i}(h\otimes  f)=\Theta_{\ol H,i}(h)\otimes f+h\otimes \ol\Theta_{i}(f)$ for $h\in \ol H$ and $f\in \ol{\cal E}(1)$. Next, for each $h\in\ol H$, Lemma \ref{lem:E(1)Description}(ii) implies that  $\res\ep|_H(h)$ can be written uniquely in the form $\res\ep|_H(h)=\sum_{j_1,\dots,j_d\ge0} h_{\ul j}(\ol \tau^{(1)})^{[\ul j]}$, where $h_{\ul j}$ is a sequence of elements of $\res H$ satisfying $x_{\ul 0}=x$ and $x_{\ul j}\to 0$ $p$-adically as $\vert \ul n\vert\to\infty$; here we write $(\ol\tau^{(1)})^{[\ul j]}=\prod_{1\leq i\leq d}(\ol\tau_i^{(1)})^{[j_i]}$ for a multi-index $\ul j=(j_1,\dots,j_d)$. By the above formula for $\Theta_{\ol G,i}(h\otimes f)$ and Lemma \ref{lem:qHiggsOnE(1)}, we obtain 
\begin{equation}
\Theta_{\ol G,i}(\res\ep|_H(h))=\sum_{j_1,\dots,j_d\ge0}\Theta_{\ul H,i}(h_{\ul j})\otimes(\ol\tau^{(1)})^{[\ul j]}-\sum_{{\substack{j_1,\dots,j_d\ge0\\j_i\neq 0}}}x_{\ul j}\otimes (\ol\tau^{(1)})^{[\ul j-\ul 1_i]},
\end{equation}
which is a Higgs analogue of \eqref{eq:ND(1)NablaFormula}.

From the vanishing $\Theta_{\ol G,i}(\res\ep|_H(h))=0$, for
$i=1,\dots,d$, we see that $h_{\ul j}=\prod_{1\leq i\leq d}\Theta_{\ol H,i}^{j_i}(h)$. Therefore $\Theta_{\ol H,i}$ is $p$-adically quasi-nilpotent, as desired.
\end{proof}

\begin{remark}
The arguments of \S\ref{ss_gen_reps_as_strat} also work for an object of $\Rep^\mu_{\Gamma,\text{conv}}(A^{\bx (1)})$ and $\cal E(1)$ after some modifications similar to the proof of Theorem \ref{thm:ConvqHiggStrat}, and allow us to construct  a functor $\Strat_{\cal E(\bullet)}:\Rep_{\Gamma,\text{conv}}^{\mu}(A^{\bx (1)})\to \Strat (\cal E(\bullet))$ such that the composition $\ev^{\Strat}_{A^{\bx(1)}}\circ \Strat_{\cal E(\bullet)}$ is the identity. Combined with Theorem \ref{thm:ConvqHiggStrat}, this gives another  proof of Theorem \ref{thm:main} for $\ev_{A^{\bx (1)}}$ parallel to the proof of Theorem~\ref{thm:main}  for $\ev_{A^{\bx}}$ given in \S\ref{ss_proof_of_main_thm1}.
\end{remark}

\newpage

\section{Connections and admissibility over crystalline period~rings}\label{section_crystalline}
In this section we study generalised representations over crystalline period rings and their relation to connections and filtered F-crystals, the link in the latter case being a version of Faltings' associatedness. As in the previous sections we continue to work locally, letting $R$ be a small $p$-adic formally smooth $\roi$-algebra.

In \S\ref{ss_q_over_A_crys} we fix a framing of $R$ and study the category of generalised representations of $\Gamma$ on modules over the framed crystalline period ring $A_\crys^\bx(R)$ which are trivial modulo $\mu$, as a crystalline analogue of Sections \ref{section_small_reps}--\ref{section_q_connections}. We show that such generalised representations are equivalent to both of q-connections and actual connections:
\[\Rep_{\Gamma}^\mu(A_\crys^\bx(R))\quis\textrm{\rm qMIC}(A_\crys^\bx(R))\quis\textrm{\rm MIC}(A_\crys^\bx(R)).\]
Here the first equivalence is a formal consequence of the results established in \S\ref{ss_modules_with_q}, while the overall equivalence uses divided powers to replace the the action of the generators of $\Gamma$ by their logarithms; see Theorem \ref{theorem_q-A_crys} for details. After imposing an additional $p$-adic quasi-nilpotence conditions, these categories are moreover equivalent to $\CR(R/A_\crys)$, namely crystals on the crystalline site of $R/A_\sub{crys}$, with the equivalence given by evaluating any crystal on the pd-thickening $A_\crys^\bx(R)$ of $R$; the advantage of this category of crystals is that it is independent of any choice of framing.

The main goal of the section is to continue avoiding any choice of framing by working with $\res R$ and show how generalised representations in $\op{Rep}_\Delta^\mu(A_\inf(\res R))$ give rise to crystals. Analogously to \cite[\S Vf]{Faltings1989}, we say that a crystal $\cal F \in\CR(R/A_\crys)$ is {\em associated} to a generalised representation $M\in\Rep_\Delta^\mu(A_\inf(\res R))$ when they have isomorphic images as crystalline generalised representations in $\Rep^\mu_\Delta(A_\crys(\res R))$ along the respective functors  ``evaluate on pd-thickening $A_\crys(\res R)$'' $\op{ev}_{A_\crys(\res R)}$ and base change $-\otimes_{A_\inf(\res R)}A_\crys(\res R)$. The main theorems (Theorems \ref{theorem_crystal_genrep}--\ref{theorem_Admissibility}) imply that such a crystal is determined by the associated crystalline generalised representation, and that any generalised representation which admits a Frobenius structure is the associate of some crystal; there is therefore a resulting functor \begin{equation}\Rep^\mu_\Delta(A_\inf(\res R),\phi)\To  \op{F-CR}(R/A_\crys)\label{eqn_associated}\end{equation} associating F-crystals to our main category of generalised representations. In Theorem \ref{theorem_M_et_up_to_isogeny} we show that the filtered F-crystal associated to a generalised representation $M$ in this way is sufficiently rich to recover up to isogeny the Galois representation attached to $M$.

The machinery above is actually developed in the generality of filtered F-crystals; not only is this required for the proofs in the non-filtered context, but it plays a crucial role in \S\ref{ss_DieudonneII} where we complete the proof that the Dieudonn\'e functor of \S\ref{sss_Dieudonne} is essentially surjective.

\begin{remark}[Globalising]\label{remark_global_associatedness}
In the global context, given a smooth, $p$-adic formal $\roi$-scheme $\frak X$, we would like to generalise the results of this section by replacing $\Rep^\mu_\Delta(A_\inf(\res R))$ by the forthcoming category of relative Breuil--Kisin--Fargues modules $\BKF(\frak X)$ (see \S\ref{ss_relative_BKF}). This would include the analogous notion of associatedness, thereby leading to a functor $\BKF(\frak X,\phi)\to  \op{F-CR}(\frak X/A_\crys)$ with the following property: the F-crystal associated to a Breuil--Kisin--Fargues module $\bb M$ can be used to recover the associated lisse $\bb Q_p$-sheaf $(\sigma_\sub{\'et}^*\bb M)[\tfrac1p]$ which will be described in \S\ref{ss_etale}.

However, we restrict primarily to the affine case for two reasons. Firstly, sheafifying the conditions of Definition \ref{definition_filtration_2} to the global setting is rather laborious. Secondly, in Theorem \ref{theorem_local_with_phi} we will identify $\BKF(\frak X,\phi)$ with prismatic F-crystals, at which point it should be possible to define the desired functor $\BKF(\frak X,\phi)\to  \op{F-CR}(\frak X/A_\crys)$ by pulling back F-crystals along the map of sites $(\frak X/(A_\crys,p))_\Prism \to (\frak X^{(1)}/(A_\inf,\tilde\xi))_\Prism$.
\end{remark}

\subsection{$q$-connections over $A_\sub{crys}$ and their logarithms as crystals}\label{ss_q_over_A_crys}
We begin with a local study of ($q$-)connections over crystalline period rings. The key point is that suitable actions by $\bb Z_p^d$ (or, equivalently, $q$-connections by Proposition \ref{proposition_reps_vs_q_connections}) over $A_\crys$ give rise to connections by replacing the actions of the generators $\gamma_i$ by their logarithms $\log(\gamma_i)$; this is far from a novel observation, having appeared previously in \cite[\S12]{BhattMorrowScholze1} and \cite[\S4]{ColmezNiziol2017}.

Let $R$ be a $p$-adic formally smooth $\roi$-algebra, and fix a framing $\square:\roi\pid{\ul T^{\pm1}}\to R$. Let $A_\crys^\bx(R)$ be the associated {\em framed crystalline period ring}, namely the unique $p$-adic formally \'etale $A_\crys\pid{\ul U^{\pm1}}$-algebra lifting the framing map, similarly to \S\ref{ss_framed}; equivalently, $A_\crys^\bx(R)$ is the $p$-adic completion of $A_\inf^\bx(R)\otimes_{A_\inf}A_\crys$ (we mention that this $p$-adic completion is automatically $\mu$ and $\xi$-adically complete; indeed, each of these elements has divided powers and is therefore nilpotent modulo any power of $p$). Observe that the existing $\Gamma$-action on $A_\inf^\bx(R)$ extends to an action by $A_\crys$-algebra automorphisms on $A_\crys^\bx(R)$; the Frobenius $\phi$ on $A_\inf^\bx(R)$ similarly extends. In short, we have applied the deformation approach of Remark \ref{remark_q_via_deformation} to construct a setting for $q$-de Rham cohomology \[A_\crys\to A_\crys^\bx(R)\circlearrowleft\Gamma,\qquad [\ep]\in A_\crys,\qquad U_1,\dots,U_d\in A_\crys^\bx(R).\]

The goal of this subsection is to establish the following theorem, in which some of the notation remains to be explained:

\begin{theorem}\label{theorem_q-A_crys}
There exist equivalences of symmetric monoidal categories \[\Rep_{\Gamma}^\mu(A_\crys^\bx(R))\quis\textrm{\rm qMIC}(A_\crys^\bx(R))\quis\textrm{\rm MIC}(A_\crys^\bx(R))\]
where \[\textrm{\rm qMIC}(A_\crys^\bx(R)):=\categ{4.3cm}{finite projective modules with flat q-connection over $A_\crys^\bx(R)$}\quad \textrm{\rm MIC}(A_\crys^\bx(R)):=\categ{4.3cm}{finite projective modules with flat connection over $A_\crys^\bx(R)$}.\] The functors are defined below.
\end{theorem}

\begin{remark}
The first and second functors in the theorem depend on our fixed sequence of $p$-power roots of unity defining $\mu$, but the composition does not. Indeed, changing the sequence of $p$-power roots has the effect of changing $\ep\in\bb Z_p(1)\subseteq\roi^{\flat\times}$ into $\ep^a$ for some $a\in\bb Z_p^\times$, and similarly $\gamma_i$ into $\gamma_i^a$ and $t=\log([\ep])$ into $at$; therefore $\tfrac1t\log(\gamma_i)$ does not change.

However, the three categories do depend on the chosen framing.
\end{remark}

Both functors in the theorem preserve the underling module, only changing its additional structure. The first functor is as in Proposition \ref{proposition_reps_vs_q_connections} and Corollary \ref{corollary_reps_vs_q_connections}, given by replacing the action of $\bb Z_p^d$ on $M$ by the flat q-connection \[\nabla:N\to N\otimes_{A^\bx_\crys}\Omega^1_{A^\bx_\crys(R)/A_\crys},\quad m\mapsto\sum_{i=1}^d\frac1\mu(\gamma_i(m)-m)\otimes\dlog(U_i).\] We now explain the second functor; so let $(N,\nabla)\in \op{qMIC}(A_\crys^\bx(R))$ and, as in \S\ref{ss_modules_with_q}, write $\nabla=\sum_{i=1}^d\nabla_{i}^\sub{log}(-)\otimes\dlog(U_i)$. Then define \[\nabla_\cal L:=\sum_{i=1}^d\nabla_{\cal L,i}^\sub{log}(-)\otimes\dlog(U_i):N\to N\otimes_{A_\crys^\bx(R)}\Omega^1_{A_\crys^\bx(R)/A_\crys}\] by \[\nabla_{\cal L,i}^\sub{log}:=\frac1t\log(1+\mu\nabla_{i}^\sub{log}),\] where $t:=\log([\ep])\in A_\sub{crys}$. We will see in the course of the following proof that $\nabla_{\cal L}$ is a well-defined flat connection.

\begin{proof}[Proof of Theorem \ref{theorem_q-A_crys}]
The first equivalence of Theorem \ref{theorem_q-A_crys} is a special case of Proposition \ref{proposition_reps_vs_q_connections}, just like Corollary \ref{corollary_reps_vs_q_connections}. It remains to discuss the second equivalence; to simplify notation, and in anticipation of the point of view of \S\ref{ss_OA_admissibility}, it is helpful to recast it using the first equivalence. Namely we must show that the ``logarithm'' functor
\[\cal L:\Rep_\Gamma^\mu(A_\crys^\bx(R))\to \MIC(A_\crys^\bx(R)),\qquad N\mapsto (N,\nabla_\cal L=\sum_{i=1}^d\nabla_{\cal L,i}^\sub{log}(-)\otimes\dlog(U_i))\] is a well-defined equivalence of categories, where \[\nabla_{\cal L,i}^\sub{log}:=\frac1t\log(\gamma_i)=\frac{-1}t\sum_{m=1}^\infty \frac{(1-\gamma_i)^m}{m}:N\to N.\]

We begin by noting that the $\nabla_{\cal L,i}^\sub{log}$ are well-defined endomorphisms of $N$: indeed, \[\frac{(\gamma_i-1)^m}{m}=\mu\frac{([\ep]-1)^{m-1}}{m!}(m-1)!\left(\frac{\gamma_i-1}{\mu}\right)^m,\] where $([\ep]-1)^{m-1}/m!$ tends to $0$ $p$-adically in $A_\sub{crys}$, and $t$ and $\mu$ differ by a unit of $A_\sub{crys}$ \cite[Lem.~12.2]{BhattMorrowScholze1}.

We now check that the $\nabla_{\cal L,i}^\sub{log}$ satisfy a Leibnitz rule. The logarithm formula for $\nabla_{\cal L,i}^\sub{log}$ may be inverted via the exponential function since $t^m/m!\to0$ $p$-adically in $A_\crys$ as $m\to\infty$: namely $\gamma_i=\exp(t\nabla_{\cal L,i}^\sub{log})$, and therefore more generally $\gamma_i^a=\exp(at\nabla_{\cal L,i}^\sub{log})$ for any $a\in\bb Z$. Expanding this we see that \[\frac1a(\gamma_i^a-1)=t\nabla_{\cal L,i}^\sub{log}+a\sum_{m=2}^\infty a^{m-2}\frac{t^m}{m!}(\nabla_{\cal L,i}^\sub{log})^m:N\to N\] is well-defined, where the final term is the $m$-fold iteration of $\nabla_{\cal L,i}^{\sub{log}}$. This expression shows in particular, by taking $a=p^b$ to be successively larger powers of $p$, that \[\lim_{b\to\infty} p^{-b}(\gamma_i^{p^b}-1)(n)=t\nabla_i^\sub{log}(n)\] for all $n\in N$. Evaluating $\gamma_i^{p^b}-1$ at $nf$, where $n\in N$ and $f\in A_\sub{crys}^\bx(R)$, we have \[(\gamma_i^{p^b}-1)(nf)=(\gamma_i^{p^b}-1)(n)\cdot f+\gamma_i^{p^b}(n)\cdot (\gamma_i^{p^b}-1)(f)\] where the final $\gamma_i$ refers to the action on $A_\sub{crys}^\bx(R)$; now inverting $tp^b$ and letting $b\to\infty$ yields a Leibnitz rule \[\nabla_{\cal L,i}^\sub{log}(nf)=\nabla_{\cal L,i}^\sub{log}(n)f+n d_i^\sub{log}(f)\] where $d_i^\sub{log}:=\tfrac1t\log(\gamma_i):A_\sub{crys}^\bx(R)\to A_\sub{crys}^\bx(R)$ (everything we have said for $N$ holds for the trivial representation $A_\sub{crys}^\bx(R)$ itself; in particular, $d_i^\sub{log}$ is well-defined and $\lim_{b\to\infty} p^{-b}(\gamma_i^{p^b}-1)(f)=td_i^\sub{log}(f)$).

We claim that $\sum_{i=1}^dd_i^\sub{log}(-)\dlog(U_i):A_\sub{crys}^\bx(R)\to \Omega^1_{A_\crys^\bx(R)/A_\crys}$ is the usual universal continuous de Rham derivative $d:A_\sub{crys}^\bx(R)\to \Omega^1_{A_\crys^\bx(R)/A_\crys}$. Since each $d_i^\sub{log}$ is a continuous $A_\sub{crys}$-linear derivation, it may be written as $d_i^\sub{log}=f_i\circ d$ for some unique continuous $A_\sub{crys}$-linear map $f_i: \Omega^1_{A_\crys^\bx(R)/A_\crys}\to A_\sub{crys}$. Then $f_i(dU_j)=d_i^{\log}(U_j)=0$ for $j\neq i$ as $\gamma_i(U_j)=U_j$, and on the other hand $f_i(dU_i)=d_i^{\log}(U_i)=t^{-1}\sum_{m=1}^{\infty}\frac{(-1)^m}{m}(\gamma_i-1)(U_i)=t^{-1}\log[\varepsilon]\cdot U_i=U_i$. Hence we have $d=\sum_iU_i^{-1}f_i(-)\otimes dU_i=\sum_id_i^{\log}(-)\otimes d\log U_i$, as required to prove the claim.

So we have indeed shown that \[\nabla_\cal L:=\sum_{i=1}^d\nabla_{\cal L,i}^\sub{log}(-)\otimes\dlog(U_i):N\to N\otimes_{A_\crys^\bx(R)}\Omega^1_{A_\crys^\bx(R)/A_\crys}\] is a well-defined connection on $N$; it is flat since the $\gamma_i$, hence also the $\nabla_i^\sub{log}$, commute (cf.\ the standard argument of Lemma \ref{lemma_q_coordinates}). That is, we have shown that the functor $\Rep_\Gamma^\mu(A_\crys^\bx(R))\to \MIC(A_\crys^\bx(R))$ is well-defined. We have also  shown that the $\Gamma$-action can be recovered by the rule $\gamma_i:=\exp(t\nabla_{\cal L,i}^\sub{log})$. So, to show that these two procedures define an equivalence between the desired categories, it remains only to check that the inverse is well-defined, namely the following: given $(N,\nabla)\in\op{MIC}(A^\bx_\sub{crys}(R))$, then the endomorphism $\exp(t\nabla_{i}^\sub{log})$ of $N$ is $\gamma_i$-semilinear. But the Leibnitz rule implies $\frac{1}{j!}(t\nabla_i^{\log})^j(nf)=\sum_{j=j_1+j_2}\frac{1}{j_1!}(t\nabla_i^{\log})^{j_1}(n)\frac{1}{j_2!}(td_i^{\log})^{j_2}(f)$ for $n\in N$ and $f\in A_\crys^\bx(R)$, whence taking the sum over all $j>0$ obtains $\exp(t\nabla_{\cal L,i}^{\log})(nf) =\exp(t\nabla_{\cal L,i}^{\log})(n) \exp(td_{\cal L,i}^{\log})(f)=\exp(t\nabla_{\cal L,i}^{\log})(n)\gamma_i(f)$ as desired.
\end{proof}

\begin{remark}[Frobenius structures]\label{remark_Frob_q-Acrys}
Theorem \ref{theorem_q-A_crys} admits an analogue incorporating Frobenius structures: \[\Rep_{\Gamma}^\mu(A_\crys^\bx(R),\phi)\quis\op{qMIC}(A_\crys^\bx(R),\phi)\quis\MIC(A_\crys^\bx(R),\phi)\] Here the first two categories are obtained by replacing $A_\inf^\bx(R)$ by $A_\crys^\bx(R)$ in the analogous categories of Definition \ref{definition_phi}; on the other hand, $\MIC(A_\crys^\bx(R),\phi)$ is the more familiar category consisting of pairs $(N,\phi_N)$, where $N\in \MIC(A_\crys^\bx(R))$ and $\phi_N:(\phi^*N)[\tfrac1p]\to N[\tfrac1p]$ is an isomorphism of $A_\crys^\bx(R)$-modules with connection (note that we could replace $p$ by $\tilde\xi$ here: they differ by a unit of $A_\crys$ \cite[Lem.~12.2(iii)]{BhattMorrowScholze1}).

These equivalences of the categories with Frobenius structures follows from Theorem \ref{theorem_q-A_crys} by similar arguments as we already used in Theorem \ref{theorem_descent_to_framed_with_phi} and Proposition \ref{proposition_reps_vs_q_connections_with_phi}, so we omit the details. There are similarly  equivalences replacing $\phi$ by $\exists\phi$, where we use the obvious modification of the notation of Remark \ref{remark_cats_with_non-fixed_phi}.
\end{remark}

\begin{remark}[$p$-adic quasi-nilpotence]\label{remark_Frob_implies_p_nilp}
We will write $\MIC_\sub{conv}(A_\crys^\bx(R))\subseteq \MIC(A_\crys^\bx(R))$ for the full subcategory consisting of objects $(N,\nabla)$ whose connection $\nabla$ is $p$-adiclly quasi-nilpotent, i.e., for each $n\in N$ there exists $m\gg0$ such that $\nabla_i^m(n)\in pN$ for $i=1,\dots,d$. For example, by repeating the argument of Lemma \ref{lemma_phi_implies_xi_nilp}(iii), one sees that this holds in particular if $N$ may be upgraded to an object of $\MIC(A_\crys^\bx(R),\phi)$, i.e., if $N\in\MIC(A_\crys^\bx(R),\exists\phi)$.

We similarly define $\Rep_{\Gamma,\sub{conv}}^\mu(A_\crys^\bx(R))$ and $\op{qMIC}_\sub{conv}(A_\crys^\bx(R))$ by asking for $p$-adic quasi-nilpotence of the endomorphisms $\tfrac{\gamma_i-1}{U_i\mu}$ and $\nabla_i$ respectively. We then claim that the equivalences of Theorem \ref{theorem_q-A_crys} restrict to equivalences \[\Rep_{\Gamma,\sub{conv}}^\mu(A_\crys^\bx(R))\quis\textrm{\rm qMIC}_\sub{conv}(A_\crys^\bx(R))\quis\textrm{\rm MIC}_\sub{conv}(A_\crys^\bx(R)).\] The first equivalence is tautological, while the second paragraph of the proof of Theorem \ref{theorem_q-A_crys} shows that $\nabla_{\cal L,i}\equiv\tfrac\mu t\tfrac{\gamma_i-1}{U_i\mu}$ mod $(\tfrac{\mu^{m-1}}{m!}:\,m\ge2)$, which shows the second equivalence since a power of the latter ideal is contained in $pA_\sub{crys}$ (this follows from the facts that each $\tfrac{\mu^{m-1}}{m!}$ has divided powers, so is $p$-adically nilpotent, and that they tend to $0$ $p$-adically as $m\to\infty$).


\end{remark}

Our next goal is Proposition \ref{proposition_crystal_framedrep}, which will reinterpret Theorem \ref{theorem_q-A_crys} in terms of crystals. Although we have already seen one explicit description of the inverse of the functor $\cal L$ (namely, $\gamma_i=\exp(t\nabla_{\cal L,i}^{\sub {log}})$), we will require the following alternative formula:

\begin{proposition}\label{lemma_Acrys_rep_formula}
Let $N\in\Rep_{\bb Z_p^d}^\mu(A_\crys^\bx(R))$, with corresponding connection $\nabla_{\cal L}$ as in Theorem \ref{theorem_q-A_crys}. Then $\gamma_i=\sum_{m=0}^\infty\mu^{[m]}U_i^m\nabla_{\cal L,i}^m$ for each $i=1,\dots,d$, where $\nabla_{\cal L,i}=U_i^{-1}\nabla_{\cal L,i}^\sub{log}$ are the non-logarithmic coordinates of the connection $\nabla_{\cal L}$.
\end{proposition}
\begin{proof}
This will follow from the existing description once the following formal identity has been established. Let $D$ be the $\bb Z_p[T,X]$-subalgebra of $\bb Q_p[T,X]$ generated by the elements $\tfrac{1}{m!}(\tfrac{T^{p-1}}p)^m$, for $m\ge1$; let $\hat D=\projlim_n D/p^nD$ be its $p$-adic completion. Then we claim that \begin{equation}\sum_{m\ge0}T^{[m]}\prod_{j=0}^{m-1}(X-j) =\sum_{m\ge 0}(\log(1+T))^{[m]}X^m\label{eqn_identity}\end{equation} in $\hat D$.

We prove the claim. It is straightforward to show that $\log(1+T)=\sum_{m=1}^{\infty}(-1)^{m-1}\tfrac {T^n}m$ converges in $\hat D$ to an element of $T\hat D$, and that both sides of the desired equality converge in $\hat D$. For $m>0$, let $f_m$ be the composition of the inclusion map $R\hookrightarrow \bb Q_p[T,X]$ and the projection map $\bb Q_p[T,X]\to\bb Q_p[T,X]/(T^m,X^m)$. Then the image of $f_m$ is a finitely generated $\bb Z_p$-module, and therefore $f_m$ naturally extends to a homomorphism $\hat f_m\colon \hat R\to \bb Q_p[T,X]/(T^m,X^m)$. By taking the inverse limit over $m$, we obtain an injective homomorphism $\hat f\colon \hat R\hookrightarrow \bb Q_p[[T,X]]$. The power series appearing in the definition of $\log(1+T)$ and both sides of the equality in the claim are sent to finite sums by $\hat f_m$. Hence it suffices to prove the claim in $\bb Q_p[[T,X]]$ with respect to the $(T,X)$-adic topology. Put $\prod_{j=0}^{m-1}(X-j)=\sum_{\ell=0}^{m-1}a_\ell^{(m)}X^\ell$. Then we have $(1+T)^m(\frac{d}{dT})^m=\prod_{j=0}^{m-1}((1+T)\frac{d}{dT}-j)=\sum_{\ell=0}^{m-1}a_\ell^{(m)}((1+T)\frac{d}{dT})^\ell$. Since $(1+T)\frac{d}{dT}((\log(1+T))^{[m]})=(\log(1+T))^{[m-1]}$ for $m\ge1$, we have $(1+T)^m(\frac{d}{dT})^m(\log(1+T))^{[l]} =\sum_{\ell=0}^{\min\{m-1,l\}}a_\ell^{(m)}(\log(1+T))^{[l-\ell]}$, whose value at $T=0$ is $a_l^{(m)}$ if $l\leq m-1$ and $0$ otherwise. Hence the value of $(\frac{d}{dT})^m\sum_{l\ge0}(\log(1+T))^{[l]}X^l$ at $T=0$ is $\sum_{l=0}^{m-1}a_l^{(m)}X^l=\prod_{j=0}^{m-1}(X-j)$. 

With the claimed identity established, we may complete the proof of the proposition; fix some $i=1,\dots,d$. Since $p^{-1}\mu^{p-1}$ has divided powers in $A_\crys$, we may define an action of $D$ on $N$ by declaring that $T$ acts as multiplication by $\mu$ while $X$ acts as the endomorphism $\nabla_{\cal L,i}^\sub{log}$. Since $N$ is $p$-adically complete and separated, this action extends to an action of $\hat D$ on $N$. By considering the action of both sides of (\ref{eqn_identity}) on an element $n\in N$ and using $U_i^{m}\nabla_i =\prod_{j=0}^{m-1}(\nabla_{\cal L, i}^\sub{log}-j)$, we obtain $\sum_{m=0}^{\infty}\mu^{[m]}U_{i}^m\nabla_i^{ m}(n) =\sum_{m=0}^{\infty}t^{[m]}(\nabla_{\cal L,i}^\sub{log})^{ m}(n)$ as desired.
\end{proof}


To discuss the crystalline interpretation of Theorem \ref{theorem_q-A_crys} we must introduce some standard notation. For $n\ge1$, we will write $R_n=R/p^nR$, $A_{\crys,n}=A_\crys/p^nA_\crys$, and $A_{\crys,n}^\bx(R)=A_\crys^\bx(R)/p^n$; we let $\Fil^rA_{\crys,n}$, for $r\ge0$, denote the divided power filtration on $A_{\crys,n}$ with respect to $A_{\crys,n}\to\roi/p^n$.

Let $\CRYS(R_n/A_{\crys,n})$ be the big crystalline site of $\Spec(R_n)$ over the PD-pair $(A_{\crys,n},pA_{\crys,n}+\Fil^1A_{\crys,n})$, let $\cal O_{R_n/A_{\crys,n}}$ be its structure sheaf, and let $\CR(R_n/A_{\crys,n})$ be the category of crystals of $\cal O_{R_n/A_{\crys,n}}$-modules locally free of finite type. The homomorphisms $R_{n+1}\to R_n$ and $A_{\crys,n+1}\to A_{\crys,n}$ induce the pullback functor $i_{n,n+1}^*\colon \CR(R_{n+1}/A_{\crys,n+1})\to\CR(R_n/A_{\crys,n})$. We define $\CRYS(R_1/A_{\crys,n})$ and $\CR(R_1/A_{\crys,n})$ similarly, and the pullback functor $i_n^*\colon \CR(R_n/A_{\crys,n})\to \CR(R_1/A_{\crys,n})$ is an equivalence \cite[Thm.~IV.1.4.1]{Berthelot1974}.

We define a finite locally free crystal $\cal F$ on $R/A_\crys$ to be the following set of data: an object $\cal F_n$ of $\CR(R_n/A_{\crys,n})$ for each $n\geq 1$ and isomorphisms $i_{n,n+1}^*\cal F_{n+1}\isoto \cal F_n$. A  locally free crytal on $R_1/A_\crys$ is defined similarly by using $\CR(R_1/A_{\crys,n})$. Writing $\CR(R/A_\crys)$ and $\CR(R_1/A_\crys)$ for the categories of finite locally free crystals on $R/A_\crys$ and on $R_1/A_\crys$ respectively, the obvious pullback functor $i^*\colon \CR(R/A_\crys)\to \CR(R_1/A_\crys)$ is again an equivalence of categories.

By \cite[Thm.~IV.1.6.5]{Berthelot1974} and Lemma \ref{lemma_limits_of_finite_projective},
there is an equivalence of categories 
\begin{equation}\label{equqation_crystal_mic}\CR(R/A_\crys)\quis \MIC_\sub{conv}(A_\crys^\bx(R))\end{equation}
obtained by taking the inverse limit of the evaluation of $\cal F_n$ on the objects $\Spec(R_n)\xrightarrow{\id}\Spec(R_n)$ of $\CRYS(R_n/A_{\crys,n})$, equipped with the natural flat connection.

For an object $\cal F=(\cal F_n)$ of $\CR(R/A_\crys)$, we also define a finite projective $A_{\crys,n}^\bx(R)$-module $\cal F_n(A_\crys^\bx(R))$ with a semilinear action of $\Gamma$ to be the evaluation of $\cal F_n$ on the object $\Spec(R_n)\to \Spec(A_{\crys,n}^\bx(R))$ of $\CRYS(R_n/A_{\crys,n})$, equipped with its natural right action by $\Gamma$. Let $\cal F(A_\crys^\bx(R))$ denote the inverse limit $\projlim_n\cal F_n(A_\crys^\bx(R))$, which is a finite projective $A_\crys^\bx(R)$-module with a semilinear action of $\Gamma$ by Lemma \ref{lemma_limits_of_finite_projective}.

The following proposition shows that the constructions of Theorem \ref{theorem_q-A_crys}, (\ref{equqation_crystal_mic}), and the previous paragraph are compatible:

\begin{proposition}
\label{proposition_crystal_framedrep}
\begin{enumerate}
\item For $\cal F\in \CR(R/A_\crys)$, the generalised representation $\cal F(A_\crys^\bx(R))$ belongs to $\Rep_{\Gamma,\sub{conv}}^\mu(A_\crys^\bx(R))$; denote the resulting functor by \[\ev_{A_\crys^\bx(R)}:\op{CR}(R/A_\crys)\to\Rep_{\Gamma,\sub{conv}}^\mu(A_\crys^\bx(R)),\quad \cal F\mapsto \cal F(A_\crys^\bx(R)).\]
\item The following diagram, in which all functors are equivalences, is commutative up to natural isomorphism:
\[
\xymatrix{\CR(R/A_\crys)\ar[r]^-{\sub{(\ref{equqation_crystal_mic})}}\ar[dr]_{\ev_{A_\crys^\bx(R)}} 
&\MIC_\sub{conv}(A_\crys^\bx(R))\ar[d]^{\cal L^{-1}}\\
&\Rep_{\Gamma,\sub{conv}}^\mu(A_\crys^\bx(R))}
\]
(where $\cal L$ is the composition of equivalences from Remark \ref{remark_Frob_implies_p_nilp}).
\end{enumerate}
\end{proposition}

\begin{proof}
Let $\cal F=(\cal F_n)_{n\geq 1}$ be an object of $\CR(R/A_\crys)$, and let $(N,\nabla)$ be the object of $\MIC_\sub{conv}(A_\crys^\bx(R))$ corresponding to $\cal F$ by the equivalence of categories \eqref{equqation_crystal_mic}. By definition we have $\cal F(A_\crys^\bx(R))=N$ as $A_\crys^\bx(R)$-modules. By Proposition \ref{lemma_Acrys_rep_formula} and the equivalences of Remark \ref{remark_Frob_implies_p_nilp}, it therefore remains only to prove that the action of $\gamma_i$ on $\cal F(A_\crys^\bx(R))$ coincides with that of $\sum_{m=0}^\infty\mu^{[m]}U_i^m\nabla_{i}^m$ on $N$.

Fix $n\ge1$. Any automorphism of the PD-ring $A_{\crys,n}^\bx(R)$ over $A_{\crys,n}$ induces an automorphism of the object $\Spec(R_n)\to \Spec(A_{\crys,n}^\bx(R))$ of $\CRYS(R_n/A_{\crys,n})$, and the pullback by this latter automorphism gives an isomorphism $\iota_f\colon \cal F_n(A_\crys^\bx(R))\otimes_{A_{\crys,n}^\bx(R),f}A_{\crys,n}^\bx(R)\isoto \cal F_n(A_\crys^\bx(R))$. For two such automorphisms $f$ and $f'$ inducing the identity on $R_n$, the composition $\iota_{f'}^{-1}\circ \iota_f$ is given by the formula
\[\iota_{f'}^{-1}\circ \iota_f(x\otimes 1)=\sum_{m_1,\dots,m_d\ge0}\prod_{i=1}^d\nabla_i^{ m_i}(x)\otimes\prod_{i=1}^d(f(U_i)-f'(U_i))^{[m_i]},\qquad x\in N\]
(see \cite[Rmk.~18(ii)]{Tsuji_simons} and the proofs of \cite[Thms.~4.12\&6.6]{BerthelotOgus1978}).

Applying this formula to the actions of both $\gamma_i$ and the identity, we see that the action of $\gamma_i$ on $\cal F_n(A_\crys^\bx(R))$ is indeed given by $\gamma_i(x)=\sum_{m\ge0}\mu^{[m]}U_i^m\nabla_i^{ m}(x)$, as desired.
\end{proof}

\begin{remark}[Frobenius structures]\label{remark_F_structures_Acrys}
The absolute Frobenius of $R_1$ and $\varphi$ on $A_\crys$ induce Frobenius pullbacks $\varphi^*\colon \CR(R_1/A_{\crys,n})\to \CR(R_1/A_{\crys,n})$ and $\varphi^*\colon \CR(R_1/A_\crys)\to \CR(R_1/A_\crys)$. A {\em Frobenius structure} on a finite locally free crystal $\cal F$ on $R/A_\crys$ is an isomorphism $\varphi_{\cal F}\colon \varphi^*i^*\cal F_{\bb Q}\isoto i^*\cal F_{\bb Q}$ in the category $\CR(R/A_\crys)_{\bb Q}$. Let $\CR(R/A_\crys,\varphi)$ denote the category of finite locally free crystals with Frobenius. 

Suppose that $\cal F$ is equipped with such a Frobenius structure $\varphi_{\cal F}$. Since $\cal F_n$ is a crystal, $\cal F_n(A_\crys^\bx(R))$ is isomorphic to the evaluation of $i_n^*\cal F_n$ on $\Spec(R_1)\to \Spec(A_{\crys,n}^\bx(R))$. The $\Gamma$-equivariance of $\varphi$ on $A_\crys^\bx(R)$ implies that we have a $\Gamma$-equivariant and $A_\crys^\bx(R)$-linear isomorphism \[(\varphi^*i_n^*\cal F_n)(\Spec(R_1)\to \Spec(A_{\crys,n}^\bx(R)))\cong  \cal F_n(A_\crys^\bx(R))\otimes_{A_{\crys,n}^\bx(R),\varphi}A_{\crys,n}^\bx(R).\] Therefore $\varphi_{\cal F}$ induces a $\Gamma$-equivariant $A_\crys^\bx(R)$-linear isomorphism \[\varphi^*(\cal F(A_\crys^\bx(R)))[\tfrac{1}{p}]\isoto\cal F(A_\crys^\bx(R))[\tfrac{1}{p}].\]

Recalling from Remark \ref{remark_Frob_implies_p_nilp} that $p$-adic quasi-nilpotence is automatic in the presence of a Frobenius structure, we thus obtain an analogue of Proposition \ref{proposition_crystal_framedrep}: namely, there is a well-defined commutative diagram of equivalences
\[
\xymatrix{
\CR(R/A_\crys,\varphi)\ar[r]^(.45){\sim}\ar[dr]_{\ev_{A_\crys^\bx(R)}} 
&\MIC(A_\crys^\bx(R),\varphi)\ar[d]^{\cal L^{-1}}_{\sim}\\
&\Rep_\Gamma^\mu(A_\crys^\bx(R),\varphi)}
\]
(and similarly for the $\exists\phi$ variants of the categories).
\comment{
 By Remark \ref{remark_Frob_implies_p_nilp}, the above equivalence induces
\[
\CR(R/A_\crys,\varphi)\quis \MIC(A_\crys^\bx(R),\varphi),\qquad
\CR(R/A_\crys,\exists\varphi)\quis \MIC(A_\crys^\bx(R),\exists\varphi).
\]
}
\end{remark}

\subsection{Preliminaries on filtrations}\label{subsection_saturated}
In this subsection we first present variants of Theorem \ref{theorem_q-A_crys} and Proposition \ref{proposition_crystal_framedrep} in which filtrations are incorporated; then we discuss the notion of a filtration being ``saturated'', which will be required for our study of associatedness in \S\ref{ss_OA_admissibility} and for the Dieudonn\'e theory in \S\ref{ss_DieudonneII}.

\begin{definition}\label{definition_filtration_1}
We first define the decreasing filtration $\Fil^rA_\crys^\bx(R)$, for $r\in\bb Z$, to be the closure of $\Fil^rA_\crys\cdot A_\crys^\bx(R)$ in $A_\crys^\bx(R)$ with respect to the $p$-adic topology. 

Let $\Rep_{\Gamma}^\mu(A_\crys^\bx(R),\Fil)$, $\op{qMIC}(A_\crys^\bx(R),\Fil)$, and $\MIC(A_\crys^\bx(R),\Fil)$ denote the categories consisting of objects $N$ of the respective category without Fil and equipped with a decreasing filtration $\Fil^rN$, for $r\in\bb Z$, satisfying the following three conditions:
\begin{enumerate}[(F1)]
\item $\Fil^rN$ is a $p$-adically closed submodule of $N$ for each $r\in\bb Z$;
\item $\Fil^rA_\crys^\bx(R)\Fil^sN\subset \Fil^{r+s}N$ for all $r,s\in\bb Z$ (we will refer to this condition by saying that the filtration is {\em multiplicative});
\item Griffith's transversality in the following sense:
\begin{enumerate}
\item for the first category $(\gamma_i-1)(\Fil^rN)\subset \mu \Fil^{r-1}N$ for $i=1,\dots, d$ and $r\in\bb Z$;
\item for the second category $\nabla(\Fil^rN)\subset \Fil^{r-1} N\otimes_{A_\crys^\bx(R)}\q\Omega^1_{A_\crys^\bx(R)/A_\crys}$ for $r\in\bb Z$;
\item for the third category $\nabla(\Fil^rN)\subset \Fil^{r-1}N\otimes_{A_\crys^\bx(R)}\Omega^1_{A_\crys^\bx(R)/A_\crys}$ for $r\in\bb Z$.
\end{enumerate}
\end{enumerate}
\end{definition}

We then have the following filtered version of Theorem \ref{theorem_q-A_crys}:

\begin{theorem}\label{theorem_q-A_crys_filtered}
The functors of Theorem \ref{theorem_q-A_crys} induce equivalences of categories
\[\Rep_{\Gamma}^\mu(A_\crys^\bx(R),\Fil)\quis\op{qMIC}(A_\crys^\bx(R),\Fil)\quis\MIC(A_\crys^\bx(R),\Fil).\] 
\end{theorem}
\begin{proof}
The first equivalence is tautological since $\nabla_i^\sub{log}=\mu^{-1}(\gamma_i-1)$. The second equivalence follows from the formulae (see the second paragraph of the proof of Theorem \ref{theorem_q-A_crys}):
\begin{align*}
\nabla_{\cal L,i}^\sub{log}&=\frac{\mu}{t}\sum_{m=1}^{\infty}(-1)^{m-1}\frac{\mu^{m-1}}{m!}(m-1)!(\nabla_i^\sub{log})^m,\\
\nabla_i^\sub{log}&=\frac{t}{\mu}\sum_{m=1}^{\infty}\frac{t^{m-1}}{m!}(\nabla_{\cal L,i}^\sub{log})^m,
\end{align*}
where $\mu^{m-1}/m!$ and $t^{m-1}/m!$ are contained in $\Fil^{m-1}A_\crys$ and tend to $0$ as $m\to \infty$, and $\mu\in tA_\crys^{\times}$.
\end{proof}

We next add filtrations to Proposition \ref{proposition_crystal_framedrep}, beginning by defining filtered crystals.

\begin{definition}
For an object $T$ of $\CRYS(R_n/A_{\crys,n})$, the structure sheaf $\roi_T$ is equipped with decreasing filtration defined by $\Fil^r\cal O_T=0$ for $r<0$, and $=J_T^{[r]}$ for $r\ge0$ where $J_T$ is the PD-ideal of $\cal O_T$. A {\em filtered crystal} on $R_n/A_{\crys,n}$ is a crystal $\cal F$ equipped with a decreasing filtration $\Fil^r\cal F$, for $r\in\bb Z$, such that (i) $\Fil^r\cal O_T\Fil^s\cal F_T\subset \Fil^{r+s}\cal F_T$ for all $r,s\in\bb Z$ and all objects $T$ of $\CRYS(R_n/A_{\crys,n})$, and (ii) $\Fil^r\cal F_{T'}=\sum_{s\in\bb Z}\Fil^s\cal O_{T'}\cdot
\op{Im}(u^{-1}(\Fil^{r-s}\cal F_T)\to\cal F_{T'})$ for all $r\in\bb Z$ and all morphisms $u\colon T'\to T$ in $\CRYS(R_n/A_{\crys,n})$.

Let $\CR(R_n/A_{\crys,n},\Fil)$ denote the category of filtered crystals $\cal F$ on $R_n/A_{\crys,n}$ such that the underlying crystal is locally free of finite type and $\Fil^r\cal F_T$ is a quasi-coherent $\cal O_T$-module for all $r\in\bb Z$ and all objects $T$ of $\CRYS(R_n/A_{\crys,n})$. We have a pullback functor $i_{n,n+1}^*\colon \CR(R_{n+1}/A_{\crys,n+1},\Fil)\to \CR(R_n/A_{\crys,n},\Fil)$ for each $n\ge1$ and, similarly to the case without filtration, we define $\CR(R/A_\crys,\Fil)$ to be the category of data consisting of objects $\cal F_n$ of $\CR(R_n/A_{\crys,n},\Fil)$ and 
isomorphisms $i_{n,n+1}^*\cal F_{n+1}\isoto \cal F_n$ for $n\ge1$.
\end{definition}

The equivalence of categories \eqref{equqation_crystal_mic} refines to an equivalence \cite[Thm.~17]{Tsuji_simons}
\begin{equation}\label{equation_Crystal_Mic_filtered}
\CR(R/A_\crys,\Fil)\quis \MIC_\sub{conv}(A_\crys^\bx(R),\Fil),
\end{equation} where the target is the full subcategory of $\MIC(A_\crys^\bx(R),\Fil)$ consisting of objects with $p$-adically quasi-nilpotent connection. We similarly define $\Rep_{\bb Z_p,\sub{conv}}^\mu(A_\crys^\bx(R),\Fil)$, and may now present the filtered version of Proposition \ref{proposition_crystal_framedrep}:

\begin{proposition}\label{proposition_crystal_framedrep_filtered}
The functor $\op{ev}_{A_\sub{crys}^\bx(R)}$ of Proposition \ref{proposition_crystal_framedrep}(i) upgrades to a filtered version $\op{ev}_{A_\sub{crys}^\bx(R)}:\CR(R/A_\crys,\Fil)\to \Rep_{\Gamma,\sub{conv}}^\mu(A_\crys^\bx(R),\Fil)$ such that the following diagram of equivalences commutes up to natural isomorphism:
\[\xymatrix{\CR(R/A_\crys,\Fil)\ar[r]^-{\sub{(\ref{equation_Crystal_Mic_filtered})}}\ar[dr]_{\ev_{A_\crys^\bx(R)}} 
&\MIC_\sub{conv}(A_\crys^\bx(R),\Fil)\ar[d]^{\cal L^{-1}}\\
&\Rep_{\Gamma,\sub{conv}}^\mu(A_\crys^\bx(R),\Fil)
}\]
\end{proposition}
\begin{proof}
Given an object $\cal F=(\cal F_n)$ of $\CR(R/A_\crys,\Fil)$, the inverse limit $\projlim_n \Fil^r\cal F_n(A_{\sub{crys},n}^\bx(R))$ of the evaluations of $\Fil^r\cal F_n$ on $\Spec(R_n)\to \Spec(A_{\crys,n}^\bx(R))$ defines a decreasing filtration $\Fil^r\cal F(A_\crys^\bx(R))$ on $\cal F(A_\crys^\bx(R))$. This is a multiplicative filtration by closed submodules; to finish refining $\cal F(A_\sub{crys}^\bx(R))$ to an object of $\Rep_{\Gamma,\sub{conv}}^\mu(A_\crys^\bx(R),\Fil)$ it remains to check that this filtration on $\cal F(A_\crys^\bx(R))$ satisfies condition (F3)(a) of Definition \ref{definition_filtration_1}.

However, ignoring this question for a moment, it already follows from the compatibility without filtration, i.e., Proposition \ref{proposition_crystal_framedrep}(ii), that the diagram commutes up to natural isomorphism (to make a precise statement, we should temporarily replace $\Rep_{\Gamma,\sub{conv}}^\mu(A_\crys^\bx(R),\Fil)$ by the larger category where the filtrations are only required to satisfy conditions (F1) and (F2) of Definition \ref{definition_filtration_1}).

This compatibility means the following in particular: letting $\nabla$ denote the connection on $\cal F(A_\crys^\bx(R))$ given by (\ref{equation_Crystal_Mic_filtered}), then the $\Gamma$-action on $\cal F(A_\crys^\bx(R))$ is given by the second formula in the proof of Theorem \ref{theorem_q-A_crys_filtered}, i.e., $\tfrac{\gamma_i-1}{\mu}=\tfrac{t}\mu\sum_{m=1}^\infty\tfrac{t^{m-1}}{m!}(\nabla^\sub{log}_i)^m$ (where $\nabla_i^\sub{log}$ are the logarithmic components of the honest connection $\nabla$). But $\nabla^\sub{log}_i(\Fil^r\cal F(A_\crys^\bx(R)))\subseteq \Fil^{r-1}\cal F(A_\crys^\bx(R))$ since the functor (\ref{equation_Crystal_Mic_filtered}) outputs filtrations satisfying (F3)(c), and the filtration is $p$-adically closed, so we easily see that $(\gamma_i-1)(\Fil^r\cal F(A_\crys^\bx(R))\subseteq \mu\Fil^{r-1}\cal F(A_\crys^\bx(R))$ as desired.
%
%
\end{proof}

Now we turn to saturated filtrations; since we need the definitions and basic properties over various rings (notably $A_\crys^\bx(R)$, $A_\crys(\res R)$, and the forthcoming $\cal R, \roiA_\crys(\cal R)$ of \S\ref{ss_OA_admissibility}; in all cases the non-zero divisor $f$ below will be $\xi$), we work axiomatically when possible.

\begin{definition}\label{definition_filtration_2}
Let $A$ be a ring equipped with a descending, multiplicative, exhaustive (i.e., $A=\bigcup_{r\in\bb Z}\Fil^r A$) $\bb Z$-indexed filtration (a ``filtered ring'' for short), and $f\in A$ a fixed non-zero divisor which belongs to $\Fil^1A$. Given an $f$-torsion-free $A$-module $N$ equipped with a $\bb Z$-indexed filtration, we will be interested in the following conditions:
\begin{enumerate}
\item[(F4)] $\Fil^rN=\{n\in N : f n\in\Fil^{r+1}N\}$ for all $r\in\bb Z$ (in which case we say that the filtration is {\em saturated}).
\item[(F5)] Writing $A^+:=\bigcup_{r\in\bb Z}f^{-r}\Fil^rA\subseteq A[\tfrac1f]$, then $N^+:=\bigcup_{r\in\bb Z}f^{-r}\Fil^rN\subseteq N[\tfrac1f]$ is finite projective as an $A^+$-module and generates the $A[\tfrac1f]$-module $N[\tfrac1f]$. (We will only consider this condition when the filtration on $N$ is multiplicative, in the sense that $\Fil^r A\Fil^sN\subseteq\Fil^{r+s}N$, since otherwise $N^+$ need not even be an $A^+$-module.)
\end{enumerate}
\end{definition}

A saturated filtration is completely determined by the submodule $N^+$, in the following sense:

\begin{lemma}\label{lemma_saturated}
Let $A,f$ be as in Definition \ref{definition_filtration_2}. Then there is an equivalence of categories
\begin{align*}\categ{4cm}{$f$-torsion-free $A$-modules $(N,\Fil^*N$) equipped with a saturated multiplicative filtration}&\quis \categ{5cm}{pairs $(N,N^+)$, where $N$ is an $f$-torsion-free $A$-module and $N^+$ is an $A^+$-submodule of $N[\tfrac1f]$}\\
(N,\Fil^*N)&\mapsto (N,N^+:=\bigcup_{r\in\bb Z}f^{-r}\Fil^rN)\\
(N,\op{Fil}^rN:=f^rN^+\cap N\textrm{ for }r\in\bb Z) &\reflectbox{ $\mapsto$ } (N,N^+)
\end{align*}
\end{lemma}
\begin{proof}
We leave the straightforward verification to the reader.
\end{proof}

\begin{remark}
We mention two additional properties of the correspondence of Lemma \ref{lemma_saturated}:
\begin{enumerate}
\item $N^+$ generates the $A[\tfrac1f]$-module $N[\tfrac1f]$ if and only if the corresponding filtration on $N$ is exhaustive. In particular, we assumed that the filtration on $A$ itself was exhaustive to ensure that $A^+[\tfrac1f]=A[\tfrac1f]$.
\item Secondly, the correspondence actually holds more generally without the $f$-torsion-freeness assumption on $N$, given by $(N, \Fil^*N)\mapsto (N,N^+:=\bigcup_{r\in\bb Z}f^{-r}e(\Fil^rN))$ and $(N,N^+)\mapsto (N,\Fil^rN:=\{n\in N:e(n)\in f^rN^+\})$, where $e:N\to N[\tfrac1\xi]$ is the canonical map. Indeed, this correspondence reduces to the $f$-torsion-free case by noting that any saturated filtration (which we define as in Definition \ref{definition_filtration_2}(F4) even in the presence of $f$-torsion) satisfies $\bigcap_{r\in\bb Z}\Fil^rN\supseteq\op{ker}e$.
\end{enumerate}
\end{remark}

\begin{definition}[Saturation]\label{definition_saturation}
Let $A,f$ be as in Definition \ref{definition_filtration_2}, and let $N$ be an $f$-torsion-free $A$-module equipped with a multiplicative filtration $\Fil^*N$. The {\em saturation} of the filtration is the new filtration $\Fil^*_\sub{sat}N$ given by setting $N^+:=\bigcup_{r\in\bb Z}f^{-r}\Fil^r N$ and then applying the leftwards functor of Lemma~\ref{lemma_saturated}, i.e.,  $\Fil^r_\sub{sat}N:=\{n\in N: f^sn\in\Fil^{r+s}N\text{ for some }s\in\bb Z\}$.
\end{definition}

We will need to know that conditions (F4) and (F5) of Definition \ref{definition_filtration_2} automatically imply closedness of the filtration (we remark that, in the following lemma, $p$-adic closedness could be replaced by $J$-acic closedness for any fixed ideal $J\subseteq A$):

\begin{lemma}\label{lemma_saturated_implies_complete}
Let $A,f$ be as in Definition \ref{definition_filtration_2}, and assume in addition that the filtration on $A$ itself is saturated and that $\Fil^rA$ is $p$-adically closed in $A$ for all $r\in\bb Z$. Let $N$ be a finitely generated $f$-torsion-free $A$-module and $N^+$ a finite projective $A^+$-submodule of $N[\tfrac1f]$ which generates it as a $A[\tfrac1f]$-module. Then $f^rN^+\cap N$ is a $p$-adically closed submodule of $N$ for each $r\in\bb Z$.
\end{lemma}
\begin{proof}
Since $N^+$ is finite projective we may pick $a\gg0$ and $A^+$-linear maps $N^+\xto{\iota}A^{+a}\xto{\pi}N^+$ such that $\pi\circ\iota=\op{id}_{N^+}$. Then observe that any given $n\in N[\tfrac1f]$ belongs to $f^rN^+$ if and only if $\iota(n)\in f^rA^{+a}$: necessity is obvious and sufficiency follows from $n=\pi(\iota(n))$.

Picking $b\gg0$ such that $\iota(N)\subseteq f^{-b}A^a$ (as $N$ is finitely generated), it follows that $\Fil^rN$ is the preimage of $f^{-b}A^a\cap f^rA^{+a}=(f^{-b}\Fil^{r+b}A)^a$ along the map $\iota:N\to f^{-b}A^a$; since this map is $p$-adically continuous and the filtration on $A$ is $p$-adically closed, we deduce the same for the filtration on $N$ as desired.
\end{proof}

We now specialise to cases of interest to us. Recall the ring $\res R$ from Remark \ref{remark_Delta}, which is equipped with an action by the geometric fundamental group $\Delta=\pi_1^\sub{\'et}(R[\tfrac1p])$, such that base change along $A_\inf(R_\infty)\to A_\inf(\res R)$ defines an equivalence $\Rep^{\mu}_\Gamma(A_\inf(R_\infty))\to \Rep^{\mu}_\Delta(A_\inf(\res R))$ (and similarly with Frobenius structures by a simple argument of a similar style to Theorem \ref{theorem_descent_to_framed_with_phi}).

\begin{definition}
Let $\op{MIC}(A_\crys^\bx(R),\mathrm{SatFil})\subseteq\op{MIC}(A_\crys^\bx(R),\Fil)$ denote the full subcategory consisting of objects $N$ for which the filtration satisfies (in addition to conditions (F1)--(F3) of Definition \ref{definition_filtration_1}) conditions (F4) and (F5) of Definition \ref{definition_filtration_2} with $A=A_\crys^\bx(R)$, $f=\xi$.

Similarly, let $\Rep_\Delta^\mu(A_\crys(\res R),\op{SatFil})$ denote the category consisting of objects $M$ of $\Rep_\Delta^\mu(A_\crys(\res R))$ equipped with a decreasing, $p$-adically closed, multiplicative, $\bb Z$-indexed filtration satisfying conditions (F4) and (F5) of Definition \ref{definition_filtration_2} (with $A=A_\crys(\res R)$, $f=\xi$), i.e., we just drop the Griffiths transversality condition, which does not make sense over non-framed period rings.

Variants of these categories with Frobenius structures are defined in the obvious way, with no compatibility required between the filtration and Frobenius.
\end{definition}

We next explain several constructions of filtrations which will appear in our study of generalised representations and associatedness. To summarise, the remaining lemmas of the subsection will build a commutative diagram:
\begin{equation}\label{diagram_Ainfcrys_frame_nonframe_Rep_filtered}
\xymatrix@C=3cm@R=2cm{
\Rep_\Gamma^\mu(A_\inf^\bx(R),\phi)\ar[r]^{-\otimes_{A_\inf^\bx(R)}A_\crys^\bx(R)}_{\sub{Lemma } \ref{lemma_phi_to_filtr_framed}}\ar[d]^{\rotatebox{-90}{$\simeq$}}_{-\otimes_{A_\inf^\bx(R)}A_\inf(\res R)}&
\Rep_\Gamma^\mu(A_\crys^\bx(R),\op{SatFil},\phi)\ar[d]^{-\otimes_{A_\crys^\bx(R)}A_\crys(\res R)}_{\sub{Lemma }\ref{lemma_saturation_of_product}}\\
\Rep_\Delta^\mu(A_\inf(\res R),\phi)\ar[r]_{-\otimes_{A_\inf(\res R)}A_\crys(\res R)}^{\sub{Lemma }\ref{lemma_phi_to_filtr}}&
\Rep_\Delta^\mu(A_\crys(\res R),\op{SatFil},\phi)}
\end{equation}

First we convert a Frobenius structure over $A_\inf$ into a filtration:

\begin{lemma}\label{lemma_phi_to_filtr}
Given $M\in \Rep_\Delta^\mu(A_\inf(\res R),\phi)$, equip $M\otimes_{A_\inf(\res R)}A_\crys(\res R)$ with the saturated filtration associated (via Lemma \ref{lemma_saturated}) to the $A_\crys(\res R)^+$-module $\phi^{-1}_M(M) \otimes_{A_\inf(\res R)}A_\crys(\res R)^+$, where $\phi_M:M[\tfrac1\xi]\isoto M[\tfrac1{\tilde\xi}]$ is the Frobenius structure map of $M$. Then this filtration is $p$-adically closed, multiplicative, and satisfies conditions (F4) and (F5) of Definition  \ref{definition_filtration_2}, i.e., $M\otimes_{A_\inf(\res R)}A_\crys(\res R)\in \Rep_\Delta^\mu(A_\crys(\res R),\op{SatFil},\phi)$.
\end{lemma}
\begin{proof}
The $A_\inf(\res R)$-module $M^+:=\phi_M^{-1}(M)$ is finite projective since the Frobenius is an automorphism of $A_\inf(\res R)$, whence $M^+\otimes_{A_\inf(\res R)}A_\crys(\res R)^+$ is a submodule of $M\otimes_{A_\inf(\res R)}A_\crys(\res R)[\tfrac1\xi]$ and so Lemma \ref{lemma_saturated} may indeed be applied to produce a filtration on $M\otimes_{A_\inf(\res R)}A_\crys(\res R)$. This filtration is multiplicative and saturated by Lemma \ref{lemma_saturated}, $p$-adically closed by Lemma \ref{lemma_saturated_implies_complete}, and satisfies (F5) since $M^+$ is finite projective and generates $M[\tfrac1\xi]$ as an $A_\sub{inf}(\res R)[\tfrac1\xi]$-module (as we may check this after applying the isomorphism $\phi_M$).
\end{proof}

The fact that the analogous process works in the framed context is more subtle:

\begin{lemma}\label{lemma_phi_to_filtr_framed}
Given $N\in\Rep^\mu_\Gamma(A_\inf^\bx(R),\phi)$, equip $N\otimes_{A_\inf^\bx( R)}A_\crys^\bx( R)$ with the saturated filtration associated (via Lemma \ref{lemma_saturated}) to the $A_\crys^\bx(R)^+$-module $\phi^{-1}_N(N) \otimes_{A_\inf^\bx( R)}A_\crys^\bx( R)^+$, where $\phi_N:N[\tfrac1\xi]\to N[\tfrac1{\tilde\xi}]$ is the Frobenius structure map of $N$. Then this filtration satisfies conditions (F1)--(F5) of Definitions \ref{definition_filtration_1} and \ref{definition_filtration_2}, i.e., $N\otimes_{A_\inf^\bx( R)}A_\crys^\bx( R)\in \Rep_\Gamma^\mu(A_\crys^\bx( R),\op{SatFil},\phi)$.
\end{lemma}
\begin{proof}
Write $N^+:=\phi_N^{-1}(N)$ for the $A_\inf^\bx(R)$-submodule of $N[\tfrac1\xi]$ defining the filtration. We claim that $N^+$ is finite projective and that $\phi_N$ restricts to an isomorphism $N^+\otimes_{A_\inf^\bx(R),\phi}A_\inf^\bx(R)\xto{\phi_N}N$; these will both follow from our q-Simpson correspondence as follows. Bearing in mind that the map $W$ (as in the beginning of \S\ref{ss_q_Higgs}) is an isomorphism in the current context, and ignoring $\Gamma$-actions, Corollary \ref{corollary_descent_of_reps_to_phi_twist} implies that there exists a finite projective $A_\inf^\bx(R)$-module $H$, equipped with a Frobenius structure $\phi_H:H\otimes_{A_\inf^\bx(R),\phi}A_\inf^\bx(R)[\tfrac1\xi]\isoto H[\tfrac1\xi]$ and an identification $H\otimes_{A_\inf^\bx(R),\phi}A_\inf^\bx(R)=N$ compatible with Frobenii. So the $A_\inf^\bx(R)$-submodule $\phi_H^{-1}(H)$ of $H\otimes_{A_\inf^\bx(R),\phi}A_\inf^\bx(R)[\tfrac1\xi]=N[\tfrac1\xi]$ is contained in $N^+$, and the isomorphism $\phi_N:N\otimes_{A_\inf^\bx(R),\phi}A_\inf^\bx(R)[\tfrac1{\tilde\xi}]\isoto N[\tfrac1{\tilde\xi}]$ carries the submodule $\phi_H^{-1}(H)\otimes_{A_\inf^\bx(R),\phi}A_\inf^\bx(R)$ isomorphically to $N$, and carries the larger submodule $N^+\otimes_{A_\inf^\bx(R),\phi}A_\inf^\bx(R)$ into $N$. So necessarily $N^+=\phi_H^{-1}(H)$, which proves our claims.


The claim also implies that $N^+$ generates the $A_\inf^\bx(R)[\tfrac1\xi]$-module $N[\tfrac1\xi]$, as we may check this after base changing along the finite flat Frobenius endomorphism of $A_\inf^\bx(R)$.

So the saturated filtration associated to $N^+\otimes_{A_\inf^\bx( R)}A_\crys^\bx( R)^+$ satisfies conditions (F2) and (F4) by Lemma \ref{lemma_saturated}, condition (F5) by the previous paragraphs, and therefore condition (F1) by Lemma \ref{lemma_saturated_implies_complete}.

Finally, note that $N^+$ is stable under the $\Gamma$-action and that $\Gamma$ acts as the identity on $N^+/\phi^{-1}(\mu)$ (either using $N^+=H$ and Corollary \ref{corollary_descent_of_reps_to_phi_twist}, or just because it acts as the identity on the target of the isomorphism $\phi_N:N^+/\phi^{-1}(\mu)\otimes_{A_\inf^\bx(R)/\phi^{-1}(\mu),\phi}A_\inf^\bx(R)/\mu\isoto N/\mu$). Using this triviality of the $\Gamma$-action on $N^+$, the Griffiths transversality condition (F3)(a) easily follows.
%
\end{proof}

We define the vertical base change functor $-\otimes_{A_\crys^\bx(R)}A_\crys(\res R)$ in the diagram (\ref{diagram_Ainfcrys_frame_nonframe_Rep_filtered}) by equipped $N\otimes_{A_\crys^\bx(R)}A_\crys(\res R)$ with the saturation of the $p$-adic closure of the product filtration. By applying the following lemma to $A_\crys^\bx(R)\to A_\crys(\res R)$ we see that this really does produce a well-defined functor between the categories indicated in (\ref{diagram_Ainfcrys_frame_nonframe_Rep_filtered}).

\begin{lemma}\label{lemma_saturation_of_product}
Let $A\to A'$ be a map of exhaustive filtered rings, and $f\in \Fil^1A$ a non-zero-divisor whose image (still denoted by $f$) in $A'$ is also a non-zero-divisor; assume that the filtration on $A'$ is saturated and $p$-adically closed. Let $N$ be an $A$-module equipped with a multiplicative filtration satisfying condition (F5), and equip $N\otimes_AA'$ with the $p$-adic closure of the product filtration. Then
\begin{enumerate}
\item the canonical map $N^+\otimes_{A^+}A'^+\to (N\otimes_AA')^+$ is an isomorphism;
\item the saturation of the filtration on $N\otimes_AA'$ is $p$-adically closed, multiplicative, and satisfies condition (F5).
\end{enumerate}
\end{lemma}
\begin{proof}
Condition (F5) clearly follows from part (i) since $N^+$ is assumed to be a finitely generated projective $A^+$-module; but then $p$-adically closedness of the saturation will follow from Lemma \ref{lemma_saturated_implies_complete}. So it remains only need to prove part (i).

Write $M:=N\otimes_AA'$ for simplicity; let $\op{Fil}_\sub{pre}^*M$ denote the product filtration on $M$, and $\op{Fil}^*M$ its $p$-adic closure (which is our default, not necessarily saturated filtration on any base change); let $\op{Fil}_\sub{sat-pre}^*M$ and $\op{Fil}_\sub{sat}^*M$ be their respective saturations. So for each $r\in\bb Z$ we have the following obvious inclusions:
\[\xymatrix@=5mm{
\Fil_\sub{pre}^rM\ar@{}[r]|-\subseteq \ar@{}[d]|-{\rotatebox{-90}{$\subseteq$}} & \Fil^r_\sub{sat-pre}M\ar@{}[d]|-{\rotatebox{-90}{$\subseteq$}} \\
\Fil^rM\ar@{}[r]|-\subseteq & \Fil^r_{\sub{sat}}M
}\]
For $?\in\{\text{pre},\emptyset,\,\text{sat-pre},\text{sat}\}$ let $M_?^+:=\bigcup_{r\in\bb Z}\xi^{-r}\Fil_?^rM$, noting that $M^+=M^+_\sub{sat}$ and $M_\sub{pre}^+=M_{\sub{sat-pre}}^+$ by the definition of saturation. Using that $N^+$ is a finite projective $A^+$-module, the identification $N\otimes_{A}A'[\tfrac1f]=M[\tfrac1f]$ is easily seen to restrict to an identification $N^+\otimes_{A^+}A'^+=M^+_\sub{pre}$, whence $M^+_\sub{pre}=M^+_\sub{sat-pre}$ is a finite projective $A'^+$-module. Therefore Lemma \ref{lemma_saturated_implies_complete} implies that the filtration $\Fil^*_\sub{sat-pre}$ is $p$-adically closed, whence $\Fil^rM\subseteq\Fil^r_\sub{sat-pre}M$ for each $r\in\bb Z$ and so $M^+\subseteq M_\sub{sat-pre}^+$. The converse inclusion being clear, we deduce that $M^+=M_\sub{sat-pre}^+$; since we have already shown that $N^+\otimes_{A^+}A'^+=M^+_\sub{pre}=M^+_\sub{sat-pre}$, this completes the proof.
\end{proof}


Finally, we will need the compatibility of the previous constructions:

\begin{lemma}
Diagram (\ref{diagram_Ainfcrys_frame_nonframe_Rep_filtered}) commutes.
\end{lemma}
\begin{proof}
Of course the commutativity of the diagram is clear if we forget about filtration; we must show that the two possible compositions for constructing a filtration are the same. Let $N\in\Rep_\Gamma^\mu(A_\inf^\bx(R),\phi)$ and let $M$ be its image under the left vertical functor. Let $N^+$ be the $A_\inf^\bx(R)$-submodule of $N[\tfrac1\xi]$ appearing in the construction of Lemma \ref{lemma_phi_to_filtr_framed}, and similarly $M^+$ the $A_\inf(\res R)$-submodule of $M[\tfrac1\xi]$ appear in Lemma \ref{lemma_phi_to_filtr}. 

By construction the filtration on $N\otimes_{A_\inf^\bx(R)}A_\crys^\bx(R)$ is the saturated filtration induced by $N^+\otimes_{A_\inf^\bx(R)}A_\crys^\bx(R)^+$; then applying the right vertical functor, it follows from the proof of Lemma \ref{lemma_saturation_of_product} that the composite functor $\xymatrix@=4pt{\ar[r] & \ar[d] \\ &}$ equips $N\otimes_{A_\inf^\bx(R)}A_\crys(\res R)$ with the saturated filtration induced by the $A_\crys(\res R)^+$-module $N^+\otimes_{A_\inf^\bx(R)}A_\crys(\res R)^+$. So we must show that $N^+\otimes_{A_\inf^\bx(R)}A_\crys(\res R)^+=M^+\otimes_{A_\inf(\res R)}A_\crys(\res R)^+$ as submodules of $N\otimes_{A_\inf^\bx(R)}A_\crys(\res R)[\tfrac1\xi]=M\otimes_{A_\inf(\res R)}A_\crys(\res R)[\tfrac1\xi]$.

Base changing the isomorphism $N^+\otimes_{A_\inf^\bx(R),\phi}A_\inf^\bx(R)\xto{\phi_N}N$ (see the proof of Lemma \ref{lemma_phi_to_filtr_framed}) along the maps $A_\inf^\bx(R)\to A_\inf(\res R)\to A_\inf(\res R)[\tfrac1\xi]$ shows that the isomorphism $M\otimes_{A_\inf(\res R),\phi}A_\inf(\res R)[\tfrac1\xi]\xto{\phi_M}M[\tfrac1\xi]$ restricts to an isomorphism $(N^+\otimes_{A_\inf^\bx(R)}A_\inf(\res R))\otimes_{A_\inf(\res R),\phi}A_\inf(\res R)\xto{\phi_M}M$. This being the characterising property of the submodule $M^+$ we have shown than $N^+\otimes_{A_\inf^\bx(R)}A_\inf(\res R)=M^+$, which is stronger than the desired identification.
\end{proof}

\subsection{Filtered F-crystals and Galois representations}\label{ss_OA_admissibility}
We are now prepared to study filtered $F$-crystals in terms of small generalised representations over $A_\crys(\res R)$, independently of any framing. In particular, we will show in Theorem \ref{theorem_crystal_genrep} that at most one F-crystal in $\CR(R/A_\crys)$ is ``associated'' to any generalised representation in $\Rep_\Delta^\mu(A_\crys(\res R))$; the proof requires the development of admissibility with respect to a suitable period ring with connection. In the presence of filtrations, when the generalised representation comes from $\Rep_\Delta^\mu(A_\inf(\res R),\phi)$ we show moreover in Theorem \ref{theorem_M_et_up_to_isogeny} that the associated geometric Galois representation may be recovered up to isogeny from the filtered F-crystal.

We begin by formulating the notion of associatedness. Given an object $\cal F=(\cal F_n)$ of $\CR(R/A_\crys)$, the inverse limit of the evaluations of $\cal F_n$ on $\Spec(\res R/p^n)\to \Spec(A_\crys(\res R)/p^n)$ gives (similarly to $\cal F(A_\crys^\bx(R))$ from \S\ref{ss_q_over_A_crys}) a finite projective $A_\crys(\res R)$-module equipped with a semilinear action of $\Delta$. There is a canonical isomorphism $\cal F(A_\crys^\bx(R))\otimes_{A_\crys^\bx(R)}A_\crys(\res R)\isoto \cal F(A_\crys(\res R))$, whence Proposition \ref{proposition_crystal_framedrep} implies that $\cal F(A_\crys(\res R))$ belongs to $\Rep_\Gamma^\mu(A_\crys(\res R))$ and we obtain the commutative diagram 
\begin{equation}\label{diagram_Acrys_ev_crystal}
\xymatrix@C=50pt{
\CR(R/A_\crys)\ar[r]^{\ev_{A_\crys^\bx(R)}}\ar[dr]_{\ev_{A_\crys(\res R)}} &
\Rep_\Gamma^\mu(A_\crys^\bx(R))\ar[d]^{-\otimes_{A_\crys^\bx(R)}A_\crys(\res R)}\\
&\Rep_\Delta^\mu(A_\crys(\res R))
}
\end{equation}
where $\ev_{A_\crys(\res R)}$ is defined by $\cal F\mapsto \cal F(A_\crys(\res R))$.

We have obvious variants with $\exists \varphi$, $\varphi$, or $\Fil$ added, where $\Rep_\Delta^\mu(A_\crys(\res R),\Fil)$ is defined in the same way as $\Rep_\Gamma^\mu(A_\crys^\bx(R),\Fil)$ in Definition \ref{definition_filtration_1} but with the condition (iii) omitted; the functor $-\otimes_{A_\crys^\bx(R)}A_\crys(\res R)$ in the filtered case is defined by setting $\Fil^r(N\otimes_{A_\crys^\bx(R)}A_\crys(\res R))$, for $r\in\bb Z$ and $N\in \Rep_\Gamma^\mu(A_\crys^\bx(R))$, to be the $p$-adic closure of the product filtration.

\begin{remark}[Variant with saturated filtrations]\label{remark_associated1_filtered}
We explain a variant of (\ref{diagram_Acrys_ev_crystal}) with SatFil. Firstly, define $\CR(R/A_\crys,\op{SatFil})$ to be the full subcategory of $\CR(R/A_\crys,\op{Fil})$ of objects whose image under the equivalence (\ref{equation_Crystal_Mic_filtered}) belongs to $\MIC_\sub{conv}(A_\crys^\bx(R),\op{SatFil})$; note that this does not depend on the chosen framing since (\ref{equation_Crystal_Mic_filtered}) and our filtration conditions (F1)--(F5) only depend on $A_\crys^\bx(R)$ as a formally smooth lifting of $R$ over $A_\crys$, which is unique up to isomorphism (see also just after (\ref{equation_crystal_micprime}), where $\MIC_\sub{conv}(A_\crys^\bx(R),\op{SatFil})$ is defined without any reference to a framing).

The equivalence $\cal L^{-1}$ of Proposition \ref{proposition_crystal_framedrep_filtered} preserves the underlying module and filtration, so clearly restricts to an equivalence $\cal L^{-1}:\MIC_\sub{conv}(A_\crys^\bx(R),\op{SatFil})\quis \Rep_{\Gamma,\sub{conv}}^\mu(A_\crys^\bx(R),\op{SatFil})$. Composing with the equivalence of the previous paragraph (and forgetting any quasinilpotence condition), this defines \[\op{ev}_{A_\crys^\bx(R)}:\CR(R/A_\crys,\op{SatFil})\To \Rep_{\Gamma}^\mu(A_\crys^\bx(R),\op{SatFil}).\]

Next, given $N\in \Rep_{\Gamma}^\mu(A_\crys^\bx(R),\op{SatFil})$, we equip $M:=N\otimes_{A_\crys^\bx(R)}A_\crys(\res R)$ with the saturation $\op{Fil}_\sub{sat}^*M$ of the $p$-adic closure of the product filtration; Lemma \ref{lemma_saturation_of_product} shows that $(M,\Fil^*_\sub{sat}M)\in \Rep_\Delta^\mu(A_\crys(\res R),\op{SatFil})$, and thus we define the vertical base change functor in the diagram
\begin{equation}\label{diagram_Acrys_ev_crystal_filtered}
\xymatrix@C=50pt{
\CR(R/A_\crys,\op{SatFil})\ar[r]^{\ev_{A_\crys^\bx(R)}}\ar[dr]_{\ev_{A_\crys(\res R)}} &
\Rep_\Gamma^\mu(A_\crys^\bx(R),\op{SatFil})\ar[d]^{-\otimes_{A_\crys^\bx(R)}A_\crys(\res R)}\\
&\Rep_\Delta^\mu(A_\crys(\res R),\op{SatFil})
}
\end{equation}
The diagonal arrow is defined by saturating the filtration on $\cal F(A_\crys(\res R))=\cal F(A_\crys^\bx(R))\otimes_{A_\crys^\bx(R))}A_\crys(\res R)$ induced by the filtration on $\cal F$, so that the diagram commutes and refines (\ref{diagram_Acrys_ev_crystal}).
\end{remark}

We may now define associatedness:

\begin{definition}\label{definition_associated}
Analogously to \cite[\S Vf]{Faltings1989}, we say that a crystal $\cal F$ in $\CR(R/A_\crys)$ is {\em associated} to an object $M$ of $\Rep_\Delta^\mu(A_\inf(\res R))$  when there exists an isomorphism $\ev_{A_\crys(\res R)}(\cal F)\cong M\otimes_{A_\inf(\res R)}A_\crys(\res R)$ in $\Rep^\mu_\Delta(A_\crys(\res R))$. Similarly with Frobenius structures.

In the presence of filtrations, we say that a filtered crystal with Frobenius $\cal F\in \CR(R/A_\crys,\op{SatFil}, \varphi)$ is associated to an object $M$ of $\Rep_\Delta^\mu(A_\inf(\res R),\phi)$ if there exists an isomorphism $\ev_{A_\crys(\res R)}(\cal F)\cong M\otimes_{A_\inf(\res R)}A_\crys(\res R)$ in $\Rep_\Delta^\mu(A_\crys(\res R),\op{SatFil},\phi)$. Here the filtration on $\ev_{A_\crys(\res R)}(\cal F)$ is defined as in Remark \ref{remark_associated1_filtered} and the filtration on $M\otimes_{A_\inf(\res R)}A_\crys(\res R)$ is defined as in Lemma \ref{lemma_phi_to_filtr}.

\end{definition}

The main goal of this section is the following analogue of \cite[Lem.~5.5(iii)]{Faltings1989}, implying that such an $\cal F$ is uniquely (up to canonical isomorphism) determined by $M$ and the chosen isomorphism $\ev_{A_\crys(\res R)}(\cal F)\cong M\otimes_{A_\inf(\res R)}A_\crys(\res R)$:

\begin{theorem}\label{theorem_crystal_genrep}
The functor $\ev_{A_\crys(\res R)}\colon \CR(R/A_\crys)\to \Rep_\Delta^\mu(A_\crys(\res R))$ is fully faithful; similarly for its variant with $\varphi$ or $\mathrm{SatFil}$ added.
\end{theorem}

We may more easily establish the following result about the essential image of the functor, affirming that all relative Breuil--Kisin--Fargues modules are ``crystalline''. We mention that we must state this essential image claim using $\exists\phi$ variants of categories, as we know of no other way to impose a quasi-nilpotence condition on generalised representations in $\Rep_\Delta^\mu(A_\crys(\res R))$ independently of the choice of framing.

\begin{theorem}\label{theorem_Admissibility}
The essential image of the functor $\ev_{A_\crys(\res R)}: \op{CR}(R/A_\crys,\exists\phi)\to\Rep_\Delta^\mu(A_\crys(\res R),\exists\phi)$ contains $M\otimes_{A_\inf(\res R)}A_\crys(\res R)$ for all $M\in \Rep_\Delta^\mu(A_\inf(\res R),\exists \varphi)$; similarly for its variant with $\exists\phi$ replaced by $\phi$ or by $\op{SatFil},\phi$ (in the latter case equipping $M\otimes_{A_\inf(\res R)}A_\crys(\res R)$ with the filtration of Lemma~\ref{lemma_phi_to_filtr}).
\end{theorem}
\begin{proof}
Extension of scalars give the following commutative diagram 
\begin{equation}\label{diagram_Ainfcrys_frame_nonframe_Rep}
\xymatrix{
\Rep_\Gamma^\mu(A_\inf^\bx(R))\ar[r]\ar[d]_{\sim}&
\Rep_\Gamma^\mu(A_\crys^\bx(R))\ar[d]\\
\Rep_\Delta^\mu(A_\inf(\res R))\ar[r]&
\Rep_\Delta^\mu(A_\crys(\res R)),}
\end{equation}
where the left vertical equivalence is Remark \ref{remark_Delta} and Theorem~\ref{theorem_descent_to_framed}; similarly its variants with $\exists \varphi$ or $\varphi$ added. With $\exists \varphi$ or $\varphi$, the functor $\ev_{A_\crys^\bx(R)}$  is an equivalence of categories by Remark \ref{remark_F_structures_Acrys} and (\ref{diagram_Acrys_ev_crystal}), whence the essential image claim follows at once.

It remain to treat the variant for $\op{SatFil},\phi$. In that case we replace (\ref{diagram_Ainfcrys_frame_nonframe_Rep}) with (\ref{diagram_Ainfcrys_frame_nonframe_Rep_filtered}) and argue in the same way: indeed, just as in the previous paragraph, the top horizontal arrow of (\ref{diagram_Acrys_ev_crystal_filtered}) is an equivalence if we add Frobenius structures to ensure quasi-nilpotence, and so the essential image claim follows in the same way.
\end{proof}

The proof of Theorem \ref{theorem_crystal_genrep} will be based on a study of admissibility with respect to a suitable period ring with connection. Let $\cal R$ be a $p$-adic formally smooth $A_\crys$-algebra lifting $R$, i.e., $\cal R\otimes_{A_\crys}\roi=R$; we also choose and fix a lifting $\varphi\colon \cal R\to \cal R$ of the absolute Frobenius of $R/p$ compatible with the Frobenius $\varphi$ of $A_\crys$. (For example, we could take $\cal R=A_\crys^\bx(R)$ equipped with its usual Frobenius, but we wish to stress that the framing is not required for what follows.) We define a filtration $\Fil^r\cal R$ to be the $p$-adic closure of $\Fil^rA_\crys\cdot \cal R$, for $r\in\bb Z$; this is a separated filtration on $\cal R$ whose graded pieces are $p$-torsion-free.\footnote{{\em Proof}: The flatness of $A_\crys/p^n\to\cal R/p^n$ implies that $\Fil^rA_\crys/p^n\otimes_{A_\crys/p^n}\cal R/p^n\isoto (\Fil^r A_\crys/p^n)\cal R/p^n$, whence $\op{gr}^r\cal R$ is isomorphic to $\op{gr}^rA_\crys\hat\otimes_\roi\cal R=\op{gr}^rA_\crys\otimes_\roi\cal R$, which is indeed $p$-torsion-free. To prove the separatedness it is then enough to check that $\bigcap_{r\ge0}\Fil^r\cal R\subseteq p\cal R$, or even that $\bigcap_{r\ge0}\Fil^r\cal R/p=0$ in $\cal R/p$; but $A_\crys/p$ is a split nilpotent extension of a field, whence the smooth $A_\crys/p$-algebra $\cal R/p$ is necessarily free as a module and so the desired separatedness reduces to the fact that $\bigcap_{r\ge0}\Fil^rA_\crys/p=0$ in $A_\crys/p$.} Similarly to \eqref{equqation_crystal_mic}, there is an equivalence of categories
\begin{equation}\label{equation_crystal_micprime}
\CR(R/A_\sub{crys})\quis \MIC_\sub{conv}(\cal R)
\end{equation}
and similarly its variants with $\exists\varphi$, $\varphi$ and/or $\Fil$ or $\op{SatFil}$ added. Here we define the categories $\MIC(\cal R,\varphi)$ and $\MIC(\cal R,\Fil)$ in the same way as Remark \ref{remark_Frob_q-Acrys} and Definition \ref{definition_filtration_1}, and we define $\MIC(\cal R,\op{SatFil})\subseteq\MIC(\cal R,\Fil)$ by adding the extra conditions (F4) and (F5) of Definition \ref{definition_filtration_2} (with respect to $A=\cal R$ and $f=\xi$). As at the start of Remark \ref{remark_associated1_filtered}, the full subcategory $\CR(R/A_\crys,\op{SatFil})\subseteq \CR(R/A_\crys,\op{Fil})$ is defined independently of $\cal R$ so that the filtered version of (\ref{equation_crystal_micprime}) restricts to an equivalence $\op{CR}(\cal R,\op{SatFil})\quis \MIC_\sub{conv}(\cal R,\op{SatFil})$.

The key to Theorem \ref{theorem_crystal_genrep} will be to show that the composition \begin{equation}\MIC_\sub{conv}(\cal R)\stackrel{\sub{(\ref{equation_crystal_micprime})}}\simeq\CR(R/A_\crys)
 \xto{\ev_{A_\crys(\res R)}} \Rep_\Delta^\mu(A_\crys(\res R))\label{eqn_composition}\end{equation}
is fully faithful. This will requires an alternative description of this composition in terms of the following period ring with flat connection from relative $p$-adic Hodge theory (c.f., \cite{Brinon2008}):

\begin{definition}\label{definition_OAcrys}
\begin{enumerate}
\item Let $\roiA_\crys(\cal R)$ be the $p$-adic completion of the divided power envelope of \[A_\crys(\res R)\otimes_{A_\crys}\cal R\To \res R/p\] compatible with the divided power structure on $pA_\crys+\Fil^1 A_\crys$, and write $p_1: A_\crys(\res R)\to\roiA_\crys(\cal R)$ and $p_2:\cal R\to\roiA_\crys(\cal R)$ for the resulting structure maps of $A_\crys$-algebras.
\item (Connection) By applying the construction in the last paragraph of \cite[IV \S1.3]{Berthelot1980} to $\Spec(\ol R/p)\hookrightarrow \Spec(A_{\crys}(\ol R)\otimes_{A_{\crys}}\cal R/p^n) \to \Spec(A_{\crys}(\ol R)/p^n)$, we obtain an $A_{\crys}(\ol R)$-linear derivation \[\nabla_{\cal OA}\colon \cal OA_{\crys}(\cal R) \to \cal OA_{\crys}(\cal R)\otimes_{\cal R}\Omega^1_{\cal R/A_{\crys}}\] compatible with the universal continuous $A_{\crys}$-derivation $d\colon \cal R\to \Omega^1_{\cal R/A_{\crys}}$. This is also a flat connection on the right $\cal R$-module $\cal OA_{\crys}(\cal R)$.

\item ($\Delta$-action) The $\Delta$-action on $\res R$ extends to an action on $\roiA_\crys(\cal R)$ which is left $A_\crys(\res R)$-semilinear, right $\cal R$-linear, and commutes with $\nabla_{\!\roiA}$.
\item (Frobenius) The endomorphism $\varphi$ on $A_\crys(\res R)$ and $\cal R$ induce an endomorphism on $\roiA_\crys(\cal R)$, which is denoted by $\varphi_{\!\roiA}$. It is $\Delta$-equivariant and compatible with $\nabla_{\!\roiA}$.
\item  (Filtration) The ring $\roiA_\crys(\cal R)$ is isomorphic to the inverse limit of the PD-envelopes of $A_\crys(\res R)\otimes_{A_\crys}\cal R/p^n\to R/p^n$ compatible with the PD-structure on $pA_{\crys,n}+\Fil^1A_{\crys,n}$. We define a decreasing filtration $\Fil^r\roiA_\crys(\cal R)$ on $\roiA_\crys(\cal R)$ to be the inverse limit of the $r$th divided power of the PD-ideal of the PD-envelope for each $n\geq 1$. It is $\Delta$-stable and satisfies Griffiths transversality $\nabla_{\!\roiA}(\Fil^r\roiA_\crys(\cal R))\subseteq \Fil^{r-1}\roiA_\crys(\cal R)\otimes_{\cal R}\Omega^1_{\cal R/A_\crys}$ for $r\in\bb Z$.
\end{enumerate}
\end{definition}

Let $D$ be an object of $\MIC_\sub{conv}(\cal R)$. The product of the flat connections on $D$ and $\roiA_\crys(\cal R)$ defines a flat connection on the right $\cal R$-module $D\otimes_{\cal R,p_2}\roiA_\crys(\cal R)$, which is compatible with $\Delta$-actions. So we may define a $A_\crys(\res R)$-module $\bb M_\crys(D)$ equipped with a semilinear $\Delta$-action by \[\bb M_\crys(D)=(D\otimes_{\cal R,p_2}\roiA_\crys(\cal R))^{\nabla=0}.\] The inclusion map $\bb M_\crys(D)\to D\otimes_{\cal R,p_2}\roiA_\crys(\cal R)$ induces the counit map of $\roiA_\crys(\cal R)$-modules \[\bb M_\crys(D)\otimes_{A_\crys(\res R),p_1}\roiA_\crys(\cal R)\to D\otimes_{\cal R,p_2}\roiA_\crys(\cal R)\] compatible with $\Delta$-action (diagonally on the domain; $\id\otimes \Delta$ on the codomain), and connections ($\id \otimes \nabla_{\!\roiA}$ on the domain; tensor product connection on the codomain). We say that $D$ is {\it admissible} if this counit map is an isomorphism.

\begin{remark}[Variants of $\bb M_\sub{crys}(D)$ with Frobenius and filtration]\label{remark_Mcrys_with_phi_and_Fil}
If $D$ is equipped with a Frobenius structure $\varphi_D$, then $\varphi_D\otimes\varphi_{\!\roiA}$ is an endomorphism of $D\otimes_{\cal R, p_2}\roiA_\crys(\cal R)[\frac{1}{p}]$ semilinear with respect to $\varphi_{\!\roiA}$. This restricts to an endomorphism of $\bb M_\crys(D)[\frac{1}{p}]$, which is clearly semilinear with respect to the Frobenius on $A_\crys(\res R)$; that fact that it defines a Frobenius structure on the generalised representation $\bb M_\crys(D)$ will follow from the Frobenius variant of Proposition \ref{proposition_MIC_Rbarrep}(ii).

If $D$ is equipped with a filtration, i.e., belongs to $\MIC(\cal R,\Fil)$, then we define a $\Delta$-stable decreasing filtration on $D\otimes_{\cal R,p_2}\roiA_\crys(\cal R)$ as the $p$-adic closure of the product filtration, and then we equip $\bb M_\sub{crys}(D)$ with the induced submodule filtration. The counit map is compatible with the $p$-adic closures of the product filtrations on each side, and we say that $D$ is admissible in this context if the counit is a filtered isomorphism.

In the case of SatFil, we instead equip $D\otimes_{\cal R,p_2}\roiA_\crys(\cal R)$ with the saturation of the $p$-adic closure of the product filtration.
\end{remark}

\begin{proposition}\label{proposition_MIC_Rbarrep}
\begin{enumerate}
\item For any $D\in \MIC_\sub{conv}(\cal R)$, the generalised representation $\bb M_\crys(D)$ belongs to $\Rep_\Delta^\mu(A_\crys(\res R))$.
\item The composition (\ref{eqn_composition}) is naturally isomorphic to $\bb M_\crys$.
\item Every object of $\MIC_\sub{conv}(\cal R)$ is admissible.
\end{enumerate}

\noindent The obvious analogues of (i)--(iii) with $\phi$ or/and Fil also hold.
\end{proposition}
\begin{proof}
For each $n\ge1$ we write $\res R_n=\res R/p^n$, $A_{\crys,n}(\res R)=A_\crys(\res R)/p^n$, 
$\cal R_n=\cal R/p^n$, and $\roiA_{\crys,n}(\cal R)=\roiA_{\crys}(\cal R)/p^n$. Let $\CRYS(\res R_n/A_{\crys,n}(\res R))$ be the big crystalline site of $\Spec(\res R_n)$ over the PD-pair ($A_{\crys,n}(\res R),pA_{\crys,n}(\res R)+\Fil^1A_{\crys, n}(\res R)$). 
For $D$ an object of $\MIC_\sub{conv}(\cal R)$ set $D_n=D/p^n$ endowed with the induced flat connection, and let $\cal F_n$ be the object of $\CR(R_n/A_{\crys,n})$ associated to $D_n$; let $\res{\cal F}_n$ be the pullback of $\cal F_n$ to $\CRYS(\res R_n/A_{\crys,n}(\res R))$, along the morphism $R_n\to\res R_n$ lying over the PD-morphism $A_{\crys,n}\to A_{\crys,n}(\res R)$.

Let $\res{\cal F}_n(\roiA_\crys(\cal R))$ be the evaluation of the crystal $\res{\cal F}_n$ on $\Spec(\roiA_{\crys,n}(\cal R))$, which we view as the PD-envelope of the closed immersion $\Spec(\res R_n)\into \Spec(A_{\crys,n}(\res R)\otimes_{A_{\crys,n}}\cal R_n)$ compatible with the PD-structure on the base $A_{\crys,n}(\res R)$. Since $A_{\crys,n}(\res R)\otimes_{A_{\crys,n}}\cal R_n$ is a smooth $A_{\crys,n}(\res R)$-algebra, $\res{\cal F}_n(\roiA_\crys(\cal R))$ is equipped with a natural flat connection over $A_{\crys,n}(\res R)$. By comparing the construction of the connections on $\cal F_n(\cal R)=D_n$ and $\res{\cal F}_n(\roiA_\crys(\cal R))$ via stratifications, we see that the canonical isomorphism $D_n\otimes_{\cal R_n}\roiA_{\crys,n}(\cal R)\isoto \res{\cal F}_n(\roiA_\crys(\res R))$ is compatible with connections, where the domain is equipped with the product connection. (Moreover, this is a filtered isomorphism if $D$ is an object of $\MIC_\sub{conv}(\cal R,\Fil)$.) Since $\cal F_n(A_\crys(\res R))$ is isomorphic to the global sections of $\res{\cal F}_n$, the description of the global sections of a crystal in terms of the corresponding module with connection gives us isomorphisms of $A_\crys(\res R)$-modules
\[\alpha_n \colon \cal F_n(A_\crys(\res R))\isoto \res{\cal F}_n(\roiA_\crys(\cal R))^{\nabla=0}\cong (D_n\otimes_{\cal R_n}\roiA_{\crys,n}(\cal R))^{\nabla=0}.\]
By looking at the behaviour of the isomorphisms appearing in $\al_n$ under the pullback by the action of an element of $\Delta$ on $\res R_n$ and $A_{\crys,n}(\res R)$, we see that they are $\Delta$-equivariant. Taking the limit over $n$ proves that $\op{ev}_{A_\sub{crys}(\res R)}(\cal F)$ is isomorphic as a generalised representation to $\bb M_\sub{crys}(D)$, thereby proving both (ii) and (i) since we already know that the former lies in $\Rep_\Delta^\mu(A_\sub{crys}(\res R))$.

We may also regard $\cal F_n(A_\crys(\res R))$ as the sections of $\res{\cal F}_n$ on $\Spec(A_{\crys,n}(\res R))$, and then the first isomorphism is induced by pullback along the morphism $p_1:\Spec(\roiA_{\crys,n}(\cal R))\to \Spec(A_{\crys,n}(\res R))$ in $\CRYS(\res R_n/A_{\crys,n}(\res R))$. Since $\res{\cal F}_n$ is a quasi-coherent crystal, this implies that $\alpha_n$ induces an isomorphism 
\[\cal \beta_n\colon \cal F_n(A_{\crys}(\res R))\otimes_{A_{\crys,n}(\res R),p_1}\roiA_{\crys,n}(\cal R)\isoto D_n\otimes_{\cal R_n}\roiA_{\crys,n}(\cal R),\] thereby proving (iii).

It remains to discuss Frobenius structures and filtrations. Firstly, if $D$ is an object of $\MIC(\cal R,\varphi)$ then, replacing $R_n$ and $\res R_n$ by 
$R_1$ and $\res R_1$, and $\cal F_n$ and $\res{\cal F}_n$ by their 
pullbacks to $\Spec(R_1)$ and $\Spec(\res R_1)$, we see that the isomorphisms $\alpha_n$ and $\beta_n$ are compatible with the Frobenius structures on $(\cal F_n(A_\crys(\res R)))_n$ and $(D_n)_n$.) Secondly, suppose that $D$ is an object of $\MIC_{\sub{conv}}(\cal R,\Fil)$. Then $\beta_n$ is a filtered isomorphism with respect to the product filtrations, and so to prove all claims with filtration it remains only to show that $\alpha_n$ is a filtered isomorphism. This reduces to verifying that the filtration on $\cal F_n(A_\crys(\res R))$ is induced by the product filtration on $\cal F_n(A_{\crys}(\res R))\otimes_{A_{\crys,n}(\res R)}\roiA_{\crys,n}(\cal R)$. That is, given $a\in \cal F_n(A_\crys(\res R))$ such that its image $a\otimes 1$ in $\cal F_n(A_{\crys}(\res R))\otimes_{A_{\crys,n}(\res R)}\roiA_{\crys,n}(\cal R)$ lies in $\Fil^r$ of the latter, we must show that $a\in \Fil^r\cal F_n(A_\crys(\res R))$. But, as in the first paragraph of the proof of Lemma \ref{lemma_explicit_description_of_OA} below, a choice of lifts of the variables allows us to define a morphism of PD-algebras (so compatible with the filtrations) $\roiA_{\sub{crys},n}(\cal R)\to A_{\sub{crys},n}(\res R)$; this in turn induces a section, compatible with filtrations, of the map $\cal F_n(A_\crys(\res R))\to \cal F_n(A_{\crys}(\res R))\otimes_{A_{\crys,n}(\res R)}\roiA_{\crys,n}(\cal R)$ and so proves the claim.
\end{proof}

To prove Theorem \ref{theorem_crystal_genrep} we finally require the following:

\begin{proposition}\label{proposition_Delta_fixed_points_of_OA}
The structure map $p_2:\cal R\to\roiA_\crys(\cal R)^\Delta$ is a filtered isomorphism.
\end{proposition}
\begin{proof}
We postpone the proof of this result to \S\ref{subsection_proof_of_proposition}.
\end{proof}

\begin{proof}[Proof of Theorem \ref{theorem_crystal_genrep}]
Without filtrations, it suffices by Proposition \ref{proposition_MIC_Rbarrep}(ii) to prove that $\bb M_\crys$ is fully faithful. Let $D$ be an object of $\MIC_\sub{conv}(\cal R)$. By Proposition \ref{proposition_MIC_Rbarrep}(iii), the counit map \begin{equation}\bb M_\crys(D)\otimes_{A_\crys(\res R),p_1}\roiA_\crys(\cal R)
\to D\otimes_{\cal R,p_2}\roiA_\crys(\cal R)\label{eqn_counit}\end{equation} is an isomorphism. By taking $\Delta$-invariants and using Lemma \ref{proposition_Delta_fixed_points_of_OA}, we obtain an isomorphism of $\cal R$-modules with flat connection \[ (\bb M_\crys(D)\otimes_{A_\crys(\res R),p_1}\roiA_\crys(\cal R))^{\Delta}\isoto D,\] which completes the proof. Moreover, this isomorphism is compatible with Frobenii if $D$ is equipped with one.

It remains to treat filtrations. Letting $D\in \MIC_\sub{conv}(\cal R,\op{SatFil})$, we will first show that the filtration on $D$ can be recovered from $\bb M_\crys(D)$ with its saturated filtration. Let $\bb M_\crys(D)^+:=\bigcup_{r\in\bb Z}\xi^{-r}\Fil^r\bb M_\crys(D)\subseteq \bb M_\crys(D)[\tfrac1\xi]$ be the $A_\crys(\res R)^+=\bigcup_{r\in\bb Z}\xi^{-r}\Fil^rA_\crys(\res R)$-module corresponding to the filtration on $\bb M_\crys(D)^+$; following Lemma \ref{lemma_saturated} and Definition \ref{definition_saturation} we will freely use such notation for other filtrations, even if they are not necessarily saturated. Lemma \ref{lemma_saturation_of_product}(i) along $p_1:A_\crys(\res R)\to\roiA_\crys(\cal R)$ shows that the canonical map $\bb M_\crys(D)^+\otimes_{A_\crys(\res R)^+,p_1}\roiA_\crys(\cal R)^+\to (\bb M_\crys(D)\otimes_{A_\crys(\res R),p_1}\roiA_\crys(\cal R))^+$ is an isomorphism, where the target is defined using the $p$-adic closure of the product filtration on $\bb M_\crys(D)\otimes_{A_\crys(\res R),p_1}\roiA_\crys(\cal R)$ (here and below, the product filtration on $\bb M_\crys(D)\otimes_{A_\crys(\res R),p_1}\roiA_\crys(\cal R)$ is induced via $\Fil^*\bb M_\crys(D)$, not $\Fil^*_\sub{sat}\bb M_\crys(D)$).

Next, the filtered version of Proposition \ref{proposition_MIC_Rbarrep} states that the counit map (\ref{eqn_counit}) is a filtered isomorphism, having equipped both sides with the $p$-adic closures of product filtrations. This induces an isomorphism $(M_\crys(D)\otimes_{A_\crys(\res R),p_1}\roiA_\crys(\cal R))^+\isoto (D\otimes_{\cal R,p_2}\roiA_\crys(\cal R))^+$.

Finally, using Lemma \ref{lemma_saturation_of_product}(i) again, this time to base change along $p_2:\cal R\to\roiA_\crys(\cal R)$, we deduce that $D^+\otimes_{\cal R^+,p_2}\roiA_\crys(\cal R)^+\to (D\otimes_{\cal R^+,p_2}\roiA_\crys(\cal R))^+$ is an isomorphism, where the target is defined using the $p$-adic closure of the product filtration $D\otimes_\cal R\roiA_\crys(\cal R)$.

Assembling the isomorphism yields a natural isomorphism of finite projective $\roiA_\crys(\cal R)^+$-modules \[D^+\otimes_{\cal R^+,p_2}\roiA_\crys(\cal R)^+\cong \bb M_\crys(D)^+\otimes_{A_\crys(\res R)^+,p_1}\roiA_\crys(\cal R)^+,\] compatible with $\Delta$-action. Since $D^+$ is a finite projective $\cal R^+$-module and $(\roiA_\crys(\cal R)^+)^\Delta=(\cal R^+)^\Delta$ by Proposition \ref{proposition_Delta_fixed_points_of_OA}, taking $\Delta$ invariants gives a natural isomorphism of $\cal R^+$-modules \[D^+\cong (\bb M_\crys(D)^+\otimes_{A_\crys(\res R)^+,p_1}\roiA_\crys(\cal R)^+)^\Delta.\] Since the filtration on $D$ is saturated by hypothesis, it is determined by $D^+$ as in Lemma \ref{lemma_saturated}, so the previous isomorphism completes the proof of our claim that the filtration on $D$ can be recovered from $(\bb M_\crys(D),\Fil^*_\sub{sat}\bb M_\crys(D))$.

We may now complete the proof that $\ev_{A_\crys(\res R)}\colon \CR(R/A_\crys, \mathrm{SatFil})\to \Rep_\Delta^\mu(A_\crys(\res R),\op{SatFil})$ is fully faithful; recall from Remark \ref{remark_associated1_filtered} that this functor is defined by applying $\ev_{A_\crys(\res R)}\colon \CR(R/A_\crys, \mathrm{Fil})\to \Rep_\Delta^\mu(A_\crys(\res R),\op{Fil})$ and then saturating the resulting filtration. Therefore Proposition \ref{proposition_MIC_Rbarrep} for not necessarily saturated filtrations immediately implies the following analogue for saturated filtrations: the composition \[\MIC_\sub{conv}(\cal R,\op{SatFil})\simeq\CR(R/A_\crys, \op{SatFil}) \xto{\ev_{A_\crys(\res R)}} \Rep_\Delta^\mu(A_\crys(\res R), \op{SatFil})\] is naturally isomorphic to $D\mapsto (\bb M_\crys(D),\Fil^*_\sub{sat}\bb M_\sub{crys})$. But this is indeed fully faithful as desired, since we showed in the first part of the proof that it is fully faithful without filtrations, and in the second part of the proof that the filtration on $D$ may be recovered from $\Fil^*_\sub{sat}\bb M_\sub{crys}(D)$.
%
\end{proof}

We finish our study of filtered F-crystals over by showing that they are sufficient to recover, up to isogeny, the Galois representation associated to an object of $\Rep_\Delta^\mu(A_\inf(\res R),\phi)$:

\begin{theorem}\label{theorem_M_et_up_to_isogeny}
Let $M\in \Rep_\Delta^\mu(A_\inf(\res R),\phi)$, and assume $M$ is finite free as an $A_\inf(\res R)$-module. Then:
\begin{enumerate}
\item $L:=M[\tfrac1\mu]^{\phi=1}$ is a finite free $\bb Z_p$-module equipped with a continuous $\Delta$-action, and the canonical map $L\otimes_{\bb Z_p}A_\inf(\res R)[\tfrac1\mu]\to M[\tfrac1\mu]$ is an isomorphism.
\item Letting $\cal F\in \CR(R/A_\crys,\op{SatFil}, \varphi)$ be the unique F-crystal with saturated filtration associated to $M$, then the $\Delta$-representation $L[\tfrac1p]$ may be recovered from $\cal F$. More precisely, letting $M_\sub{crys}:=\op{ev}_{A_\crys(\res R)}(\cal F)\in \Rep_\Delta^\mu(A_\crys(\res R),\op{SatFil})$, the canonical map \[L[\tfrac1p]\To M_\crys[\tfrac1\mu]^{\phi=1}\cap \big(M_\crys^+\otimes_{A_\crys(\res R)^+}\Fil^0B_\crys(\res R))\] is an isomorphism compatible with $\Delta$-actions; here $M_\crys^+=\bigcup_{r\in\bb Z}\xi^{-r}\Fil^rM_\sub{crys}$ as usual, and we also set $\Fil^0B_\crys(\res R):=\bigcup_{r\in\bb Z}\mu^{-r}\Fil^rA_\crys(\res R)$ so that the intersection on the right takes place in $M_\crys[\tfrac1\mu]$.
\end{enumerate}
\end{theorem}
\begin{proof}
We have formulated the theorem in terms of $\cal F$ for the sake of naturality, but note that $\cal F$ is strictly speaking not necessary: $M_\sub{crys}$ is given by $M\otimes_{A_\inf(\res R)}A_\crys(\res R)$, which is all we will use in the proof.

Part (i) is proved by applying the proof of Proposition \ref{proposition_etale} (with $A^+=\res R$); this is where we need the finite freeness assumption, which is any case true Zariski locally on $\Spf R$.

The inverse image of $M[\tfrac1\mu]$ along the Frobenius structure $\phi_M:M[\tfrac1\mu]\to M[\tfrac1{\phi(\mu)}]$ is $M^+[\tfrac1{\phi^{-1}(\mu)}]$, where $M^+$ is the inverse image of $M$ along $\phi_M:M[\tfrac1\xi]\to M[\tfrac1{\tilde\xi}]$. The isomorphism in part (i) therefore restricts to an isomorphism $L\otimes_{\bb Z_p}A_\inf(\res R)[\tfrac1{\phi^{-1}(\mu)}]\isoto M^+[\tfrac1{\phi^{-1}(\mu)}]$. Since $\tfrac1{\phi^{-1}(\mu)}=\tfrac\xi\mu$ belongs to $\Fil^0B_\sub{crys}(\res R)$, and $M_\crys^+=M^+\otimes_{A_\inf(\res R)}A_\crys(\res R)$ by definition of the filtration on $M_\crys$ (given by Lemma \ref{lemma_phi_to_filtr}), base changing further yields an isomorphism \[L\otimes_{\bb Z_p}\Fil^0B_\sub{crys}(\res R)\isoto M_\crys^+\otimes_{A_\crys(\res R)^+}\Fil^0B_\crys(\res R)\]  compatible with Frobenii and $\Delta$-actions. The proof is then completed by the identity $A_\crys(\res R)[\tfrac1\mu]^{\phi=1}\cap \Fil^0B_\crys(\res R)=\bb Q_p$, for which we refer to \cite[Thm.~A3.26]{Tsuji1999}.
 \end{proof}

\subsection{Proof of Proposition \ref{proposition_Delta_fixed_points_of_OA}}\label{subsection_proof_of_proposition}
In this subsection we prove Proposition \ref{proposition_Delta_fixed_points_of_OA}. The proof reduces to a framed version of the proposition via the following lemmas; let $H:=\ker(\Delta\to\Gamma)$, which acts $R_\infty$-linearly on $\res R$.

\begin{lemma}\label{lemma_H-fixed_points}
The canonical map $A_\sub{crys}(R_\infty)\to A_\sub{crys}(\res R)^H$ is a filtered isomorphism.
\end{lemma}
\begin{proof}
First we recall that $R_\infty\to\res R^H$ is an isomorphism. Indeed, the almost purity theorem implies that it is almost an isomorphism, and the ``almost'' can be dropped since $R_\infty$ is $p$-torsion-free and $R_\infty/p\to\res R/p$ is injective.

By induction on $r$, using the short exact sequence \[0\To R_\infty\xto{\xi^{[r]}} A_\sub{crys}(R_\infty)/\op{Fil}^rA_\sub{crys}(R_\infty)\To A_\sub{crys}(R_\infty)/\op{Fil}^{r-1}A_\sub{crys}(R_\infty)\To 0\] and the $H$-invariants of the analogous sequence for $\res R$, we then see that $A_\sub{crys}(R_\infty)/\op{Fil}^rA_\sub{crys}(R_\infty)\to (A_\sub{crys}(\res R)/\op{Fil}^rA_\sub{crys}(\res R))^H$ is an isomorphism for all $r\ge1$ (which implies that the map $A_\crys(R_\infty)\to A_\crys(\res R)$ is strictly compatible with filtrations).

We claim that the diagram
\[\xymatrix{
A_\sub{crys}(R_\infty)\ar[r]\ar[d] & A_\sub{crys}(\res R)\ar[d]\\
\projlim_rA_\sub{crys}(R_\infty)/\op{Fil}^rA_\sub{crys}(R_\infty)\ar[r] & \projlim_rA_\sub{crys}(\res R)/\op{Fil}^rA_\sub{crys}(\res R)
}\]
is cartesian, after which the proof is completed by taking $H$-invariants in the diagram. Since all terms $A_\sub{crys}$ and $A_\sub{crys}/\Fil^r$ appearing in the diagram are $p$-adically complete and $p$-torsion-free (whence the same is also true of $\projlim_rA_\crys(R_\infty)/\Fil^rA_\crys(R_\infty)$, and similarly for $\res R$ in place of $R_\infty$), the claim reduces to checking that its reduction mod $p$, namely 
\[\xymatrix{
A_\sub{crys}(R_\infty)/p\ar[r]\ar[d] & A_\sub{crys}(\res R)/p\ar[d]\\
\projlim_rA_\sub{crys}(R_\infty)/(\op{Fil}^rA_\sub{crys}(R_\infty),p)\ar[r] & \projlim_rA_\sub{crys}(\res R)/(\op{Fil}^rA_\sub{crys}(\res R),p)
}\]
is cartesian. But this easily follows from the injectivity of $R_\infty^\flat/\xi^p\to \res R^\flat/\xi^p$ and the following compatible descriptions, for $S=R_\infty$ and $\res R$: 
\[ \bigoplus_{n\ge0}S^\flat/\xi^p\cong A_\sub{crys}(S)/p,\qquad \prod_{n\ge0}S^\flat/\xi^p\cong \projlim_rA_\sub{crys}(S)/(\op{Fil}^rA_\sub{crys}(S),p),\qquad (a_n)_n\mapsto\sum_{n\ge0}a_n\xi^{[np]}\]
(The first of these descriptions is obtained by base changing along the flat map $\bb F_p[X]\to S^\flat$, $X\mapsto \xi$ the usual description of the pd-envelope of $\bb F_p[X]\to \bb F_p$ as a $\bb F_p[X]/X^p$-module; namely, it is the free module with basis $X^{[np]}$ for $n\ge0$. The second description then follows, since in the first description $\Fil^{rp}A_\sub{crys}(S)/p$ corresponds to $\bigoplus_{n\ge r}S^\flat/\xi^p$ for any $r\ge0$.)
\comment{
\[S^\flat/\xi^p[\delta_1,\delta_2,\dots]/(\delta_1^p,\delta_2^p,\dots)\isoto A_\sub{crys}(S)/p,\qquad \delta_i\mapsto \underbrace{\gamma_p\circ\cdots\circ\gamma_p}_{i\sub{ times}}(\xi)\text{ mod }p,\] which is a filtered isomorphism if we equip the left side with filtration \[\op{Fil}^r:=\begin{cases}(\xi^r,\delta_1,\delta_{2},\dots)& 0\le r<p \\ (\delta_{\left\lfloor\tfrac rp \right\rfloor},\delta_{\left\lfloor\tfrac rp\right\rfloor+1},\dots) & r\ge p\end{cases}.\]
}

\end{proof}

Let $\roiA_{\sub{crys},\infty}(\cal R)$ be the $p$-adic completion of the divided power envelope of $A_\sub{crys}(R_\infty)\otimes_{A_\sub{crys}}\cal R\to R_\infty/p$ compatible with the divided power structure on $pA_\crys+\Fil^1 A_\crys$, equipped with the analogous divided power filtration to that of Definition \ref{definition_OAcrys}(v). To control the canonical map $\roiA_{\sub{crys},\infty}(\cal R)\to \roiA_{\sub{crys}}(\cal R)$ we will use the following explicit description of each side.

Let $A_\sub{crys}(\res R)\{\tau_1,\dots,\tau_d\}$ denote the $p$-adically completed PD-polynomial algebra in variables $\tau_1,\dots,\tau_d$ over $A_{\sub{crys}}(\res R)$, equipped with PD-structure by viewing it as the completed PD-envelope of $A_\sub{crys}(\res R)[\ul \tau]\to \res R$ compatible with the divided power structure on $pA_\crys+\Fil^1 A_\crys$ (so that its PD-ideal $\Fil^1A_\sub{crys}(\res R)\{\ul\tau\}$ is the $p$-adic completion of the ideal generated by $\Fil^1A_\sub{crys}(\res R)$ and the variables). Let $\tilde T_1,\dots,\tilde T_d\in\cal R$ be fixed arbitrary lifts of the variables $T_1,\dots,T_d\in R$ associated to the framing. Then there is a well-defined morphism of $A_\sub{crys}(\res R)$-PD-algebras \begin{equation}A_\sub{crys}(\res R)\{\tau_1,\dots,\tau_d\}\To \roiA_\sub{crys}(\cal R),\qquad \tau_i\mapsto U_i\otimes 1-1\otimes\tilde T_i\label{eqn_PD_map1}\end{equation} since $U_i$ and $\tilde T_i$ have the same image in $\res R$, namely $T_i$; recall here that $U_1,\dots,U_d\in A_\inf(R_\infty)$ are the Teichm\"uller lifts of $T_1^\flat,\dots,T_d^\flat$. In the same way there is a morphism of $A_\sub{crys}(R_\infty)$-PD-algebras \begin{equation}A_\sub{crys}(R_\infty)\{\tau_1,\dots,\tau_d\}\To \roiA_{\sub{crys},\infty}(\cal R),\qquad \tau_i\mapsto U_i\otimes 1-1\otimes\tilde T_i,\label{eqn_PD_map2}\end{equation} such that (\ref{eqn_PD_map1}) is the $p$-completed base change of (\ref{eqn_PD_map2}) along $A_\sub{crys}(R_\infty)\to A_\sub{crys}(\res R)$.

\begin{lemma}\label{lemma_explicit_description_of_OA}
The morphisms (\ref{eqn_PD_map1}) and (\ref{eqn_PD_map2}) are isomorphisms.
\end{lemma}
\begin{proof}
For concreteness of notation we discuss (\ref{eqn_PD_map1}), the same argument working for (\ref{eqn_PD_map2}). The framing map $A_\sub{crys}\pid{\ul {\tilde T}^{\pm1}}\to \cal R$ induced by the choice of lifts is $p$-adically formally \'etale, so may be used to define a morphism of $A_\sub{crys}$-PD-algebras $\cal R\to A_\sub{crys}(\res R)\{\ul \tau\}$ such that $\tilde T_i\mapsto U_i-\tau_i$ and such that the composition with $A_\sub{crys}(\res R)\{\ul \tau\}\to \res R$ (sending $\tau_i\mapsto 0$) is the canonical map $\cal R\to\res R$. Combined with the $A_\sub{crys}(\res R)$-algebra structure and PD-structure of $A_\sub{crys}(\res R)\{\ul \tau\}$, this yields an $A_\sub{crys}(\res R)$-PD-morphism in the other direction to (\ref{eqn_PD_map1}).

The composition $A_\crys(\res R)\{\ul\tau\}\to\roiA_\crys(\cal R)\to A_\crys(\res R)\{\ul\tau\}$ is the identity because it is directly seen to be so on the variables.

In the other direction the composition \begin{equation}A_\sub{crys}(\res R)\otimes_{A_\crys}\cal R\to A_\sub{crys}(\res R)\{\ul\tau\}\to \roiA_\sub{crys}(\cal R)\to \roiA_\sub{crys}(\cal R)/p^n\label{eqn_PD_map3}\end{equation} is the canonical map for any $n\ge1$ by the following argument: it is a map of $A_\crys(\res R)[\ul{\tilde T}^{\pm1}]$-algebras (since the composition $\cal R\to A_\sub{crys}(\res R)\{\ul\tau\}\to\roiA_\sub{crys}(\cal R)$ preserves the $\tilde T_i$ by construction) which agrees with the canonical map after the surjection $\roiA_\sub{crys}(\cal R)/p^n\to\res R/p$ with nilpotent kernel, and $A_\crys(\res R)[\ul{\tilde T}^{\pm1}]\to A_\crys(\res R)\otimes_{A_\crys}\cal R$ is \'etale modulo $p^n$. By the uniqueness of extensions of maps to PD-envelopes, the PD-homomorphism $\roiA_\sub{crys}(\cal R)\to \roiA_\sub{crys}(\cal R)/p^n$ induced by the composition (\ref{eqn_PD_map3}) is therefore also the canonical map, and letting $n\to\infty$ we deduce that the composition $\roiA_\sub{crys}(\cal R)\to A_\sub{crys}(\res R)\{\ul\tau\}\to \roiA_\sub{crys}(\cal R)$ is the identity.
\end{proof}

We equip $A_\sub{crys}(\res R)\{\ul \tau\}$ with a descending filtration by defining $\Fil^rA_\sub{crys}(\res R)\{\ul \tau\}$, for $r\ge0$, to be the $p$-adic completion of $\bigoplus_{n_1,\dots,n_d\ge0}(\Fil^{r-(n_1+\dots+n_d)}A_\crys(\res R))\tau^{[n_1]}\cdots\tau^{[n_d]}$. We equip $A_\sub{crys}(R_\infty)\{\ul \tau\}$ with the analogous filtration. The morphisms (\ref{eqn_PD_map1}) and (\ref{eqn_PD_map2}), as well as their inverses constructed in the previous proof, are compatible with these filtrations; it follows formally from Lemma \ref{lemma_explicit_description_of_OA} that (\ref{eqn_PD_map1}) and (\ref{eqn_PD_map2}) are therefore filtered isomorphisms. From this we deduce:

\begin{corollary}
The canonical map $\roiA_{\sub{crys},\infty}(\cal R)\to \roiA_{\sub{crys}}(\cal R)^H$ is a filtered isomorphism.
\end{corollary}
\begin{proof}
Lemma \ref{lemma_H-fixed_points} implies that the inclusion $A_\sub{crys}(R_\infty)\subseteq A_\sub{crys}(\res R)$ is strict with respect to the filtrations, so the same is true for $A_\sub{crys}(R_\infty)\{\ul\tau\}\subseteq A_\sub{crys}(\res R)\{\ul\tau\}$, or equivalently for $\roiA_{\sub{crys},\infty}(\cal R)\subseteq \roiA_{\sub{crys}}(\cal R)$ thanks to the filtered version of Lemma \ref{lemma_H-fixed_points}. Therefore it is enough to show that the canonical map is an isomorphism ignoring filtrations.

But under the isomorphisms (\ref{eqn_PD_map1}) and (\ref{eqn_PD_map2}) the group $H$ acts trivially on the variables $\tau_1,\dots,\tau_d$, so the desired isomorphism reduces to Lemma \ref{lemma_H-fixed_points}.
\end{proof}

In light of the previous corollary, Proposition \ref{proposition_Delta_fixed_points_of_OA} reduces to the following:

\begin{proposition}
The structure map $p_2:\cal R\to \roiA_{\sub{crys},\infty}(\cal R)^\Gamma$ is a filtered isomorphism.
\end{proposition}
\begin{proof}
Having already descended from $\roiA_\sub{crys}(\cal R)$ to $\roiA_{\sub{crys},\infty}(\cal R)$, we now descend further to a framed version: namely, let $\roiA_\sub{crys}^\bx(\cal R)$ be the $p$-adic completion of the PD-envelope of $A_\sub{crys}^\bx(R)\otimes_{A_\sub{crys}}\cal R\to R/p$ compatible with the divided power structure on $pA_\crys+\Fil^1 A_\crys$. Just as in Lemma \ref{lemma_explicit_description_of_OA}, there is an isomorphism $A_\sub{crys}^\bx(R)\{\tau_1,\dots,\tau_d\}\isoto \roiA_\sub{crys}^\bx(\cal R)$, $\tau_i\mapsto U_i\otimes 1-1\otimes\tilde T_i$, though we will not use it: instead, swapping the roles of $A_\sub{crys}^\bx(R)$ and $\cal R$, the same argument shows that there is an isomorphism of $\cal R$-algebras $\cal R\{\tau_1,\dots,\tau_d\}\isoto \roiA_\sub{crys}^\bx(\cal R)$, $\tau_i\mapsto U_i\otimes 1-1\otimes\tilde T_i$, which we will use.

Recall that $\roiA_{\sub{crys},\infty}(\cal R)$ and $\roiA_\sub{crys}^\bx(\cal R)$ are $p$-adically completed pd-envelopes of $A_\sub{crys}(R_\infty)\otimes_{A_\sub{crys}}\cal R\to R_\infty/p$ and $A_\sub{crys}^\bx(R)\otimes_{A_\sub{crys}}\cal R\to R/p$ respectively. The first of these maps is given by the $p$-adic completion of the base change of the second along $\bb Z_p[\ul U^{\pm1}]\to \bb Z_p[\ul U^{\pm1/p^\infty}]$, where we view $A_\crys^\bx(R)$ as an $\bb Z_p[\ul U^{\pm1}]$-algebra via the framing; this base change assertion reduces to the analogous assertion for $A_\inf$ in place of $A_\crys$, which was explained in \S\ref{ss_framed}, thanks to the identifications $A_\inf(R_\infty)/(p^n,\xi^{pn})\otimes_{A_\inf/(p^n,\xi^{pn})}A_\crys/p^n\isoto A_\crys(R_\infty)/p^n$ and $A_\inf^\bx(R)/(p^n,\xi^{pn})\otimes_{A_\inf/(p^n,\xi^{pn})}A_\crys/p^n\isoto A_\crys^\bx(R)/p^n$ for $n\ge1$. But since this base change is flat, we deduce that $\roiA_{\sub{crys},\infty}(\cal R)=\roiA_{\sub{crys}}^\bx(\cal R)\hat\otimes_{A_\sub{inf}[\ul U^{\pm1}]} A_\sub{inf}[\ul U^{\pm1/p^\infty}]$ or, in other words, that \[\roiA_{\sub{crys},\infty}(\cal R)=\hat\bigoplus_{k_1,\dots,k_d\in\bb Z[\tfrac1p]\cap[0,1)}\roiA_\sub{crys}^\bx(\cal R)U_1^{k_1}\cdots U_d^{k_d}.\]

Combing the previous decomposition with the final polynomial description of the first paragraph, we have \[\roiA_{\sub{crys},\infty}(\cal R)=\hat\bigoplus_{k_1,\dots,k_d\in\bb Z[\tfrac1p]\cap[0,1)}\cal R\{\tau_1,\dots,\tau_d\}U_1^{k_1}\cdots U_d^{k_d}\tag{\dag}\] The $\cal R$-linear $\Gamma$-action respects this decomposition, given on the $p$-power roots of the variables $U_i$ in the usual way by rescaling by the corresponding root of $[\ep]$, and on the $\tau_i$ by the PD-automorphisms characterised by $\gamma_i(\tau_j)
=[\ep]\tau_i+\mu\tilde T_i$ if $i=j$ and $=\tau_j$ if $i\neq j$. To complete the proof we may therefore fix $i=1,\dots,d$ and $k_i\in\bb Z[\tfrac1p]\cap[0,1)$, and we must prove that \[(\cal R\{\tau_i\}U_i^{k_i})^{\gamma_i=\op{id}}=\begin{cases}\cal R & k_i=0 \\ 0 & k_i\neq 0.\end{cases}.\]

So let $a=\sum_{j\ge0}a_j\tau_i^{[j]}U_i^{k_i}$ belong to the left side, where $a_j\in \cal R$. Applying $\gamma_i$, using $\gamma_i(\tau_i^{[j]})=([\ep]\tau_i+\mu\tilde T_i)^{[j]}$ and the usual formula for divided powers of sums and products, and equating $\tau^{[j]}$-coefficients in the expression $a=\gamma_i(a)$, one sees that $a_j=\sum_{\ell\ge j}a_\ell[\ep^{k_i}][\ep^j]\mu^{[\ell-j]}\tilde T^{\ell-j}_i$ for all $j\ge0$, or in other words that \begin{equation}(1-[\ep^{j+k_i}])a_j=\sum_{\ell>j}a_\ell[\ep^{j+k_i}]\mu^{[\ell-j]}\tilde T_i^{\ell-j}.\label{equation_t_i}\end{equation}

We first treat the case $k_i\neq 0$, when we wish to show that $a_j=0$ for all $j\ge0$; since the filtration on $\cal R$ is separated (as explained before (\ref{equation_crystal_micprime})), we may show $a_j\in\bigcap_{s\ge0}\Fil^s\cal R$. To do this we proceed by induction on $s\ge 0$, proving that $a_j\in\Fil^s\cal R$ for all $j\ge 0$. The assertion being trivial when $s=0$, we suppose it is true for some $s$ and deduce at once from (\ref{equation_t_i}) that $(1-[\ep^{j+k_i}])a_j\in\Fil^{s+1}\cal R$ for all $j\ge0$. But $1-[\ep^{j+k_i}]$ has non-zero image in $A_\crys/\Fil^1A_\crys=\roi$ since $k_i\neq 0$, and the graded steps of the filtration are $p$-torsion-free (again explained just before (\ref{equation_crystal_micprime})), so in fact $a_j\in\Fil^{s+1}\cal R$ as desired.

We next consider the case $k_i=0$, where we must prove that $a_j=0$ for all $j>0$. We again proceed by induction on $s$, showing that $a_j\in\Fil^s\cal R$ for all $j>0$; again this is trivial when $s=0$, so we suppose it is true for some $s>0$. From (\ref{equation_t_i}) we have \[a_{j+1}=a_j[\ep^{-j}]\tfrac{1-[\ep]^j}\mu\tilde T_i^{-1}-\sum_{\ell>j+1}a_\ell\tfrac{\mu^{[\ell-j]}}{\mu}\tilde T_i^{\ell-j-1}\] for all $j\ge0$; recall that $\tfrac{\mu^{[\ell-j]}}\mu=\tfrac{\mu^{\ell-j-1}}{(\ell-j)!}\in\Fil^1\cal R$ for $\ell>j+1$ and tends to $0$ $p$-adically as $\ell\to\infty$. If $j=0$ then the first term on the right side vanishes and so it follows that $a_1\in\Fil^{s+1}\cal R$; using $\tfrac{1-[\ep]^j}\mu\in\cal R$ for $j>0$ we then see by induction on $j$ that $a_j\in\Fil^{s+1}\cal R$ for all $j>0$, as desired.
\end{proof}

\subsection{Dieudonn\'e theory, II: essential surjectivity}\label{ss_DieudonneII}
In this subsection we prove that the functor $\Phi_R^\sub{BKF}$ constructed in \S\ref{sss_Dieudonne}
is essentially surjective,
thereby providing a proof of Theorem \ref{theorem_p_div} in the case $p\neq 2$ by using 
Lau's Dieudonn\'e theory over complete intersection semiperfect $\bb F_p$-algebras \cite{Lau2018}. We adopt the notation introduced there and assume throughout this subsection that $p\neq 2$.
Since the proof is based on the contravariant crystalline
Dieudonn\'e functor, it is convenient to work with duals of $\Phi_R^\sub{BKF}$ and $\Phi^\sub{inf}_A$, i.e, with the functors
\[\Phi_R^{\sub{BKF}\vee}(-)=\Hom(\Phi_R^{\sub{BKF}}(-),A_\inf(R_\infty))\colon \op{BT}(R)\to\Rep^\mu_{\Gamma}(A_\inf(R_\infty),\phi,[0,1])^\sub{op}\] and\[ \Phi^{\sub{inf}\vee}_A(-)=\Hom(\Phi_A^\inf(-),A_\inf(A))\colon \op{BT}(A)\to \op{BKF}
(A,\varphi,[0,1])^\sub{op}.\]
Here $\Rep^\mu_{\Gamma}(A_\inf(R_\infty),\phi,[0,1])$ denotes the full subcategory
of $\Rep^\mu_\Gamma(A_\inf(R_\infty),\phi)$ consisting of objects $(M,\phi_M)$ satisfying
$\tilde\xi M\subset \varphi_M(M)\subset M$; similarly, for any perfectoid ring $A$, we write $\op{BKF}(A,\varphi,[0,1])$ for the full subcategory of $\op{BKF}(A,\varphi)$ satisfying the analogous height condition $\tilde\xi M\subset \varphi_M(M)\subset M$ (c.f., Definition \ref{definition_BKF_of_perfectoid} for the covariant case).

Let $M\in\Rep^\mu_{\Gamma}(A_\inf(R_\infty),\phi,[0,1])$; forgetting the $\Gamma$-action on $M$, Theorem \ref{theorem_p_div_perfectoid} shows that there exists a unique $p$-divisible group $G_\infty$ over $R_\infty$ such that $\Phi^{\sub{inf}\vee}_{R_\infty}(G_{\infty})=M$. Our first step is to show that $G_\infty$ descends from $R_\infty$ to $R$; that is, we must construct a descent isomorphism $p_2^*G_\infty\cong p_1^*G_\infty$, where $R_\infty(1)$ denotes the $p$-adic completion of $R_\infty\otimes_RR_\infty$, and $p_1$ and $p_2$ denote the
first and second projections $\Spec(R_\infty(1))\to\Spec(R_\infty)$, respectively. There will be three key ingredients to this construction: Grothendieck--Messing's deformation theory \cite{Messing1972} to partly reduce the problem to $R_\infty(1)/p$; Lau's Diedonn\'e theory over complete intersection semiperfect $\bb F_p$-algebras such as $R_\infty(1)/p$; and the existence of a crystal associated to the relative Breuil--Kisin--Fargues module $M$, as in Theorem \ref{theorem_Admissibility}.
Once a descent of $G_\infty$ has been constructed, the second step will be to verify that it is compatible with the $\Gamma$-action on $M$. We now begin the proof in earnest, starting with the first step.

For $\nu\ge0$, let $R_\infty(\nu)$ denote the $p$-adic completion of the tensor product over $R$
of $\nu+1$ copies of $R_\infty$; let $p_1,p_2:\Spec(R_\infty(1))\to \Spec(R_{\infty})$ denote the
projections to the first and second components. In order to descend $G_\infty$ to $R$, we must construct an isomorphism of $p$-divisible groups $p_2^*G_\infty\cong p_1^*G_\infty$ which satisfies the cocycle condition after pulling back further to $R_\infty(2)$.

We first interpret the problem in terms of windows over the crystalline period ring
of $R_\infty(\nu)$ by using Lau's Dieudonn\'e theory over complete intersection semiperfect algebras
and Grothendieck--Messing deformation theory.
The algebra $R_\infty(\nu)/p$ is a semiperfect $\bb F_p$-algebra which can be written as the quotient of a perfect $\bb F_p$-algebra by a regular sequence, i.e., it is a complete intersection in the sense of \cite[Def.~4.1]{Lau2018} (in particular it is a quasiregular semiperfect ring in the sense of \cite{BhattMorrowScholze2}), as we will see in the proof of Lemma
\ref{lem:SatFilAcrysRinftyProduct} below.
Let $R_\infty(\nu)^\flat:=\projlim_\phi R_\infty(\nu)/p$ be the inverse limit perfection of $R_\infty(\nu)/p$, which is a perfect $\bb F_p$-algebra surjecting onto $R_\infty(\nu)/p$ (with kernel generated by a regular sequence \cite[Lem.~4.2]{Lau2018}), and let $A_\crys(R_\infty(\nu))$ denote the $p$-adic completion of the pd-envelope of $W(R_\infty(\nu)^\flat)\to R_\infty(\nu)$ compatibly
with the pd-structure on $p\bb Z_p$; we recall that $A_\crys(R_\infty(\nu))$ is $p$-torsion-free \cite[Lem.~2.6.1]{CaisLau2017} \cite[Thm.~8.14(i)]{BhattMorrowScholze2}. Write $\Fil A_\crys(R_{\infty}(\nu)):=\ker(A_\crys(R_{\infty}(\nu))\to R_\infty(\nu))$ and $\Fil_pA_{\crys}(R_{\infty}(\nu)):=\ker(A_\crys(R_{\infty}(\nu))\to R_\infty(\nu)/p)=\Fil A_\crys(R_{\infty}(\nu))+pA_\crys(R_{\infty}(\nu)$.
Let $\phi$ denote the Frobenius endomorphism of $A_\crys(R_{\infty}(\nu))$
induced by the usual Frobenius of $W(R_{\infty}(\nu)^{\flat})$; note that
$\phi(\Fil A_\crys( R_{\infty}(\nu)))\subset pA_\crys(R_{\infty}(\nu))$.
Setting $\phi_1=p^{-1}\phi\vert_{\Fil_p A_\crys( R_{\infty}(\nu))}$, we have therefore constructed a ``pd-frame''
\[\ul A_\crys( R_{\infty}(\nu)/p)=(A_\crys( R_{\infty}(\nu)),\Fil_p A_\crys(R_{\infty}(\nu)),  R_{\infty}(\nu)/p,\phi,\phi_1)\]
as in \cite[\S5.3]{Lau2018} (we refer to \cite[\S2]{Lau2018} for the precise definition of a pd-frame). Lau's Dieudonn\'e theory over semiperfect rings \cite[Lem.~4.13, Corol.~6.5]{Lau2018} states that evaluation of contravariant Dieudonn\'e crystals on the pd-thickening $A_\crys(R_{\infty}(\nu))\to R_\infty(\nu)/p$ induces an equivalence of categories
\begin{equation}\label{eq:BTWinEquivSemiPerf}
\Phi_{R_{\infty}(\nu)/p}^{\crys}: 
\op{BT}(R_{\infty}(\nu)/p)\xrightarrow{\sim}
\op{Win}(\ul A_\crys(R_{\infty}(\nu)/p))^\sub{op},
\end{equation}
where $\op{Win}(-)$ denotes the category of windows over the frame $\ul A_\crys( R_{\infty}(1)/p)$ \cite[\S2]{Lau2018}.

By replacing $\Fil_p$ by $\Fil$ in the definition of $\ul A_\crys(R_\infty(\nu))$ above,
we obtain another pd-frame
$$\ul A_\crys(R_\infty(\nu))=(A_\crys( R_{\infty}(\nu)),\Fil A_\crys(R_{\infty}(\nu)),  R_{\infty}(\nu),\phi,
\phi_1\vert_{\Fil}).$$
By the equivalence \eqref{eq:BTWinEquivSemiPerf}  and
Grothendieck-Messing deformation theory \cite[V.~Thm.~(1.6)]{Messing1972}
for $R_{\infty}(\nu)\to R_\infty(\nu)/p$,
the same argument as the proof of \cite[Prop.~9.7]{Lau2018} shows that
evaluation of  contravariant Dieudonn\'e crystals on the pd-thickening
$A_\crys(R_\infty(\nu))\to R_\infty(\nu)$ gives the equivalence of
categories
\begin{equation}\label{eq:BTWinEquivSemiPerf2}
\Phi_{R_{\infty}(\nu)}^{\crys}: \op{BT}(R_{\infty}(\nu))\xrightarrow{\sim}
\op{Win}(\ul A_\crys(R_{\infty}(\nu)))^\sub{op}.
\end{equation}

The functor
$\Phi_{R_{\infty}}^{\inf\vee}$ is defined as a composition of equivalences
\begin{equation}\label{eq:pDivGpBKEquiv}
\op{BT}(R_{\infty})
\xrightarrow[\Phi_{R_{\infty}}^{\crys}]{\sim}\op{Win}(\ul A_\crys(R_{\infty}))^\sub{op}
\xleftarrow{\sim}\op{Win}(\ul{A}_\inf(R_{\infty}))^\sub{op}
\xrightarrow{\sim}\op{BKF}(R_{\infty},\varphi,[0,1])^\sub{op}
\end{equation}
\cite[Props.~9.3 \& 9.7, (9.6)]{Lau2018} where the frame $\ul{A}_\inf(R_{\infty})$ is
defined by
\[\ul{A}_\inf(R_{\infty})=(A_\inf(R_{\infty}),\xi A_\inf(R_{\infty}),R_{\infty},\phi,\phi_1^{\inf}),
\quad\phi_1^{\inf}=\tilde\xi^{-1}\phi\vert_{\xi A_\inf(R_{\infty})}.\]
Note that $\op{BKF}(R_{\infty},\varphi,[0,1])$ is equivalent to the category defined in \cite[Def.~9.4]{Lau2018} by restricting along $\phi$; under this normalisation, the last functor in \eqref{eq:pDivGpBKEquiv},
which is an isomorphism, maps $(P,\Fil P,\phi_P,\phi_{P,1})$ to $(P, \phi_P)$; its
inverse is given by $\Fil P=\varphi_P^{-1}(\tilde\xi P)$ and
$\phi_{P,1}=\tilde\xi^{-1}\phi\vert_{\Fil P}$.

Let $M_{\crys}$ denote the window over the frame $\ul A_\crys(R_\infty)$ corresponding 
to $(M,\varphi_M)\in \op{BKF}(R_{\infty},\varphi,[0,1])$ (with $\Gamma$-action forgotten)
via the middle and right equivalences of categories in \eqref{eq:pDivGpBKEquiv}.
Let $p_1, p_2\colon \Spec(A_\crys(R_\infty(1)))\to \Spec(A_\crys(R_\infty))$ denote
the first and second projections. Then, in conclusion,
the descent goal is equivalent to constructing an isomorphism of 
windows $p_2^*M_\crys\xrightarrow{\cong}p_1^*M_\crys$ over $\ul A_\crys(R_\infty(1))$
satisfying the cocycle condition after pullback to
the frame $\ul A_\crys(R_\infty(2))$.

We construct such an isomorphism by realizing $M_\crys$ as the evaluation
of an object $\cal F=((\cal F_n,\Fil^{\bullet}\cal F_n)_{n\geq 1},\varphi_{\cal F})$ of the category
$\CR(R/A_\crys,\op{SatFil},\varphi)$ associated to the object
$M\in \Rep_{\Gamma}(A_\inf(R_\infty),\varphi,[0,1])$ we fixed at the beginning.
By Remark \ref{remark_F_structures_Acrys}, Remark \ref{remark_associated1_filtered}, and
\eqref{diagram_Ainfcrys_frame_nonframe_Rep_filtered} with $\ol R$ replaced by $R_{\infty}$
for the latter two (which is easily done), we obtain the following commutative diagram.
\begin{equation}\label{eq:CrysAinfRepDiagram}
\xymatrix@C=50pt{
\CR(R/A_\crys,\op{SatFil},\phi)\ar[r]_{\sim}^{\ev_{A_\crys^\bx(R)}}
\ar[dr]_{\ev_{A_\crys(R_\infty)}}
&
\Rep_{\Gamma}^\mu(A_\crys^\bx(R),\op{SatFil},\varphi)
\ar[d]^{\sub{Lemma }\ref{lemma_saturation_of_product}}
&
\Rep_{\Gamma}^\mu(A_\inf^\bx(R),\varphi)
\ar[l]_(.45){\sub{Lemma } \ref{lemma_phi_to_filtr_framed}}
\ar[d]^{\rotatebox{-90}{$\simeq$}}
\\
&\Rep_\Gamma^\mu(A_\crys(R_\infty),\op{SatFil},\varphi)&
\Rep_\Gamma^\mu(A_\inf(R_\infty),\varphi)
\ar[l]_(.45){\sub{Lemma }\ref{lemma_phi_to_filtr}}
}
\end{equation}
We define $\cal F$ to be the object corresponding to $M$ via the
right vertical equivalence and the upper two horizontal functors
 in the
above diagram.

Let us compare  the evaluation $\ev_{A_\crys(R_\infty)}(\cal F)$ of $\cal F$ with $M_\crys$, by showing that there exists a natural isomorphism 
\begin{equation}\label{eq:CrysEvComparison}
\ev_{A_\crys(R_\infty)}(\cal F)\cong M_\crys\quad\text{in }\Rep_{\Gamma}^\mu
(A_\crys(R_\infty),\op{SatFil},\varphi),
\end{equation}
where we will explain after Lemma \ref{lem:WinFilSatFil} how to equip the window $M_\crys$ with a $\Gamma$-action and view it as an object of $\Rep_{\Gamma}^\mu
(A_\crys(R_\infty),\op{SatFil},\varphi)$. Thanks to the commutative diagram \eqref{eq:CrysAinfRepDiagram}, the problem is reduced
to comparing the bottom horizontal functor in \eqref{eq:CrysAinfRepDiagram} with
the middle and the right equivalences in \eqref{eq:pDivGpBKEquiv}. We will need the following construction (the filtration constructed on $N$ in the next lemma will be called {\em the saturated filtration associated to} the submodule $\Fil N$):

\begin{lemma}\label{lem:WinFilSatFil}
Let $A$,  $f$, and $A^+$ be as in Definition \ref{definition_filtration_2},
and assume that the filtration on $A$
is saturated. Let $N$ be a finite projective $A$-module equipped with an
$A$-submodule $\Fil N$ such that there exists a decomposition $N=N'\oplus N''$
as $A$-modules
satisfying $\Fil N=N'\oplus \Fil^1 A\, N''$, and define 
a $\bb Z$-indexed multiplicative filtration on $N$ by 
$\Fil^rN=\Fil^rA\, N+\Fil^{r-1}A  \Fil N$ for $r\in \bb Z$. Then:
\begin{enumerate}
\item $\Fil N=\Fil^1N$ and $\Fil^rN=\Fil^{r-1}A\, N'\oplus \Fil^r A\, N''$ for all $r\in \bb Z$;
\item $N^+:=\bigcup_{r\in \bb Z} f^{-r}\Fil^r N$ is equal to $f^{-1}(N'\otimes_AA^+)\oplus
(N''\otimes_AA^+)$;
\item the filtration $\Fil^rN$ on $N$ satisfies (F4) and (F5) in Definition
\ref{definition_filtration_2}.
\end{enumerate}
Moreover, if $\Fil^rA$ is $p$-adically closed in $A$ for every $r\in \bb Z$,
then the same is true of $\Fil^rN$ for every $r\in \bb Z$.
\end{lemma}
\begin{proof} This is straightforward. Note that $N'$ and $N''$ are finite projective $A$-modules.
\end{proof}

We may now explain and prove (\ref{eq:CrysEvComparison}). We define a category $\op{Mod}_\sub{fproj}(A_\crys(R_\infty),\op{SatFil},\varphi)$ by forgetting the action of $\Gamma$
in the definition of $\Rep_{\Gamma}^\mu(A_\crys(R_\infty),\op{SatFil},\varphi)$. To be precise,  an object of $\op{Mod}_\sub{fproj}(A_\crys(R_\infty),\op{SatFil},\varphi)$ is a finite projective $A_\crys(R_\infty)$-module $P$ equipped with both an isomorphism $\phi_P:\phi^*P[\tfrac1p]\isoto P[\tfrac1p]$ and a filtration satisfying (F1)--(F2) of Definition \ref{definition_filtration_1} and (F4)--(F5) of Definition \ref{definition_filtration_2}.

We regard $\op{Win}(\ul A_\crys(R_\infty))$ as a full subcategory of $\op{Mod}_\sub{fproj}(A_\crys(R_\infty),\op{SatFil},\varphi)$ because an object $(P,\Fil P,\varphi_P,\varphi_{P,1})$ of the former is determined by the object $(P,\Fil^rP,\varphi_P)$ of the latter, where $\Fil^rP$
is the saturated filtration associated to $\Fil P$ as in Lemma \ref{lem:WinFilSatFil}.
By forgetting the action of
$\Gamma$ in the construction of the bottom horizontal functor in \eqref{eq:CrysAinfRepDiagram},
we obtain a functor $\BKF(R_\infty,\varphi)
\to \op{Mod}_\sub{fproj}(A_\crys(R_\infty),\op{SatFil},\varphi)$, which together
with the middle and right functors in \eqref{eq:pDivGpBKEquiv} forms the following commutative diagram by the description of $\Fil^rN$ and $N^+$ in Lemma
\ref{lem:WinFilSatFil}.
\begin{equation}\label{eq:WinSatFil}
\xymatrix{
\op{BKF}(R_\infty,\varphi,[0,1])\ar[r]^{\sim}\ar@{^(->}[d]&
\op{Win}(\ul A_\inf(R_\infty))\ar[r]^{\sim}&
\op{Win}(\ul A_\crys(R_\infty))\ar@{^(->}[d]\\
\BKF(R_\infty,\varphi)\ar[rr]&&
\op{Mod}_\sub{fproj}(A_\crys(R_\infty),\op{SatFil},\varphi).
}
\end{equation}
Thus, equipping $M_\crys\in \op{Win}(\ul A_\crys(R_\infty))$ with the
action of $\Gamma$ induced by that on $M$, we obtain the desired natural isomorphism (\ref{eq:CrysEvComparison}). We remark that the compatibility of this isomorphism with the action of $\Gamma$ will be used in the second step of the proof.

Next we show that evaluation of $\cal F$ on the pd-thickening
$A_\crys(R_\infty(1))\to R_\infty(1)$ gives the
desired descent isomorphism. For an integer $r>0$,
we define $\Fil^rA_\crys(R_\infty(\nu))$ to be the $p$-adic closure of the
image of the $r^\sub{th}$ divided power of the pd-ideal of the pd-envelope of
$W(R_\infty(\nu)^\flat)\to R_\infty(\nu)$ compatible with the pd-structure on $p\bb Z_p$.
Put $\Fil^rA_\crys(R_\infty(\nu))=A_\crys(R_\infty(\nu))$ for all integers $r\leq 0$, and note that $\Fil^1A_\crys(R_\infty(\nu))=\Fil A_\crys(R_\infty(\nu))$. To apply
Lemma \ref{lem:WinFilSatFil} to $A_\crys(R_\infty(\nu))$, and Lemma
\ref{lemma_saturation_of_product}
to several pd-homomorphisms between two of $A_\crys(R_\infty(\nu))$,
we need the following fact.
\begin{lemma}\label{lem:SatFilAcrysRinftyProduct}
The element $\xi\in \Fil^1A_\crys(R_\infty(\nu))$ is a non-zero divisor in
$A_\crys(R_\infty(\nu))$, and the filtration
$\Fil^rA_\crys(R_\infty(\nu))$, for $r\in\bb Z$, on $A_\crys(R_\infty(\nu))$ is saturated
with respect to $\xi$.
\end{lemma}
\begin{proof}
(cf.~\cite[the proof of Lemma 2.6.1]{CaisLau2017}) It suffices to show
that $\bigcap_{r\in \bb Z}\Fil^rA_\crys(R_\infty(\nu))=0$ and that
multiplication by $\xi$ on $A_\crys(R_\infty(\nu))$ induces an injective homomorphism
$[\xi]\colon \op{gr}^rA_\crys(R_\infty(\nu))\to \op{gr}^{r+1}A_\crys(R_\infty(\nu))$ for each $r\in \bb N$.
For a positive integer $n\ge0$, let $(D_n, J_n)$ be the pd-envelope 
of $\theta_n\colon W_n(R_\infty(\nu)^\flat)=W(R_\infty(\nu)^\flat)/p^n\to R_\infty(\nu)/p^n$
compatible with the pd-structure on $p\bb Z_p$. 
Since the $p$-adic completion 
$\tilde R_{\infty}(\nu)$ of the tensor product over $\roi$ of $\nu+1$ copies of $R_{\infty}$
is a perfectoid ring, and the closed immersion
$\Spec(R_{\infty}(\nu)/p)\hookrightarrow \Spec(\tilde R_{\infty}(\nu)/p)$
is defined by the ideal generated by a regular sequence $p_1^*(T_i)-p_{j}^*(T_i)$, for
$1\le i\le j$, $2\le j\le \nu+1$, on an affine open neighborhood of
$\Spec(\tilde R_{\infty}(\nu)/p)$, \cite[Lem.~4.2]{Lau2018} implies that
the kernel of $R_\infty(\nu)^\flat\to R_\infty(\nu)/p$ is generated
by a regular sequence $f_1,\ldots, f_N$ with $f_1=(\xi \mod p)$. Choose liftings
$g_1,\ldots, g_N\in \ker(\theta\colon W(R_\infty(\nu)^\flat)\to R_\infty(\nu))$ of $f_1,\ldots, f_N$,
which exist because $\theta$ is surjective; we assume $g_1=\xi$.
Since $W(R_\infty(\nu)^\flat)$ and $R_\infty(\nu)$ are both $p$-torsion free, we see that
$g_1,\ldots, g_N$ form a regular sequence in $W_n(R_\infty(\nu)^\flat)$ and generate
the kernel of $\theta_n$ for $n>0$. Let $A_n$ (resp.~$A_n^\sub{pd}$) be the polynomial ring (resp.~the pd-polynomial ring) over $\bb Z/p^n$
in $N$ variables $X_1,\ldots, X_N$, and put $I_n=\sum_{i=1}^NX_iA_n$.
Then, by the local criteria of flatness, the homomorphism
$A_n/I_n^{p^nN}\to W_n(R_\infty(\nu)^\flat)/\ker(\theta_n)^{p^nN}$, $X_i\mapsto f_i$ is flat;
taking the pd-envelopes with respect to the images of $I_n$ and
$\ker(\theta_n)$, we obtain an isomorphism
$A_n^\sub{pd}\otimes_{A_n/I^{p^nN}}W_n(R_\infty(\nu)^\flat)/\ker(\theta_n)^{p^nN}\xrightarrow{\cong} D_n$. This shows, for each $r\ge0$, that the $R_\infty(\nu)/p^n$-module $J_n^{[r]}/J_n^{[r+1]}$ is
free with basis $g_1^{[r_1]}\cdots g_N^{[r_N]}$, where $r_1+\cdots+r_N=r$, whence the same is true
for the $R_\infty(\nu)$-module $\gr^iA_\crys(R_\infty(\nu))$ and the desired injectivity of $[\xi]$ holds
as $g_1=\xi$. The vanishing of $\bigcap_r\Fil^r$ is reduced to 
$\varprojlim_rJ_n^{[r]}=0$ for all $n\geq 1$, and then to $\varprojlim_rJ_1^{[r]}=0$ because
$J^{[r]}_n$ is flat over $\bb Z/p^n$ and $J^{[r]}_{n+1}/p^nJ_{n+1}^{[r]}=J_n^{[r]}$ by the above
isomorphism. But the isomorphism for $n=1$ implies that $D_1$ is a free
$R_\infty(\nu)^\flat/(f_1^p,\ldots, f_N^p)$-module with basis
$f_1^{[pj_1]}\cdots f_N^{[pj_N]}$, whence we have $\varprojlim_rJ_1^{[r]}=0$
as desired.
\end{proof}

Define $\op{Mod}_\sub{fproj}(A_\crys(R_\infty(\nu)),\op{SatFil},\varphi)$
by simply replacing $R_\infty$ with $R_\infty(\nu)$ in the definition
of $\op{Mod}_\sub{fproj}(A_\crys(R_\infty),\op{SatFil},\varphi)$. 
By Lemma \ref{lem:SatFilAcrysRinftyProduct},  we may regard $\op{Win}(\ul A_\crys(R_\infty(\nu)))$
as a full subcategory of $\op{Mod}_\sub{fproj}(A_\crys(R_\infty(\nu)),\op{SatFil},\varphi)$
by identifying an object $(P,\Fil P,\varphi_{P},\varphi_{P,1})$ of the former
with the object $(P,\Fil^rP,\varphi_P)$ of the latter it determines, where $\Fil^rP$ is the saturated
filtration associated to $\Fil P$ as in Lemma \ref{lem:WinFilSatFil}.
Similarly to Remark \ref{remark_associated1_filtered}, we can define
a functor
$$\ev_{A_\crys(R_\infty(\nu))}\colon\CR(R/A_\crys,\op{SatFil},\varphi)\longrightarrow \op{Mod}_\sub{fproj}
(A_\crys(R_\infty(\nu)),\op{SatFil},\varphi)$$
by taking the inverse limit over $n\ge1$ of the evaluations on
$\Spec(R_\infty(\nu)/p),\Spec(R_\infty(\nu)/p^n)\hookrightarrow
\Spec(A_\crys(R_\infty(\nu))/p^n)$ and saturating the filtration.
By Lemma \ref{lem:SatFilAcrysRinftyProduct}, we can apply Lemma
\ref{lemma_saturation_of_product} to the morphism $s\colon \Spec(A_\crys(R_\infty(\nu')))\to \Spec(A_\crys(R_\infty(\nu)))$
associated to any non-decreasing map
$[\nu]\to [\nu']$, and so obtain a functor
$$s^*\colon \op{Mod}_\sub{fproj}(A_\crys(R_\infty(\nu)),\op{SatFil},\varphi)
\longrightarrow \op{Mod}_\sub{fproj}(A_\crys(R_\infty(\nu')),\op{SatFil},\varphi),$$
which is compatible with both $\ev_{A_\crys(R_\infty(\nu))}$ and
$\ev_{A_\crys(R_\infty(\nu'))}$, and also with the base change functor 
$s^*\colon \op{Win}(\ul A_\crys(R_\infty(\nu)))\longrightarrow
\op{Win}(\ul A_\crys(R_\infty(\nu')))$ of \cite[Lem.~2.10]{Lau2010}.

By the above compatibility for the projections
$p_1,p_2\colon \Spec(A_\crys(R_\infty(1)))\to \Spec(A_\crys(R_\infty))$
and $p_{12}, p_{23},p_{13}\colon \Spec(A_\crys(R_\infty(2)))
\to \Spec(A_\crys(R_\infty(1)))$, we obtain an isomorphism
$$\varepsilon\colon
p_2^*(\ev_{A_{\crys}(R_\infty)}(\cal F))
\cong \ev_{A_\crys(R_\infty(1))}(\cal F)
\cong p_1^*(\ev_{A_\crys(R_\infty)}(\cal F))\quad
\text{in }\op{Win}(\ul A_\crys(R_\infty(1)))$$
satisfying the cocycle  condition
$p_{13}^*(\varepsilon)=p_{12}^*(\varepsilon)
\circ p_{23}^*(\varepsilon)$ in 
$\op{Win}(\ul A_\crys(R_\infty(2))).$
Combined with \eqref{eq:CrysEvComparison}, this finishes the descent goal: indeed, via the equivalence 
\eqref{eq:BTWinEquivSemiPerf2}, the inverse of $\ep$ induces an isomorphism
$$\varepsilon^\sub{BT}\colon p_2^*(G_{\infty})\xrightarrow{\cong} p_1^*(G_{\infty})$$
of $p$-divisible groups over $R_{\infty}(1)$ satisfying the cocycle condition
$p_{13}^*(\varepsilon^\sub{BT})=p_{12}^*(\varepsilon^\sub{BT})
\circ p_{23}^*(\varepsilon^\sub{BT})$
over $R_{\infty}(2)$. By faithfully flat descent,
we obtain a $p$-divisible group $G$ over $R$.

Let $\pi$ denote the morphism $\Spec(R_{\infty})\to \Spec(R)$.
The second step is to show that the isomorphism
$\Phi^{\sub{inf}\vee}_{R_\infty}(\pi^*G)\cong \Phi^{\sub{\inf}\vee}_{R_\infty}(G_\infty)
=M$ in $\BKF(R_\infty,\varphi,[0,1])$ is compatible with the $\Gamma$-action on $M$.
By transferring this isomorphism via the right and middle equivalences in \eqref{eq:pDivGpBKEquiv}
and using the $\Gamma$-equivariance of \eqref{eq:CrysEvComparison}, we are reduced to verifying
that the composition
\begin{equation}\label{eq:FFDescentGaloisEquiv}
\Phi_{R_{\infty}}^{\crys}(\pi^*G)
\cong \Phi_{R_{\infty}}^\crys(G_{\infty})
\cong M_{\crys}\cong \ev_{A_\crys(R_\infty)}(\cal F)\end{equation}
in $\op{Win}(\ul A_\crys(R_{\infty}))$ is compatible with the action of $\Gamma$, i.e.,
for each $\gamma\in \Gamma$, the semilinear action of $\gamma$ on $\pi^*G$
induced by its action on $\ev_{A_\crys(R_\infty)}(\cal F)$ via \eqref{eq:FFDescentGaloisEquiv}
coincides with $\gamma^*\pi^*G\xrightarrow{\cong} (\pi\circ\gamma)^*G=\pi^*G$.
To show this claim, we need to give a construction of the action 
of $\gamma$ on $\ev_{A_\crys(R_{\infty})}(\cal F)$ from the isomorphism 
$\varepsilon$ used in the construction of $G$. 
Let $(1,\gamma)\colon \Spec(A_\crys(R_\infty))\to \Spec(A_\crys(R_{\infty}(1)))$
be the pd-morphism induced by the morphism
$(1,\gamma)\colon \Spec(R_{\infty})\to \Spec(R_{\infty}(1))$.
Then we have $p_1\circ (1,\gamma)=\id$ and $p_2\circ (1,\gamma)=\gamma$ 
on $\Spec(A_\crys(R_{\infty}))$. By the definition of $\varepsilon$ and the action of $\gamma$ on
$\ev_{A_\crys(R_{\infty})}(\cal F)$, these two equalities 
imply that the action $\gamma^*(\ev_{A_\crys(R_{\infty})}(\cal F))\xrightarrow{\cong}
\ev_{A_\crys(R_{\infty})}(\cal F)$ of $\gamma$ on $\ev_{A_\crys(R_{\infty})}(\cal F)$
is given by the pullback $(1,\gamma)^*(\varepsilon)$ of $\varepsilon$. This shows that the semilinear action of $\gamma$ on $\pi^*G$ 
induced by that on $\ev_{A_\crys(R_{\infty})}(\cal F)$
is the pullback of the composition $p_2^*\pi^*G\cong p_2^*G_{\infty}
\xrightarrow[\varepsilon^\sub{BT}]{\cong} p_1^*G_{\infty}\cong p_1^*\pi^*G$
by $(1,\gamma)\colon \Spec(R_{\infty})\to \Spec(R_{\infty}(1))$. By the
construction of $G$, the composition above coincides with
the isomorphism induced by $p_2\circ \pi=p_1\circ \pi$, whence
its pullback by $(1,\gamma)$ is the isomorphism $\gamma^*\pi^*G
\cong \pi^*G$ induced by $\pi\circ \gamma=\pi$, as desired. This completes the proof of Theorem \ref{theorem_p_div} (in the case $p\neq 2$).

\begin{remark}[Crystalline Dieudonn\'e functor over $R$]
In the proof above, the object $\cal F$ of the category $\CR(R/A_\crys,\op{SatFil},\varphi)$ associated
to the object $M$ of $\Rep_{\Gamma}^\mu(A_\inf(R_\infty),\varphi)$ plays a key role. 
This object $\cal F$ is related to the contravariant Dieudonn\'e crystal $\bb D_\crys(G_0)$ of
$G_0:=G\times_{\Spec(R)}\Spec(R/p)$ as follows. By restricting the base to
$A_\crys$, $\bb D_\crys(G_0)$ gives an object $\cal G=((\cal G_n)_{n\geq 1},\varphi_{\cal G})$
of $\CR(R/A_\crys,\varphi)$, which is equipped with a direct factor
$\Fil\ol{\cal G}_n$ of $\ol{\cal G}_n:=\cal G_n\otimes_{\cal O_{R_n/A_\crys}}\cal O_{R_n}$:
the pullback of the Hodge filtration in $\Gamma(R,\bb D_{\crys}(G_0))$ associated
to the lifting $G$ of $G_0$.  We can regard $\cal G $ with $(\Fil\ol{\cal G}_n)_{n\geq 1}$
as an object of $\CR(R/A_\crys,\op{SatFil},\varphi)$ similarly to the right vertical
inclusion in \eqref{eq:WinSatFil}. By the construction of the functor $\Phi_{R_{\infty}}^{\crys}$
and the compatibility of the crystalline Dieudonn\'e functors and Hodge filtrations
with base changes, we obtain a canonical isomorphism 
$$\Phi_{R_{\infty}}^{\crys}(\pi^*G)\cong \ev_{A_{\crys(R_{\infty})}}(\cal G)$$
in $\op{Win}(\ul A_\crys(R_\infty))\subset \op{Mod}_\sub{fproj}(A_\crys(R_\infty),\op{SatFil}, \varphi)$
compatible with the action of $\Gamma$, i.e., in the category
$\Rep_{\Gamma}^\mu(A_\crys(R_\infty),\op{SatFil},\varphi)$.
We see that the functor
$\ev_{A_\crys(R_\infty)}$ in \eqref{eq:CrysAinfRepDiagram} is 
fully faithful either by replacing $\res{R}$ by $R_\infty$
in the proof of Theorem \ref{theorem_crystal_genrep}, or by reducing it to
Theorem \ref{theorem_crystal_genrep}. 
(For the latter proof, we use the compatibility of $\ev_{A_\crys(\res R)}$
and $\ev_{A_\crys(R_\infty)}$ with the faithful functor 
$\Rep^{\mu}_{\Gamma}(A_\crys(R_\infty),\op{SatFil},\varphi)
\to \Rep^{\mu}_{\Delta}(A_\crys(\res R),\op{SatFil},\varphi)$
defined by Lemma \ref{lemma_saturation_of_product}). Thus comparing the above isomorphism 
with \eqref{eq:FFDescentGaloisEquiv}, we see that $\cal F\cong \cal G$ in $\CR(R/A_\crys,\op{SatFil},\varphi)$. 
\end{remark}

\newpage
\section{Globalising small representations via the pro-\'etale site}\label{section_BKF_on_proetale}
So far in the paper we have studied various approaches to the category of generalised representations $\Rep_\Gamma^\mu(A_\inf(R_\infty))$ in the local context, namely for a small, formally smooth $\roi$-algebra $R$ equipped with a choice of framing. Even using Remark \ref{remark_Delta} to phrase our results in terms of $\Rep_\Gamma^\mu(A_\inf(\res R))$, our generalised representations would still depend on the choice of an algebraic closure of $\Frac R$. The goal of this section is twofold: (1) we will use the pro-\'etale site to present an intrinsic approach to our generalised representations of interest which is independent of any choices, and (2) we will globalise this approach to arbitrary smooth, $p$-adic formal $\roi$-schemes.

So let $\frak X$ be a smooth, $p$-adic formal $\roi$-scheme; let $X$ denote its adic generic fibre as a rigid analytic space and let $\nu:X_\sub{pro\'et}\to\frak X_\sub{Zar}$ be the projection map of sites. We begin this section by introducing our relative Breuil--Kisin--Fargues modules (without Frobenius for the moment) on $\frak X$:

\begin{definition}\label{definition_BKF_global}
Let $\bb M$ be a sheaf of $\bb A_{\sub{inf},X}$-modules on $X_\sub{pro\'et}$, and recall the elements $\xi_r\in A_\inf$ from the Notations. We say that $\bb M$ is {\em trivial modulo $\xi_r$}if the sheaf of $\nu_*(\bb A_{\inf,X}/\xi_r)$-modules $\nu_*(\bb M/\xi_r)$ is locally finite free and the counit \[\nu^{-1}\nu_*(\bb M/\xi_r)\otimes_{\nu^{-1}\nu_*(\bb A_{\inf,X}/\xi_r)}\bb A_{\inf,X}/\xi_r\To \bb M/\xi_r\] is an isomorphism of sheaves on $X_\sub{pro\'et}$. We say that $\bb M$ is {\em trivial modulo $<\mu$} if it is trivial modulo $\xi_r$ for all $r\ge1$.

The category of locally finite free $\bb A_{\sub{inf},X}$-modules which are trivial modulo $<\mu$ will be denoted by $\BKF(\frak X)$ and called {\em Breuil--Kisin--Fargues modules without Frobenius} on $\frak X$.
\end{definition}

It is clear from the definition that the construction $\frak X\mapsto \BKF_{\sub{wo}\phi}(\frak X)$ is a stack for the Zariski topology; namely, given an open cover of $\frak X$, then a relative Breuil--Kisin--Fargues module on $\frak X$ is uniquely determined by its restriction to each open together with the glueing isomorphisms on the intersections (required to satisfy the usual cocycle conditions on the triple intersections). The goal of this section is therefore to show, given any open $\Spf R\subseteq \frak X$ where $R$ is a small formally smooth $\roi$-algebra as in \S\ref{section_small_reps}, that relative Breuil--Kisin--Fargues modules on $\Spa(R[\tfrac1p],R)$ are given by the generalised representations from \S\ref{section_small_reps}; to be precise, Theorem \ref{theorem_Ainf_reps_vs_pro_etale} will show that taking global sections (on a suitable affinoid perfectoid cover of the adic generic fibre of $\Spf R$) induces an equivalence of categories 
\begin{equation}\BKF(\Spf R)\quis\Rep_\Gamma^{<\mu}(A_\inf(R_\infty)).\label{eqn_BKF_to_Rep}\end{equation}

The main obstacle to establishing this equivalence is the lack of a theory of vector bundles over the integral structure sheaf on adic spaces. More precisely, given an affinoid perfectoid $\cal U$ in the pro-\'etale site of $X$, it is not true that locally finite free sheaves of $\hat\roi_X^+|_{\cal U}$-modules correspond to finite projective modules over the corresponding integral perfectoid ring $\Gamma(\cal U,\hat\roi_X^+)$. However, the triviality conditions on our sheaves and representations allow this problem to be overcome via judicious use of the almost purity theorem. We do this first in \S\ref{subsection_localI} for modules over sheaves such as $\hat\roi_X^+$, $W_r(\hat\roi_X^+)$, and $W_r(\hat\roi_{X^{\flat}}^{+})$, then bootstrap up to $\bb A_{\inf,X}$ in \S\ref{subsection_localII}.

\subsection{Small sheaves of modules over $\hat\roi_X^+$ etc.}\label{subsection_localI}
We wish to study sheaves on modules on $X_\sub{pro\'et}$ over the usual integral structure/period sheaves, namely \[\hat\roi_X^+,\quad W_r(\hat\roi_X^+)=\bb A_{\sub{inf},X}/\xi_r,\quad W_r({\hat\roi_{X^{\flat}}^{+}})=\bb A_{\sub{inf},X}/p^s,\quad \bb A_{\sub{inf},X}.\] (See \cite[\S5]{BhattMorrowScholze1} \cite{Scholze2013}.) Rather than constantly changing notation or arguing with different cases, we adopt the following axiomatic set-up to simultaneously treat the first three cases (the final case will be treated separately in \S\ref{subsection_localII} since it does not quite satisfy the usual axioms for almost mathematics and is complete both $\xi$-adically and $p$-adically).

Let $B$ be a ring, $\frak M\subseteq B$ an ideal, and $\bb B$ a sheaf of $B$-algebras on $X_\sub{pro\'et}$. These are subject to the following hypotheses:
\begin{enumerate}
\item[(B1)] Firstly, we assume that $\frak M$ is the standard type of ideal for almost mathematics which arises from a valuation ring: namely, we assume that $\frak M=\frak M^2$ is an increasing union $\bigcup_{i\ge0} \pi_iB$ of principal ideals of $B$ generated by non-zero-divisors $\pi_i$, with the property that $\pi_i^p\in \pi_{i-1}B$ for each $i\ge1$. To avoid any fixed choices of the $\pi_i$, we write $\frak M^\circ\subseteq\frak M$ for the subset of elements $\pi\in\frak M$ with the property that $\frak M\subseteq\sqrt{\pi B}$. The set $\frak M$ satisfies the following: (i) $\pi_i\in \frak M^{\circ}$; (ii) For any $\pi, \pi'\in \frak M^{\circ}$, there exists a positive integer $n$ such that $\pi^n\vert \pi'$ and $(\pi')^n\vert \pi$; (iii) For any $\pi\in \frak M^{\circ}$ and $\pi'\in \frak M$, we have $\pi'\in \frak M^{\circ}$ if and only if $\pi'\vert \pi^n$ for $n\gg 0$. For the remaining properties let $\pi$ be any element of $\frak M^\circ$ (the properties are easily seen not to depend on the particular $\pi$.)
\item[(B2)] Assume that $\bb B$ is complete, in that the canonical map $\bb B\to\projlim_s\bb B/\pi^s$ is an isomorphism.
\item[(B3)] For all affinoid perfectoids $V\in X_\sub{pro\'et}$, assume the following:
\begin{enumerate}
\item $\pi$ is a non-zero divisor of $\Gamma(V,\bb B)$ (hence of $\bb B$ since we make this assumption for all $V$).
\item $H^i_\sub{pro\'et}(V,\bb B)$ is almost zero for $i>0$. (It then follows formally, for any $f\in B$ which is a non-zero-divisor of $\bb B$, that $H^i_\sub{pro\'et}(V,\bb B/f)$ is almost zero for all $i\ge1$, and that the injection $\Gamma(V,\bb B)/f\to \Gamma(V,\bb B/f)$ is almost surjective.)
\end{enumerate}
\item[(B4)] Finally, assume that the pro-\'etale sheaf $\bb B/\pi\bb B$ is the pull-back of a sheaf from the \'etale site of $X$, i.e., this necessarily means that $w^*w_*(\bb B/\pi\bb B)\isoto \bb B/\pi\bb B$, where $w:X_\sub{pro\'et}\to X_\sub{\'et}$ is the canonical morphism of sites, by \cite[Corol.~3.17]{Scholze2013} ([loc.~cit.] also shows that the higher direct images $R^iw_*(\bb B/\pi\bb B)$ then vanish for $i>0$, whence an easy induction shows that the isomorphism remains valid if we replace $\pi$ by any power, hence by any element of $\frak M^\circ$).
\end{enumerate}

\begin{example}\label{example_B}
The following are the standard choices of $\frak M\subseteq B$, $\bb B$ satisfying these hypotheses, for any $r\ge1$:
\begin{enumerate}
\item $W_r(\frak m)\subseteq W_r(\roi)$, $W_r(\hat\roi_X^+)$;
\item $W_r(\frak m^\flat)\subseteq W_r(\roi^\flat)$, $W_r({\hat\roi_{X^{\flat}}^{+}})$.
\end{enumerate}
Indeed, the almost mathematics hypothsis (B1) follow from \cite[Lem.~10.1 \& Corol.~10.2]{BhattMorrowScholze1}. By induction on the short exact sequence $0\to W_r\to W_{r+1}\to W_1\to 0$ the other hypotheses (B2)--(B4) then reduce to the case $r=1$, in which case they follow from the definitions of $\hat\roi_X^+$ and ${\hat\roi_{X^{\flat}}^{+}}$ and from \cite[Lems.~4.10 \& 5.11]{Scholze2013}.
\end{example}

We will encounter often the condition appearing in (B4), namely that a given sheaf $\cal F$ on $X_\sub{pro\'et}$ is the pull-back of a sheaf on $X_\sub{\'et}$, or equivalently that the counit $w^*w_*(\cal F)\isoto\cal F$ is an isomorphism. Such sheaves on $X_\sub{pro\'et}$ will henceforth be called {\em discrete} since the resulting Galois actions on their global sections are continuous for the discrete topology, as we will see before Example \ref{example_affinoid_perfectoid_covers}.

For the rest of \S\ref{subsection_localI}--\ref{subsection_localII} we assume that we are in the local situation that $\frak X=\Spf R$, where $R$ is a $p$-adically complete, formally smooth, small $\roi$-algebra $R$ as at the start of \S\ref{section_small_reps}; to remind ourselves that we are working locally we will write $U=\Spa (R[\tfrac1p],R)$ instead of $X$ for the rigid analytic generic fibre of $\frak X$.

In this subsection we study modules over sheaves of rings $\bb B$ satisfying hypotheses (B1)--(B4). The following definition presents the types of modules which will be of interest to us.

\begin{definition}\label{trivial_modulo_pi}
Let $\frak M\subseteq B$, $\bb B$ satisfy hypotheses (B1)--(B4), and let $\bb M$ be a sheaf of $\bb B$-modules.

Letting $\pi$ be any element of $\frak M^\circ$, we say that $\bb M$ is {\em complete} if and only if $\bb M\isoto\projlim_s\bb M/\pi^s$, and that it is {\em topologically torsion-free} if and only if it is $\pi$-torsion-free; these two notions do not depend on the chosen $\pi\in\frak M^\circ$.

Fix $\varpi\in\frak M^\circ$. We say that $\bb M$ is {\em trivial modulo $\varpi$} if and only if $\Gamma(U,\bb M/\varpi)$ is a finite projective $\Gamma(U,\bb B/\varpi)$-module and the canonical base change map \[\Gamma(U,\bb M/\varpi)\otimes_{\Gamma(U,\bb B/\varpi)}\bb B/\varpi\To\bb M/\varpi\] is an isomorphism. This property does depend crucially on $\varpi$. We will see in Proposition \ref{proposition_small_sheaves}(v) that if $\bb M$ is complete, topologically torsion-free, and trivial modulo some element of $\frak M^\circ$, then it is locally finite projective sheaf of modules.

We will also be interested in the following weaker condition: say that $\bb M$ is {\em trivial modulo $<\varpi$} if and only if it is trivial modulo $\varpi'$ for all $\varpi'\in\frak M$ such that $\varpi\in\varpi'\frak M$ (note this automatically implies $\varpi'\in \frak M^{\circ}$).

A sheaf which is trivial modulo some element of $\frak M^\circ$ in this way may informally be called {\em small}, in agreement with Faltings' terminology for generalised representations \cite{Faltings2005, Faltings2011}. The relation between the two notions will be made precise in Theorem \ref{theorem_correspondence_for_B}.
\end{definition}

We now begin our systematic study of small sheaves:

\begin{proposition}\label{proposition_small_sheaves}
Let $\frak M\subseteq B$, $\bb B$ satisfy hypotheses (B1)--(B4), and let $\bb M$ be a complete, topologically torsion-free sheaf of $\bb B$-modules which is trivial modulo some element of $\frak M^\circ$. Then, for any affinoid perfectoid $\cal U\in U_\sub{pro\'et}$, the following hold:
\begin{enumerate}
\item $H^i_\sub{pro\'et}(\cal U,\bb M)$ is almost zero for all $i>0$, and $\bb M$ is derived $\pi$-adically complete for any $\pi\in\frak M^\circ$.
\item Given any $a,b\in \frak M^\circ$ with $a\in b\frak M^{\circ}$, then the two maps \[\Gamma(\cal U,\bb M)\to\Gamma(\cal U,\bb M/b),\qquad \Gamma(\cal U,\bb M/a)\to\Gamma(\cal U,\bb M/b)\] have the same image, namely $\Gamma(\cal U,\bb M)/b$;
\item $\Gamma(\cal U,\bb M)/\pi$ is finite projective over $\Gamma(\cal U,\bb B)/\pi$ for any $\pi\in \frak M^\circ$.
\item $\Gamma(\cal U,\bb M)\otimes_{\Gamma(\cal U,\bb B)}\bb B/\pi|_\cal U\Isoto \bb M/\pi|_\cal U$ for any $\pi\in \frak M^\circ$.
\item $\Gamma(\cal U,\bb M)$ is a finite projective $\Gamma(\cal U,\bb B)$-module and the canonical map \[\Gamma(\cal U,\bb M)\otimes_{\Gamma(\cal U,\bb B)}\bb B|_\cal U\To\bb M|_\cal U\] is an isomorphism of sheaves on $U_{\sub{pro\'et}/\cal U}$. (Since this holds for all affinoid perfectoids $\cal U$, it follows that $\bb M$ is a locally finite projective sheaf of $\bb B$-modules.)
\end{enumerate}
If $\bb M$ is trivial modulo $\varpi$ and $\Gamma(\cal U,\bb M/\varpi)$ is a finite free $\Gamma(\cal U,\bb B/\varpi)$-module for some $\varpi\in \frak M^\circ$, then we may replace ``projective'' by ``free'' in the claim (iii) and in the first claim of (v).
\end{proposition}
\begin{proof}
(i): Let $\varpi\in\frak M^\circ$ be an element such that $\bb M$ is trivial mod $\varpi$. Since $H^i_\sub{pro\'et}(\cal U,\bb M/\varpi)=\Gamma(U,\bb M/\varpi)\otimes_{\Gamma(U,\bb B/\varpi)}H^i_\sub{pro\'et}(\cal U,\bb B/\varpi)$, these higher cohomologies are almost zero by assumption (B3b); similarly if we replace $\varpi$ by $\varpi^s$, for any $s\ge1$, by an easy induction. Moreover, from the exact sequence $0\to\bb M/\varpi\xto{\times\varpi^s}\bb M/\varpi^{s+1}\to\bb M/\varpi^s\to 0$ (which holds since $\bb M$ is topologically torsion-free) one sees that $\Gamma(\cal U,\bb M/\varpi^{s+1})\to \Gamma(\cal U,\bb M/\varpi^s)$ is almost surjective. Since we have established these properties for an arbitrary affinoid perfectoid (which form a basis of $U_\sub{pro\'et}$), we may now apply the almost version of \cite[Lem.~3.18]{Scholze2013} (as in \cite[Lem.~4.10(v)]{Scholze2013}) to $\bb M=\projlim_s\bb M/\varpi^s$ to deduce that $H^i_\sub{pro\'et}(\cal U,\bb M)$ is almost zero for $i>0$. 

In fact, we have shown that the map $R\Gamma_\sub{pro\'et}(\cal U,\bb M)\to \op{Rlim}_sR\Gamma_\sub{pro\'et}(\cal U,\bb M/\varpi^s)$ is an almost quasi-isomorphism of $B$-modules, and so in particular the cone is derived $\varpi$-adically complete. But the map becomes a quasi-isomorphism upon derived $\varpi$-adic completion, and so this cone is in fact zero. Since this is true for all affinoid perfectoid $\cal U$ (which form a basis of $U_\sub{pro\'et}$), it follows that $\bb M\quis\op{Rlim}_s\bb M/\varpi^s$. Finally, in this quasi-isomorphism we are free to replace $\varpi$ by any other element $\pi$ of $\frak M^\circ$, since the powers of the two elements are intertwined.

(ii): For any $a,b\in \frak M^\circ$ such that $a\in b\frak M^\circ$, the exact sequence $0\to\bb M\xto{\times a/b}\bb M \to\bb M/\tfrac ab\to 0$ and (i) show that $H^1_\sub{pro\'et}(\cal U,\bb M)\into H^1_\sub{pro\'et}(\cal U,\bb M/\tfrac ab)$. Taking cohomology in the commutative diagram
\[\xymatrix{
0\ar[r] & \bb M\ar[r]^b\ar[d]&\bb M\ar[r]\ar[d]&\bb M/b\ar[r]\ar[d]&0\\
0\ar[r] & \bb M/\tfrac ab\ar[r]^b&\bb M/a\ar[r]&\bb M/b\ar[r]&0,
}\]
and performing a simple diagram chase now yields (ii).

(iii): In the rest of the proof, fix an element $\pi\in\frak M$ such that $\varpi\in \pi\frak M$ (because we are in a setting for almost mathematics, such $\pi$ does exist). We first observe that the canonical map \begin{equation}\Gamma( U,\bb M/\varpi)\otimes_{\Gamma( U,\bb B/\varpi)}\Im\big(\Gamma(\cal U,\bb B/\varpi)\to \Gamma(\cal U,\bb B/\pi)\big)\To \Im\big(\Gamma(\cal U,\bb M/\varpi)\to \Gamma(\cal U,\bb M/\pi)\big)\label{eqn_trivial}\end{equation} is an isomorphism: indeed, we may interchange the tensor product and $\Im$ on the left side, then use the smallness hypothesis on $\bb M$ to write it as the right side. Appealing to part (ii), for both $\bb M$ and $\bb B$ itself, we have therefore established a natural isomorphism \[\Gamma( U,\bb M/\varpi)\otimes_{\Gamma( U,\bb B/\varpi)}\Gamma(\cal U,\bb B)/\pi\cong \Gamma(\cal U,\bb M)/\pi,\] whence $\Gamma(\cal U,\bb M)/\pi$ is finite projective over $\Gamma(\cal U,\bb B)/\pi$. It then follows from $\pi$-torsion-freeness of $\Gamma(\cal U,\bb M)$ and $\Gamma(\cal U,\bb B)$ that $\Gamma(\cal U,\bb M)/\pi^s$ is finite projective over $\Gamma(\cal U,\bb B)/\pi^s$ for all $s\ge1$, which implies (iii) since any element of $\frak M^\circ$ divides some power of $\pi$.

(iv): The canonical map $\Gamma(\cal U,\bb M/\varpi)\otimes_{\Gamma(\cal U,\bb B/\varpi)}\bb B/\varpi|_\cal U\to\bb M/\varpi|_\cal U$ of sheaves on $U_{\sub{pro\'et}/\cal U}$ is clearly an isomorphism (by the triviality of $\bb M$ modulo $\varpi$), and similarly if we replace $\varpi$ by $\pi$ (since triviality mod $\varpi$ implies triviality mod $\pi$). Taking the image of the $\varpi$-identification in the version for $\pi$, using (\ref{eqn_trivial}), shows that
\[\Im\big(\Gamma(\cal U,\bb M/\varpi)\to\Gamma(\cal U,\bb M/\pi)\big) \otimes_{\Im(\Gamma(\cal U,\bb B/\varpi)\to\Gamma(\cal U,\bb B/\pi))}\bb B/\pi|_\cal U\Isoto\bb M/\pi|_\cal U.\] Rewriting this using (ii) yields \[\Gamma(\cal U,\bb M)\otimes_{\Gamma(\cal U,\bb B)}\bb B/\pi|_\cal U\Isoto \bb M/\pi|_\cal U.\] A trivial induction again allows us to replace $\pi$ by $\pi^s$ for any $s\ge1$, hence by any element of $\frak M^\circ$. If $\Gamma(\cal U,\bb M/\varpi)$ is finite free over $\Gamma(\cal U,\bb B/\varpi)$, then the same argument as above shows the claim (iii) with ``projective'' replaced by ``free''.

(v): $\Gamma(\cal U,\bb M)$ is $\pi$-adically complete since it identifies with an inverse limit $\projlim_s\Gamma(\cal U,\bb M/\pi^s)$ of $\pi$-power-torsion modules. It is $\pi$-torsion-free since $\bb M$ is topologically torsion-free by assumption, and we already know that $\Gamma(\cal U,\bb M)/\pi$ is finite projective over $\Gamma(\cal U,\bb B)/\pi$ by (iii), whence it follows that $\Gamma(\cal U,\bb M)$ is finite projective over $\Gamma(\cal U,\bb B)$. The desired isomorphism trivially follows by taking a limit of the isomorphisms in (iv) over powers of $\pi$. If $\Gamma(\cal U,\bb M/\varpi)$ is finite free over $\Gamma(\cal U,\bb B/\varpi)$, then $\Gamma(\cal U,\bb M)/\pi^s$ is finite free over $\Gamma(\cal U,\bb B)/\pi^s$ for all $s\geq 1$ by (iii). By taking an inverse limit, we see that $\Gamma(\cal U, \bb M)$ is finite free over $\Gamma(\cal U,\bb B)$.
\end{proof}

We now focus attention on an affinoid perfectoid cover $\cal U$ of $U$ of the form 
\begin{itemize}
\item[($\Pi$)] $\cal U=\projlimf_iU_i$, where each $U_i\to U$ is a finite, Galois \'etale cover with Galois group $G_i$; we put $G:=\projlim_iG_i$.
\end{itemize}
Given any sheaf $\cal F$ on $U_\sub{pro\'et}$, the action of $G$ on $\cal U$ formally induces an action of $G$ on the global sections $\Gamma(\cal U,\cal F)$. If $\cal F$ is discrete, then this action is continuous for the discrete topology on $\Gamma(\cal U,\cal F)$; indeed, $\Gamma(\cal U,\cal F)=\indlim_i\Gamma(U_i,\cal F)$ \cite[Lem.~3.16]{Scholze2013} and the $G$-action on $\Gamma(U_i,\cal F)$ factors through $G_i$ (by the very definition of the action).

\begin{example}\label{example_affinoid_perfectoid_covers}
There are two standard choices for such an affinoid perfectoid cover $\cal U$:
\begin{enumerate}
\item $\cal U=U_\infty$ the usual affinoid perfectoid cover of $U$ associated to a chosen framing $\roi\pid{T_1^{\pm1},\dots,T_d^{\pm 1}}\to R$, in which case $G=\bb Z_p(1)^d$.
\item $\cal U=\res U$ the affinoid perfectoid cover arsing from the filtered colimit of all finite extensions of $R$ which are unramified outside $p$ (inside some fixed algebraic closure of $\Frac R$), in which case $G=\pi_1^\sub{\'et}(\Spec R[\tfrac1p])$; as in Remark \ref{remark_Delta}, here we implicitly assume that $\Spec(R/pR)$ is connected and we omit the base point from $\pi_1^\sub{\'et}$.
\end{enumerate}
\end{example}

Our aim is to prove that suitable generalised $G$-representations correspond exactly to certain sheaves on $U_\sub{pro\'et}$. We first need the following cohomological vanishing, which is a simple consequence of almost purity (B3b):

\begin{lemma}\label{lemma_almost_purity_group}
Let $\frak M\subseteq B$, $\bb B$ satisfy hypotheses (B1)--(B4), and let $\cal U$ be an affinoid perfectoid cover of $U$ as in ($\Pi$). Let $M$ be a complete, topologically torsion-free\footnote{This means $\pi$-adically complete and $\pi$-torsion-free for any/all $\pi\in\frak M^\circ$.} $\Gamma(\cal U,\bb B)$-module equipped with a continuous, semi-linear $G$-action which is trivial modulo some element of $\frak M^\circ$. Then, for any affinoid perfectoid $V\in U_\sub{pro\'et}$, the continuous group cohomology $H^i_\sub{cont}(G,M\otimes_{\Gamma(\cal U,\bb B)}\Gamma(\cal U\times_UV,\bb B))$ almost vanishes for $i>0$. 
\end{lemma}
\begin{proof}
Although we could not find a reference, the standard argument (which essentially appears in step 2 of the proof of Theorem \ref{theorem_correspondence_for_B}) shows that there is a Hochschild--Serre identification $R\Gamma_\sub{\'et}(Y,\cal F)=R\Gamma(\Gal(X/Y),R\Gamma_\sub{\'et}(X,\cal F))$ for any finite Galois \'etale cover $X\to Y$ of Noetherian adic spaces and sheaf $\cal F$ on $Y_\sub{\'et}$. In particular, writing $V=\projlimf_j V_j$, we have $R\Gamma_\sub{\'et}(V_j,\cal F)=R\Gamma(G_i,R\Gamma_\sub{\'et}(U_i\times_UV_j,\cal F))$ for any sheaf $\cal F$ on $U_\sub{\'et}$. By taking the filtered colimit over $i$ and $j$, this yields \[R\Gamma_\sub{pro\'et}(V,\cal F)=R\Gamma_\sub{cont}(G,R\Gamma_\sub{pro\'et}(\cal U\times_UV,\cal F))\] for any discrete sheaf $\cal F$ on $U_\sub{pro\'et}$.

We apply the previous Hochschild--Serre identification to $\cal F=\bb B/\pi$ for each $\pi\in\frak M^\circ$. But the higher cohomologies of $R\Gamma_\sub{pro\'et}(V,\bb B/\pi)$ and $R\Gamma_\sub{pro\'et}(\cal U\times_UV,\bb B/\pi)$ almost vanish, by hypothesis (B3b) since $V$ and $\cal U\times_UV$ are affinoid perfectoid; so we deduce that the canonical map \[\Gamma(V,\bb B/\pi)\To R\Gamma_\sub{cont}(G,\Gamma(\cal U\times_UV,\bb B/\pi))\tag{\dag}\] is an almost quasi-isomorphism. Hence $H^i_\sub{cont}(G,\Gamma(\cal U\times_UV,\bb B/\pi))$ is almost zero for $i>0$.

From the almost isomorphism $\Gamma(\cal U\times_UV,\bb B)/\pi\to \Gamma(\cal U\times_UV,\bb B/\pi)$ of (B3b), we may then deduce that $H^i_\sub{cont}(G,\Gamma(\cal U\times_UV,\bb B)/\pi)$ is almost zero for $i>0$. We now assume that $\pi$ was chosen small enough so that $M$ is trivial modulo $\pi$; then clearly  $H^i_\sub{cont}(G,M\otimes_{\Gamma(\cal U,\bb B)}\Gamma(\cal U\times_UV,\bb B)/\pi)$ is also almost zero for $i>0$. Note that $M$ is a finite projective $\Gamma(\cal U,\bb B)$-module by the assumptions on $M$. But then by a trivial induction we may replace $\pi$ by any power $\pi^s$ and complete the proof by letting $s\to\infty$ (using $R\Gamma_\sub{cont}(G,M\otimes_{\Gamma(\cal U,\bb B)}\Gamma(\cal U\times_UV,\bb B))=\op{Rlim}_sR\Gamma_\sub{cont}(G,M\otimes_{\Gamma(\cal U,\bb B)}\Gamma(\cal U\times_UV,\bb B)/\pi^s)$.

\end{proof}

We are now prepared to establish the following equivalence of categories. The target category in the following theorem consists, by definition, of finite projective $\Gamma(\cal U,\bb B)$-modules $M$ equipped with a semi-linear, continuous action by $G$ which is trivial modulo $<\varpi$ in the sense that it is trivial modulo $\varpi'$ (in the usual sense of Definition \ref{definition_trivial}) for all $\varpi'\in\frak M$ such that $\varpi\in\varpi'\frak M$.

\begin{theorem}\label{theorem_correspondence_for_B}
Let $\frak M\subseteq B$, $\bb B$ satisfy hypotheses (B1)--(B4), and let $\cal U\in U_\sub{pro\'et}$ be an affinoid perfectoid cover of $U$ as in ($\Pi$). Fix $\varpi\in\frak M^\circ$. Then taking global sections induces an equivalence of categories
\begin{align*}
\categ{6cm}{locally finite projective sheaves of $\bb B$-modules on $U_\sub{pro\'et}$ trivial modulo $<\varpi$}&\Isoto\Rep^{<\varpi}_G(\Gamma(\cal U,\bb B))\\ 
\bb M&\mapsto \Gamma(\cal U,\bb M)
\end{align*}
Moreover, for each such $\bb M$, there is a natural almost quasi-isomorphism of complexes of $\Gamma(U,\bb B)$-modules \[R\Gamma_\sub{cont}(G,\Gamma(\cal U,\bb M))\To R\Gamma_\sub{pro\'et}(U,\bb M).\]
\end{theorem}
\begin{proof}
Step 1: We begin by checking that the functor is well-defined. In light of Proposition \ref{proposition_small_sheaves}(v), it remains to show that the $G$-action on $\Gamma(\cal U,\bb M)$ is continuous and trivial mod $<\varpi$; we first deal with the latter. Given any elements $\pi,\pi'\in\frak M$ such that $\pi'\in\pi\frak M$ and $\varpi\in\pi'\frak M$, we saw in the course of proving Proposition \ref{proposition_small_sheaves}(iii) that $\Gamma( U,\bb M/\pi')\otimes_{\Gamma( U,\bb B/\pi')}\Gamma(\cal U,\bb B)/\pi\cong \Gamma(\cal U,\bb M)/\pi$. This is of course compatible with the $G$-actions on each side, where $G$ acts trivially on the finite projective $\Gamma(U,\bb B/\pi')$-module $\Gamma(U,\bb M/\pi')$, and so shows that the semi-linear $G$-action on the $\Gamma(\cal U,\bb B)$-module $\Gamma(\cal U,\bb M)$ is indeed trivial mod $\pi$. This being true for all such $\pi$, we have established the desired triviality mod $<\varpi$.

Next we check continuity. We must show that the $G$-action on $\Gamma(\cal U,\bb M)/\pi$ is continuous (for the discrete topology) for all $\pi\in\frak M^\circ$; noting that $\Gamma(\cal U,\bb M)/\pi\into \Gamma(\cal U,\bb M/\pi)$, it is sufficient to prove continuity of the action on the latter. But $\bb M/\pi$ is a discrete sheaf (this follows from hypothesis (B4) if $\pi$ is small enough so that $\bb M$ is trivial modulo $\pi$, then follows for general $\pi$ by the same argument as explained in (B4)), so the continuity of the $G$-action on its sections was already explained just before Example \ref{example_affinoid_perfectoid_covers}. This completes the proof that the functor of the theorem is well-defined.

Step 2: We construct the inverse to the functor. We first study the situation modulo $\pi$, for each $\pi\in\frak M^\circ$, then will take the limit; to this end let $\bb D$ be a discrete sheaf of rings of $U_\sub{pro\'et}$, e.g., $\bb D=\bb B/\pi$. Given a $\Gamma(\cal U,\bb D)$-module $M$ equipped with a continuous (wrt.~the discrete topology on $M$), semi-linear $G$-action, we form a presheaf $\tilde M$ of $\bb D$-modules on $U_\sub{pro\'et}$ via the rule \[U_\sub{pro\'et}\ni V\mapsto \tilde M(V):=(M\otimes_{\Gamma(\cal U,\bb D)}\Gamma(\cal U\times_U V,\bb D))^G.\] Here $G$ is acting diagonally on the right (note that $G$ acts $V$-invariantly on $\cal U\times_U V$, hence the diagonal action makes sense and the fixed points are a $\Gamma(V,\bb D)$-module). We claim that this construction forms a well-defined functor
\begin{align*}
\categ{6cm}{flat $\Gamma(\cal U,\bb D)$-modules with continuous, semi-linear $G$-action}&\To\categ{6cm}{discrete sheaves of $\bb D$-modules on $U_\sub{pro\'et}$}\\ 
M&\mapsto \tilde M
\end{align*}
with left inverse $\Gamma(\cal U,-)$ (here ``continuous'' is again wrt.\ the discrete topology on $M$).

First note that $\tilde M$ is actually a sheaf: if $ V'\to V$ is any cover in $U_\sub{pro\'et}$, then the exact sequence \[0\To \Gamma(\cal U\times_U V,\bb D)\To \Gamma(\cal U\times_U V',\bb D)\rightrightarrows\Gamma(\cal U\times_U V'\times_{ V} V',\bb D)\] remains exact upon applying the left exact functor $(M\otimes_{\Gamma(\cal U,\bb D)}-)^G$. Also, the sheaf $\tilde M$ is easily seen to be discrete, by writing $V=\projlimf_jV_j$ and noting that $\Gamma(\cal U\times_UV,\bb D)=\indlim_j\Gamma(\cal U\times_UV_j,\bb D)$ by the discreteness of $\bb D$.

Finally, assuming that the action of $G$ on $M$ is continuous, we claim that $M=\Gamma(\cal U,\tilde M)$. More generally, for any $V\in U_\sub{pro\'et}$ lying over $\cal U$ we will show that $\Gamma(V,\tilde M)=M\otimes_{\Gamma(\cal U,\bb D)}\Gamma(V,\bb D)$. Firstly, the continuity hypothesis means that $M=\indlim_i M^{H_i}$ where $H_i:=\ker(G\to G_i)$; similarly, discreteness of the sheaf $\bb D$ implies that $\Gamma(\cal U\times_UV,\bb D)=\indlim_i\Gamma(U_i\times_UV,\bb D)$. We thus obtain the second of the following identities: \[\Gamma(V,\tilde M)=(M\otimes_{\Gamma(\cal U,\bb D)}\Gamma(\cal U\times_U V,\bb D))^G=\indlim_i(M^{H_i}\otimes_{\Gamma(U_i,\bb D)}\Gamma(U_i\times_U V,\bb D))^{G_i}\] But $U_i\to U$ is a finite, Galois \'etale cover, and so $U_i\times_UU_i=G_i\times U_i$, whence $U_i\times_UV=G_i\times V$; taking sections reveals that $\Gamma(U_i\times_UV,\bb D)\cong G_i\times\Gamma(V,\bb D)$ as a $G_i$-module (with $G_i$ acting trivially on $\Gamma(V,\bb D)$), whence the $G_i$-invariants on the right side of the previous line equals $M^{H_i}\otimes_{\Gamma(U_i,\bb D)}\Gamma(V,\bb D)$. Taking $\indlim_i$ yields $M\otimes_{\Gamma(U,\bb D)}\Gamma(V,\bb D)$, as required to prove the claim.

Step 3: By taking the limit, we will now check that the following functor is well-defined and has left inverse $\Gamma(\cal U,-)$:
\begin{align*}
\Rep^{<\varpi}_G(\Gamma(\cal U,\bb B))&\To\categ{6cm}{locally finite projective sheaves of $\bb B$-modules on $U_\sub{pro\'et}$ trivial modulo $<\varpi$}\\
M&\mapsto \tilde M:V\mapsto \big(M\otimes_{\Gamma(\cal U,\bb B)}\Gamma(\cal U\times_UV,\bb B)\big)^G
\end{align*}
(Hopefully the same notation $\tilde{}$ will not cause confusion.) In fact, given $\pi\in\frak M^\circ$ and setting $M_s:=M\otimes_{\Gamma(\cal U,\bb B)}\Gamma(\cal U,\bb B/\pi^s)$ we have \[\big(M\otimes_{\Gamma(\cal U,\bb B)}\Gamma(\cal U\times_UV,\bb B)\big)^G= \projlim_s(M_s\otimes_{\Gamma(\cal U,\bb B/\pi^s)}\Gamma(\cal U\times_UV,\bb B/\pi^s))^G\] for all $V\in U_\sub{pro\'et}$. So $\tilde M$ is an inverse limit of the discrete constructions considered above, whence we immediately see that it is a sheaf with $\Gamma(\cal U,\tilde M)=\projlim_s M_s=M$ (for the second identity, we use the isomorphism of pro systems $\{\Gamma(\cal U,\bb B)/\pi^s)\}_s\isoto \{\Gamma(\cal U,\bb B/\pi^s))\}_s$ furnished by Proposition \ref{proposition_small_sheaves}(ii)). Note also that the sections $\Gamma(V,\tilde M)$ are $\pi$-torsion-free since $\bb B$ is $\pi$-torsion-free and $M$ is flat over $\Gamma(\cal U,\bb B)$; so $\tilde M$ is $\pi$-torsion-free.

Next we show that $\tilde M$ is trivial mod $<\varpi$. For this we first need some analogues for group cohomology of the arguments of Propositon \ref{proposition_small_sheaves}(ii)\&(iii). Let $V\in U_\sub{pro\'et}$ be any affinoid perfectoid and let $\pi\in\frak M$ be any element such that $\varpi\in\pi\frak M$; then let $\pi'$ be an element strictly between them, i.e., $\varpi\in\pi'\frak M$ and $\pi'\in\pi\frak M$.  From the $H^1$ almost vanishing of Lemma \ref{lemma_almost_purity_group} (and arguing exactly as in the proof of Proposition \ref{proposition_small_sheaves}(ii)) we deduce that the image of \[(M\otimes_{\Gamma(\cal U,\bb B)}\Gamma(\cal U\times_U V,\bb B)/\pi')^G\To (M\otimes_{\Gamma(\cal U,\bb B)}\Gamma(\cal U\times_U V,\bb B)/\pi)^G\] identifies with $(M\otimes_{\Gamma(\cal U,\bb B)}\Gamma(\cal U\times_U V,\bb B))^G/\pi$. But, recalling that $M$ is trivial mod $\pi'$, the previous line may be rewritten as \[(M/\pi')^G\otimes_{(\Gamma(\cal U,\bb B)/\pi')^G}(\Gamma(\cal U\times_U V,\bb B)/\pi')^G\To (M/\pi')^G\otimes_{(\Gamma(\cal U,\bb B)/\pi')^G}(\Gamma(\cal U\times_U V,\bb B)/\pi)^G,\] whose image is \[(M/\pi')^G\otimes_{(\Gamma(\cal U,\bb B)/\pi')^G}\Gamma(\cal U\times_UV,\bb B)^G/\pi\] (by applying the previous image identity to $M=\Gamma(\cal U,\bb B)$). But, assuming in addition that $V$ admits a morphism to $\cal U$, then $\Gamma(\cal U\times_UV,\bb B)^G=\Gamma(V,\bb B)$ (replacing $\bb B$ by $\bb D=\bb B/\pi^s$, this was shown in the final paragraph of the proof of Step 2; then let $s\to\infty$; note that the equality is given by the pull-back by the projection $\cal U\times_UV\to V$, which does not involve $\cal U\to V$); so, assembling the results of this paragraph, we have produced a identification functorial in $V$ \[(M\otimes_{\Gamma(\cal U,\bb B)}\Gamma(\cal U\times_U V,\bb B))^G/\pi\cong (M/\pi')^G\otimes_{(\Gamma(\cal U,\bb B)/\pi')^G}\Gamma(V,\bb B)/\pi.\] By now varying $V$ over all affinoid perfectoids lying over $\cal U$ (which form a basis of $U_\sub{pro\'et}$, since $\cal U$ is a cover of $U$), this corresponds to an isomorphism of sheaves \[\tilde M/\pi \cong (M/\pi')^G\otimes_{(\Gamma(\cal U,\bb B)/\pi')^G}\bb B/\pi.\] The right side is clearly a trivial sheaf of $\bb B/\pi$-modules, thereby completing the proof that $\tilde M$ is trivial mod $\pi$.

To prove that our functor $M\mapsto\tilde M$ is well-defined, it remains to show that $\tilde M$ is complete. Let $V\in U_\sub{pro\'et}$ be affinoid perfectoid and $\pi\in\frak M^\circ$. Proposition \ref{proposition_small_sheaves}(ii) implies that the canonical map $\Gamma(\cal U\times_UV,\bb B/\pi^{s+2})\to \Gamma(\cal U\times_UV,\bb B/\pi^{s})$ factors through $\Gamma(\cal U\times_UV,\bb B)/\pi^{s+1}$; this factoring formally remains valid after applying $(M\otimes_{\Gamma(\cal U,\bb B)}-)^G$. But, similarly to the previous paragraph, the $H^1$ almost vanishing of Lemma \ref{lemma_almost_purity_group} then implies that the resulting map $(M\otimes_{\Gamma(\cal U,\bb B)}\Gamma(\cal U\times_UV,\bb B)/\pi^{s+1})^G \to(M\otimes_{\Gamma(\cal U,\bb B)}\Gamma(\cal U\times_UV,\bb B)/\pi^{s})^G$ has image $(M{\otimes_{\Gamma(\cal U,\bb B)}}\Gamma(\cal U\times_UV,\bb B))^G/\pi^s$. Concatenating shows that the map
\[(M\otimes_{\Gamma(\cal U,\bb B)}\Gamma(\cal U\times_UV,\bb B/\pi^{s+2}))^G \to(M\otimes_{\Gamma(\cal U,\bb B)}\Gamma(\cal U\times_UV,\bb B/\pi^{s}))^G\] has image $(M\otimes_{\Gamma(\cal U,\bb B)}\Gamma(\cal U\times_UV,\bb B))^G/\pi^s$. Since this holds for all affinoid perfectoid $V$, it follows that the canonical maps of sheaves $\tilde M/\pi^s\to (M\otimes_{\Gamma(\cal U,\bb B)}\Gamma(\cal U\times_U -,\bb B/\pi^s))^G$ become an isomorphism of pro sheaves when we let $s\to\infty$. In particular, it becomes an isomorphism when we apply $\projlim_s$, thereby proving that $\tilde M\isoto\projlim_s\tilde M/\pi^s$ (since $\tilde M$ is by definition $\projlim_s$ of the right sides).

Step 4: To complete the proof of the categorical equivalence of the theorem, it remains to show that $\tilde{\Gamma(\cal U,\bb M)}=\bb M$, naturally for $\bb M$ in the left category.  The canonical maps \[(\Gamma(\cal U,\bb M)\otimes_{\Gamma(\cal U,\bb B)}\Gamma(\cal U\times_UV,\bb B))^G\To \Gamma(\cal U\times_U V, \bb M)^G=\Gamma(V,\bb M)\] (for the equality, see the parenthetical comment ``replacing $\bb B$ by...'' in Step 3), for $V\in U_\sub{pro\'et}$ admitting a morphism to $\cal U$, induce a morphism $i:\tilde{\Gamma(\cal U,\bb M)}\to \bb M$ in the left category. We claim $i$ becomes an isomorphism after applying $\Gamma(\cal U,-)$: since we already know that $\Gamma(\cal U,-)$ is a left inverse of $\tilde{\quad}$ (i.e., the first isomorphism in the following line), it is enough to show that the composition
\[\Gamma(\cal U,\bb M)\isoto (\Gamma(\cal U,\bb M)\otimes_{\Gamma(\cal U,\bb B)}\Gamma(\cal U\times_U\cal U,\bb B))^G\To \Gamma(\cal U\times_U \cal U, \bb M)^G\stackrel{\op{pr}_2^*}=\Gamma(\cal U\bb M)\] is the identity; but this is clear since the inverse of the first map (constructed in Step 3) is induced by pulling back along the diagonal $\Gamma(\cal U\times_U\cal U,\bb B)\to\Gamma(\cal U,\bb B)$. The claim shows that $i$ becomes an isomorphism after restricting to $\cal U$, by Proposition \ref{proposition_small_sheaves}(v), hence is an isomorphism since $\cal U$ is a cover of $U$.

Step 5: We establish the group cohomology almost quasi-isomorphism, using an Hochschild--Serre argument similar to the proof of Lemma \ref{lemma_almost_purity_group}. Firstly, since each $\bb M/\pi^s$ is discrete, we have natural Hochschild--Serre identifications $R\Gamma_\sub{pro\'et}(U,\bb M/\pi^s)\simeq R\Gamma(G_i,R\Gamma_\sub{pro\'et}(U_i,\bb M/\pi^s))$; taking $\indlim_i$ yields $R\Gamma_\sub{pro\'et}(U,\bb M/\pi^s)\simeq R\Gamma_\sub{cont}(G,R\Gamma_\sub{pro\'et}(\cal U,\bb M/\pi^s))$, whence an almost quasi-isomorphism $R\Gamma_\sub{pro\'et}(U,\bb M/\pi^s)\stackrel{\sub{al}}{\simeq} R\Gamma_\sub{cont}(G,\Gamma(\cal U,\bb M)/\pi^s)$ since the higher pro-\'etale cohomologies of $\bb M$ on $\cal U$ are almost zero by Proposition \ref{proposition_small_sheaves}(i). Finally, taking $\op{Rlim}_s$ (and using derived $\pi$-adic completeness of $\bb M$, also by Proposition \ref{proposition_small_sheaves}(i)) yields the desired almost quasi-isomorphism $R\Gamma_\sub{pro\'et}(U,\bb M)\stackrel{\sub{al}}\simeq R\Gamma_\sub{cont}(G,\Gamma(\cal U,\bb M))$.
\end{proof}

We finish this subsection with a lemma which will be required soon in order to extend the above theory to sheaves of modules over $\bb A_\sub{inf}$:

\begin{lemma}\label{lemma_trivial_mod_all_implies_trivial}
Let $\frak M\subseteq B$, $\bb B$ satisfy hypotheses (B1)--(B4), and in addition
\begin{itemize}
\item[(B5)] given $\pi\in\frak M^\circ$, the $\pi$-power torsion in $H^1_\sub{pro\'et}(U,\bb B)$ is killed by $\pi^N$ for $N\gg0$ )this boundedness condition is of course independent of the chosen $\pi$).
\end{itemize}
Let $\bb M$ be a locally finite projective sheaf of $\bb B$-modules which is trivial modulo every element of $\frak M^\circ$. Then $\bb M$ is trivial in the sense that $\Gamma(U,\bb M)$ is a finite projective $\Gamma(U,\bb B)$-module and the canonical map $\Gamma(U,\bb M)\otimes_{\Gamma(U,\bb B)}\bb B\to\bb M$ is an isomorphism.
\end{lemma}
\begin{proof}
Fix $\pi\in\frak M^\circ$ and let $s\ge1$. Hypothesis (B5) implies that the canonical map $H^1(U,\bb B)[\pi^s]\to H^1(U,\bb B)/\pi^t$ is injective for $t\gg s$, whence the composition $H^1(U,\bb B)[\pi^s]\to H^1(U,\bb B/\pi^t)$ is also injective. An analogous diagram chase as to that used in the proof of Proposition \ref{proposition_small_sheaves}(ii) then shows that $\Gamma(U,\bb B)$ and $\Gamma(U,\bb B/\pi^t)$ have the same image in $\Gamma(U,\bb B/\pi^s)$, namely $\Gamma(U,\bb B)/\pi^s$; in particular, the image is independent of $t\gg s$.

We now extend some of the results of the previous paragraph to the sheaf $\bb M$. For $t_2\ge t_1\ge s$ the triviality of $\bb M$ modulo $\pi^{t_2}$ states that $\bb M/\pi^{t_2}=\Gamma(U,\bb M/\pi^{t_2})\otimes_{\Gamma(U,\bb B/\pi^{t_2})}\bb B/\pi^{t_2}$, whence \[\Gamma(U,\bb M/\pi^{t_1})=\Gamma(U,\bb M/\pi^{t_2})\otimes_{\Gamma(U,\bb B/\pi^{t_2})}\Gamma(U,\bb B/\pi^{t_1}).\] Assuming that $t_1\gg s$, the previous paragraph therefore implies that $\Gamma(U,\bb M/\pi^{t_1})$ and $\Gamma(U,\bb M/\pi^{t_2})$ have the same image in $\Gamma(U,\bb M/\pi^s)$, given by $\Gamma(U,\bb M/\pi^{t_1})\otimes_{\Gamma(U,\bb B/\pi^{t_1})}\Gamma(U,\bb B)/\pi^s$. Taking $s=1$, fixing $t_1\gg 1$, and writing $t=t_2$, we have therefore proved that the pro-abelian group $\{\Gamma(U,\bb M/\pi^t)\}_t$ satisfies the Mittag--Leffler condition and constructed an exact sequence of pro abelian groups \[0\To \{\Gamma(U,\bb M/\pi^t)\}_t\xto{\pi} \{\Gamma(U,\bb M/\pi^t)\}_t\To \Gamma(U,\bb M/\pi^{t_1})\otimes_{\Gamma(U,\bb B/\pi^{t_1})}\Gamma(U,\bb B)/\pi\To 0.\] Taking the inverse limit as $t\to\infty$ then shows that \[\Gamma(U,\bb M)/\pi=\Gamma(U,\bb M/\pi^{t_1})\otimes_{\Gamma(U,\bb B/\pi^{t_1})}\Gamma(U,\bb B)/\pi,\] which is a finite projective $\Gamma(U,\bb B)/\pi$-module. Since $\Gamma(U,\bb M)$ is $\pi$-torsion-free and $\pi$-adically complete, it follows that it is a finite projective $\Gamma(U,\bb B)$-module. Moreover, the canonical map $\Gamma(U,\bb M)\otimes_{\Gamma(U,\bb B)} \bb B\to\bb M$ identifies mod $\pi$ with $\Gamma(U,\bb M/\pi^{t_1})\otimes_{\Gamma(U,\bb B/\pi^{t_1})} \bb B/\pi\to\bb M/\pi$, which is indeed an isomorphism by triviality mod $\pi^{t_1}$; therefore the canonical map in the statement of the lemma an isomorphism modulo any power of $\pi$ by induction, hence is an isomorphism by passage to the limit.
\end{proof}

\begin{remark}\label{remark_p_vs_[p]}
Hypothesis (B5) is satisfied for $\bb B=W_r(\hat\roi_X^+)$. First note that $[\zeta_{p^r}]-1\in W_r(\frak m)^\circ$: indeed, powers of this element are intertwined with powers of $p$ \cite[Rmk.~3]{Morrow_BMSnotes_simons}, which are intertwined with powers of $[p]\in W_r(\frak m)^\circ$ \cite[Lem.~3.2(ii)]{BhattMorrowScholze1}. So it is enough to show that  $H^1_\sub{pro\'et}(U,W_r(\hat\roi_X^+))$ modulo its $[\zeta_{p^r}]-1$-torsion is $p$-torsion-free. But this quotient identifies with $L\eta_{[\zeta_{p^r}]-1}R\Gamma_\sub{pro\'et}(U,W_r(\hat\roi_X^+))$ by a basic property of the $L\eta$ functor \cite[Lem.~6.4]{BhattMorrowScholze1}; the latter is then isomorphic to the $p$-adic completion of $W_r\Omega_{R/\roi}^1$ by the Cartier--Hodge--Tate comparison \cite[Thms.~9.4(ii) \& 11.1]{BhattMorrowScholze1}. Finally note that the $p$-adic completion of $W_r\Omega_{R/\roi}^1$ is $p$-torsion-free: by \'etale base change \cite[Lem.~10.4]{BhattMorrowScholze1} this reduces to the $p$-torsion-freeness of $W_r\Omega_{\roi[\ul T^{\pm1}]/\roi}^1$, which follows from its explicit decomposition \cite[Thm.~10.12]{BhattMorrowScholze1}.

On the other hand, hypothesis (B5) is not satisfied for $W_r(\hat\roi_{X^{\flat}}^{+})$.
\end{remark}

\subsection{Small sheaves of modules over $\bb A_{\inf}$}\label{subsection_localII}
Now we begin our study of analogous sheaves over $\bb A_{\sub{inf},U}$, where we adopt the following analogue of the smallness condition from Definition \ref{trivial_modulo_pi}:

\begin{definition}\label{definition_trivial_mod_mu}
Let $\bb M$ be a sheaf of $\bb A_{\sub{inf},U}$-modules on $U_\sub{pro\'et}$. We say that it is {\em trivial modulo $<\mu$} if, for each $r\ge1$, the $\Gamma(U,\bb A_{\sub{inf},U}/\xi_r)$-module $\Gamma(U,\bb M/\xi_r)$ is finite projective and the counit \[\Gamma(U,\bb M/\xi_r)\otimes_{\Gamma(U,\bb A_{\sub{inf},U}/\xi_r)}\bb A_{\sub{inf},U}/\xi_r\To \bb M/\xi_r\] is an isomorphism.

The category of locally finite free $\bb A_{\sub{inf},U}$-modules which are trivial modulo $<\mu$ will be denoted by $\BKF(R)$ and called {\em Breuil--Kisin--Fargues modules without Frobenius} over $R$.


\end{definition}

\begin{remark}\label{remark_BKF_global_local}
Before continuing, we check that Definition \ref{definition_trivial_mod_mu} agrees with Definition \ref{definition_BKF_global} in our local set-up; namely, we show that \[\BKF(\Spf R)=\BKF(R).\] To do this, it is enough to check that locally finite free $\nu_*(\bb A_{\inf,U}/\xi_r)$-modules on $\Spf R$ correspond to finite projective modules over $\Gamma(\Spf R,\nu_*(\bb A_{\inf,U}/\xi_r))=\Gamma(U,\bb A_{\inf,U}/\xi_r)$ by taking global sections and that the inverse correspondence is given by taking $-\otimes_{\Gamma(U,\bb A_{\inf,U}/\xi_r)}\nu_*(\bb A_{\inf,U}/\xi_r)$. But this is the usual description of locally finite free sheaves on an affine formal scheme, recalling that $\Spf R$ and $\Spf W_r(R)$ have the same underlying topological spaces  and that $\theta:\nu_*(\bb A_{\inf,U}/\xi_r)\isoto W_r(\roi_{\frak X})=\roi_{W_r(R)}$.
\end{remark}

In the following lemma we check the role of the usual elements $\xi_r$ in almost mathematics over $W_r(\roi^\flat)=\bb A_\sub{inf}/p^r$. In particular, part (ii) shows that Definition \ref{definition_trivial_mod_mu} modulo $p^r$ recovers Definition \ref{trivial_modulo_pi}.
 
\begin{lemma}\label{lemma_<mu_mod_p}
Fix $r\ge1$ and write $\res\xi_s,\res\mu\in W_r(\roi^\flat)$ for the image of the elements $\xi_s,\mu\in A_\sub{inf}$, for $s\ge1$.
\begin{enumerate}
\item $\phi^m(\res\xi_s)\in W_r(\frak m^\flat)^\circ$ for any $s\ge r$ and $m\in\bb Z$.
\item Given an element $\pi\in W_r(\frak m^\flat)$, then $\res\mu\in \pi W_r(\frak m^\flat)$ if and only if $\res\xi_s\in\pi W_r(\frak m^\flat)$ for some $s\ge r$ (equivalently, for all $s\gg r$).
\end{enumerate}
\end{lemma}
\begin{proof}
(i) Since $\phi$ is an automorphism of $W_r(\roi^\flat)$ and $W_r(\frak m^\flat)$ we immediately reduce to the case $m=0$.

We will use that $W_r(\frak m^\flat)$ is the increasing union of the principal ideals generated by the non-zero-divisors $\phi^{-k}(\res\mu)$ for $k\ge0$; note also that $\phi^{-k}(\res \mu)^{p^k}\equiv \res\mu$ mod $p$, whence $\phi^{-k}(\res \mu)^{p^{k}r}\in \res\mu W_r(\roi^\flat)$ since $p^r=0$ in $W_r(\roi^\flat)$. This shows that $\res\mu\in W_r(\frak m^\flat)^\circ$, and therefore $\phi^{-s}(\res\mu)\in W_r(\frak m^\flat)^\circ$ for any $s\ge 1$ since $\phi$ is an automorphism. Since $\xi_s\equiv p^s$ mod $\phi^{-s}(\mu)$ we deduce that $\res\xi_s$ is a multiple of $\phi^{-s}(\res\mu)$ for any $s\geq r$. But conversely we also know that $\res\xi_s\phi^{-s}(\res\mu)=\res \mu$.

(ii) Since $\res\mu$ is divisible by $\res\xi_s$, the implication $\Leftarrow$ is trivial. We prove the converse by induction on $r\ge1$, the case $r=1$ being an easy valuation argument. To proceed by induction, note that the exact sequence \[0\To\frak m^\flat\xto{\times p^{r-1}}W_r(\frak m^\flat)\xTo{R} W_{r-1}(\frak m^\flat)\To 0\] of $W_r(\roi^\flat)$-modules induces an exact sequence \[0\To\frak m^\flat/R^{r-1}(\pi)\frak m^\flat\xto{f}W_r(\frak m^\flat)/\pi\xTo{R} W_{r-1}(\frak m^\flat)/R(\pi)\To 0\] since $R(\pi)$ is a non-zero-divisor of $W_{r-1}(\roi^\flat)$ (as it divides the image of $\mu$ in $W_{r-1}(\roi^\flat)$). By induction we may therefore write $f(x)\equiv\res\xi_s$ mod $\pi W_r(\frak m^\flat)$ for some $x\in\frak m^\flat$ and $s\ge r-1$. Multiplying this identity by $\phi^{-s}(\res \mu)$, and using $\phi^{-s}(\res\mu)\res\xi_s=\res\mu\in\pi W_r(\frak m^\flat)$, we see that $\phi^{-s}(\res\mu)f(x)=0$ in $W_r(\frak m^\flat)/\pi$, i.e., $\phi^{-s}(\mu)x\in R^{r-1}(\pi)\frak m^\flat$. But the case $r=1$ then implies $\phi^{-s}(\xi_{s'})x\in R^{r-1}(\pi)\frak m^\flat$ for some $s'\ge 1$. In conclusion $\xi_{s+s'}=\phi^{-s}(\xi_{s'})\xi_s\equiv 0$ mod $\pi W_r(\frak m^\flat)$, as desired.
\end{proof}

We are now prepared to establish an analogue of Proposition \ref{proposition_small_sheaves} over $\bb A_\sub{inf}$:

\begin{proposition}\label{proposition_small_Ainf_sheaves}
Let $\bb M$ be a sheaf of $\bb A_{\sub{inf},U}$-modules. Then the following are equivalent:
\begin{enumerate}[(a)]
\item $\bb M$ is locally finite projective and trivial modulo $<\mu$;
\item it is possible to write $\bb M=\projlim_s\bb M_s$, where each $\bb M_s$ is a locally finite projective sheaf of $\bb A_{\sub{inf},U}/p^s$-modules which is trivial modulo $<\mu$(in the sense of Definition \ref{trivial_modulo_pi} for Example \ref{example_B}(ii)), and where the transition maps $\bb M_s\to\bb M_{s-1}$ are required to satisfy $\bb M_s/p^{s-1}\isoto\bb M_{s-1}$ for all $s>1$.
\end{enumerate}
Suppose that $\bb M$ satisfies these equivalent conditions and let $\cal U\in U_\sub{pro\'et}$ be any affinoid perfectoid. Then:
\begin{enumerate}
\item $H^i_\sub{pro\'et}(\cal U,\bb M)$ is killed by $[\frak m^\flat]$ for $i>0$.
\item $H^1_\sub{pro\'et}(\cal U,\bb M)$ has no non-zero elements killed by either $p$ or $\xi_r$, any $r\ge1$.
\item $\Gamma(\cal U,\bb M)$ is a finite projective $\Gamma(\cal U,\bb A_{\sub{inf},U})$-module and the canonical map \[\Gamma(\cal U,\bb M)\otimes_{\Gamma(\cal U,\bb A_{\sub{inf},U})}\bb A_{\sub{inf},U}|_{\cal U}\To \bb M|_{\cal U}\] is an isomorphism.
\item If $\nu_*(\bb M/\xi)$ is a finite free $\nu_*(\bb A_{\inf,U}/\xi)$-module, then $\bb M\vert_{\cal U}$ is a finite free $\bb A_{\inf,U}\vert_{\cal U}$-module.
\end{enumerate}
In particular, any $\bb M$ satisfying condition (a) is a locally finite free sheaf of $\bb A_{\inf,U}$-modules (i.e., in the definition of relative Breuil--Kisin--Fargues modules, we could replace the condition that $\bb M$ be locally finite free by the a priori weaker condition that it is only locally finite projective).
\end{proposition}
\begin{proof}
The implication (a)$\Rightarrow$(b) follows from the $p$-adic completeness of $\bb M$ (the sheaf $\bb A_{\sub{inf},U}$ is $p$-adically complete, whence any locally finite free module over it is also) and the resulting presentation $\bb M=\projlim_s\bb M/p^s$: note that triviality of $\bb M$ modulo $<\mu$ (in the sense of Definition \ref{definition_trivial_mod_mu}) implies triviality of $\bb M/p^s$ modulo $<\mu$ mod $p^s$ (in the sense of Definition \ref{trivial_modulo_pi}), thanks to Lemma \ref{lemma_<mu_mod_p}(ii).

Henceforth assume (b). We will simultaneously establish (a) and the properties (i)--(iii). The main step to show (a) is the following claim: given integers $s\ge r\ge r_0\ge 1$, we claim that there is a natural identification \[\Gamma(\cal U,\bb M_r)\otimes_{\Gamma(\cal U,\bb A_{\sub{inf},U}/p^r)}\Gamma(\cal U,\bb A_{\sub{inf},U}/p^{r_0})/\xi_s\isoto \Gamma(\cal U,\bb M_{r_0})/\xi_s.\tag{\dag}\] Firstly, $\xi_s\in W_r(\frak m^\flat)^\circ$ by Lemma \ref{lemma_<mu_mod_p}(i), whence Proposition \ref{proposition_small_sheaves}(iv)\&(v) implies that $\Gamma(\cal U,\bb M_r)$ is finite projective over $\Gamma(\cal U,\bb A_{\sub{inf},U}/p^r)$ and that \[\Gamma(\cal U,\bb M_r)\otimes_{\Gamma(\cal U,\bb A_{\sub{inf},U}/p^r)}\bb A_{\sub{inf},U}/(p^r,\xi_s)|_{\cal U}\isoto \bb M_r/\xi_s|_{\cal U};\] going mod $p^{r_0}$ and applying $\Gamma(\cal U,-)$ we deduce that \[\Gamma(\cal U,\bb M_r)\otimes_{\Gamma(\cal U,\bb A_{\sub{inf},U}/p^r)}\Gamma(\cal U,\bb A_{\sub{inf},U}/(p^{r_0},\xi_s))\isoto\Gamma(\cal U,\bb M_{r_0}/\xi_s).\] Now fix an auxiliary integer $t\ge s+r_0$. Then the previous isomorphism is also valid if we replace $s$ by $t$ (since $t$ is also $\ge r$), and Proposition \ref{proposition_small_sheaves}(ii) (which is valid since $\res\xi_{t}=\res\xi_s\phi^{-s}(\res\xi_{t-s})$, where $\phi^{-s}(\res\xi_{t-s})\in W_{r_0}(\frak m^\flat)$ by Lemma \ref{lemma_<mu_mod_p}(i)) states that the image of the $t$-version in the $s$-version is precisely the desired identification (\dag).

Fixing $s\ge r>1$, we now use (\dag) for the triples $s\ge r\ge r$ and $s\ge r\ge r-1$; it follows easily that $\Gamma(\cal U,\bb M_r)/\xi_s\to \Gamma(\cal U,\bb M_{r-1})/\xi_s$ is surjective if $\Gamma(\cal U,\bb A_{\sub{inf},U}/p^{r})\to \Gamma(\cal U,\bb A_{\sub{inf},U}/p^{r-1})$ is surjective, which indeed it is by the identities of \cite[Lem.~5.6]{BhattMorrowScholze1}.
But $\Gamma(\cal U,\bb M_{r-1})$ is finite over $\Gamma(\cal U,\bb A_{\inf,U}/p^{r-1})$ and $\Gamma(\cal U,\bb A_{\inf,U}/p^{r-1})$ is $\xi_s$-adically complete, so it follows that $\Gamma(\cal U,\bb M_r)\to \Gamma(\cal U,\bb M_{r-1})$ is surjective.

Therefore by induction $\Gamma(\cal U,\bb M_r)\to \Gamma(\cal U,\bb M_t)$ is surjective for any $r\ge t$, i.e., the exact sequence $0\to\bb M_{r-t}\xto{p^t}\bb M_r\to\bb M_t\to 0$ remains exact upon applying $\Gamma(\cal U,-)$; then taking the inverse limit over $r$ (and using that surjectivity of the transition maps has already been checked) shows that the sequence $0\to \bb M\xto{p^t}\bb M\to \bb M_t\to 0$ is exact upon applying $\Gamma(\cal U,-)$. Since this holds for all affinoid perfectoids $\cal U$, we deduce that $0\to \bb M\xto{p^t}\bb M\to \bb M_t\to 0$ is exact, i.e., $\bb M/p^t\isoto\bb M_t$. It then follows that $H^1_\sub{pro\'et}(\cal U,\bb M)$ is $p$-torsion-free, since we showed surjectivity of $\Gamma(\cal U,\bb M)\to \Gamma(\cal U,\bb M_t)$.

Moreover, since the higher cohomologies of each $\bb M_t$ on $\cal U$ are killed by $[\frak m^\flat]$, by Proposition \ref{proposition_small_sheaves}(i), taking the limit using \cite[Lem.~3.18]{Scholze2013} shows that the higher cohomologies of $\bb M$ on $\cal U$ are killed by $[\frak m^\flat]$, i.e., (i). In particular, $H^1_\sub{pro\'et}(\cal U,\bb M)$ is killed by $[\frak m^\flat]$; but we have already shown it is $p$-torsion-free, whence it is $\xi_r$- and $\tilde\xi_r$-torsion-free since $\tilde\xi_r\equiv\xi_r\equiv p^r$ mod $[\frak m^\flat]$, i.e., (ii).

We have shown that $\Gamma(\cal U,\bb M)/p^t=\Gamma(\cal U,\bb M_t)$, which by Proposition \ref{proposition_small_sheaves}(v) is a finite projective $\Gamma(\cal U,\bb A_{\sub{inf},U})/p^t$-module satisfying $\Gamma(\cal U,\bb M)/p^t\otimes_{\Gamma(\cal U,\bb A_{\sub{inf},U})/p^t}\bb A_{\sub{inf},U}/p^t\vert_{\cal U}\isoto\bb M_t\vert_{\cal U}$. Therefore, by completeness, $\Gamma(\cal U,\bb M)$ is a finite projective $\Gamma(\cal U,\bb A_{\sub{inf},U})$-module and taking the limit shows that $\Gamma(\cal U,\bb M)\otimes_{\Gamma(\cal U,\bb A_{\sub{inf},U})}\bb A_{\sub{inf},U}\vert_{\cal U}\isoto \bb M\vert_{\cal U}$, i.e., (iii). Since the affinoid perfectoids form a basis, this also shows that $\bb M$ is locally finite projective. 

Next we show that $\bb M=\projlim_s\bb M_s$ is trivial modulo $<\mu$. For any $r\ge 1$, the sheaf of $W_r(\hat\roi_X^+)$-modules $\bb M/\xi_r$ is trivial modulo all powers of $p$ (since $\bb M/(\xi_r,p^s)=\bb M_s/\xi_r$ is trivial by hypothesis); since $p\in W_r(\frak m)^\circ$ by Remark \ref{remark_p_vs_[p]}, we deduce from Lemma \ref{lemma_trivial_mod_all_implies_trivial} that $\bb M/\xi_r$ is trivial.

(iv): As $\bb M$ is trivial modulo $\xi$, the assumption implies that $\bb M/\xi$ is a finite free $\bb A_{\inf,U}/\xi$-module. By applying the last claim of Proposition \ref{proposition_small_sheaves} to the sheaf of $\bb A_{\inf}/p$-modules $\bb M/p$ and $\varpi=(\xi \mod p)$, we see that $\Gamma(\cal U,\bb M/p)$ is a finite free $\Gamma(\cal U,\bb A_{\inf,U}/p)$-module. Hence the proof of (iii) above shows that the $\Gamma(\cal U, \bb A_{\inf,U})$-module $\Gamma(\cal U,\bb M)$ is finite free, and therefore, the $\bb A_{\inf,U}\vert_{\cal U}$-module $\bb M\vert_{\cal U}$ is finite free.

The final ``In particular'' claim is an immediate consequence of (iv), by restricting to an open cover of $\Spf R$ on each piece of which $\nu_*(\bb M/\xi)$ is finite free over $\nu_*(\bb A_{\inf,U}/\xi)$.
\end{proof}

We reach the main promised theorem of the section, describing relative Breuil--Kisin--Fargues modules on $\Spf R$ in terms of the generalised representations which were studied in \S\ref{section_small_reps}:

\begin{theorem}\label{theorem_Ainf_reps_vs_pro_etale}
Let $\cal U$ be an affinoid perfectoid cover of $U$ as in ($\Pi$). Taking global sections induces an equivalence of categories:
\begin{align*}
\BKF(R)=\categ{3cm}{locally finite free $\bb A_{\sub{inf},U}$-modules which are trivial modulo $<\mu$}&\Isoto\Rep_G^{<\mu}(\Gamma(\cal U,\bb A_{\sub{inf},U}))=\categ{5cm}{finite projective $\Gamma(\cal U,\bb A_{\sub{inf},U})$-modules equipped with a continuous, semi-linear $G$-action which is trivial modulo $<\mu$}\\ 
\bb M&\mapsto \Gamma(\cal U,\bb M)
\end{align*}
Moreover, for each such $\bb M$, there is a natural morphism of complexes of $\Gamma(U,\bb A_{\sub{inf},U})$-modules \[R\Gamma_\sub{cont}(G,\Gamma(\cal U,\bb M))\To \hat{R\Gamma_\sub{pro\'et}(U,\bb M)},\] where the hat denotes derived $p$-adic completion, such that all cohomology groups of the cone are killed by $W(\frak m^\flat)$. In the special case $\cal U=U_\infty$, so that $G=\bb Z_p(1)^d$, this morphism becomes a quasi-isomorphism after applying $L\eta_{\pi}$ for any non-zero divisor $\pi\in\bigcup_{j\ge0}\phi^{-j}(\mu)A_\inf$.
\end{theorem}
\begin{proof}
We begin by showing that the functor is well-defined. Given $\bb M$ in the left category, the global sections $\Gamma(\cal U,\bb M)$ are a finite projective $\Gamma(\cal U,\bb A_{\sub{inf},U})$-module by Proposition \ref{proposition_small_Ainf_sheaves}(iii). Each quotient $\bb M/(p^s,\xi^s)$ is a discrete sheaf since $\bb M$ is trivial modulo $\xi$,
$\bb A_{\inf,U}/(p^s,\xi)$ is discrete, and $\bb M/p^s$ is $\xi$-torsion free. This implies that the natural $G$-action on $\Gamma(\cal U,\bb M)$ is continuous for the $(p,\xi)$-adic topology: compare with Step 1 of the proof of Theorem \ref{theorem_correspondence_for_B} and use the fact that $\Gamma(\cal U,\bb M)/p^s=\Gamma(\cal U,\bb M/p^s)$ from Proposition \ref{proposition_small_Ainf_sheaves}(ii). The $G$-action is trivial modulo $\xi_r$, for any $r\ge1$, since \[\Gamma(\cal U,\bb M)/\xi_r=\Gamma(U,\bb M/\xi_r)\otimes_{\Gamma(U,\bb A_{\sub{inf},U}/\xi_r)}\Gamma(\cal U,\bb A_{\sub{inf},U})/\xi_r,\] where the $G$-action on $\Gamma(U,\bb M/\xi_r)$ is trivial; here we are implicitly using the identification $\Gamma(\cal U,-)/\xi_r=\Gamma(\cal U,-/\xi_r)$ for both $\bb M$ and $\bb A_{\sub{inf},U}$ from Proposition \ref{proposition_small_Ainf_sheaves}(ii)

We next prove that the functor is fully faithful. Let $\bb M$ and $\bb M'$ belong to the left category. Then
\begin{align*}
\Hom_{\bb A_{\sub{inf},U}}(\bb M,\bb M')&=\projlim_s\Hom_{\bb A_{\sub{inf},U}/p^s}(\bb M/p^s,\bb M'/p^s)\\
&=\projlim_s\Hom_{\Gamma(\cal U,\bb A_{\sub{inf},U}/p^s),G}(\Gamma(\cal U,\bb M/p^s),\Gamma(\cal U,\bb M'/p^s))
\end{align*}
where $\Hom_{\Gamma(\cal U,\bb A_{\sub{inf},U}/p^s),G}$ refers to the obvious notion of $G$-equivariant morphisms in the category  
\[\categ{7cm}{finite projective $\Gamma(\cal U,\bb A_{\sub{inf},U}/p^s)$-modules with continuous, semi-linear $G$-action trivial modulo $<\mu$}\] The second equality above is an instance, for each $s\ge1$, of Theorem \ref{theorem_correspondence_for_B}, noting that $\bb M/p^s$ and $\bb M'/p^s$ are trivial modulo $\mu$ by the proof of the implication (a)$\Rightarrow$(b) of Proposition \ref{proposition_small_Ainf_sheaves}. Moving the quotient by $p^s$ outside the global sections (by Proposition \ref{proposition_small_Ainf_sheaves}(ii)), this may be rewritten as
\[\projlim_s\Hom_{\Gamma(\cal U,\bb A_{\sub{inf},U})/p^s,G}(\Gamma(\cal U,\bb M)/p^s,\Gamma(\cal U,\bb M')/p^s)=\Hom_{\Gamma(\cal U,\bb A_{\sub{inf},U}),G}(\Gamma(\cal U,\bb M),\Gamma(\cal U,\bb M'))\] where $\Hom_{\Gamma(\cal U,\bb A_{\sub{inf},U}),G}$ refers to morphisms in the category on the right side of the theorem. This completes the proof of fully faithfulness.

Regarding cohomology, we have already noted in the course of the proof that $\Gamma(\cal U,\bb M)/p^s=\Gamma(\cal U,\bb M/p^s)$ and that Theorem \ref{theorem_correspondence_for_B} may be applied to each $\bb M/p^s$. Therefore the final assertion of that theorem provides a natural almost quasi-isomorphism of complexes of $\Gamma(\cal U,\bb A_{\inf,U})/p^s$-modules \[R\Gamma_\sub{cont}(G,\Gamma(\cal U,\bb M))/p^s\To R\Gamma_\sub{pro\'et}(U,\bb M)/p^s.\] Taking the derived inverse limit as $s\to\infty$ provides the desired morphism \[R\Gamma_\sub{cont}(G,\Gamma(\cal U,\bb M))\To \hat{R\Gamma_\sub{pro\'et}(U,\bb M)}.\] In the special case $\cal U=U_\infty$, the quasi-isomorphism claim is an application of Lemma \ref{lemma_Leta_isom}(ii).

It remains to show that the functor is essentially surjective; so let $M$ belong the right category. For each $s\ge 1$, the quotient $M_s:=M/p^sM$ is a finite projective $\Gamma(\cal U,\bb A_{\sub{inf},U}/p^s)$-module with a continuous, semi-linear $G$-action trivial modulo $<\mu$ (by Lemma \ref{lemma_<mu_mod_p}). The equivalence of Theorem \ref{theorem_correspondence_for_B} tells us that $M_s$ corresponds to a locally finite projective sheaf $\bb M_s$ of $\bb A_{\sub{inf},U}/p^s$-modules trivial modulo $<\mu$ on $U_\sub{pro\'et}$ with the following two properties: $\bb M_s|_{\cal U}=M_s\otimes_{\Gamma(\cal U,\bb A_{\sub{inf},U})/p^s}\bb A_{\sub{inf},U}|_{\cal U}/p^s$ by Proposition \ref{proposition_small_sheaves}(v), and the canonical $G$-action on $\Gamma(\cal U,\bb M_s)=M_s$ is the existing one on $M_s$.

The inverse system $\cdots\to\bb M_s\to\bb M_{s-1}\to\cdots$ satisfies the conditions of Proposition \ref{proposition_small_Ainf_sheaves}(b), whence $\bb M:=\projlim_s \bb M_s$ belongs to the left category by the implication (b)$\Rightarrow$(a) and part (iv) of Proposition \ref{proposition_small_Ainf_sheaves}. But the previous two properties show that $\Gamma(\cal U,\bb M)=\projlim_s\Gamma(\cal U,\bb M_s)$ identifies with $M$ as an object of the right category, thereby proving the essential surjectivity.
\end{proof}

\subsection{Relative Breuil--Kisin--Fargues modules}\label{ss_relative_BKF}
Given a $p$-adically complete, formally smooth, small $\roi$-algebra $R$ with choice of framing, the following diagram summarises our main equivalences thus far:
\begin{equation}\xymatrix@C=1.2cm{
\BKF(\Spf R)\ar[r]^-{\Gamma(U_\infty,-)}_{\simeq\sub{Thm.~\ref{theorem_Ainf_reps_vs_pro_etale}}}&\Rep_\Gamma^{<\mu}(A_\inf(R_\infty))&\\
&\Rep_\Gamma^\mu(A_\inf(R_\infty))\ar[u]_{\simeq \sub{Thm.~\ref{theorem_descent_to_framed}}}&&\\
&\Rep_\Gamma^\mu(A_\inf^\bx(R))\ar[u]^{-\otimes_{A_\sub{inf}^\bx(R)}A_\sub{inf}(R_\infty)}_{\simeq\sub{Thm.~\ref{theorem_<mu_implies_mu}}}\ar[r]&\mathrm{qMIC}(A_\inf^\bx(R))\ar[l]^{\gamma_i=1+\mu\nabla^\sub{log}_i}_{\simeq \sub{Corol.~\ref{corollary_reps_vs_q_connections}}}&\\
&&\mathrm{qHIG}(A_\inf^{\bx}(R)^{(1)})\ar[u]_{(F,F_\Omega)^*}^{\sub{Thm.~\ref{theorem_Simpsons}}}& \op{CR}_{\Prism}(R^{(1)}/(A_{\inf},\tilde\xi))\ar@{_(->}[l]_-{\op{ev}_{A_\inf^{\bx}(R)^{(1)}}}^{\sub{Thm.~\ref{thm:main}}}
}\label{eqn_intro_later}\end{equation}
There are three goals of this short subsection. First we show that the overall composition of functors $\op{CR}_{\Prism}(R^{(1)}/(A_{\inf},\tilde\xi))\to \BKF(\Spf R)$ is independent of the chosen framing and may be globalised:

\begin{theorem}\label{theorem_prism_to_BKF}
For any separated, smooth, $p$-adic formal $\roi$-scheme $\frak X$, there exists a natural functor \[\op{CR}_{\Prism}(\frak X^{(1)}/(A_{\inf},\tilde\xi))\To \BKF(\frak X)\] given on small affine opens by the above composition (up to a natural equivalence); this functor is fully faithful.
\end{theorem}
\begin{proof}
As usual let $X$ denote the rigid analytic generic fibre of $\frak X$. Let $\cal B$ denote the collection of affinoid perfectoids $V\in X_\sub{pro\'et}$ with the property that the image of $V\to X$ (which is an open map \cite[Lem.~3.10(iv)]{Scholze2013}) is contained in $\Spa(R[\tfrac1p],R)$ for some affine open $\Spf R\subseteq \frak X$. Given such an affinoid perfectoid $V$, with corresponding perfectoid ring $A=\Gamma(V,\hat\roi_X^+)$ (note that $A$ is indeed perfectoid by \cite[Lem.~3.20]{BhattMorrowScholze1}), the resulting map $\Spf A\to \Spf R\to \frak X$ does not depend on the choice of the affine open $\Spf R$: this is easily checked using the separatedness hypothesis to assume that any other choice $\Spf R'$ is actually contained in $\Spf R$. Note also that $\cal B$ forms a basis of $X_\sub{pro\'et}$: indeed, arbitrary affinoid perfectoids $V$ form a basis, and each $V$ admits a cover by the affinoid perfectoids $V\times_X\Spa(R[\tfrac1p],R)\in\cal B$ as $\Spf R$ runs over an affine open cover of $\frak X$ (see Lems.~3.10 \& 4.6 and Corol.~4.7 of \cite{Scholze2012} for the various facts used here).

Given $\cal F \in \op{CR}_{\Prism}(\frak X^{(1)}/(A_{\inf},\tilde\xi))$, we define a sheaf $\cal F_\sub{BKF}$ of $\bb A_{\sub{inf},X}$-modules on $X_\sub{pro\'et}$ as follows. For any affinoid perfectoid $V\in \cal B$, with corresponding perfectoid ring $A=\Gamma(V,\hat\roi_X^+)$, Frobenius twisting the above canonical map $\Spf(A)\to\frak X$ defines a map $\Spf(A_\inf(A)/\tilde\xi)=\Spf(A^{(1)})\to\frak X^{(1)}$ which thereby makes $\Gamma(V,\bb A_{\inf,X})=A_\inf(A)$ into an object of the prismatic site $(\frak X^{(1)}/(A_\inf,\tilde\xi))_\Prism$. The sections of the prismatic crystal $\cal F$ therefore define a finite projective $\Gamma(V,\bb A_{\inf,X})$-module \[\cal F_\sub{BKF}^\sub{pre}(V):=\Gamma(\cal F, \Spf(A_\inf(A))\hookleftarrow\Spf(A^{(1)})\to\frak X^{(1)}),\] and we let $\cal F_\sub{BKF}$ be the sheaf of $\bb A_{\inf,X}$-modules on $X_\sub{pro\'et}$ obtained by sheafifying $V\mapsto \cal F_\sub{BKF}^\sub{pre}(V)$, as $V$ runs over the basis $\cal B$.

To prove that this defines the desired functor we must show that (1) $\cal F_\sub{BKF}$ belongs to $\BKF(\frak X)$ and (2) on any small affine open $\Spf R\subseteq\frak X$, the restriction $\cal F_\sub{BKF}|_{\Spa(R[\tfrac1p],R)}$ coincides up to a natural equivalence with the composition of the functors in diagram (\ref{eqn_intro_later}) applied to $\cal F|_{\Spf R^{(1)}}\in \op{CR}_{\Prism}(R^{(1)}/(A_{\inf},\tilde\xi))$. However, claim (1) is local on $\frak X$ and so it will follow from claim (2). More precisely, we may therefore now suppose that $\frak X=\Spf R$ where $R$ is as at the start of the subsection, we let $\bb M\in \BKF(\Spf R)$ be the composition of the functors in (\ref{eqn_intro_later}) applied to $\cal F$, and we must show that there exists an isomorphism $\cal F_\sub{BKF}\cong\bb M$ of sheaves of $\bb A_{\inf,U}$-modules on $U_\sub{pro\'et}$, where we write $U=X$ to be consistent with the notation used previously in the local situation.

Let $M\in \Rep_\Gamma^\mu(A_\inf(R_\infty))$ denote the image of $\cal F$ under the zigzag of morphisms in diagram (\ref{eqn_intro_later}); in other words, in the notation of \S\ref{section_prismatic_crystals}, particularly diagram (\ref{eqn_crystal_diagram}), $M=\op{ev}_{A_\infty^\bx}(\cal F)$. Also note that $M=\cal F^\sub{pre}_\sub{BKF}(U_\infty)$, where $U_\infty$ is the usual pro\'etale cover of $U=\Spa(R[\tfrac 1p],R)$ associated to the framing as in Example \ref{example_affinoid_perfectoid_covers}(i)) so that, according to the construction of Theorem \ref{theorem_Ainf_reps_vs_pro_etale} (see the penultimate paragraph of the proof, as well as step $3$ of the proof of Theorem \ref{theorem_correspondence_for_B}), the sheaf $\bb M$ satisfies $\bb M\isoto\projlim_s\bb M/p^s$ and \[\Gamma(V,\bb M/p^s)=(M\otimes_{A_\inf(R_\infty)}\Gamma(U_\infty\times_UV,\bb A_{\inf,U}/p^s))^\Gamma\] for all $V\in U_\sub{pro\'et}$. Assuming $V$ is affinoid then $U_\infty\times_UV$ is affinoid perfectoid \cite[Lem.~4.6]{Scholze2013} and we have $\Gamma(U_\infty\times_UV,\bb A_{\inf,U})/p^s\isoto \Gamma(U_\infty\times_UV,\bb A_{\inf,U}/p^s)$; so the crystal property of $\cal F$ allows us to rewrite the previous line as $(\cal F_\sub{BKF}^\sub{pre} (U_\infty\times_UV)/p^s)^\Gamma$, where the $\Gamma$-action is induced by its action on $U_\infty$, which clearly receives a map from $\cal F_\sub{BKF}^\sub{pre}(V)/p^s$ if $V$ is also affinoid perfectoid.

Letting $s\to\infty$, we have defined natural maps $f_V:\cal F_\sub{BKF}^\sub{pre}(V)\to \bb M(V)$ of $\Gamma(V,\bb A_{\inf,U})$-modules for all affinoid perfectoids $V\in U_\sub{pro\'et}$.  The map $f_{U_\infty}$ is by construction the canonical isomorphism $\cal F_\sub{BKF}^\sub{pre}(U_\infty)=M\isoto \Gamma(U_\infty,\bb M)$ arising from the fact that $\bb M$ and $M$ correspond under the equivalence of Theorem \ref{theorem_Ainf_reps_vs_pro_etale}. Moreover, for any affinoid perfectoid $V$ over $U_\infty$, the canonical maps $\cal F^\sub{pre}_\sub{BKF}(U_\infty)\otimes_{\Gamma(U_\infty,\bb A_{\inf,U})}\Gamma(V,\bb A_{\inf,U})\to \cal F^\sub{pre}_\sub{BKF}(V)$ and $\Gamma(U_\infty,\bb M)\otimes_{\Gamma(U_\infty,\bb A_{\inf,U})}\Gamma(V,\bb A_{\inf,U})\to \Gamma(V,\bb M)$ are isomorphisms; in the first case this is due to the crystal property of $\cal F$, in the second case it is Theorem \ref{proposition_small_Ainf_sheaves}(iii). In conclusion $f_V$ is an isomorphism for all affinoid perfectoids $V$ over $U_\infty$; since such affinoid perfectoids form a basis of $U_\sub{pro\'et}$, sheafifying then defines the desired isomorphism $f:\cal F_\sub{BKF}\isoto\bb M$.

Finally, to check that the functor is fully faithful we may work locally locally: then the composition $(F,F_\Omega)^*\circ\op{ev}_{A_\inf^{\bx}(R)^{(1)}}$ in diagram (\ref{eqn_intro_later}) is fully faithful by Theorem \ref{thm:main} (for $\op{ev}^\phi_{A^{\bx(1)}}$), and the remaining four functors in the diagram are even equivalences.
\end{proof}

The second goal of this subsection is to incorporate the Frobenius to define true Breuil--Kisin--Fargues modules in the global set-up:

\begin{definition}
Let $\frak X$ be a smooth, $p$-adic formal $\roi$-scheme. A {\em relative Breuil--Kisin--Fargues module} on $\frak X$ is a pair $(\bb M,\phi_{\bb M})$, where $\bb M\in\BKF(\frak X)$ and $\phi_{\bb M}:(\phi^*\bb M)[\tfrac1{\tilde\xi}]\to \bb M[\tfrac1{\tilde\xi}]$ is an isomorphism of sheaves of $\bb A_{\inf,X}[\tfrac1{\tilde\xi}]$-modules. The category of such will be denoted by $\BKF(\frak X,\phi)$.
\end{definition}

As in the situation lacking Frobenii, it is clear that relative Breuil--Kisin--Fargues modules form a stack for the Zariski topology on $\frak X$, so it is sufficient to describe them on any small open of $\frak X$. The following equivalences and globalisation follow easily from the analogous results without Frobenius structures:

\begin{theorem}\label{theorem_local_with_phi}
\begin{enumerate}
\item Let $R$ be a $p$-adically complete, formally smooth, small $\roi$-algebra, and fix a framing as at the start of \S\ref{section_small_reps}; then the variants of the functors of diagram (\ref{eqn_intro_later}) incorporating Frobenius structures are all equivalences of categories.
\item For any separated, smooth, $p$-adic formal $\roi$-scheme $\frak X$, there is a natural equivalence of categories \[\op{F-CR}_{\Prism}(\frak X^{(1)}/(A_{\inf},\tilde\xi))\quis \op{BKF}(\frak X,\phi)\] by incorporating Frobenius structures into Theorem \ref{theorem_prism_to_BKF}.
\end{enumerate}
\comment{
there are equivalences of categories
\[\hspace{-2.3cm}\BKF(\Spf R,\phi)\quiss^{\Gamma(U_\infty,-)}_\sub{Thm.~\ref{theorem_Ainf_reps_vs_pro_etale}}
\Rep_\Gamma^\mu(A_\inf(R_\infty),\phi)\quissleft^{-\otimes_{A_\sub{inf}^\bx(R)}A_\sub{inf}(R_\infty)}_\sub{Thm.~\ref{theorem_descent_to_framed_with_phi}}
\Rep_\Gamma^\mu(A_\inf^\bx(R),\phi)\quissleft^{\gamma_i=1+\mu\nabla^\sub{log}_i}_\sub{Prop.~\ref{proposition_reps_vs_q_connections_with_phi}}
\mathrm{qMIC}(A_\inf^\bx(R),\phi)\quissleft_\sub{Corol.~\ref{corollary_q_Simp}}^{(F,F_\Omega)^*}
\mathrm{qHIG}(\sigma^*A_\inf^\bx(R),\phi).\]
}
\end{theorem}
\begin{proof}
(i): We begin by showing that the global sections functor \[\Gamma(U_\infty,-):\BKF(\Spf R,\phi)\to\Rep_\Gamma^\mu(A_\inf(R_\infty),\phi)\] is an equivalence. Theorem \ref{theorem_Ainf_reps_vs_pro_etale} states that it is an equivalence if we drop the Frobenius structures from both sides. To  show it is fully faithful it therefore suffices to check that for any $(\bb M_i,\phi_{\bb M_i})\in \BKF(\Spf R,\phi)$, $i=1,2$, and morphism $f:\bb M_1\to\bb M_2$ in $\BKF(\Spf R)$, then $f$ respects the Frobenius structures if the induced morphism $\Gamma(U_\infty,f):\Gamma(U_\infty,\bb M_1)\to \Gamma(U_\infty,\bb M_2)$ in $\Rep_\Gamma^\mu(A_\inf(R_\infty))$ respects the Frobenius structures. But observe that if $\Gamma(U_\infty,f)$ respects the Frobenius structure, then so does $f|_{U_\infty}$ by Proposition \ref{proposition_small_Ainf_sheaves}(iii), which is enough since $U_\infty$ is a cover of $U$.

To prove essential surjectivity, we suppose that $\bb M\in\BKF(\Spf R)$ and that $\phi_M:\Gamma(U_\infty,\bb M)[\tfrac1\xi]\isoto \Gamma(U_\infty,\bb M)[\tfrac1{\tilde\xi}]$ is a Frobenius structure on its sections over $U_\infty$; we must show that $\phi_M$ is induced by a unique Frobenius structure $\phi_\bb M:\bb M[\tfrac1\xi]\isoto \bb M[\tfrac1{\tilde\xi}]$. After replacing $\phi_M$ by $\tilde\xi^a\phi_M$ for $a\gg0$ we may suppose that $\phi_M$ restricts to $\phi_M:\Gamma(U_\infty,\bb M)\to \Gamma(U_\infty,\bb M)$, which we view as an $A_\inf(R_\infty)$-linear map $\phi_M:\Gamma(U_\infty,\bb M\otimes_{\bb A_{\inf,U},\phi}\bb A_{\inf,U})\to \Gamma(U_\infty,\bb M)$. Note that $\bb M\otimes_{\bb A_{\inf,U},\phi}\bb A_{\inf,U}$ is trivial modulo $\phi(\xi_{r+1})=\tilde\xi \xi_r$ for all $r\ge1$, hence trivial modulo $\xi_r$ for all $r\ge1$, i.e., $\bb M\otimes_{\bb A_{\inf,U},\phi}\bb A_{\inf,U}\in \BKF(\Spf R)$. Therefore the known equivalence without Frobenius structures implies that $\phi_M$ is induced by a unique $\bb A_{\inf,U}$-linear map $\bb M\otimes_{\bb A_{\inf,U},\phi}\bb A_{\inf,U}\to \bb M$, which is an isomorphism after inverting $\tilde\xi$ by applying the same argument to $\tilde\xi^b\phi_M^{-1}$ for $b\gg0$, as desired.

The remaining functors in the variant of diagram (\ref{eqn_intro_later}) with Frobenius have already been checked to be equivalences: see respectively Theorem \ref{theorem_descent_to_framed_with_phi}, Proposition \ref{proposition_reps_vs_q_connections_with_phi} and the subsequent paragraph, Corollary \ref{corollary_q_Simp} and its subsequent paragraph, and finally Theorem \ref{thm:main} and Lemma \ref{lemma_phi_implies_xi_nilp} for $\op{ev}_{A_\sub{inf}^\bx(R)^{(1)}}$.

\comment{
\todo{Still need to fix this!!}
\marginpar{MM reply to Q (22). It seems to me that stackiness of prismatic crystals is a consequence of $(p,I)$-faithfully flat descent of vector bundles; this is discussed in Anschutz--Le Bras, so I added the precise references (although they state their results only for affines, the proofs appear to be global).}
\marginpar{MM: See my email for discussion about quasisyntomic topology}
}

(ii): We first note that the functor sending opens $\frak U$ of $\frak X$ to the category of prismatic crystals $\op{CR}_{\Prism}(\frak U^{(1)}/(A_{\inf},\tilde\xi))$ is a stack for the Zariski topology. After refining a given open cover $\{\frak U_i\}$ of $\frak X$ we may assume that each $\frak U_i$ is affine, and we suppose we have prismatic crystals $\cal F_i\in \op{CR}_{\Prism}(\frak U_i^{(1)}/(A_{\inf},\tilde\xi))$ for each $i$ together with isomorphisms $\cal F_i|_{\frak U_i\cap\frak U_j}\cong\cal F_j|_{\frak U_i\cap\frak U_j}$ in $\op{CR}_{\Prism}(\frak U_i^{(1)}\cap \frak U_j^{(1)}/(A_{\inf},\tilde\xi))$ which are compatible on the triple intersections. Then, given any object $\frak B=(\Spf(B)\leftarrow\Spf(B/\tilde\xi B)\to\frak X^{(1)})$ of $\op{CR}_{\Prism}(\frak X^{(1)}/(A_{\inf},\tilde\xi))$, each open $\Spf(B_i/\tilde\xi)\times_{\frak X^{(1)}}\frak U_i^{(1)}$ is equal to $\Spf(B_i/\tilde\xi B_i)$ for some affine open $\Spf(B_i)\subseteq\Spf(B)$; note that $B_i$ inherits a $\delta$-map from $B$ \cite[Lem.~2.18]{BhattScholze2019}, thereby making $\frak B_i=(\Spf(B_i)\leftarrow\Spf(B_i/\tilde\xi B_i)\to\frak U_i^{(1)})$ into an object of the prismatic site of $\frak U_i^{(1)}$. Using $(p,\tilde\xi)$-completely faithfully flat descent as in Corollary \ref{cor:DescentFlatCover}(ii), it is not difficult to glue the finite projective $B_i$-modules $\cal F_i(\frak B_i)$ to a finite projective $B$-module $\cal F(\frak B)$, and then to check that $\frak B\mapsto\cal F(\frak B)$ defines the desired prismatic crystal $\cal F$ on $\frak X^{(1)}$. See also \cite[Prop.~4.1.8]{AnschutzLeBras2019} for stackiness of prismatic crystals in the (affine) quasisyntomic topology.

\comment{ OLD ARGUMENT TO REMOVE

Indeed, letting $\nu:(\frak X^{(1)}/(A_\inf,\tilde\xi))_\Prism\to \frak X_\sub{qsyn}$ be the morphism of sites to the category of quasisyntomic formal $\frak X$-schemes, as in \cite[\S4.1]{AnschutzLeBras2019}, the proof of Prop.~4.1.4 of {\em op.~cit.} shows that $\op{CR}_{\Prism}(\frak U^{(1)}/(A_{\inf},\tilde\xi))$ is equivalent to the category of finite locally free $\nu_*\roi_\Prism$-modules on $\frak X_\sub{qsyn}$. The same holds for each open $\frak U\subseteq\frak X$, whence the claim follows formally (even for the quasisyntomic topology in place of the Zariski topology) as in \cite[Prop.~4.1.8]{AnschutzLeBras2019}. A similar claim holds when we incorporate Frobenius structures.
}

The proof of part (ii) therefore reduces to the affine case, where it follows by the taking the composition of all the functors in part (i).
\end{proof}

\begin{example}[Dieudonn\'e theory: global case]\label{example_global_Dieudonne}
We sketch how the Dieudonn\'e theory of \S\ref{sss_Dieudonne} and \S\ref{ss_DieudonneII} may be globalised.

Let $\frak X$ be a separated, smooth, $p$-adic formal $\cal O$-scheme, and write $\op{BKF}(\frak X,\varphi,[-1,0])$ for the full subcategory of $\op{BKF}(\frak X,\varphi)$ consisting of $(\bb M,\varphi_{\bb M})$ satisfying $\bb M\subset \varphi_{\bb M}(\varphi^*\bb M)\subset \tilde \xi^{-1}\bb M$. Then Theorem \ref{theorem_p_div} may be globalised to an equivalence of categories $$\Phi^{\BKF}_{\frak X}\colon \op{BT}(\frak X)\quis \op{BKF}(\frak X,\varphi, [-1,0])$$ as follows: for a $p$-divisible group $G$ over $\frak X$, Theorem \ref{theorem_p_div_perfectoid} yields an object $\Phi^\inf_{\hat{\cal O}_X^+(V)}(G\times_{\frak X}\Spf(\hat{\cal O}_X^+(V)))$
of $\op{BKF}(\hat{\cal O}_X^+(V),\phi, [-1,0])$ 
for each affinoid perfectoid $V$ in the basis $\cal B$ used in the proof of Theorem \ref{theorem_prism_to_BKF}, and this construction is shown to be a well-defined equivalence by the same arguments as Theorems  \ref{theorem_prism_to_BKF} and \ref{theorem_local_with_phi}.


We can also globalise Proposition \ref{prop:TateModule} as follows. Let $G$ be a $p$-divisible group
over $\frak X$, and let $TG$ denote the Tate module regarded as a locally free lisse $\bb Z_p$-sheaf
on $X$. Then, by the same argument as the construction of the isomorphism in the
proof of Proposition \ref{prop:TateModule}, we obtain a canonical isomorphism
$$TG(V)=\varprojlim_nG[p^n](V)\cong
\Phi^\inf_{\hat{\cal O}_X^+(V)}(G\times_{\frak X}\Spf(\hat{\cal O}_X^+(V)))[\tfrac{1}{\mu}]^{\varphi=1}$$
for each perfectoid affinoid $V\in \cal B$ satisfying the condition (as $\cal U$)
in the first paragraph of the proof of Proposition \ref{proposition_etale} (note that such affinoid perfectoids still form a basis). 
Varying $V$, we obtain
a canonical isomorphism
$$TG\cong  \Phi^{\BKF}_{\frak X}(G)[\tfrac{1}{\mu}]^{\varphi=1}
=\sigma_\sub{\'et}^*(\Phi^{\BKF}_{\frak X}(G)),$$
where $\sigma_\sub{\'et}^*$ is the \'etale specialization
functor constructed in \S\ref{ss_etale}.
\end{example}

\begin{example}[Relative prismatic cohomology: global case]
Let $f:\cal Y\to\frak X$ be a proper smooth morphism between separated, smooth, $p$-adic formal $\roi$-schemes; assume that the relative de Rham cohomologies $R^if_*\Omega^\blob_{\cal Y/\frak X}$ on $\frak X$ are locally free $\roi_{\frak X}$-modules for all $i\ge0$.

Using base change in prismatic cohomology and the definition of the functor in Theorem \ref{theorem_prism_to_BKF}, we see that the main construction of \S\ref{sssection_relative} may be globalised. That is, for each $i\ge0$, there exists an ``$i^\sub{th}$ higher direct image'' object of $\op{BKF}(\frak X,\phi)$, defined by sheafifying \[V\mapsto H^i_\Prism(\cal Y^{(1)}\times_{\frak X^{(1)}}V^{(1)}/(\bb A_{\sub{inf},X}(V),\tilde\xi))\] as $V$ runs over the basis $\cal B$ of affinoid perfectoids used in the proof of Theorem \ref{theorem_prism_to_BKF}. Under the equivalence of Theorem \ref{theorem_local_with_phi}(ii), this corresponds to the actual $i^\sub{th}$ higher direct image of the prismatic structure sheaf $\roi_\Prism$ along the morphism of sites $f_*:(\cal Y^{(1)}/(A_\inf,\tilde\xi))_\Prism\to (\frak X^{(1)}/(A_\inf,\tilde\xi))_\Prism$.
\end{example}

\newpage
\section{Cohomology and specialisation functors following \cite{BhattMorrowScholze1}}\label{section_cohomology}
This section is a continuation of Section \ref{section_BKF_on_proetale}. So again let $\frak X$ be a smooth $p$-adic formal $\roi$-scheme, and denote its adic generic fibre by $X$. Given $\bb M\in\BKF(\frak X)$, we may mimic the main construction of \cite{BhattMorrowScholze1} by introducing the complex of sheaves of $A_\inf$-modules on the Zariski site of $\frak X$ \[A\Omega_{\frak X}(\bb M):=L\eta_\mu\big(\hat{R\nu_*\bb M}\big)\in D(\frak X_\sub{Zar},A_\inf),\] where the hat denotes the derived $p$-adic completion. This complex of sheaves on $\frak X$ is locally given by the q-de Rham complex of a module with q-connection (similarly, by the q-Higgs complex of a module with q-connection in the presence of a Frobenius structure on $\bb M$):

\begin{theorem}\label{theorem_cohomology_as_q_dr_or_Higgs}
Fix an affine open $\Spf R\subseteq\frak X$, where $R$ is a formally smooth, small $\roi$-algebra, and fix a framing $\square:\roi\pid{\ul T^{\pm1}}\to R$. Then there is a natural, for $\bb M\in\BKF(\frak X)$, equivalence \[R\Gamma_\sub{Zar}(\Spf R,A\Omega_\frak X(\bb M))\simeq N\otimes_{A_\inf^\bx(R)}\hat\Omega^\blob_{A_\inf^\bx(R)/A_\inf},\] where $N\in\textrm{\rm qMIC}(A^\bx_\inf(R))$ is the $A^\bx_\inf(R)$-module with q-connection associated to $\bb M|_{\Spf R}\in\BKF(\Spf R)$ via the equivalences of Theorems \ref{theorem_descent_to_framed} and \ref{theorem_<mu_implies_mu}, Corollary  \ref{corollary_reps_vs_q_connections}, and (\ref{eqn_BKF_to_Rep}).

Similarly, there is a natural, for $\bb M\in\BKF(\frak X,\phi)$, equivalence \[R\Gamma_\sub{Zar}(\Spf R,A\Omega_\frak X(\bb M))\simeq H\otimes_{A_\inf^\bx(R)^{(1)}}\hat\Omega^\blob_{A_\inf^\bx(R)^{(1)}/A_\inf},\] where the right side is the q-Higgs complex of the module with q-Higgs field $H\in\textrm{\rm qHIG}( A^{\bx}_\inf(R)^{(1)},\phi)$ associated to $\bb M|_{\Spf R}\in\BKF(\Spf R,\phi)$ via the equivalences summarised in Theorem \ref{theorem_local_with_phi}.

\end{theorem}
\begin{proof}
We will show in Proposition \ref{proposition_local_to_global} that the natural local-to-global map $L\eta_\mu R\Gamma_\sub{pro\'et}(\Sp R[\tfrac1p],\bb M)\to R\Gamma_\sub{Zar}(\Spf R,A\Omega_{\frak X}(\bb M))$ is an equivalence; the former identifies with $L\eta_\mu R\Gamma_\sub{cont}(\Gamma,M)$ by Theorem \ref{theorem_Ainf_reps_vs_pro_etale}, where $M:=\Gamma(U_\infty,\bb M)\in\Rep^{\mu}_\Gamma(A_\inf(R_\infty))$. The result then follows from the fact that all the categorical equivalences in the statement of the theorem have been shown to be compatible with cohomology: see Proposition \ref{proposition_Leta_framed}, Corollary \ref{corollary_q_Simp}, and Theorem \ref{theorem_Ainf_reps_vs_pro_etale}.
\end{proof}

The second goal of this section is to construct specialisation functors\footnote{The notation $\sigma^*$ is meant to conjure up pulling back along a morphism of sites; see Remark \ref{remark_comp_to_prismatic}.}\footnote{Terminology in the literature seems to vary slightly, but by ``locally finite free F-crystal'' we mean a locally finite free crystal $\bb E$ on the crystalline site $(\frak X_k/W(k))_\sub{crys}$, equipped with an isomorphism in the isogeny category $\phi_{\bb E}:(F_{\frak X_k}^*\bb E)[\tfrac1p]\isoto\bb E[\tfrac1p]$. By twisting the Frobenius one can typically reduce assertions to the effective case where $\phi_{\bb E}:F_{\frak X_k}^*\bb E\to \bb E$ is already defined before passing to the isogeny category of crystals.}
\begin{align*}
\sigma^*_\sub{dR}:\BKF(\frak X)&\To \categ{5cm}{vector bundles with flat $\roi$-linear connection on $\frak X$}\\
\sigma^*_\sub{crys}:\BKF(\frak X,\phi)&\To\categ{4.5cm}{locally finite free F-crystals on the special fibre $\frak X_k$}\\
\sigma^*_\sub{\'et}:\BKF(\frak X,\phi)&\To\categ{5cm}{locally free lisse $\bb Z_p$-sheaves on the rigid analytic generic fibre $X$}
\end{align*}

The vector bundle with connection $\sigma^*_\sub{dR}\bb M$ has an associated de Rham complex \[\sigma^*_\sub{dR}\bb M\otimes_{\roi_\frak X}\Omega^\blob_{\frak X/\roi}=[\sigma^*_\sub{dR}\bb M\xto{\nabla}\sigma^*_\sub{dR}\bb M\otimes_{\roi_\frak X}\Omega^1_{\frak X/\roi}\xto{\nabla}\cdots]\in D^{\ge0}(\frak X_\sub{Zar},\roi).\] Similarly, the locally free crystal $\sigma^*_\sub{crys}\bb M$ gives rise to \[Ru_*(\sigma^*_\sub{crys}\bb M)\in D^{\ge0}(\frak X_\sub{Zar},W(k)),\] where $u:(\frak X_k/W(k))_\sub{crys}\to \frak X_{k,\sub{Zar}}=\frak X_{\sub{Zar}}$ is the usual projection morphism of sites. These specialisations are related to the cohomology $A\Omega_\frak X(\bb M)$ by the following theorem, which is an analogue with coefficients of \cite[Thm.~1.10]{BhattMorrowScholze1}.

\begin{theorem}\label{theorem_specialisations}
As above, let $\frak X$ be a smooth $p$-adic formal $\roi$-scheme. Then there are natural, for $\bb M\in\BKF(\frak X,\phi)$ (we just need $\bb M\in\BKF(\frak X)$ in the first statement), equivalences of complexes of sheaves on $\frak X_\sub{Zar}$
\begin{align*}
A\Omega_{\frak X}(\bb M)\dotimes_{A_\inf,\theta}\roi&\simeq \sigma^*_\sub{dR}\bb M\otimes_{\roi_\frak X}\Omega^\blob_{\frak X/\roi}\\
\hat{A\Omega_{\frak X}(\bb M)\dotimes_{A_\inf}W(k)}&\simeq Ru_*(\sigma^*_\sub{crys}\bb M)\\ 
\hat{A\Omega_\frak X(\bb M)[\tfrac1\mu]}^{\phi_{\bb M}=1}&\simeq R\nu_*(\sigma^*_\sub{\'et}\bb M)
\end{align*}
(where the hats denote derived $p$-adic completion), assuming for the third comparison that $C$ is algebraically closed.
\end{theorem}

Assuming in addition that $\frak X$ is proper over $\roi$, we define the {\em $A_\inf$-cohomology of $\frak X$ with coefficients} in a given $\bb M$ to be $R\Gamma_{A_\inf}(\frak X,\bb M):=R\Gamma_\sub{Zar}(\frak X,A\Omega_\frak X(\bb M))$; the following analogue of \cite[Thm.~1.8]{BhattMorrowScholze1} holds: 

\begin{corollary}
Let $\frak X$ be a proper smooth $p$-adic formal $\roi$-scheme and $\bb M\in\BKF(\frak X,\phi)$. Then $R\Gamma_{A_\inf}(\frak X,\bb M)$ is a perfect complex of $A_\inf$-modules and taking cohomology in Theorem \ref{theorem_specialisations} yields equivalences
\begin{align*}
R\Gamma_{A_\inf}(\frak X,\bb M)\dotimes_{A_\inf}\roi&\simeq R\Gamma_\sub{dR}(\frak X,\sigma^*_\sub{dR}\bb M)\\
R\Gamma_{A_\inf}(\frak X,\bb M)\dotimes_{A_\inf}W(k)&\simeq R\Gamma_\sub{crys}(\frak X_k,\sigma^*_\sub{crys}\bb M)\\
R\Gamma_{A_\inf}(\frak X,\bb M)\dotimes_{A_\inf}A_\inf[\tfrac1\mu]&\simeq R\Gamma_\sub{\'et}(X,\sigma^*_\sub{\'et}\bb M)\otimes_{\bb Z_p} A_\inf[\tfrac1\mu],
\end{align*}
assuming for the third comparison that $C$ is algebraically closed.
\end{corollary}
\begin{proof}
These equivalences, as well as perfectness of the cohomology, are deduced from the definition of $A\Omega_\frak X(\bb M)$ and the previous theorem exactly as in the proof of \cite[Thm.~14.3]{BhattMorrowScholze1} (note that \cite[Thm.~5.7]{BhattMorrowScholze1} holds for non-constant lisse $\bb Z_p$-sheaves, thanks to \cite[Thm.~5.11]{Scholze2013} or by mimicking Bhatt's deduction of the primitive comparison theorem from the above de Rham comparison \cite[Rmk.~8.4]{Bhatt2016}).
\end{proof}

With the technical foundations of Sections \ref{section_small_reps}--\ref{section_BKF_on_proetale} already in place, the constructions and proofs of this section will be relatively straightforward modifications of those already found in \cite{BhattMorrowScholze1}.

\begin{remark}[Comparison with prismatic cohomology]\label{remark_comp_to_prismatic}
Assume $\frak X$ is separated in this remark. Identifying $\bb M\in \BKF(\frak X,\phi)$ with a prismatic crystal $\cal F \in \op{F-CR}_{\Prism}(\frak X^{(1)}/(A_{\inf},\tilde\xi))$ via the equivalence of Theorem \ref{theorem_local_with_phi}, an analogue of \cite[Thm.~17.2]{BhattScholze2019} should show that there are natural equivalences  (independent of any choice of framing)  $R\Gamma_\sub{Zar}(\Spf R, A\Omega_{\frak X}(\bb M))\simeq R\Gamma_\Prism(R^{(1)}/(A_\sub{inf},\tilde\xi),\cal F)$, for all small affine opens $\Spf R\subseteq\frak X$; these should then globalise to a comparison of cohomologies $R\Gamma_{A_\inf}(\frak X,\bb M)\simeq R\Gamma_\Prism(\frak X^{(1)}/(A_\sub{inf},\tilde\xi),\cal F)$.

Similarly, the specialisation functors $\sigma^*_\sub{dR},\,\sigma^*_\sub{crys}$ should correspond to pulling back along morphisms of sites to $(\frak X^{(1)}/(A_\inf,\tilde\xi))_\Prism$ respectively from the crystalline site of $\frak X/\roi$ and the crystalline site of $\frak X_k/W(k)$.

However, we have not checked these expectations.
\end{remark}

\begin{remark}[No globally defined Higgs bundle]\label{remark_no_global_Higgs}
Given an $A^\bx_\inf(R)^{(1)}$-module with q-Higgs field $H$, the mod $\tilde\xi$ reduction  defines a usual Higgs field on the $R^{(1)}$-module $\res H:=H/\tilde\xi$; that is, an $R^{(1)}$-linear map $\res\Theta:\res H\to\res H\otimes_{R^{(1)}}\Omega^1_{R^{(1)}/\roi}$. Setting $\tilde\Omega_\frak X(\bb M):=A\Omega_\frak X(\bb M)/\tilde\xi$ and taking the second equivalence of Theorem \ref{theorem_cohomology_as_q_dr_or_Higgs} modulo $\tilde\xi$, we moreover see that $R\Gamma(\Spf R,\tilde\Omega_\frak X(\bb M))$ is represented by the Higgs complex of $\res H$. It is natural to ask whether these various Higgs bundles may be glued over all small opens in order to constructed a global Higgs bundle on $\frak X^{(1)}$, whose Higgs complex is equivalent to $\tilde\Omega_\frak X(\bb M)$ viewed as a complex of sheaves on $\frak X^{(1)}$, thereby defining a specialisation functor $\sigma^*_\sub{Higgs}$ from $\BKF(\frak X)$ to Higgs bundles on $\frak X^{(1)}$. Unfortunately this cannot be possible in the global case without extra data such as a chosen smooth lifting of $\frak X$ to $A_\inf/\xi^2$ (which will not always exist), c.f.~\cite[Rmk.~8.4]{BhattMorrowScholze1} \cite[Rmk.~4.12]{BhattScholze2019}.
\end{remark}

\subsection{Technical local-to-global lemmas}
In this subsection $R$ is a small formally smooth $\roi$-algebra, equipped with a fixed choice of framing, and $U:=\Spa(R[\tfrac1p],R)$ is its adic generic fibre. We fix a relative Breuil--Kisin--Fargues module without Frobenius $\bb M\in\BKF(R)$ and an integer $s\ge0$, and establish some technical lemmas concerning the d\'ecalage functor in order to prove Theorem \ref{theorem_specialisations}. These are all analogues of the lemmas which may be found in \cite[\S9]{BhattMorrowScholze1} and \cite[\S7]{Morrow_BMSnotes_simons}, so we will occasionally be brief and refer there for additional details.

There are global-to-local morphisms for the d\'ecalage functor
\[L\eta_{\phi^{-s}(\mu)} \hat{R\Gamma_\sub{pro\'et}(U,\bb M )}\To R\Gamma_\sub{Zar}(\Spf R,L\eta_{\phi^{-s}(\mu)}\hat{R\nu_*\bb M})\tag{t1}\]
\[L\eta_{[\zeta_{p^{r+s}}]-1} R\Gamma_\sub{pro\'et}(U,\bb M/\tilde\xi_r )\To R\Gamma_\sub{Zar}(\Spf R,L\eta_{[\zeta_{p^{r+s}}]-1}R\nu_*(\bb M/\tilde\xi_r))\tag{t2}.\] Here, and in what follows, we implicitly use the identification $\tilde\theta_r:\bb A_{\sub{inf},U}/\tilde\xi_r\isoto W_r(\hat\roi_U^+)$ which sends $\phi^{-s}(\mu)$ to $[\zeta_{p^{r+s}}]-1$. There is also a base change morphism
\[L\eta_{\phi^{-s}(\mu)} \hat{R\Gamma_\sub{pro\'et}(U,\bb M )}/\tilde\xi_r \To L\eta_{[\zeta_{p^{r+s}}]-1}R\Gamma_\sub{pro\'et}(U, \bb M/\tilde\xi_r)\tag{t3},\] and these fit into a commutative diagram
\[\xymatrix@C=0.7cm@R=1.5cm{
R\Gamma_\sub{Zar}(\Spf R,L\eta_{\phi^{-s}(\mu)}\hat{R\nu_*\bb M})/\tilde\xi_r \ar[r]^-{(t7)} & R\Gamma_\sub{Zar}(\Spf R,L\eta_{[\zeta_{p^{r+s}}]-1}R\nu_*(\bb M/\tilde\xi_r))\\
L\eta_{\phi^{-s}(\mu)} \hat{R\Gamma_\sub{pro\'et}(U,\bb M )}/\tilde\xi_r \ar[r]^-{(t3)}\ar[u]^{(t1)\op{mod}\tilde\xi_r} & L\eta_{[\zeta_{p^{r+s}}]-1}R\Gamma_\sub{pro\'et}(U, \bb M/\tilde\xi_r)\ar[u]_{(t2)}\\
L\eta_{\phi^{-s}(\mu)} R\Gamma_\sub{cont}(\Gamma, \Gamma(U_\infty,\bb M))/\tilde\xi_r \ar[r]^-{(t6)} \ar[u]^{(t4)\op{mod}\tilde\xi_r}& L\eta_{[\zeta_{p^{r+s}}]-1}R\Gamma_\sub{cont}(\Gamma, \Gamma(U_\infty,\bb M/\tilde\xi_r))\ar[u]_{(t5)}
}\]
where (t4) denotes $L\eta_{\phi^{-s}(\mu)}$ of the almost quasi-isomorphism of Theorem \ref{theorem_Ainf_reps_vs_pro_etale}, (t5) denotes the analogous morphism for the $W_r(\hat\roi_U^+)$-module $\bb M/\tilde\xi_r$ as in Theorem \ref{theorem_correspondence_for_B}, and finally (t6) and (t7) are further base change morphisms associated to the identification $\tilde\theta_r$.

\begin{proposition}\label{proposition_local_to_global}
The morphisms (t1)--(t7) are all equivalences.
\end{proposition}
\begin{proof}
Before applying $L\eta_{[\zeta_{p^{r+s}}]-1}$, the morphism (t5) is an almost quasi-isomorphism by Theorem \ref{theorem_correspondence_for_B}; it becomes an actual isomorphism after applying $L\eta_{[\zeta_{p^{r+s}}]-1}$ thanks to Lemma \ref{lemma_Leta_isom}(i).

Similarly, (t4) is a quasi-isomorphism by combining the almost quasi-isomorphism of Theorem \ref{theorem_Ainf_reps_vs_pro_etale} with Lemma \ref{lemma_Leta_isom}(ii).

To check that (t6) is a quasi-isomorphism we must interchange the d\'ecalage functor and the mod $\tilde\xi_r$ (note that $\Gamma(U_\infty,\bb M)/\tilde\xi_r=\Gamma(U_\infty,\bb M/\tilde\xi_r)$ by Proposition \ref{proposition_small_Ainf_sheaves}(ii)); to do that it is enough, by \cite[Rem.~9(ii)]{Morrow_BMSnotes_simons}, to show that all cohomology groups $H^i_\sub{cont}(\Gamma,\Gamma(U_\infty,\bb M)/\phi^{-s}(\mu))$ are $p$-torsion-free. But the generalised representation $\Gamma(U_\infty,\bb M)$ is trivial modulo $<\mu$ by Theorem \ref{theorem_Ainf_reps_vs_pro_etale}, hence trivial modulo $\mu$ by Theorem \ref{theorem_<mu_implies_mu}, hence trivial modulo $\phi^{-s}(\mu)$ (if $s>0$ we could also make the simpler argument that the generalised representation is trivial modulo $\phi^{-1}(\mu)$, hence modulo $\phi^{-s}(\mu)$, by Corollary \ref{corollary_mu_smallness}). To prove the desired $p$-torsion-freeness, we therefore reduce to the case $\bb M=\bb A_{\inf,U}$, in which case it is Lemma \ref{lemma_mu_1}(ii)

From the commutativity of the diagram it now follows that (t3) is a quasi-isomorphism.

Next we show that (t2) is an isomorphism. Given a $p$-adically complete formally \'etale $R$-algebra $R'$ and setting $R'_\infty:=R_\infty\hat\otimes_RR'=\roi\pid{\ul T^{\pm1/p^\infty}}\hat\otimes_{\roi\pid{\ul T^{\pm1}}}R'$, we claim that the canonical map \begin{equation}\big(L\eta_{[\zeta_{p^{r+s}}]-1} R\Gamma_\sub{cont}(\Gamma,\Gamma(U_\infty,\bb M/\tilde\xi_r))\big)\hat\otimes_{W_r(R)}W_r(R')\to L\eta_{[\zeta_{p^{r+s}}]-1} R\Gamma_\sub{cont}(\Gamma,\Gamma(\Sp R'_\infty[\tfrac1p],\bb M/\tilde\xi_r))\label{equation_partt2}\end{equation} is a quasi-isomorphism: in fact, Elkik's results \cite{Elkik1973} imply that $W_r(R')$ is the $p$-adic completion of an \'etale $W_r(R)$-algebra, tensoring with which may be interchanged with both the $L\eta$ and the group cohomology (c.f., the proof of \cite[Lem.~9.9]{BhattMorrowScholze1}). This reduces the claim to checking that \[\Gamma(U_\infty,\bb M/\tilde\xi_r)\hat\otimes_{W_r(R)}W_r(R')\to \Gamma(\Sp R'_\infty[\tfrac1p],\bb M/\tilde\xi_r)\] is an isomorphism. But this follows from the identification $\bb M/\tilde\xi_r|_{U_\infty}=\Gamma(U_\infty,\bb M/\tilde\xi_r)\otimes_{W_r(R_\infty)}W_r(\hat\roi_U^+)|_{U_\infty}$ from Proposition \ref{proposition_small_sheaves}(v) and fact that $W_r(R_\infty')=W_r(R_\infty)\hat\otimes_{W_r(R)}W_r(R')$, which is a consequence of the behaviour of Witt vectors on finite \'etale morphisms \cite[Thm.~2.4]{vanderKallen1986}.

The previous paragraph (in conjunction with equivalence (t5) for both $R$ and $R'$) has shown that for any $\Spf R'\in(\Spf R)_\sub{\'et}$, the sections $R\Gamma_\sub{\'et}(\Spf R',-)$ of $L\eta_{[\zeta_{p^{r+s}}]-1}R\Gamma_\sub{pro\'et}(U, \bb M/\tilde\xi_r)\hat\otimes_{W_r(R)}W_r(\roi_{\Spf R})$ are precisely $L\eta_{[\zeta_{p^{r+s}}]-1}R\Gamma_\sub{pro\'et}(\Sp R'[\tfrac1p], \bb M/\tilde\xi_r)$. The same argument as in the second part of the proof of \cite[Corol.~8.13]{BhattMorrowScholze1} now shows that (t2) is an equivalence.

It remains to prove that (t1) is an equivalence (from which (t7) will follow); this follows from the same sheafification argument already given in the proof of \cite[Prop.~9.14]{BhattMorrowScholze1}, so we will be brief. Firstly, Corollary~\ref{corollary_completeness} implies that \[L\eta_{\phi^{-s}(\mu)} R\Gamma_\sub{cont}(\Gamma, \Gamma(U_\infty,\bb M))\To \op{Rlim}_rL\eta_{\phi^{-s}(\mu)} R\Gamma_\sub{cont}(\Gamma, \Gamma(U_\infty,\bb M))/\tilde\xi_r\] is an equivalence; using the already established equivalences of the diagram this may be rewritten as 
\[L\eta_{\phi^{-s}(\mu)} \hat{R\Gamma_\sub{pro\'et}(U,\bb M )}\quis \op{Rlim}_rL\eta_{[\zeta_{p^{r+s}}]-1}R\Gamma_\sub{pro\'et}(U,\bb M/\tilde\xi_r).\] This remains true if we replace $R$ by any formally \'etale $R$-algebra $R'$ and $U$ by $\Sp R'[\tfrac1p]$, i.e., the \'etale presheaves \begin{equation}(\Spf R)_\sub{\'et}\ni\Spf R'\mapsto L\eta_{\phi^{-s}(\mu)}\hat{R\Gamma_\sub{pro\'et}(\Sp R'[\tfrac1p],\bb M)}\text{  and  } \op{Rlim}_rL\eta_{[\zeta_{p^{r+s}}]-1}R\Gamma(\Sp R'[\tfrac1p],\bb M/\tilde\xi_r)\label{eqn_presheaf}\end{equation} coincide (on affines, off which we do not in any case define them). To sheafify this, let $j:\frak X_\sub{Zar}\to\frak X_\sub{Zar}^\sub{psh}$ be the map of topoi from sheaves to presheaves, so that $j^*$ is sheafification; set $\nu^\sub{psh}:=j\nu:X_\sub{pro\'et}\to\frak X^\sub{psh}_\sub{Zar}$. Then $L\eta_{\phi^{-s}(\mu)}\hat{R\nu_*\bb M}=j^*L\eta_{\phi^{-s}(\mu)}R\nu_*^\sub{psh}\hat{\bb M}$, where $\hat{\bb M}$ is the derived $p$-adic completion of $\bb M$; this reduces (t1) to showing that the counit $L\eta_{\phi^{-s}(\mu)}R\nu_*^\sub{psh}\hat{\bb M}\to Rj_*j^*L\eta_{\phi^{-s}(\mu)}R\nu_*^\sub{psh}\hat{\bb M}$ is an equivalence. But the identification of the presheaves (\ref{eqn_presheaf}) means that $L\eta_{\phi^{-s}(\mu)}\hat{R\nu_*\bb M}=\op{Rlim}_rL\eta_{[\zeta_{p^{r+s}}]-1}R\nu_*^\sub{psh}(\bb M/\tilde\xi_r)$, so then \cite[Lem.~9.15]{BhattMorrowScholze1} reduces the desired counit equivalence to showing that the analogous counit map $L\eta_{[\zeta_{p^{r+s}}]-1}R\nu_*^\sub{psh}(\bb M/\tilde\xi_r)\to Rj_*j^*L\eta_{[\zeta_{p^{r+s}}]-1}R\nu_*^\sub{psh}(\bb M/\tilde\xi_r)$ is an equivalence for each $r\ge1$, i.e., that $R\nu_*^\sub{psh}(\bb M/\tilde\xi_r)$ is already a sheaf. But combining equivalence (\ref{equation_partt2}) with (t5) (for both $R$ and $R'$) shows that \[L\eta_{[\zeta_{p^{r+s}}]-1}R\nu_*^\sub{psh}(\bb M/\tilde\xi_r)\simeq\big(L\eta_{[\zeta_{p^{r+s}}]-1} R\Gamma_\sub{pro\'et}(U,\bb M/\tilde\xi_r)\big)\hat\otimes_{W_r(R)}W_r(\roi_{\Spf R}),\] which is indeed a sheaf.
\end{proof}

\subsection{The Bockstein connection: de Rham and crystalline specialisations}
As at the start of the section, let $\frak X$ be a smooth $p$-adic formal $\roi$-scheme; let $\bb M\in\BKF(\frak X)$ and $r\ge1$. By triviality of $\bb M$ modulo $\xi_r$, the $W_r(\roi_\frak X)$-module $\nu_*(\bb M/\xi_r)$ on $\frak X$ is locally finite free; in this subsection we equip this with a $W_r(\roi)$-linear, flat de Rham--Witt connection \[\nabla_{\xi_r}:\nu_*(\bb M/\xi_r)\To \nu_*(\bb M/\xi_r)\otimes_{W_r(\roi_\frak X)}W_r\Omega^1_{\frak X/\roi}\] and identify $A\Omega_\frak X(\bb M)/\xi_r$ with the resulting de Rham--Witt complex \begin{equation}\nu_*(\bb M/\xi_r)\xto{\nabla_{\xi_r}} \nu_*(\bb M/\xi_r)\otimes_{W_r(\roi_\frak X)}W_r\Omega^1_{\frak X/\roi}\xto{\nabla_{\xi_r}} \nu_*(\bb M/\xi_r)\otimes_{W_r(\roi_\frak X)}W_r\Omega^2_{\frak X/\roi}\xto{\nabla_{\xi_r}}\cdots.\label{eqn_dRW}\end{equation} Here we use the relative de Rham--Witt complex of Langer--Zink \cite{LangerZink2004} (see also \cite[\S10]{BhattMorrowScholze1} for a summary) and de Rham--Witt connections of Bloch, \'Etesse, and Langer--Zink \cite{Etesse1988, Bloch2004, LangerZink2005}.

We will use  the complex of sheaves of $W_r(\roi_{\frak X})$-modules \begin{equation}\tilde{W_r\Omega}_\frak X=(L\eta_{\phi^{-r}(\mu)}\hat{R\nu_*\bb A_{\inf,X}})/\xi_r\stackrel{\theta_r}{\isoto}L\eta_{[\zeta_{p^r}]-1}R\nu_*W_r(\hat\roi_X^+);\label{Wrtilde}\end{equation} from \cite[Thm.~9.4]{BhattMorrowScholze1} (we stress that we have renormalised the definition of \cite{BhattMorrowScholze1}, where instead $\tilde{W_r\Omega}_\frak X=A\Omega_\frak X/\tilde\xi_r$, via $\phi^r$), and their description (in which we ignore Breuil--Kisin twists for simplicity):

\begin{theorem}[{``$p$-adic Cartier isomorphism'' \cite[Thm.~11.1]{BhattMorrowScholze1}}]\label{theorem_p-adic_Cartier}
There are natural isomorphisms $\cal H^n(\tilde{W_r\Omega}_\frak X)\cong W_r\Omega_\frak X^n$, for each $r\ge 1$ and $n\ge0$, with the following compatibilities:
\begin{enumerate}
\item Equipping the cohomology sheaves of $\tilde{W_r\Omega}_\frak X$ with the Bockstein associated to $\xi_r$, the isomorphism $\cal H^*(\tilde{W_r\Omega}_\frak X)\cong W_r\Omega_\frak X^n$ is one of differential graded $W_r(\roi)$-algebras
\item The Frobenius $F$ on the de Rham--Witt complex is induced by the Frobenius $\phi$ on $\bb A_{\inf,X}$.
\item The restriction $R:W_{r+1}\Omega_\frak X^n\to W_{r}\Omega_\frak X^n$ corresponds to \[\left(\tfrac{[\zeta_{p^{r}}]-1}{[\zeta_{p^{r+1}}]-1}R\right)^n:\tau^{\le n}L\eta_{[\zeta_{p^{r+1}}]-1}R\nu_*W_{r+1}(\hat\roi_X^+)\to \tau^{\le n}L\eta_{[\zeta_{p^{r}}]-1}R\nu_*W_r(\hat\roi_X^+).\]
\end{enumerate}
(See also \cite[\S2.2]{Morrow_Vanishing} for some discussion of these compatibilities.)
\end{theorem}


We need the following analogue for $\bb M$ of equivalence (\ref{Wrtilde}):


\begin{lemma}\label{lemma_base_change1}
The canonical base change map
\[\theta_r:(L\eta_{\phi^{-r}(\mu)}\hat{R\nu_*\bb M})/\xi_r\To L\eta_{[\zeta_{p^r}]-1}R\nu_*(\bb M/\xi_r)=\nu_*(\bb M/\xi_r)\otimes_{W_r(\roi_\frak X)}\tilde{W_r\Omega}_\frak X\] is an equivalence. (Note that the arrow is merely $W_r(\roi)$-linear, but the equality is $W_r(\roi_\frak X)$-linear.)

\end{lemma}
\begin{proof}
Note first that the equality is an immediate consequence of the triviality modulo $\xi_r$ of $\bb M$, and the projection formula. 

Given any small affine open $\Spf R$ of $\frak X$ with generic fibre $U$, the sections of the domain of $\theta_r$ on $\Spf R$ are \begin{equation}(L\eta_{\phi^{-r}(\mu)}R\Gamma_\sub{cont}(\Gamma,\Gamma(U_\infty,\bb M)))/\xi_r\label{eqn_exchange}\end{equation} by the equivalences (t1) and (t4) of Proposition \ref{proposition_local_to_global}. Similarly, the sections of the codomain of $\theta_r$ on $\Spf R$ are $L\eta_{[\zeta_{p^r}]-1}R\Gamma_\sub{cont}(\Gamma,\Gamma(U_\infty,\bb M/\xi_r))$ by triviality of $\bb M$ mod $\xi_r$ and \cite[Thm.~9.4(ii)]{BhattMorrowScholze1}. The goal is therefore to show that the 
``mod $\xi_r$'' in (\ref{eqn_exchange}) can be interchanged with the d\'ecalage functor; for this it is enough to check (by \cite[Rem.~9(ii) \& footnote 44]{Morrow_BMSnotes_simons}) that all the cohomology groups $H^i(\Gamma,\Gamma(U_\infty,\bb M)/\phi^{-r}(\mu))$, for $i\ge0$, are $p$-torsion-free. But the representation is trivial modulo $<\mu$ by Proposition \ref{proposition_small_Ainf_sheaves}, hence trivial modulo $\phi^{-1}(\mu)$ by Corollary \ref{corollary_mu_smallness}, hence trivial modulo $\phi^{-r}(\mu)$. To prove the desired $p$-torsion-freeness, we therefore reduce to the case of the trivial representation and appeal to Lemma \ref{lemma_mu_1}(ii).
\end{proof}

Since $L\eta_{\phi^{-r}(\mu)}\hat{R\nu_*\bb M}$ is a module over the $\bb E_\infty$-$A_\inf$-algebra $L\eta_{\phi^{-r}(\mu)}\hat{R\nu_*\bb A_{\inf,X}}$ (as the d\'ecalage functor is lax symmetric monoidal \cite[Prop.~6.7]{BhattMorrowScholze1}), and similarly modulo $\xi_r$, it follows formally that the cohomology sheaves $\cal H^*((L\eta_{\phi^{-r}(\mu)}\hat{R\nu_*\bb M})/\xi_r)$ inherit the structure of a sheaf of graded modules over $\cal H^*((L\eta_{\phi^{-r}(\mu)}\hat{R\nu_*\bb A_{\inf,X}})/\xi_r)$. Moreover, the respective Bockstein maps $\op{Bock}_{\xi_r}:\cal H^*\to\cal H^{*+1}$ upgrade the former into a differential graded module over the latter by, e.g., \cite[Lem.~132]{Tsuji_simons}. Appealing to the identifications of Lemma \ref{lemma_base_change1} and the preceding paragraph, we have shown that the Bockstein gives the graded $W_r(\roi)$-module $\nu_*(\bb M/\xi_r)\otimes_{W_r(\roi_\frak X)}W_r\Omega^\blob_{\frak X/\roi}$ the structure of a differential graded module over the differential graded algebra $W_r\Omega^\blob_{\frak X/\roi}$. Concretely, such a module structure is necessarily the data of a flat connection 
\[\nabla_{\xi_r}:\nu_*(\bb M/\xi_r)\To \nu_*(\bb M/\xi_r)\otimes_{W_r(\roi_\frak X)}W_r\Omega^1_{\frak X/\roi}\] such that $\nu_*(\bb M/\xi_r)\otimes_{W_r(\roi_\frak X)}W_r\Omega^\blob_{\frak X/\roi}$ identifies with the associated de Rham--Witt complex (\ref{eqn_dRW}).

This completes the construction of the flat de Rham--Witt connection on the locally free $W_r(\roi_\frak X)$-module $\nu_*(\bb M/\xi_r)$, thereby defining a functor
\begin{align*}
\sigma^*_{\sub{dRW},r}:\BKF(\frak X)&\To\categ{5cm}{locally finite free $W_r(\roi_\frak X)$-modules with $W_r(\roi)$-linear, flat de Rham--Witt connection}\\
\bb M&\mapsto (\nu_*(\bb M/\xi_r),\nabla_{\xi_r})
\end{align*}
We note that these are compatible in $r$ in the following sense:

\begin{lemma}\label{lemma_compatibility_in_r}
The canonical identification of modules $\nu_*(\bb M/\xi_{r+1})\otimes_{W_{r+1}(\roi_\frak X),R}W_r(\roi_\frak X)=\nu_*(\bb M/\xi_r)$ (the base change being taken along the restriction map) is compatible with de Rham--Witt connections.
\end{lemma}
\begin{proof}
This is essentially a formal consequence of the Bockstein definition of the connection and the description of the restriction map in Theorem \ref{theorem_p-adic_Cartier}, but we provide the details.

It is enough to prove the compatibility on each small affine open $\Spf R$, and we will implicitly use the local-to-global identifications of Proposition \ref{proposition_local_to_global} in the proof (though this is not really necessary: we could instead work with K-flat resolutions and local sections of the complexes).

Let $C^\blob$ be a $\mu$-torsion-free complex of $A_\inf$-modules, supported in degrees $\ge0$, which represents $\hat{R\Gamma_\sub{pro\'et}(U,\bb M)}$; so $\eta_{\phi^{-r}(\mu)}C^\blob$ represents $L\eta_{\phi^{-r}(\mu)}\hat{R\Gamma_\sub{pro\'et}(U,\bb M)}$. Let $\eta_{\phi^{-r}(\mu)}'C^\blob$ be the subcomplex of $\eta_{\phi^{-r}(\mu)}C^\blob$ given by \[(\eta_{\phi^{-r}(\mu)}'C^\blob)^i=\begin{cases}(\eta_{\phi^{-r}(\mu)}C^\blob)^i & i<n, \\ \{x\in\phi^{-r}(\mu)^nC^n : dx\in \xi_r\phi^{-r}(\mu)^{n+1}C^{n+1}\} & i=n, \\ 0 & i>n,\end{cases}\] so that the quotient $\eta_{\phi^{-r}(\mu)}'C^\blob/\xi_r\sigma^{\le n}\eta_{\phi^{-r}(\mu)}C^\blob$ is equal to $\tau^{\le n}((\eta_{\phi^{-r}(\mu)}C^\blob)/\xi_r)$. Here $\sigma^{\le n}$ denotes naive truncation, while $\tau^{\le n}$ is canonical truncation (i.e., given by $\ker d$ in degree $n$).

Multiplication by $\left(\phi^{-r}(\mu)/\phi^{-(r+1)}(\mu)\right)^n$ restricts both to $\eta_{\phi^{-(r+1)}(\mu)}C^\blob\to \eta_{\phi^{-r}(\mu)}C^\blob$ (hence also to $\xi_{r+1}\eta_{\phi^{-(r+1)}(\mu)}C^\blob\to \xi_r \eta_{\phi^{-r}(\mu)}C^\blob$), and to $\eta_{\phi^{-(r+1)}(\mu)}'C^\blob\to \eta_{\phi^{-r}(\mu)}'C^\blob$. Passing to the quotient in the previous paragraph then induces \[R_{r,n}:=\times \left(\tfrac{\phi^{-r}(\mu)}{\phi^{-(r+1)}(\mu)}\right)^n: \tau^{\le n}((\eta_{\phi^{-r}(\mu)}C^\blob)/\xi_r)\To \tau^{\le n}((\eta_{\phi^{-{r+1}}(\mu)}C^\blob)/\xi_{r+1}).\] Theorem \ref{theorem_p-adic_Cartier}(iii) and Lemma \ref{lemma_base_change1} implies that the induced map on $H^n(-)$ is simply the restriction map \[\op{can}\otimes R:\Gamma(\Spf R,\nu_*(\bb M/\xi_{r+1}))\otimes_{W_{r+1}(R)}W_{r+1}\Omega^n_{R/\roi}\To \Gamma(\Spf R,\nu_*(\bb M/\xi_r))\otimes_{W_r(R)}W_r\Omega^n_{R/\roi}.\]

To prove the compatibility with the de Rham--Witt connections, define a subcomplex $D^\blob_{r,n}\subseteq(\eta_{\phi^{-r}(\mu)})/\xi_r^2$ by \[D^\blob_{r,n}=\begin{cases}
(\eta_{\phi^{-r}(\mu)})^i/\xi_r^2 & i<n,\\
\{x\in (\eta_{\phi^{-r}(\mu)})^n/\xi_r^2 : dx \text{ is divisible by }\xi_r\} & i=n,\\
\{x\in (\eta_{\phi^{-r}(\mu)})^{n+1}/\xi_r : dx =0 \text{ in }(\eta_{\phi^{-r}(\mu)})^{n+2}/\xi_r\} & i=n+1,\\
0 & i>n+1,
\end{cases}
\]
By construction, the short exact sequence  $0\to (\eta_{\phi^{-r}(\mu)}C^\blob)/\xi_r\xto{\xi_r}(\eta_{\phi^{-r}(\mu)}C^\blob)/\xi_r^2\to (\eta_{\phi^{-r}(\mu)}C^\blob)/\xi_r\to 0$ restricts to the short exact sequence in the bottom of the following diagram; and multiplication by the same element as the previous paragraph induces a $R_{r,n}:D^\blob_{r,n+1}\to D^\blob_{r,n}$ making the diagram commute:
\[\xymatrix{
0\ar[r]& \tau^{\le n+1}((\eta_{\phi^{-(r+1)}(\mu)}C^\blob)/\xi_{r+1})\ar[r]^-{\xi_{r+1}}\ar[d]^{R_{r,n+1}}&D^\blob_{r+1,n}\ar[r]\ar[d]^{R_{r,n}}& \tau^{\le n}((\eta_{\phi^{-(r+1)}(\mu)}C^\blob)/\xi_{r+1})\ar[r]\ar[d]^{R_{r,n}}& 0\\
0\ar[r]& \tau^{\le n+1}((\eta_{\phi^{-r}(\mu)}C^\blob)/\xi_r)\ar[r]_-{\xi_r}&D^\blob_{r,n}\ar[r]& \tau^{\le n}((\eta_{\phi^{-r}(\mu)}C^\blob)/\xi_r)\ar[r]& 0
}\]
Taking the boundary maps of the horizontal short exact sequences yields the desired compatibility with the de Rham--Witt connections.
\end{proof}

The associated de Rham--Witt complexes describe $A\Omega(\bb M)/\xi_{r}$:

\begin{proposition}\label{proposition_drw_comparison}
There is a natural, for $\bb M\in\BKF(\frak X)$, equivalence between $A\Omega(\bb M)/\xi_r$ and the associated de Rham--Witt complex $\nu_*(\bb M/\xi_r)\otimes_{W_r(\roi_\frak X)}W_r\Omega^\blob_{\frak X/\roi}$ of line (\ref{eqn_dRW}). These equivalences are compatible with $r$, in the sense that the canonical reduction map $A\Omega(\bb M)/\xi_{r+1}\to A\Omega(\bb M)/\xi_{r}$ corresponds to the reduction map on the de Rham--Witt complexes.
\end{proposition}
\begin{proof}
The de Rham--Witt connection on $\nu_*(\bb M/\xi_r)$ was constructed so that the de Rham--Witt complex is isomorphic to $[\cal H^*((L\eta_{\phi^{-r}(\mu)}\hat{R\nu_*\bb M})/\xi_r),\op{Bock}_{\xi_r}]$ (by which we mean the cohomology sheaves as indicated, equipped with the associated Bockstein). The proposition therefore follows by writing $A\Omega(\bb M)=L\eta_{\xi_r}L\eta_{\phi^{-r}(\mu)}\hat{R\nu_*\bb M}$ \cite[Lem.~6.11]{BhattMorrowScholze1} and appealing to the fundamental connection between the d\'ecalage functor and the Bockstein \cite[Prop.~6.12]{BhattMorrowScholze1}. Compatibility in $r$ is a consequence of Lemma \ref{lemma_compatibility_in_r}.
\end{proof}

\begin{definition}
The {\em de Rham specialisation} functor is \[\sigma_\sub{dR}^*:=\sigma^*_{\sub{dRW},1}:\BKF(\frak X)\To\categ{5cm}{vector bundles with flat $\roi$-linear connection on $\frak X$}\]
\end{definition}

Next we take into account the Frobenius. Under the identifications $\theta_r:\bb A_{\inf,X}/\xi_r\isoto W_r(\hat\roi_X^+)$, the Witt vector Frobenius $F$ on the target corresponds to $\phi$ on the source \cite[Lem.~3.4]{BhattMorrowScholze1}; pushing forward to $\frak X$ we deduce that the same is true for $\theta_r:\nu_*(\bb A_{\inf,X}/\xi_r)\isoto W_r(\roi_\frak X)$. Therefore, given $\bb M\in\BKF(\frak X)$, we have a natural identification \begin{equation}\nu_*((\phi^*\bb M)/\xi_r)=\nu_*(\bb M/\xi_{r+1})\otimes_{W_{r+1}(\roi_\frak X),F}W_r(\roi_\frak X).\label{equation_phi_pullback}\end{equation} Note that both sides of (\ref{equation_phi_pullback}) are locally finite free $W_r(\roi_\frak X)$-modules equipped with a de Rham--Witt connection: firstly, the connection on $\nu_*((\phi^*\bb M)/\xi_r)$ is that defined by applying the de Rham--Witt specialisation functor to $\phi^*\bb M\in\BKF(\frak X)$, i.e., $\nu_*((\phi^*\bb M)/\xi_r)=\sigma^*_{\sub{dRW}, r}(\phi^*\bb M)$; secondly, the connection on $\nu_*(\bb M/\xi_{r+1})\otimes_{W_{r+1}(\roi_\frak X),F}W_r(\roi_\frak X)$ is defined by pulling back the connection on $\nu_*(\bb M/\xi_{r+1})$ along the Frobenius in the sense of the next remark.

\begin{remark}
Given a sheaf of $W_{r+1}(\roi_\frak X)$-modules $N$ equipped with a de Rham--Witt connection $\nabla$, pullback along the Frobenius $F:W_{r+1}(\roi_\frak X)\to W_{r}(\roi_\frak X)$ defines a de Rham--Witt connection $F^*\nabla$ on the sheaf of $W_{r+1}(\roi_\frak X)$-modules $F^*N=N\otimes_{W_{r+1}(\roi_\frak X),F}W_{r}(\roi_\frak X)$. Explicitly, in terms of local sections, $F^*\nabla$ is given by
\begin{align*}
F^*\nabla:N\otimes_{W_{r+1}(\roi_\frak X),F}W_{r}(\roi_\frak X)&\To N\otimes_{W_{r+1}(\roi_\frak X),F}W_r\Omega^1_{\frak X/\roi}\\
n\otimes f&\mapsto (\op{id}\otimes pF)\nabla(n)+n\otimes df,
\end{align*}
where $pF:W_{r+1}\Omega^1_{\frak X/\roi}\to W_r\Omega^1_{\frak X/\roi}$. In terms of the language of Section \ref{section_q_connections}, this construction is simply base changing along the morphism of differential (sheaves of) rings
\[\xymatrix{
W_{r+1}(\roi_\frak X)\ar[r]^d\ar[d]_F & W_{r+1}\Omega^1_{\frak X/\roi}\ar[d]_{pF}\\
W_r(\roi_\frak X)\ar[r]_d&W_r\Omega^1_{\frak X/\roi}
}\]

\end{remark}

\begin{lemma}\label{lemma_drw_compatibility}
The identification (\ref{equation_phi_pullback}) is compatible with the connections on each side.
\end{lemma}
\begin{proof}
Similarly to Lemma \ref{lemma_compatibility_in_r}, this is a formal consequence of the Bockstein construction and compatibilities of Theorem \ref{theorem_p-adic_Cartier}. Let $\Phi:L\eta_{\phi^{-(r+1)}(\mu)}\hat{R\nu_*\bb M}\quis L\eta_{\phi^{-r}(\mu)}\hat{R\nu_*( \phi^*\bb M)}$ denote the canonical $\phi$-semilinear equivalence; this induces $\Phi_i:(L\eta_{\phi^{-(r+1)}(\mu)}\hat{R\nu_*\bb M})/\xi_{r+1}^i\to (L\eta_{\phi^{-r}(\mu)}\hat{R\nu_*( \phi^*\bb M)})/\xi_r^i$ for any $i\ge1$ (we only need $i=1,2$), since $\phi(\xi_{r+1})=\xi_r\tilde\xi$.

Passing to $H^n(-)$ induces in particular \begin{align*}H^n(\Phi_1):\nu_*(\bb M/\xi_{r+1})\otimes_{W_{r+1}(\roi_\frak X)}W_{r+1}\Omega^n_{\frak X/\roi}\To& \nu_*((\phi^*\bb M)/\xi_r)\otimes_{W_{r}(\roi_\frak X)}W_{r}\Omega^n_{\frak X/\roi}\\
&=(\nu_*(\bb M/\xi_{r+1})\otimes_{W_{r+1}(\roi_\frak X),F}W_r(\roi_\frak X))\otimes_{W_r(\roi_\frak X)}W_{r}\Omega^n_{\frak X/\roi},\end{align*} where the  identification is (\ref{equation_phi_pullback}). Theorem \ref{theorem_p-adic_Cartier}(ii) implies that this map is $\op{can}\otimes F$. The desired compatibility now follows by taking the $H^0\to H^1$ boundary maps in the commutative diagram of fibre sequences
\[\xymatrix{
(L\eta_{\phi^{-(r+1)}(\mu)}\hat{R\nu_*\bb M})/\xi_{r+1}\ar[r]^{\xi_{r+1}}\ar[d]^{\tilde\xi\Phi_1} & (L\eta_{\phi^{-(r+1)}(\mu)}\hat{R\nu_*\bb M})/\xi_{r+1}^2\ar[r]\ar[d]^{\Phi_2} & (L\eta_{\phi^{-(r+1)}(\mu)}\hat{R\nu_*\bb M})/\xi_{r+1}\ar[d]^{\Phi_1}\\
(L\eta_{\phi^{-r}(\mu)}\hat{R\nu_*\bb M})/\xi_{r}\ar[r]^{\xi_{r}} & (L\eta_{\phi^{-r}(\mu)}\hat{R\nu_*\bb M})/\xi_{r}^2\ar[r] & (L\eta_{\phi^{-r}(\mu)}\hat{R\nu_*\bb M})/\xi_{r}
}\]
and recalling that $\tilde\xi\equiv p$ mod $\xi_r$.
\end{proof}

\begin{corollary}\label{corollary_frob+nilp}
Let $(\bb M,\phi_\bb M)\in\BKF(\frak X,\phi)$. Then
\begin{enumerate}
\item $\phi_\bb M$ induces an isomorphism of $W_r(\roi_\frak X)[\tfrac1p]$-modules with de Rham--Witt connection \[F^*(\sigma^*_{\sub{dRW},r+1}\bb M)[\tfrac1p]\isoto (\sigma^*_{\sub{dRW},r}\bb M)[\tfrac1p]\]
\item The connection $\nabla_{\xi}$ on $\sigma^*_\sub{dR}\bb M$ is $p$-adically nilpotent.
\end{enumerate}
\end{corollary}
\begin{proof}
Part (i) follows from Lemma \ref{lemma_drw_compatibility} and the fact that $\theta_r(\tilde\xi)=p$ (indeed, $\tilde\xi\equiv p$ mod $\mu$). Part (ii) follows from a similar argument to Lemma \ref{lemma_phi_implies_xi_nilp}(iii).
\end{proof}

We are now prepared to define the crystalline specialisation $\sigma^*_\sub{crys}\bb M$ of any $\bb M\in\BKF(\frak X,\phi)$. Base changing each $\nu_*(\bb M/\xi_r)$ to the special fibre and taking the limit over $r$ (along the restriction maps) yields \[\sigma^*_\sub{crys}(\bb M):=\projlim_r\big(\nu_*(\bb M/\xi_r)\otimes_{W_r(\roi_\frak X)}W_r(\roi_{\frak X_k})\big),\] a locally finite free $W(\roi_{\frak X_k})$-module equipped with flat de Rham--Witt connection \[\nabla:\sigma^*_\sub{crys}(\bb M)\to \sigma^*_\sub{crys}(\bb M)\otimes_{W(\roi_{\frak X_k})}W\Omega^1_{\frak X_k/k}.\] By Corollary \ref{corollary_frob+nilp}, the Frobenius structure $\phi_\bb M$ on $\bb M$ induces a Frobenius structure on $\sigma^*_\sub{crys}\bb M$, i.e., an isomorphism of modules with connection $F^*(\sigma^*_\sub{crys}\bb M)[\tfrac1p]\isoto (\sigma^*_\sub{crys}\bb M)[\tfrac1p]$ (to be precise, in order to take the limit over $r$ in Corollary \ref{corollary_frob+nilp}(i), one should note that the Frobenius structure exists after introducing a bounded $p$-power denominator which does not depend on $r$), and the connection on the $\roi_{\frak X_k}$-module $(\sigma^*_\sub{crys}\bb M)\otimes _{W(\roi_{\frak X_k})}\roi_{\frak X_k}$ is nilpotent. By a theorem of Bloch \cite{Bloch2004} such data corresponds to an F-crystal on $\frak X_k$, which we will abusively also denote by $\sigma^*_\sub{crys}\bb M$, with the property that the crystalline cohomology complex $Ru_*\sigma^*_\sub{crys}\bb M$ is equivalent to the de Rham--Witt complex $\sigma^*_\sub{crys}(\bb M)\otimes_{W(\roi_{\frak X_k})}W\Omega^\blob_{\frak X_k/k}$ of the module with connection. This defines the 
{\em crystalline specialisation} \[\sigma^*_\sub{crys}:\BKF(\frak X,\phi)\To \categ{4cm}{locally finite free $F$-crystals on $\frak X_k$}.\]

We finish the subsection by noting explicitly that we have already proved the first two equivalences of Theorem \ref{theorem_specialisations}:

\begin{proof}[Proof of first two equivalences of Theorem \ref{theorem_specialisations}]
In Proposition \ref{proposition_drw_comparison} we have already constructed a natural identification \[A\Omega(\bb M)/\xi_r\simeq\nu_*(\bb M/\xi_r)\otimes_{W_r(\roi_\frak X)}W_r\Omega^\blob_{\frak X/\roi}.\] The first two equivalences claimed in Theorem \ref{theorem_specialisations} follow respectively by applying this in the case $r=1$ or base changing further along the canonical maps $A_\inf/\xi_r\to W_r(k)$ and then taking $\projlim_r$.
\end{proof}

\subsection{The \'etale specialisation}\label{ss_etale}
We finally define the specialisation functor $\sigma^*_\sub{\'et}$, which associates to each Breuil--Kisin--Fargues module over $\frak X$ a locally free lisse $\bb Z_p$-sheaf on $X$.

Recall from \cite[\S8]{Scholze2013} that a {\em locally free lisse $\bb Z_p$-sheaf} on $X$ is a sheaf of modules over the pro-\'etale sheaf $\hat{\bb Z}_p:=\projlim_r\bb Z/p^r\bb Z$ which is pro-\'etale locally isomorphic to a finite direct sum of copies of $\hat{\bb Z}_p$. It is known that the aforementioned inverse limit has no higher derived inverse limits and that the sequence of pro-\'etale sheaves $0\to\hat{\bb Z}_p\to \bb A_{\inf,X}\xto{\phi-1}\bb A_{\inf,X}\to 0$ is short exact (e.g., \cite[Lem.~3.5]{Morrow_Vanishing}).

The following proposition does not require $X$ to arise as the generic fibre of a smooth formal scheme over $\roi$, so we state it more generally:

\begin{proposition}\label{proposition_etale}
Let $X$ be a reduced rigid analytic variety over $C$ and $\bb M$ a locally finite free sheaf of $\bb A_{\inf,X}$-modules on $X_\sub{pro\'et}$, equipped with a Frobenius semi-linear isomorphism $\phi_{\bb M}:\bb M[\tfrac1\xi]\isoto\bb M[\tfrac1{\tilde\xi}]$. Then 
\begin{enumerate}
\item $\bb L:=(\bb M\otimes_{\bb A_{\inf,X}}W(\hat\roi_{X^\flat}))^{\phi_\bb M=1}$ is a locally free lisse $\bb Z_p$-sheaf of the same rank as $\bb M$.
\item  Let $r\in\bb Z$ be large enough so that $\tilde\xi^r\bb M\subseteq\phi(\bb M)$, whence $\phi^{-1}$ restricts to $\tfrac1{\mu^r}\bb M$; then $\bb L\subseteq\tfrac1{\mu^r}\bb M$ and the sequence \[0\To\bb L\To\tfrac1{\mu^r}\bb M\xto{1-\phi^{-1}}\tfrac1{\mu^r}\bb M\To 0\] is exact.
\item The inclusion $\bb L\subseteq \tfrac1{\mu^r}\bb M$ induces an identification $\bb L\otimes_{\hat{\bb Z}_p}\bb A_{\inf,X}[\tfrac1\mu]=\bb M[\tfrac1\mu]$. More precisely, given $s\le r$ such that $\phi(\bb M)\subseteq \tilde\xi^s\bb M$, then \[\bb L\otimes_{\hat{\bb Z}_p}\mu^r\bb A_{\inf,X}\subseteq\bb M\subseteq \bb L\otimes_{\hat{\bb Z}_p}\mu^{s}\bb A_{\inf,X}.\]
\end{enumerate}
\end{proposition}
\begin{proof}
Replacing $\bb M$ by $\tfrac1{\mu^r}\bb M$ we may suppose until part (iii) that $r=0$, i.e., that $\phi^{-1}$ restricts to $\phi^{-1}:\bb M\to \bb M$. We will prove the results by arguing locally. A base for $X_\sub{pro\'et}$ is provided by those affinoid perfectoids $\cal U$ with the following property: letting $A=\Gamma(\cal U,\hat\roi_X)$ denote the associated perfectoid Tate $C$-algebra, then any faithfully flat, finite \'etale $A$-algebra (hence similarly for $A^\flat$-algebras by the Almost Purity theorem) admits a section. See, e.g., \cite[Lem.~3.4]{Morrow_Vanishing} and the beginning of the proof of Lem.~3.5 of [op.~cit.]. By restricting further we may also suppose that $\bb M|_\cal U$ is finite free, whence $M:=\Gamma(\cal U,\bb M)$ is a finite free $W(A^{\flat+})$-module equipped with an inverse-Frobenius semi-linear map $\phi_M^{-1}:M\to M$ which induces an isomorphism $\phi_M^{-1}:M[\tfrac1{\tilde\xi}]\isoto M[\tfrac1\xi]$.

Since $A^\flat$ is a perfect $\bb F_p$-algebra with no non-trivial finite \'etale covers, usual Artin--Schreier--Witt theory implies that \[\phi_M-1:M\otimes_{W(A^{\flat+})}W(A^{\flat})\To M\otimes_{W(A^{\flat+})}W(A^{\flat})\] is surjective and its kernel $L=\Gamma(\cal U,\bb L)$ is a finite free $W(A^{\flat})^{\phi=1}$-module satisfying $L\otimes_{\bb Z_p}W(A^{\flat})=M\otimes_{W(A^{\flat+})}~W(A^{\flat})$. But $W(A^{\flat})^{\phi=1}=\Gamma(\cal U,\hat{\bb Z}_p)$ thanks to the pro-\'etale Artin--Schreier--Witt sequence $0\to\hat{\bb Z}_p\to W(\hat\roi_{X^\flat})\xto{\phi-1}W(\hat\roi_{X^\flat})\to 1$, so this proves (i).

To prove (ii) we will show that $1-\phi_M^{-1}:M\to M$ is surjective and that its kernel is precisely $L$. The previous paragraph shows that this is true after $-\otimes_{W(A^{\flat+})}~W(A^{\flat})$ (note that $\phi_M$ is invertible after this base change, so we can pass between $1-\phi_M^{-1}$ and $\phi_M-1$), so the goal is to show that the square
\[\xymatrix{
M\ar[r]^{1-\phi_M^{-1}}\ar@{^(->}[d]& M\ar@{^(->}[d]\\
 M\otimes_{W(A^{\flat+})}~W(A^{\flat})\ar[r]_-{1-\phi_M^{-1}}& M\otimes_{W(A^{\flat+})}~W(A^{\flat})
}\]
is bicartesian. All terms in the diagram are $p$-torsion-free and $p$-adically complete, so it is enough to check that the diagram is bicartesian modulo $p$. To simplify notation write $\res M:=M/p$, which is a finite free $A^{\flat+}$-module equipped with an inverse-Frobenius semi-linear map $\phi_{\res M}^{-1}:\res M\to \res M$ which becomes an isomorphism after $\otimes_{A^{\flat+}}A^\flat$. We must show that $1-\phi_{\res M}^{-1}:\res M\to\res  M$ is surjective and that the kernel of $1-\phi_{\res M}^{-1}:\res M\otimes_{A^{\flat+}}A^\flat\to \res M\otimes_{A^{\flat+}}A^\flat$ is contained in $\res M$. Since we already know that $1-\phi_{\res M}^{-1}$ is surjective on $\res M\otimes_{A^{\flat+}}A^\flat$, one checks easily that it is enough (to prove both claims) to show the following: given $n\in \res M\otimes_{A^{\flat+}}A^\flat$ such that $\phi_{\res M}^{-1}(n)-n\in\res  M$, then $n\in\res  M$. We may write $n=m/\ep$ for some $\ep\in \roi^\flat$ and $m\in M$, and calculate that $\res M\ni \phi_{\res M}^{-1}(m/\ep)-m/\ep=\phi_{\res M}^{-1}(m)/\ep^{1/p}-m/\ep$, whence $\ep^{1/p}n\in\res  M$. Continuing in this way shows that $n$ is almost in $\res M$, or more precisely defines an almost zero element of the $\roi^\flat$-module $(\res M\otimes_{A^{\flat+}}A^\flat)/\res M$. But this has no non-zero almost-zero elements since $\res M$ is a finite free $A^{\flat+}$-module and $A^+$ satisfies the equivalent conditions of Lemma \ref{lemma_no_non_zero}, thanks to Example \ref{example_R_infty_no_almost_zero}(iii). This completes the proof of (ii).

To prove part (iii) we argue in a standard way by using internal hom; to keep the bounds clear we drop the hypothesis that $r=0$. We consider $H:=\Hom_{W(A^{\flat+})}(M,L\otimes_{\bb Z_p}W(A^{\flat+}))$, which is a finite free $W(A^{\flat+})$-module equipped with a Frobenius semi-linear isomorphism $\phi_H:H[\tfrac1\xi]\isoto H[\tfrac1{\tilde\xi}]$, defined by $\phi_H(f):=(1\otimes\phi)\circ f\circ\phi_M^{-1}$, for $f\in H[\tfrac1\xi]=\Hom_{W(A^{\flat+})[\tfrac1\xi]}(M[\tfrac1\xi],L\otimes_{\bb Z_p}W(A^{\flat+})[\tfrac1\xi])$. We have seen that the canonical map $\iota:L\otimes_{\bb Z_p}W(A^{\flat+})\to M$ is invertible after base change to $W(A^\flat)$, thereby defining $\iota^{-1}\in H\otimes_{W(A^{\flat+})}W(A^\flat)$; but $\iota$ is $\phi$-equivariant, whence the same is true of its inverse $\iota$, i.e., $\iota^{-1}\in (H\otimes_{W(A^{\flat+})}W(A^\flat))^{\phi_H=1}$. Letting $s$ be as in the statement of (iii), it is clear that $\tilde\xi^{-s}H\subseteq\phi(H)$, so the result proved in part (ii) shows that $(H\otimes_{W(A^{\flat+})}W(A^\flat))^{\phi_H=1}\subseteq{\mu^{s}}H$; that means that $\tfrac1{\mu^s}\iota^{-1}$ is a morphism from $M$ to $L\otimes_{\bb Z_p}W(A^{\flat+})$, as required.
\end{proof}

The proposition implies in particular to any relative Breuil--Kisin--Fargues module $(\bb M,\phi_\bb M)\in\BKF(\frak X,\phi)$, yielding our desired locally free lisse $\bb Z_p$-sheaf \[\sigma^*_{\sub{\'et}}\bb M:=\bb L=(\bb M\otimes_{\bb A_{\inf,X}}W(\hat\roi_{X^\flat}))^{\phi_\bb M=1}=\bb M[\tfrac1\mu]^{\phi_\bb M^{-1}=1}\] This is the definition of our \'etale specialisation functor \[\sigma^*_\sub{\'et}:\BKF(\frak X,\phi)\To \categ{3cm}{locally free lisse $\bb Z_p$-sheaves on $X$}.\]

\begin{proof}[Proof of third equivalence of Theorem \ref{theorem_specialisations}]
The identification $\hat{\bb A_{\inf,X}[\tfrac1\mu]}=W(\hat\roi_X)$ shows that $\hat{A\Omega_\frak X(\bb M)[\tfrac1\mu]}=R\nu_*(\bb M\otimes_{\bb A_{\inf, X}}W(\hat\roi_{X^\flat}))$, whose Frobenius-fixed points are indeed $R\nu_*\bb L$ by Proposition \ref{proposition_etale}.
\end{proof}

\bibliographystyle{acm}
\bibliography{../Bibliography}

\noindent
\parbox{0.4\linewidth}{
Matthew Morrow\\
CNRS \& IMJ-PRG,\\
SU -- 4 place Jussieu,\\
Case 247,\\
75252 Paris\\
{\tt matthew.morrow@imj-prg.fr}
}
\parbox{0.5\linewidth}{
\noindent
Takeshi Tsuji \\
Graduate School of Mathematical Sciences,\\
The University of Tokyo,\\
3-8-1 Komaba,\\
Meguro-ku, Tokyo 153-8914\\
{\tt t-tsuji@ms.u-tokyo.ac.jp}
}

\end{document}